\documentclass[11pt]{article}
\usepackage[margin=1in]{geometry}
\usepackage{xcolor}
\usepackage{amsthm}

\usepackage{url}
\usepackage{cite}
\usepackage{natbib}
\usepackage{algorithmic}

\usepackage[utf8]{inputenc}
\usepackage{mathtools}
\usepackage{amssymb}
\usepackage[ruled]{algorithm2e} 
\usepackage{comment}
\usepackage{bm}
\usepackage{caption}
\usepackage{subcaption}
\usepackage{amsthm}

\newtheorem{theorem}{Theorem}
\newtheorem*{theorem*}{Theorem}
\newtheorem*{lemma*}{Lemma}
\newtheorem{definition}{Definition}

\newtheorem{lemma}{Lemma}

\newtheorem{assumption}{Assumption}

\newtheorem*{theorem-no-num}{Theorem}
\newtheorem*{lemma-no-num}{Lemma}

\newcommand{\bfN}{\mathbf{N}}
\newcommand{\bflambda}{\bm{\lambda}}
\newcommand{\bfM}{\mathbf{M}}
\newcommand{\bfR}{\mathbf{R}}
\newcommand{\bfS}{\mathbf{S}}

\newcommand{\bfD}{\mathbf{D}}
\newcommand{\bfF}{\mathbf{F}}
\newcommand{\bfd}{\mathbf{d}}
\newcommand{\bfe}{\mathbf{e}}

\newcommand{\bbR}{\mathbb{R}}
\newcommand{\calA}{\mathcal{A}}
\newcommand{\calW}{\mathcal{W}}

\newcommand{\calX}{\mathcal{X}}

\newcommand{\calQ}{\mathcal{Q}}
\newcommand{\calP}{\mathcal{P}}
\newcommand{\calZ}{\mathcal{Z}}
\newcommand{\calC}{\mathcal{C}}
\newcommand{\calV}{\mathcal{V}}
\newcommand{\calU}{\mathcal{U}}

\newcommand{\calI}{\mathcal{I}}
\newcommand{\calF}{\mathcal{F}}
\newcommand{\calM}{\mathcal{M}}
\newcommand{\calS}{\mathcal{S}}

\newcommand{\calD}{\mathcal{D}}

\newcommand{\bbN}{\mathbb{N}}
\newcommand{\bfone}{\mathbf{1}}
\newcommand{\bbE}{\mathbb{E}}

\newcommand{\bfx}{\bm{x}}

\newcommand{\bfrho}{\bm{\rho}}
\newcommand{\bfr}{\bm{r}}
\newcommand{\bff}{\bm{f}}
\newcommand{\bfh}{\bm{h}}
\newcommand{\bfg}{\bm{g}}
\newcommand{\bfH}{\bm{H}}
\newcommand{\bfG}{\bm{G}}

\newcommand{\bfgamma}{\bm{\gamma}}
\newcommand{\bfalpha}{\bm{\alpha}}

\newcommand{\bfeta}{\bm{\eta}}
\newcommand{\bbP}{\mathbb{P}}
\newcommand{\calL}{\mathcal{L}}

\newcommand{\calN}{\mathcal{N}}

\newcommand{\etcomment}[1]{\textcolor{blue}{\textbf{ET:} #1}}

\usepackage{color}              
\usepackage[suppress]{color-edits}
\definecolor{forestgreen}{rgb}{0.13, 0.55, 0.13}
\addauthor{et}{blue}
\addauthor{pf}{forestgreen}
\addauthor{mc}{red}
\newcommand{\Xomit}[1]{}

\title{Dynamic Pricing Provides Robust Equilibria in Stochastic Ride-Sharing Networks}

\author{J. Massey Cashore\thanks{Supported by an NSERC PGS D Fellowship} \and 
        Peter I. Frazier\thanks{Supported by AFOSR grant FA9550-19-1-0283} \and 
        \'Eva Tardos\thanks{Supported by AFOSR grant FA9550-19-1-0183 and NSF grants CCF-1408673 and CCF-1563714}}
\date{
Cornell University \\
\today
}
\begin{document}

    \maketitle
\begin{abstract}
        Ridesharing markets are complex: drivers are strategic, rider demand  and driver availability are stochastic, and complex city-scale phenomena like weather induce large scale correlation across space and time.  
At the same time, past work has focused on a subset of these challenges.
We propose a model of ridesharing networks with strategic drivers, spatiotemporal dynamics, and stochasticity. Supporting both computational tractability and better modeling flexibility than classical fluid limits,
we use a two-level stochastic model 
that 
allows correlated shocks caused by weather or large public events. 

Using this model, we propose a novel pricing mechanism: stochastic spatiotemporal pricing (SSP).
We show that the SSP mechanism is asymptotically incentive-compatible and that all (approximate) equilibria of the resulting game are asymptotically welfare-maximizing when the market is large enough.  
The SSP mechanism iteratively recomputes prices based on realized demand and supply, and in this sense prices dynamically. We show that this is critical: while a static variant of the SSP mechanism (whose prices vary with the market-level stochastic scenario but not individual rider and driver decisions)
has a sequence of asymptotically welfare-optimal approximate equilibria, we demonstrate that it also has other equilibria producing extremely low social welfare. Thus, we argue that dynamic pricing is important for ensuring robustness in stochastic ride-sharing networks.

    \end{abstract}


\section{Introduction}
Ridesharing markets are complex.  Drivers are strategic \citep{cradeur_2018,lu2018surge}, rider and driver decisions are stochastic, stochastic city-scale phenomenon like weather create correlation across space and time \citep{kamga2013hailing,chou2002testing}, and factors including irrationality \citep{sheldon2016income,camerer1997taxi}, learning \citep{cook2018gender}, and model error require robust off-equilibrium performance.

At the same time, pricing methods for ridesharing published in the academic literature are developed and analyzed considering only a portion of this complexity.  Work studying stochastic demand focuses on non-strategic drivers \citep{braverman2019, banerjee2017, besbes2019spatial,yan2020dynamic,alonso2017demand,ashlagi2018maximum,ozkan2020dynamic}
and/or ignores spatial aspects of ridesharing \citep{castillo2017surge,sharing_economy,garg2020driver}.  Work focusing on strategic drivers in more realistic spatial settings assume deterministic demand \citep{spatio_temporal:2018}, take a fluid limit in which demand becomes deterministic \citep{ride_hailing:2018,besbes:2018,bimpikis2016spatial}, or are descriptive rather than normative
\citep{lu2018surge}.

Our paper studies strategic driver behaviour in a model unifying many of the real-world complexities that were previously studied in isolation: stochastic rider demand, stochastic driver availability, strategic driver decisions, and network structure with spatial and temporal components. 
\pfdelete{In this setting, in} \pfedit{In} a \pfedit{novel stochastic} large-market limit \pfedit{more} appropriate for ridesharing applications \pfedit{than past deterministic fluid limits \citep{ride_hailing:2018}}, we develop a computationally tractable pricing and matching mechanism under which all approximate equilibria are asymptotically welfare-optimal. \pfdelete{Moreover, going significantly \pfedit{beyond models with deterministic fluid limits, correlated shocks affecting many riders and drivers keep} our model stochastic in the large-market setting.} 
We additionally show that ensuring {\it all} approximate equilibria are good is non-trivial: we demonstrate that persistent stochasticity in the large-market limit creates the need for a ``re-solving'' step when defining our mechanism.
Without this, we show that a different mechanism that does not re-solve and that is more closely linked to deterministic past work \citep{spatio_temporal:2018} has approximate equilibria that lose substantial welfare even in the large-market limit.

We model stochastic supply and demand with a two-level hierarchical distribution reminiscent of stochastic programming 
\citep{prekopa2013stochastic}
that simultaneously provides computational tractability and modeling flexibility. The top level corresponds to city-level variation in demand and supply patterns caused by weather, large public events, public transit outages and other random phenomenon that affect many riders and drivers simultaneously. These phenomena create random shocks that are correlated across many riders and drivers.  The lower level corresponds to fine-grained idiosyncratic randomness that affects individual riders and drivers independently.



Idiosyncratic randomness is tackled with a key analytical tool: a large-market setting in which the number of riders and drivers scale proportionally. This is especially relevant for the key ridesharing markets, those in large cities, where the number of riders and drivers is large. 
Through this approach, some randomness vanishes (conditionally independent rider and driver variation), supporting efficient computation of mechanisms based on this large-market limit. At the same time, randomness encoded by the stochastic scenario tree modeling city-level phenomena remains. 
This fact is an important difference between the large-market limit that we study and other deterministic large-market limits previously studied in ridesharing \citep{ride_hailing:2018}, and allows us to model important correlated city-wide shocks that pure fluid limits do not.


Building on the power of this novel modeling approach, we develop a novel pricing and matching algorithm for ridesharing markets, called the stochastic spatiotemporal pricing (SSP) mechanism. This recomputes pricing and matching decisions at each time period based on the observed scenario and driver distribution by solving a multistage stochastic program arising in the large-market limit. 
This can be solved as a convex program that remains tractable for up to several thousand scenarios.
Since the SSP responds to the stochastic scenario and real-time driver locations, it is a form of dynamic pricing.

The SSP mechanism is incentive compatible and all equilibria are welfare-maximizing when idiosyncratic randomness is absent. Further, all approximate equilibria are approximately welfare-maximizing when demand {\it is} idiosyncratically random and the market is large, in the sense that incentive compatibility violations and welfare-suboptimality of any equilibria vanish in this limit. 


Indeed, this is critical in practice because such perturbations are inevitable. We show that the repeated computations are essential for achieving robust market performance in the presence of such perturbations. This insight is derived from comparing the SSP mechanism to a static analog that uses static prices derived from only a single computation, depending only on the scenario. 
While this static mechanism has at least one approximate equilibrium that is asymptotically welfare-optimal, namely, resulting from following the platform's suggestions, we observe, using a simple example, that its performance is not robust: approximate equilibria of the static mechanism can have extremely low social welfare while those of the SSP mechanism cannot.

The key distinction between the SSP and the static mechanism is that the first responds to real-time conditions, while the second is only able to respond to variation in the top-level stochastic scenario. Indeed, in the example in \S\ref{sec:example}, there is only one scenario and so the static mechanism (which can depend on the scenario in general) corresponds to a fully deterministic pricing strategy of the kind studied by \citep{spatio_temporal:2018}.
Thus, the fact that SSP (which uses re-solving and thus changes prices dynamically based on market conditions) is guaranteed to result in equilibria with high social welfare while the static mechanism 
does not, argues that dynamic pricing is better able to provide robust equilibria and adapt to market fluctuations.
This is the first analysis, of which we are aware, to argue that dynamic pricing is needed for this purpose, and contributes to the larger literature on the purpose and value of dynamic pricing in ridesharing \citep{castillo2017surge,banerjee2016dynamic,lu2018surge,cachon2017role,hall2015effects,chen2016dynamic}.

In summary, the primary contributions of this paper are:
\begin{itemize}
\item In \S\ref{sec:model} we provide a novel two-level stochastic model of ride-sharing, that allows us to study a large-market limit retaining the persistent uncertainty of the \pfedit{macroscopic} \pfdelete{top level} stochasticity of our model. \pfedit{We refer to a model created by taking this limit as the {\it stochastic fluid} model and to the original model as the {\it two-level model}.} 
\item In \S\ref{sec:opt_central} we show that the optimal driver allocation \pfedit{in the stochastic fluid model} can be computed by a \pfedit{tractable} \pfdelete{stochastic} convex program. We will then use the dual variables of this convex program for the suggested pricing method for rideshare pricing, resulting in a practical dynamic pricing mechanism, SSP, \pfedit{defined in \S\ref{sec:sspm}}. 
\pfdelete{\item In \S\ref{sec:sspm} we show (Theorem \ref{thm:fluid_ic}) that the suggested mechanism is incentive compatible in the fluid limit, despite the spatiotemporally correlated stochasticity. In fact, it is also subgame perfect due to the repeated re-computations.}
\item \pfedit{In \S\ref{sec:sspm}, we study SSP's properties.  First, in the stochastic fluid model, we show (Theorem \ref{thm:robust_eqlbm_fluid}) that all equilibria resulting from the SSP mechanism are welfare optimal and all approximate equilibria are approximately welfare optimal.}
\item \pfedit{Then, in the two-level model, we show (Theorem~\ref{thm:robust_eqlbm_two_level}) that in this detailed and realistic model, every approximate equilibrium resulting from the SSP mechanism achieves approximately optimal welfare when the market is large enough.}
\item \pfedit{Finally, in \S\ref{sec:example} we show that dynamic pricing enabled by re-solving is a key component in enabling our robustness result Theorem~\ref{thm:robust_eqlbm_two_level}. We consider a variant of the SSP that does not re-solve, ignoring current driver locations. We  
demonstrate in a simple example how a risk premium leads to an equilibrium with significantly suboptimal welfare in large finite markets, despite 
existence of another equilibrium with optimal welfare in the stochastic fluid model. This suggests that being fully dynamic (depending both on observed supply/demand and city-level stochastic scenarios) may be important for achieving good practical performance in rideshare pricing.}
\end{itemize}

\subsection{Related Work}
The literature on ridesharing has recently considered drivers' strategic choice of location in 
\citep{spatio_temporal:2018,
ride_hailing:2018,besbes:2018,bimpikis2016spatial,lu2018surge}.

Papers \citep{bimpikis2016spatial,ride_hailing:2018,besbes:2018} all consider spatially heterogeneous driver-side pricing to incentivize driver relocation, but unlike the current paper consider a deterministic continuum approximation of the number of drivers in each node (or at points in a spatial continuum in \citep{besbes:2018}) and the flow of demand between nodes. In contrast, we focus on atomic drivers and study the effects of demand uncertainty.

\citep{lu2018surge} provides empirical evidence using data from Uber that surge pricing causes drivers to relocate toward locations with higher surge.  It argues that the causal effect of surge has three components: a direct effect on earnings per trip; a real-time signal that demand is higher than expected at the surging location; and a slower signal about average location-specific demand. In the current paper we focus on the first aspect.  The second and third aspects are absent from our model due to assumptions that the demand distribution, other drivers' strategies, and the platform's mechanism are all known.


The work within this ridesharing literature most closely related to the current paper is \citep{spatio_temporal:2018}.  This paper studies driver-side pricing in a multi-location multi-period model of a ridesharing market, and considers how to set prices over space (and time) to ensure that strategic drivers make welfare-optimal empty relocation decisions and accept all dispatches.  Like the current paper, these prices are based on solving an LP relaxation of the optimal planning problem to integrality.  Critically, and unlike the current paper, \citep{spatio_temporal:2018} assumes that the platform (and drivers) have complete information about future demand from riders.  This assumption causes the welfare-optimal actions to be incentivizable. We show that this result extends to the fluid limit variant of our problem, despite the high-level uncertainty in the model. However, stochasticity \pfedit{not encoded in a top-level scenario (including the idiosyncratic low-level stochasticity in our model),} can cause the equilibria of such a mechanism without dynamic pricing to have extremely low social welfare. \pfedit{Including a perfect representation of all stochasticity into a scenario tree is practically impossible from a modeling perspective and would lead to computational intractability.} Thus, we see the model in \citep{spatio_temporal:2018} as focusing on fully deterministic ridesharing markets, \pfedit{without providing a robust way to generalize to realistic settings}. \pfedit{In contrast, we see our approach as a step toward the computational tractability, modeling flexibility and robustness required in practice.}

Within the sharing economy more broadly, \citep{bike_angels:2018} considers the problem of designing incentives in the context of bike sharing rebalancing, in which strategic agents trade off the cost (in terms of time and effort) of moving a bike between stations against the location-specific platform-controlled reward.  The work in \citep{bike_angels:2018} pays special attention to practical considerations, and they demonstrate that the intelligent design of centrally controlled incentives can have meaningful real-world impact.

\pfedit{There is also work that considers drivers' strategic behavior regarding which trips to accept \citep{chen2019pricing,chen2020pricing,garg2020driver,castro2020matching} without considering the spatial aspects of ridesharing.}

\section{Model Description}
\label{sec:model}

\Xomit{We begin with a formal description of our stochastic ridesharing network model.
We present the model in three parts.

First, in Section \ref{sec:network_structure} we define the basic spatial and temporal 
network structure 
over which market interactions occur.  An important aspect of this network structure is the
inclusion of a stochastic scenario which is sampled at each time period, modeling an uncertain
and dynamic state of the world. 

Next, in Section \ref{sec:level_two_assns} we state assumptions governing market dynamics
in each time period, including assumptions governing stochastic 
rider demand and stochastic driver supply.
There are two distinct models defined in this section.  In one model, which we refer to as 
the stochastic fluid model,
drivers and riders are modeled as continuous fluids 
\mccomment{better wording here than ``continuous fluids''?}, and the only source of randomness in the 
model is the stochastic scenario sampled at each time period.
In the other model, which we refer to as the stochastic two-level model, 
drivers and riders are modeled as discrete individuals, and there is an additional layer of randomness
governing the volume of new supply and demand at each time period.

Finally, in Section \ref{sec:model_utility} we discuss our assumptions about 
pricing and matching mechanisms, strategic behaviour and driver utility, and the total welfare generated by the marketplace.}


\medskip
\noindent
\textbf{Stochastic Network Structure.} 
\label{sec:network_structure}
Drivers move across a set of locations $\calL \ne \emptyset$ in discrete time over $T$ periods. 

Macroscopic randomness is modeled by a Markov chain $(A_t : t)$ with discrete random variables $A_t$ that will make the behavior of riders and drivers random, even in the large-market limit. We refer to 
$\omega_t=A_{1:t}$ 
as the \emph{scenario} at time $t$ \etedit{(including the history)},
where $\omega_0=\emptyset$.
$\Omega_t$ denotes the set of possible 
values for $\omega_t$.
We write $\omega_t\sim\bbP(\cdot\mid\omega_{t-1})$ to indicate the conditional distribution of $\omega_t$.
\Xomit{Because $\omega_t$ has length $t$, the time can be recovered from the scenario.}
\Xomit{If scenario $\omega_s\in\Omega_s$ is realized in time $s$, 
then only a subset of scenarios at time $t>s$ remain possible. 
This set is }
\etedit{We use
$\Omega_t(\omega_s) = \{\omega_t\in\Omega_t : \omega_t^{1:s} = \omega_s\}$ to denote the set of scenarios possible at time $t$ given $\omega_s$ at time $s$},
where $\omega_t^{1:s}$ truncates $\omega_t$ to its first $s$ elements.
We 
refer to the set of all scenarios as the {\it scenario tree}. \etedit{Each scenario $w_t$ is a node in this tree, linked to all scienarios }
$\Omega_{t+1}(\omega_t)$. 

We assume there is a fixed cost $c_{(\ell,d,\omega_t)}$ to drive from a location $\ell$ to a destination $d$ at time $t$ under the scenario $\omega_t$.

\medskip
\noindent
\textbf{Driver and Rider Entry.}
We consider a large-market limit where the volume of drivers and riders scale with a \emph{population-size parameter} $k\in\mathbb{N}$.
The number of drivers entering
the market at location $\ell$, time $t$, and scenario $\omega_t$ under population size $k$  is a random variable $M^{(k)}_{(\ell,\omega_t)}$. The number
of riders interested in traveling from location $\ell$ to location $d$ is a random variable
$D^{(k)}_{(\ell,d,\omega_t)}$. 
We sometimes omit the superscript $k$ and write $M_{(\ell,\omega_t)}$ or $D_{(\ell,d,\omega_t)}$.
\Xomit{Not all of the $D^{(k)}_{(\ell,d,\omega_t)}$ riders request a dispatch.  Whether they do depends on the price set by the platform, discussed in Section \ref{subsec:pricing_policy}. }
\etedit{A subset of the possible riders in $D^{(k)}_{(\ell,d,\omega_t)}$ will request a ride, depending on the price set by the platform.}
We assume that these random variables satisfy good concentration properties as $k$ grows large and have expected values that scale linearly in $k$, as stated in Assumption~\ref{assn:entry_distributions}.

\begin{assumption}
\label{assn:entry_distributions}
Let $(X^{(k)})_{k=1}^\infty$ be the sequence of driver- or rider-entry random variables $(M^{(k)}_{(\ell,\omega_t)})_{k=1}^{\infty}$  
or $(D^{(k)}_{(\ell,d,\omega_t)})_{k=1}^\infty$, for any location $\ell, d, \omega_t$.
We assume:
\begin{enumerate}
\item The expected value of $X^{(k)}$ grows linearly in $k$, i.e. $\bbE[X^{(k)}] = k\bbE[X^{(1)}]$ for all $k\geq 1$.
\item There exists a sequence of error terms $(\epsilon_k:k\geq 1)$ and probability terms $(q_k : k\geq 1)$, both converging to $0$ as $k\to\infty$,
 such that:
\begin{equation}
    \bbP\left(\frac{1}{k}|X^{(k)} - \bbE[X^{(k)}]| \geq \epsilon_k\right) \leq q_k.
\end{equation}
\end{enumerate}
\end{assumption}

For example, if each $X^{(k)}$ follows a $\mathrm{Poisson}(\lambda k)$ distribution, then property 1 is immediate and
property 2 follows from concentration inequalities. \mccomment{cite \url{http://www.stat.yale.edu/~pollard/Books/Mini/Basic.pdf}?}

For simplicity we assume that drivers do not exit the market.
We also assume that it takes a single time period to driver from any origin location to destination location.
Both assumptions are without loss of generality: we can add locations that model
being on-trip or leaving the market.

\medskip
\noindent
\textbf{Market State.}
\label{subsec:market_state}
Let $\calM_t$ be the set 
of drivers who have entered the marketplace by time $t$ 
and let $\ell_i\in\calL$ be the current location of \etedit{driver $i$} 
(momentarily suppressing the dependence on $t$ in the notation).  Let $S_\ell$ denote $1/k$ times the number of drivers at position $\ell$, that is  
\Xomit{Let 
\begin{equation}
\label{eq:supply_volume}
S_\ell = \frac{1}{k}\sum_{i\in\calM_t}\bfone\left\{\ell_i = \ell\right\},
\end{equation}
be} 
the volume of active drivers who are positioned at $\ell$,  
normalized by $k$.

Then, $\bfS_t = (S_\ell : \ell\in\calL)$ represents the supply volume across all locations.  
We refer to $\bfS_t$ as the supply-location vector.
We define our state variable at the beginning of each time period $t$ to be the tuple $(\omega_t,\bfS_t)$,
meaning that the platform pricing and matching decisions, as well as strategic driver decisions about where
to drive, depend only on the realized scenario $\omega_t$ and the spatial distribution of drivers specified by $\bfS_t$.
Section \mccomment{ref} shows that this is (essentially) without loss of generality, assuming that drivers are
expected utility maximizers.

The notation $\calS=\bbR^\calL_+$ indicates the set of vectors indexed by locations $\calL$ with nonnegative
components.  We think of $\Omega_t\times\calS$ as the state space at time $t$. \etcomment{do we need these?}

\medskip
\noindent
\textbf{Pricing Policy and Rider Dispatch Requests.}
\label{subsec:pricing_policy}
The platform sets \pfedit{prices at time $t$ using function $P_t$.}\pfdelete{policy is a sequence of functions $(P_1,P_2,\dots,P_T)$
such that the function $P_t$ is used to set prices at time $t$.}  \Xomit{We assume each $P_t$ is a 
function from the state space $\Omega_t\times\calS$ to a price vector in $\bbR^{\calL^2}_+$,
such that each component of the price vector specifies the trip price for each origin and destination location.
}
We 
write $P_{(\ell,d,\omega_t)}(\bfS_t)$ to mean the price \pfdelete{that the function $P_t$ specifies} for
trips from $\ell$ to $d$ when
$(\omega_t,\bfS_t)$ is the market state. More simply, we write $P_{(\ell,d)}$ when the pricing function
and associated market state is clear. 

Prices \pfdelete{have the effect of filtering} \pfedit{filter} demand for a trip.  For each pair of locations $(\ell,d)\in\calL^2$ and each
scenario $\omega_t$, every rider who is interested in a trip from $\ell$ to $d$ holds a private value $V$ for the trip. 
We assume the private values $V_j$ for all $j=1,2,\dots,D_{(\ell,d,\omega_t)}$ riders are independent, 
identically distributed random variables.  Let $F_{(\ell,d,\omega_t)}$ be the associated \pfedit{cumulative} distribution function.
We assume $F_{(\ell,d,\omega_t)}$ satisfies \pfdelete{the conditions described in} Assumption \ref{assn:rider_value_dist}.

\begin{assumption}
    \label{assn:rider_value_dist}
    \etcomment{I am confused, do you mean $F(V_{max})=1$?}
    \pfdelete{Let $F=F_{(\ell,d,\omega_t)}$ be the cumulative distribution function for the rider value distribution for the pair of locations $(\ell,d)\in\calL^2$
    and time-scenario $\omega_t\in\Omega_t$.}
    We assume:
    \begin{enumerate}
        \item There exists an upper bound $V_{max}$ \pfedit{on the valuation $V$} such that $F(V_{max})=\pfedit{1}$.
        \item $F:[0,V_{max}]\to [0,1]$ is continuous and invertible, and the inverse is Lipschitz continuous.
        \item The inverse function at $0$ satisfies $F^{-1}(0)=0$.
    \end{enumerate}
\end{assumption}

We assume that riders only request a dispatch if their value $V_j$ exceeds the price $P_{(\ell,d)}$.
Let $R_{(\ell,d,\omega_t)}$ be a random variable counting the number of riders who request a trip from $\ell$ to $d$ under
the price $P_{(\ell,d)}$.
\pfdelete{Notice that} $R_{(\ell,d,\omega_t)}$ follows a Binomial distribution with $D_{(\ell,d,\omega_t)}$ trials and $1-F_{(\ell,d,\omega_t)}(P_{(\ell,d)})$
success probability.
We \pfdelete{will often} write $R_{(\ell,d)}$ for the number of dispatch requests when the scenario $\omega_t$ \pfedit{is clear.}\pfdelete{does not need to be emphasized.}

\medskip
\noindent
\textbf{Matching Process, Driver Strategies and Add-Passenger Disutility.}
\label{subsec:driver_strategies}
After prices filter demand and dispatch requests are realized for each route, the platform
operates a matching process that \pfdelete{is responsible for} allocat\pfedit{es} dispatch requests to available drivers.
When allocated a dispatch, a driver may accept or decline. \pfedit{If the driver accepts, they drive the passenger to the dispatch's destination and collect}
\pfdelete{Accepting entails driving the
passenger to their destination and collecting} the associated payment. \pfedit{If the driver declines,}  \pfdelete{driver also have the option to decline the dispatch, and} \pfedit{they do not collect a payment and may optionally} drive to any destination they choose.
\pfdelete{without collecting a payment.}

Drivers incur an idiosyncratic \emph{add-passenger disutility},
which models the cost of adding a passenger to their car.  In each period, every driver samples their add-passenger disutility independently
from a $\mathrm{Uniform}(0,C)$ distribution for some constant $C>0$.
A driver's decision about whether to accept or
reject a dispatch \pfdelete{they've been allocated} depends also 
on their add-passenger disutility\pfdelete{value}.

\pfcomment{We allow this threshold to depend on the driver, right? We should says this.}
We model driver decisions by assuming every driver \pfdelete{positioned} at a location $\ell$ specifies a threshold $x_{(\ell,d)}\in [0,C]$ for each possible dispatch
destination $d\in\calL$. \pfedit{These threshold may depend on the driver, though we suppress this in the notation.} When a driver is allocated a dispatch towards location $d$, we assume they only accept the dispatch if their add-passenger
disutility value $X$ is smaller than $x_{(\ell,d)}$. Each driver's strategy also specifies a destination $e\in\calL$, to which
they will drive empty if they do not accept a dispatch.

For each location $\ell$ and time $t$, let $\calM_{\ell,t}$ be the \etdelete{index} set of drivers \pfdelete{positioned} at $\ell$ at time $t$.
For each driver $i\in\calM_{\ell,t}$, we write $\bfx_i = (x^i_{(\ell,d)} : d\in\calL)$ to mean the vector of disutility-acceptance
thresholds selected by driver $i$, and $e_i\in\calL$ is the relocation destination selected by driver $i$.
\etdelete{We assume that each driver selects the tuple $(\bfx_i,e_i)$ before the matching process begins and before their add-passenger disutilities
are sampled.
We write $\calX=[0,C]^\calL$ to mean the space of feasible disutility thresholds $\bfx$.}\etcomment{I think these can be deleted}

We assume that dispatches can only be served by drivers \pfdelete{who are already} positioned at the \pfedit{dispatch's} origin\pfdelete{location, so the matching process operates separately for each location}.
We also assume that each driver can be allocated at most one dispatch, that is, drivers who decline a dispatch will not get allocated an alternate dispatch in the same period.

We formalize the matching process as a function $\mathrm{MP}$ taking these arguments:
\begin{itemize}
\item A location $\ell\in\calL$, and the market state $(\omega_t,\bfS_t)$.
\item For each driver $i\in\calM_{\ell,t}$, the disutility thresholds $\bfx_i$ and the relocation destionation $e_i$.
\item For each driver $i\in\calM_{\ell,t}$, the add-passenger disutility value $X_i$.
\item For each destination $d$, the number of dispatch requests $R_{(\ell,d)}$.
\item An external source of randomness $U$ sampled independently and uniformly from $[0,1]$.
\end{itemize}

The matching process produces a set of tuples $\{(i,b_i) : i\in\calM_{\ell,t}\}$. Here, \pfcomment{I changed this from $d_i$ to $b_i$ --- could someone confirm?} $b_i\in \calL\cup\{\emptyset\}$
is the destination of the dispatch allocated to driver $i$ (where $b_i=\emptyset$ if driver $i$ was not allocated a dispatch).

\pfdelete{
The action that a driver takes in a period is described by a tuple $(\ell,d,\delta)$, where $\ell,d\in\calL$ specify origin and
destination locations of the trip taken, and $\delta\in\{0,1\}$ indicates whether the trip
serves a dispatch ($\delta=1$) or is an empty relocation trip ($\delta=0$).
}

For each driver $i\in\calM_{\ell,t}$, we use $a_i^t=(\ell,d_i,\delta_i)$ to indicate their action at time $t$.
The driver's starting location is $\ell$.  The destination location $d_i$ depends on the output of the matching
process and the driver's strategy $(\bfx_i,e_i)$.  If the add-passenger disutility $X_i$ is smaller than the threshold $x_{(\ell,b_i)}$, then
the driver accepts the dispatch and $d_i=b_i$ and $\delta_i=1$.  If not, then the driver declines the dispatch and instead drives towards $d_i=e_i$
with $\delta_i=0$.

We thus think of this entire matching process as taking the inputs to the function $\mathrm{MP}$ specified above and producing the drivers' actions
$(a_i^t : i\in\calM_{\ell,t})$:
\pfdelete{Because the actions that a driver takes are fully determined by the matching process, it will be useful to think of the matching process
as a random function that produces an action for each driver. That is, we can think of the actions $(a_i^t : i\in\calM_{\ell,t})$ as the
output of a randomized function $\mathrm{MP}$ \pfcomment{need to move the introduction of this with the definition of MP above --- perhaps can cut a little text.}:
}
$$
(a_i^t : i\in\calM_{\ell,t}) \sim \mathrm{MP}(\ell,\{(\bfx_i,e_i):i\in\calM_{\ell,t}\}, \{R_{(\ell,d)} : d\in\calL\}).
$$

We assume the platform randomizes its selection of drivers (in particular, preferential dispatch is not an available
lever to align incentives).  We also assume the actions produced by the matching process are always feasible
with respect to demand, i.e. the number of drivers serving a dispatch from $\ell$ to $d$ never
exceeds the number of dispatch requests $R_{(\ell,d)}$.

\medskip
\noindent
\textbf{Strategy Profiles and Approximate Equilibria.}
\label{subsec:strategy_profiles}
We now define a strategy profile.  Let $\calM$ be the set of indices for all drivers who could potentially join the market.  \pfdelete{Because the number of drivers who enter the market at each time period is finite, $\calM$ is countably infinite.} \pfcomment{Deleted a sentence here about the cardinality of $\calM$ being countably infinite.  The cardinality is a number of drivers, which is an integer, so that is automatic.}
For each driver $i\in\calM$, a strategy is a sequence of functions $\pi_i =(\pi_i^1,\pi_i^2,\dots,\pi_i^T)$ where each $\pi_i^t$
determines the driver's action at time $t$.
This function 
$\pi_i^t:\calL\times\calS_t \to \calX \times \calL$ takes as input a location $\ell$
and a time $t$ state $(\omega_t,\bfS_t)\in\calS_t$, and produces a disutility threshold vector $\bfx_i\in\calX$ and a relocation
destination $e_i\in\calL$.

\pfdelete{
One might consider actions that depend    on the history of observed market states up to the beginning of time $t$.
For simplicity we have instead assumed that all drivers use Markovian strategies that only depend on the state at
the beginning of the time period. This assumption is without loss of generality in the context of exact equilibria because we assume that
drivers are expected-utility maximizers.   In the current context where we consider approximate equilibria this assumption is almost without loss
of generality. \mccomment{This description about Markovian strategies being WLOG can be clarified / improved.}\pfcomment{not sure what almost wlog means. We could also delete this paragraph to save space.}
}

We define a strategy profile as $\Pi = (\pi_i:i\in\calM)$.
We model strategic driver behaviour by assuming that drivers select strategies resulting in an approximate equilibrium strategy profile $\Pi$, as defined below.
Fixing the strategy profile, 
(and also the pricing and matching policies), the market dynamics
are simply a stochastic process. 
Whether a strategy profile $\Pi$ is an approximate equilibrium (with respect to a pricing and matching policy) depends
on this stochastic process.

Fix a strategy profile $\Pi$ and assume that pricing and matching policies are fixed.  
Let $(a_i^t:t\in[T])$ be the sequence of actions taken by each driver $i\in\calM$.  Set $a_i^t=\emptyset$ for any driver $i$ who hasn't entered
the market by time period $t$; otherwise $a_i^t$ is selected based on the matching process for time period $t$.

Let $R_i^t$ be the reward collected by driver $i$ in time period $t$.  $R_i^t$ is determined from the action $a_i^t$ based on the
following relationship:
\begin{equation}
    R_i^t=\begin{cases}
    P_{(\ell,d)}^t - X_i^t - c_{(\ell,d)} & \mbox{if }a_i^t = (\ell,d,1),\\
        -c_{(\ell,d)} & \mbox{if }a_i^t = (\ell,d,0),\\
        0 & \mbox{if }a_i^t = \emptyset,
    \end{cases}
\end{equation}
where $P_{(\ell,d)}^t$ is the price set by the SSP mechanism for a dispatch from $\ell$ to $d$ in time period $t$, and $X_i^t$
is driver $i$'s add-passenger disutility from time period $t$.
The utility-to-go $U^t_i$ for driver $i$ at time period $t$ is the sum of rewards they collect over all future
time periods,
\begin{equation}
    U_i^t = \sum_{\tau=t}^TR_i^\tau.
\end{equation}
$R_i^t$ and $U_i^t$ are both random variables.

Our definition of incentive compatibility for $\Pi$ will use of a driver's expected utility to go, conditioned on their 
dispatch destination and their add-passenger disutility.
Define $\calU_i^t(\ell,b,X,(\omega_t,\bfS_t);\Pi)$ to mean the expected utility-to-go collected by driver $i$, conditioned
on being located at $\ell$, on being allocated a dispatch towards $b\in \calL\cup\{\emptyset\}$, on seeing add-passenger 
disutility $X\in [0,C]$, on the market state $(\omega_t,\bfS_t)$, and the strategy profile $\Pi$:
$$
\calU_i^t(\ell,b,X,(\omega_t,\bfS_t);\Pi) = \bbE^\Pi\left[U_i^t \mid \ell_i^t = \ell, b_i^t = b, X_i^t = X, (\omega_t,\bfS_t)\right].
$$
When the location and the state do not need to be emphasized, we write $\calU_i^t(b,X;\Pi)$ for clarity.

Informally, a strategy profile $\Pi$ is an approximate equilibrium if (almost) every driver has small incentive to deviate, regardless
of the destination they are dispatched towards (if any). We use the 
notation $\Pi-\pi_i+\pi_i'$ to represent that driver $i$ deviates from a strategy $\pi_i$ to an alternate strategy $\pi_i'$.
The incentive-to-deviate for a driver, given a particular time period and market state, is defined as the maximum utility 
gain the driver can achieve by switching to an alternate strategy. 
Recall we use the set $\calM_t$ to index the active drivers in the marketplace at time $t$.
We use the notation $\calM_t(\epsilon; (\omega_t,\bfS_t))$ to mean the set of active drivers whose conditional incentive-to-deviate is at most $\epsilon$,
given the market state $(\omega_t,\bfS_t)$:
\begin{equation}
    \label{eq:delta_deviation_incentive}
    \calM_t(\epsilon;(\omega_t,\bfS_t)) = \left\{i\in\calM_t : \sup_{b\in\calL\cup\{\emptyset\}}\sup_{X\in[0,C]}\sup_{\pi_i'} \calU_i^t(b,X;\Pi-\pi_i+\pi_i') - \calU_i^t(b,X;\Pi) \leq \epsilon\right\}.
\end{equation}
$\calM_t(\epsilon)$ implicitly depends on the market state $(\omega_t,\bfS_t)$, but we omit this dependence in the notation
for clarity.
The formal definition $\Pi$ must meet to be considered an approximate equilibrium is stated below.

\begin{definition} \label{def:approx_eqlbm} 
Consider the two-level model with population size $k$. Hold the platform and matching policies fixed and let 
$\Pi$ be a strategy-profile. For a constant $\epsilon > 0$, we say that $\Pi$ is an $(\epsilon,\delta)$-approximate equilibrium
if the number of drivers who have at least $\epsilon$-conditional incentive to deviate from any market state is smaller than $\delta k$, i.e., if
    $$|\calM_t\setminus\calM_t(\epsilon;\omega_t,\bfS_t)|\leq \delta k$$
    for every $t$ and every market state $(\omega_t,\bfS_t)\in\Omega_t\times\calS_t(\gamma_t)$.
\pfcomment{I've edited to here.}
    \mccomment{I updated the definition to let a small fraction of drivers take an arbitrarily bad action.  But now I realize the ``bound holds for all market states''
    qualifier might not be quite what we want.}
\end{definition} 
In the above definition, we use the notation $\calS_t(\gamma_t)$ to mean the set of supply-location vectors $\bfS=(S_\ell \geq 0: \ell\in\calL)$
such that the total volume of drivers in the network is smaller than $\gamma_t$, i.e. $\sum_\ell S_\ell\leq \gamma_t$.
We pick $\gamma_t$ to be a large constant so that $\bfS_t\in\calS_t(\gamma_t)$ occurs with high-probability.  Note that the choice of $\gamma_t$
need only depend on the distributions governing the number of drivers who enter the marketplace, which satisfy the concentration assumptions
described in Assumption \ref{assn:entry_distributions}.  In particular, the total volume of drivers in the marketplace at time $t$ is invariant
to the strategy that the drivers use (since we assume without loss of generality that drivers never exit the marketplace).

\Xomit{
for the two-level model, let $p_1,\dots,p_T$ be a sequence of probability terms,
and let $\epsilon_1,\dots,\epsilon_T$ be a sequence of error terms.  We define $\Pi$ to be a $(p_t,\epsilon_t)$-approximate equilibrium
at time period $t$ if the following conditions hold:
\begin{enumerate}
    \item There is a subset of the state space $\calS_t$ such that, under $\Pi$, the probability the time $t$ market state lies in $\calS_t$
            is at least $p_t$:
\begin{equation}
\bbP^\Pi\left((\omega_t,\bfS_t)\in\calS_t\right)\geq p_t.
\end{equation}
    \item In the event that the time $t$ market state does lie in the set $\calS_t$, the incentive-to-deviate from $\Pi$ is at most $\epsilon_t$
for every driver in the marketplace:
\begin{equation}
(\omega_t,\bfS_t)\in\calS_t \implies \calM_t(\epsilon_t) = \calM_t.
\end{equation}
\end{enumerate}
We define $\Pi$ to be a $(p_{1:T},\epsilon_{1:T})$-approximate equilibrium if it is a $(p_t,\epsilon_t)$-approximate equilibrium for each time period $t$.
\end{definition}
}

\medskip
\noindent
\textbf{The Stochastic Fluid Model}
\label{subsec:fluid_model}
The fluid model is our tractable approximation to the main model we have been describing so far.
We can think of the sources of randomness in our main model as falling under two buckets:
the stochastic scenario tree governs \emph{macroscopic} randomness, and the remaining
sources of randomness (driver entry, rider entry and dispatch requests) form \emph{microscopic} randomness.
With this designation in mind, we refer to our main model as the stochastic two-level model.
The stochastic two-level model retains the top-level macroscopic randomness, but it approximates
the microscopic randomness with deterministic dynamics.

The volume of riders and drivers who enter the market is deterministic in the stochastic fluid model, conditional
on the realized scenario. For a time period $t$ and scenario $\omega_t$, let
$$
\bar{M}_{(\ell,\omega_t)} = \bbE[M^{(1)}_{(\ell,\omega_t)}]\ \ \ \mbox{ and } \ \ \ 
\bar{D}_{(\ell,d,\omega_t)} = \bbE[D^{(1)}_{(\ell,d,\omega_t)}]
$$
be the volume of drivers who join the market at location $\ell$, and
the volume of riders who enter the market interested in a dispatch from $\ell$ to $d$, respectively.
While these definitions are stated in terms of the expected number of drivers and riders who
join the market under a population-size parameter $k$ equal to $1$, this is equivalent to taking
the expected number of drivers and riders who join under any population-size $k$ and dividing that
by $k$, under the first condition stated in Assumption \ref{assn:entry_distributions}.
Also define 
$$\bar{r}_{(\ell,d,\omega_t)}(P)=\bar{D}_{(\ell,d,\omega_t)}\left(1-F_{(\ell,d,\omega_t)}(P)\right)$$ 
to be the fluid volume of riders who request a trip, as a function of the trip price $P$.

Our formalization of a pricing policy is the same in the stochastic fluid model as it was in the
two-level model: a pricing policy is a sequence of functions $(P_1,P_2,\dots,P_T)$ where each
$P_t$ is a function that sees the market state $(\omega_t,\bfS_t)$ and produces a vector of trip prices 
for each route $(P_{(\ell,d)}\geq 0 : (\ell,d)\in\calL^2)$.

We model strategy profiles in the fluid model as a sequence of functions $\Sigma = (\Sigma_1,\dots,\Sigma_T)$.
In the stochastic fluid model we assume all drivers at the same location use the same disutility threshold
for accepting or rejecting a dispatch decision.  We also assume that driver relocation decisions are determined
by a collective relocation distribution for all drivers at the same location.  Each function $\Sigma_t$
takes as input a time $t$ state $(\omega_t,\bfS_t)$ and
produces a disutility threshold vector $\bfx_\ell\in\calX$ and a relocation distribution $\bfe_\ell\in\Delta(\calL)$,
for each location $\ell$.

In the stochastic fluid model, the matching process is defined as a function $\bar{\mathrm{MP}}$ that takes as 
input a location $\ell$, a market state $(\omega_t,\bfS_t)$, and a disutility threshold vector $\bfx_\ell$,
and produces a dispatch vector $\bfg_\ell = (g_{(\ell,d)} : d\in\calL)$ specifying the volume of drivers
$g_{(\ell,d)}$ who serve a dispatch from $\ell$ to $d$.
If the drivers at $\ell$ use a disutility acceptance threshold $x=x_{(\ell,d)}$ for destination $d$, 
and the matching process specifies $g=g_{(\ell,d)}$ dispatches towards $d$, we assume the fluid matching
process has to allocated dispatches towards $d$ to a pool of drivers with volume
\begin{equation}
\label{eq:fluid_matching_dispatch_volume_fn}
Z(g,x) = \begin{cases}
g\frac{C}{x} & \mbox{if }x > 0,\\
0 & \mbox{else}.
\end{cases}
\end{equation}
For the dispatch volumes to be feasible, each $g_{(\ell,d)}$ cannot exceed the dispatch request
volume $\bar{r}_{(\ell,d)}(P_{(\ell,d)})$, and the implied pool sizes cannot exceed the available driver supply:
at each location $\ell$,
\begin{equation}
\label{eq:fluid_matching_capacity_constraint}
\sum_{d\in\calL}Z(g_{(\ell,d)}, x_{(\ell,d)}) \leq S_\ell.
\end{equation}

The volume of relocation trips along each route is then determined by the relocation distribution $\bfe_\ell$ and 
the remaining volume of undispatched drivers at $\ell$.  The volume of drivers who serve a relocation trip from $\ell$
to $d$ is defined as 
$$
h_{(\ell,d)} = e_{(\ell,d)}\left(S_\ell - \sum_{d\in\calL}g_{(\ell,d)}\right).
$$
Because the relocation volumes are deterministic functions of the dispatch volumes, we can think
of the matching process as a deterministic function that takes as input the market state $(\omega_t,\bfS_t)$, 
the disutility thresholds $\bfx_\ell\in\calX^\calL$ and the relocation distributions $\bfe_\ell\in\Delta(\calL)$
used at each location $\ell$, and producing trip volumes 
$(\bfh,\bfg) = \bar{\mathrm{MP}}(\omega_t,\bfS_t,(\bfx_\ell,\bfe_\ell:\ell\in\calL))$, where $\bfh=(h_{(\ell,d)} : (\ell,d)\in\calL^2)$
specifies the volume of relocation trips along each route, and $\bfg=(g_{(\ell,d)} : (\ell,d)\in\calL^2)$ specifies the volume
of dispatch trips along each route.

We will frequently use the parameterization $\bff = \bfh + \bfg$ to mean the vector of total trip volumes across each route.
In the stochastic fluid model, note that the supply location vector for the next time period is a deterministic function of the
scenario in the next time period and the total trip volumes in the current time period.  If the time $t$ scenario is $\omega_t$ and
the time $t$ total trip volumes are given by $\bff$, if the time $t+1$ scenario is $\omega_{t+1}$ then let $\bar{\bfS}_{\omega_{t+1}}(\bff)$
specify the the time $t+1$ supply location vector, where 
\begin{equation}
\label{eq:fluid_state_transition}
(\bar{\bfS}_{\omega_{t+1}}(\bff))_\ell = \bar{M}_{(\ell,\omega_{t+1})} + \sum_{o\in\calL} f_{(o,\ell)}.
\end{equation}

\medskip
\noindent
\textbf{Equilibrium Strategies in the Stochastic Fluid Model.}
\label{subsec:fluid_equilibria}
In order to define what it means for a fluid strategy profile $\Sigma=(\Sigma_1,\dots,\Sigma_T)$ to be an equilibrium strategy profile,
we associated with $\Sigma$ a value function for each location and a Q-value for each trip type.
The value function specifies the expected utility-to-go for a driver given their location, and the Q-value
specifies the expected utility-to-go for a driver given their action.

Assume the pricing and matching policies are fixed.  We define the Q-values and the value function associated with
$\Sigma$ recursively. For the base case, let $\calV_{T+1}(\cdot)$ always equal $0$.
Let $a=(\ell,d,\delta)\in\calL^2\times\{0,1\}$ denote any action, let $X\in [0,C]$ be any add-passenger disutility value,
let $t\in [T]$ be any time period, and let $(\omega_t,\bfS_t)\in\Omega_t\times\calS$ be any market state.
Assume the value function for the next time period $\calV_{t+1}:\calL\times\Omega_t\times\calS\to\bbR_+$ has already been defined.
Let $\bfh,\bfg$ be the relocation-trip and dispatch-trip volumes produced by the matching process and the time $t$ strategy $\Sigma_t$,
and let $\bff=\bfh+\bfg$ be the total-trip volumes.

Define the Q-value for time $t$ as:
$$
\calQ_t(a,X,\omega_t,\bfS_t) = \delta(P_{(\ell,d)} - X) - c_{(\ell,d)} + \bbE\left[\calV_{t+1}(d,\omega_{t+1},\bar{\bfS}_{\omega_{t+1}}(\bff))\right],
$$
and define the value of location $\ell$ at time $t$ as:
$$
    \calV_t(\ell,\omega_t,\bfS_t) = \frac{1}{S_\ell}\sum_{d\in\calL} \left[h_{(\ell,d)} \calQ_t(\ell,d,0)  + 
                                                        g_{(\ell,d)}\calQ_t\left(\ell,d,1, \frac{x_{(\ell,d)}}{2}\right)\right] .
$$
In the above, we write
$$
\calQ_t(\ell,d,0) = \calQ_t((\ell,d,0), X, \omega_t,\bfS_t)
$$
to mean the expected utility-to-go of  drivers who take a relocation trip from $\ell$ to $d$,
and we write
$$
\calQ_t\left(\ell,d,1, \frac{x_{(\ell,d)}}{2}\right) = \calQ_t\left((\ell,d,0), \frac{x_{(\ell,d)}}{2}, \omega_t,\bfS_t\right)
$$
to mean the expected utility-to-go of drivers who serve a dispatch from $\ell$ to $d$.  Note that $\frac{x_{(\ell,d)}}{2}$ is the
average pickup disutility incurred by drivers who serve a dispatch from $\ell$ to $d$.
In this definition of $\calV_t$, we are implicitly relying on the equality
$$
\calQ_t\left(\ell,d,1, \frac{x_{(\ell,d)}}{2}\right) = \frac{1}{x_{(\ell,d)}} \int_0^{x_{(\ell,d)}} \bar{\calQ}_t\left(\ell,d,1, x\right) dx.
$$

In order for a fluid strategy $\Sigma$ to be an equilibrium with respect to a fixed pricing policy and matching process, we
require that no driver can have a profitable deviation from their specified action, from any market state.
The property that no driver can have a profitable deviation can be broken down further into the following subproperties:
\begin{itemize}
    \item A driver who takes a relocation trip cannot have the ability to take a relocation trip toward a different destination yielding higher utility.
    \item A driver who accepts a dispatch cannot have the ability to take a relocation trip yielding higher utility.
    \item A driver who declines a dispatch cannot have been able to achieve higher utility by accepting the dispatch.
\end{itemize}

We summarize these properties with the following equations:
\begin{align}
    h_{(\ell,d)} > 0& \implies \calQ_t(\ell,d,0) = \max_{d'\in\calL}\calQ_t(\ell,d',0),\nonumber\\
            g_{(\ell,d)} > 0& \implies \calQ_t(\ell,d,1,X) \geq \max_{d'\in\calL}\calQ_t(\ell,d',0) \ \ \ \  \forall X\in [0,x_{(\ell,d)}],  \label{eq:exact_fluid_ic} \\
            \bar{r}_{(\ell,d)} > 0, \ x_{(\ell,d)} < C &\implies \calQ_t(\ell,d,1,X) \leq \max_{d'\in\calL}\calQ_t(\ell,d',0)\ \ \ \  \forall X \in (x_{(\ell,d)}, C] \nonumber .
\end{align}
We also provide approximate incentive comptability conditions, which depend on an error term $\epsilon > 0$:
\begin{align}
    h_{(\ell,d)} > \epsilon& \implies \calQ_t(\ell,d,0) \geq \max_{d'\in\calL}\calQ_t(\ell,d',0) - \epsilon,\nonumber\\
            g_{(\ell,d)} > \epsilon& \implies \calQ_t(\ell,d,1,X) \geq \max_{d'\in\calL}\calQ_t(\ell,d',0) - \epsilon \ \ \ \  \forall X\in [0,x_{(\ell,d)}],  \label{eq:approx_fluid_ic} \\
            \bar{r}_{(\ell,d)} > \epsilon, \ x_{(\ell,d)} < C-\epsilon &\implies \calQ_t(\ell,d,1,X) \leq \max_{d'\in\calL}\calQ_t(\ell,d',0) + \epsilon\ \ \ \  \forall X \in (x_{(\ell,d)}, C] \nonumber .
\end{align}

\begin{definition}
\label{def:fluid_eqlbm}
    We say a fluid strategy profile $\Sigma$ is an exact equilibrium, with respect to a fixed pricing and matching policy, 
    if the incentive compatibility conditions (\ref{eq:exact_fluid_ic}) are satisfied from every market state.
    $\Sigma$ is an $\epsilon$-approximate equilibrium if the approximate incentive compatibility conditions (\ref{eq:approx_fluid_ic}) are
    satisfied from any market state.
\end{definition}





\Xomit{
\section{Example Network}
In this section we describe an example stochastic ridesharing network
satisfying the modeling assumptions outlined in the previous section. \etedit{At the high level consider the static version of the pricing method. While this is approximately incentive compatible, simulations suggest that it can actually perform quite poorly in markets of finite size, even when that size is large. The reason appears to me that despite asymptotic incentive compatibility, finite markets retain a small but consistent incentive for drivers to deviate. In an infinitely-large market, exiting the market and continuing to drive provide equal reward, but in the finite-market increased stochasticity harms the reward of driving compared to the deterministic payoff of market exit.  While the incentive to deviate is small, the harm to welfare of the resulting deviation is large.}

\etedit{We highlight this issue with focusing }
on a simple example that can
help make concrete the assumptions and dynamics about market operations
from the previous section.

We consider a simple network with $T=2$ time periods and two locations:
$A$ and $B$.  We assume that the top-level stochastic scenario in this
network indicates the presence or absence of rain, and we also
assume that it is also raining in the first time period. We encode this
assumption by making $\omega_1=R$ be the unique top-level scenario for
the first time period (thus $\Omega_1=\{R\})$), and we take $\Omega_2=\{RS,RR\}$
to be the set of top-level scenarios in the second time period, where $\omega_2=RS$
indicates that the sun is shining in the second time period and $\omega_2=RR$ indicates
that it is still raining in the second time period.

We will first describe this example in the context of the stochastic two-level
set of assumptions, and then we will describe how the same network structure can be
analyzed under the stochastic fluid model.

Assume that new supply in the first time only enters the market in location $A$,
and the volume of new-driver supply Binomially distributed. Following notation from
the previous section, let $M_{(A,\omega_1)}\sim\mathrm{Binomial}(N,p)$ and
$M_{(B,\omega_1)}=0$ indicate the volume of new-driver supply at location $A$
and location $B$, respectively, at time $1$ under (the unique) scenario $\omega_1$. 

}

\Xomit{
\section{Driver Strategy Profile and Other Definitions}
\mccomment{This section is copied over from some scratch
sections I added when working on the welfare-robustness extension
to our results -- should eventually merge these definitions with 
another section.}
\subsection{Driver Strategy Profiles}

\begin{definition}
Fix a time $t$ and scenario $\omega_t$.
An \emph{active trip type} with respect $t$ and $\omega_t$ is any tuple
$(r,\omega_t,\delta)$ where $(r,\omega_t)\in\calA_{\omega_t}^t$ is
any active ODT-scenario and $\delta\in\{0,1\}$ is an indicator specifying
whether the trip fulfils a rider dispatch ($\delta=1$) or not ($\delta=0$).
\end{definition}

\begin{definition}
Consider the stochastic two-level model at a time $t$, scenario $\omega_t$,
with population-size parameter $k$, supply-location vector $\bfS_t$, and 
dispatch-request vector $\bfd_t$.  A \emph{feasible trip-allocation vector} 
is any vector $\bfh$ with components indexed by the active trip types, 
$$
\bfh = (h_{(r,\omega_t,\delta)} : (r,\omega_t)\in\calA_{\omega_t}^t, \delta\in\{0,1\}),
$$
which satisfies the following properties:
\begin{enumerate}
    \item (Discrete Driver Feasibility): each component $h_{(r,\omega_t,\delta)}$ is a nonnegative integer multiple of $1/k$.
    \item (Dispatch Demand Feasibility): each passenger-trip component of $\bfh$ does 
    not exceed the realized amount of dispatch demand for the corresponding route, i.e., 
    $h_{(r,\omega_t,1)} \leq d_{(r,\omega_t)}$ holds for all active routes $(r,\omega_t)$.
    \item (Supply Feasibility): the number of trip-allocations allocated by $\bfh$ to outgoing trips from each active location $(l,\omega_t)$ is exactly equal to the amount of available supply at that location.  i.e. 
    $$
    S_{(l,\omega_t)} = \sum_{(r,\omega_t)\in A^+_{\omega_t}(l)} \sum_{\delta\in\{0,1\}}
    h_{(r,\omega_t,\delta)}
    $$
    holds for all active locations $(l,\omega_t)\in\calN_{\omega_t}^t$.
\end{enumerate}
\end{definition}

\begin{definition}
Consider the stochastic two-level model at a time $t$, scenario $\omega_t$,
with population-size parameter $k$, supply-location vector $\bfS_t$, and 
dispatch-request vector $\bfd_t$.
Let $\bfh$ be a feasible trip-allocation vector with respect to these parameters.
Fix any active location $(l,\omega_t)\in\calN_{\omega_t}^t$.
The \emph{trip-allocation distribution} induced by $\bfh$ at $(l,\omega_t)$
is a discrete probability distribution $\bfh(l,\omega_t)$ supported on
active trip types which originate from $(l,\omega_t)$, defined as follows:
$$
\bfh(l,\omega_t) \equiv \left(\frac{h_{(r,\omega_t,\delta)}}{S_{(l,\omega_t)}} :
                    (r,\omega_t)\in A^+_{\omega_t}(l), \delta\in\{0,1\}\right).
$$
\end{definition}

\begin{definition}
A \emph{driver strategy-profile} is a function $H(\cdot)$ which takes as input
\begin{itemize}
    \item the time $t$,
    \item the scenario $\omega_t$,
    \item the population-size parameter $k$,
    \item the supply-location vector $\bfS_t$,
    \item the dispatch-request demand vector $\bfd_t$,
\end{itemize}
and produces a trip-suggestion vector $\bfh=H(t,\omega_t,k,\bfS_t,\bfd_t)$ that
is feasible with respect to these parameters.
\end{definition}

\begin{assumption}
\label{assn:no_empty_trips}
When a driver has two or more actions with equal utility-to-go values,
they always have a preference for the action which fulfils a dispatch,
if one exists. \mccomment{drivers can still be indifferent between two relocation trips, so some strong-convexity property will likely still need to be used to address that issue}
\end{assumption}

\begin{definition}
A \emph{platform pricing protocol} is a function $P(\cdot)$ which takes as input
\begin{itemize}
    \item the time $t$,
    \item the scenario $\omega_t$,
    \item the supply-location vector $\bfS_t$,
    \item the price-inquiry demand vector $\bfD_t$,
\end{itemize}
and produces a price vector $\bfrho=P(t,\omega_t,\bfS_t,\bfD_t)$, where
$\bfrho$ specifies nonnegative trip prices for all active routes:
$$
\bfrho = (\rho_{(r,\omega_t)} : (r,\omega_t) \in\calA_{\omega_t}^t).
$$
\end{definition}

\begin{definition}
Fix a pricing protocol $P$ and a strategy profile $H$.
Consider the stochastic two-level model at a time $t$, scenario $\omega_t$,
with population-size parameter $k$, and supply-location vector $\bfS_t$.
Fix an active location $(l,\omega_t)\in\calN_{\omega_t}^t$.
The \emph{expected utility-to-go} of a driver positioned at $(l,\omega_t)$
with respect to $P$ and $H$ is denoted by $\calV^k_{(l,\omega_t)}(\bfS_t;P,H)$
and defined by the following equation:
$$
\calV^k_{(l,\omega_t)}(\bfS_t;P,H) = \begin{cases}
\bbE^k\left[\delta\rho_{(r,\omega_t)} - c_{(r,\omega_t)} + \calV^k_{(l^+(r),\omega_{t+1})}(\bfS_{t+1};P,H)\right] & \mbox{if } t<T,\\
\bbE^k\left[\delta\rho_{(r,\omega_t)} - c_{(r,\omega_t)}\right] & \mbox{if }t=T. 
\end{cases}
$$
The above expectation is taken over multiple sources of randomness:
\begin{itemize}
\item The random price-inquiry demand $\bfD_t = (D_{(r,\omega_t)} : (r,\omega_t)\in\calA_{\omega_t}^t)$ at each active route 
$(r,\omega_t)$.
\item The prices $\bfrho = P(t,\omega_t,\bfS_t,\bfD_t)$ may depend on the random price-inquiry demand.
\item The random dispatch-request demand $\bfd_t = (d_{(r,\omega_t)} : (r,\omega_t)\in\calA_{\omega_t}^t)$ depends on the amount of price-inquiry
demand $\bfD_t$ and the trip-prices $\bfrho$.
\item The trip-allocation vector $\bfh=H(t,\omega_t,k,\bfS_t,\bfd_t)$ depends
on the random dispatch-request demand.
\item The trip type $(r,\omega_t,\delta)$ appearing in the equation is randomly
sampled from the trip-allocation distribution $\bfh(l,\omega_t)$ induced by $\bfh$
at $(l,\omega_t)$.
\item The random time $t+1$ scenario $\omega_{t+1}\sim\bbP(\cdot\mid\omega_t)$.
\item The random time $t+1$ supply-location vector $\bfS_{t+1}$ depends on
    the trip-allocation vector $\bfh$, the scenario $\omega_{t+1}$,
    and the random new-driver counts $M_{(l,\omega_{t+1})}$ at each active time $t+1$ LT-scenario $(l,\omega_{t+1})\in\calN_{\omega_{t+1}}^{t+1}$.
\end{itemize}
\end{definition}

\Xomit{
\begin{definition}
Fix a pricing protocol $P$ and a strategy profile $H$.
Consider the stochastic two-level model at a time $t$, scenario $\omega_t$,
with population-size parameter $k$, and supply-location vector $\bfS_t$.
Fix an active location $(l,\omega_t)\in\calN_{\omega_t}^t$.
The \emph{deviation utility-to-go} for a driver positioned at $(l,\omega_t)$
with respect to $P$ and $H$ is denoted by $\tilde{\calV}^k_{(l,\omega_t)}(\bfS_t;P,H)$
and defined by the following equation:
$$
\tilde{\calV}^k_{(l,\omega_t)}(\bfS_t;P,H) = \begin{cases}
/\bbE^k\left[\max_{(r,\omega_t,\delta)} \delta\rho_{(r,\omega_t)} - c_{(r,\omega_t)} + \tilde{\calV}^k_{(l^+(r),\omega_{t+1})}(\tilde{\bfS}_{t+1};P,H)\right] & \mbox{if } t<T,\\
\bbE^k\left[\max_{(r,\omega_t,\delta)} \delta\rho_{(r,\omega_t)} - c_{(r,\omega_t)}\right] & \mbox{if }t=T. 
\end{cases}
$$
The sources of randomness in the above equation are almost all the same as in the previous
definition.  The main difference is that the trip type $(r,\omega_t,\delta)$ is determined
by taking a maximum, instead of being sampled from the trip-allocation distribution
$\bfh(l,\omega_t)$.  This deviation leads to a different supply-location vector at time $t+1$.
We emphasize this difference by using the notation $\tilde{\bfS}_{t+1}$ 
to indicate the supply-location vector that arises after the deviation.  
This definition assumes that only the one driver deviates and all other drivers drive
according to the strategy profile $H$.

For the purpose of deviation utility-to-go we assume that the driver can deviate to
any trip type $(r,\omega_t,\delta)$ originating from their location $(l,\omega_t)$, i.e.
we ignore possible constraints induced by limited demand availability for dispatch trips
and we ignore the logistics of connecting to a rider that was not allocated by 
the platform.  In this sense, the deviation utility-to-go provides an upper bound on the
utility achievable from deviating.
\end{definition}

\begin{definition}
Fix a pricing protocol $P$ and a strategy profile $H$.
Consider the stochastic two-level model at a time $t$, scenario $\omega_t$,
with population-size parameter $k$, supply-location vector $\bfS_t$, and dispatch-request
vector $\bfd_t$.
Let $\bfh=H(t,\omega_t,k,\bfS_t,\bfd_t)$ be the trip-allocation vector produced by
the strategy profile $H$ with respect to these parameters.

Consider a driver who is positioned at
an active location $(l,\omega_t)\in\calN_{\omega_t}^t$. Let $(r,\omega_t,\delta)$ be the
trip type the driver samples from the trip-allocation distribution $\bfh(l,\omega_t)$.
The \emph{incentive gap} with respect to $(r,\omega_t,\delta)$ is the incentive the
driver has to deviate, conditioned on the trip-allocation $(r,\omega_t,\delta)$.
The incentive gap with respect to $(r,\omega_t,\delta)$ is equal to the following expression:
$$
\max_{(\tilde{r},\omega_t,\tilde{\delta})} (\tilde{\delta}\rho_{(\tilde{r},\omega_t)} -
c_{(\tilde{r},\omega_t)}) - (\delta\rho_{(r,\omega_t)} - c_{(r,\omega_t)}) +
\bbE^k\left[\tilde{\calV}^k_{(l^+(r),\omega_{t+1})}(\tilde{\bfS}_{t+1};P,H)
- \calV^k_{(l^+(r),\omega_{t+1})}(\bfS_{t+1};P,H)\right].
$$
\end{definition}
}

\begin{definition}
Consider the stochastic two-level model with population-size $k$.  Fix a pricing protocol
$P$ and a strategy profile $H$.  Fix any time $t$, scenario $\omega_t$,
active location $(l,\omega_t)$, and supply-location vector $\bfS_t$.
Let $\epsilon>0$ be an error term, and define $p_k(\epsilon,l,\omega_t,\bfS_t;P,H)$
the be the probability that a driver
located\footnote{In the event there are $0$ drivers located at $\ell$ in time $t$ we
define $p_k(\epsilon,l,\omega_t,\bfS_t;P,H)$ to be $1$.} at $(l,\omega_t)$ receives a trip-allocation for which the incentive gap is
smaller than $\epsilon$, with respect to the market dynamics induced by $P$, $H$ and the size-parameter $k$. 

For each $k\geq 0$ let $\bfS_t^k$ be a sequence of supply-location vectors
in the support of the induced distribution over supply-location vectors
at time $t$, scenario $\omega_t$, and population-size $k$, and assume
the sequence converges $\bfS_t^k\to\bar{\bfS}_t$.

The strategy profile $H$ is an \emph{asymptotic equilibrium} with
respect to the pricing protocol $P$ if, for any $t$, $\omega_t$, $(l,\omega_t)$, and convergent 
sequence of  supply-location vectors $(\bfS_t^k:k=1,2,\dots)$,
there exists a sequence of error terms $\epsilon_k\to 0$ such that $p_k(\epsilon_k,l,\omega_t,\bfS_t^k;P,H)\to 1$.
\end{definition}

\begin{definition}
For a time period $t$, scenario $\omega_t$, and supply-location vector $\bfS_t$,
let $\bff^*(\bfS_t)$ denote the optimal solution for the stochastic maximum
flow problem (\ref{eq:dynamic_optimization}) with respect to the subnetwork
induced by $\omega_t$ and the supply locations $\bfS_t$.

The \emph{stochastic spatiotemporal pricing and matching mechanism} (SSP) is the combination of
a pricing function $P^{SSP}$ and a driver strategy profile $H^{SSP}$.

The pricing
function $P^{SSP}(\cdot)$ produces prices by computing the optimal solution
$\bff^*(\bfS_t)$ and using the optimal flow values $f^*_{(r,\omega_t)}(\bfS_t)$
on each route.  Specifically, the prices
$\bfrho = P^{SSP}(t,\omega_t,\bfS_t)$ are specified by
$$
\rho_{(r,\omega_t)} = F_{(r,\omega_t)}^{-1}\left(1-\frac{f^*_{(r,\omega_t)}}{\bar{D}_{(r,\omega_t)}}\right),
$$
where $F_{(r,\omega_t)}^{-1}$ is the inverse of the rider-value distribution
function, $f^*_{(r,\omega_t)}$and $\bar{D}_{(r,\omega_t)}$ is the mean price-inquiry demand for
the route $(r,\omega_t)$.  Note that the randomness in the SSP prices comes
only from the supply-location vector $\bfS_t$.

The strategy profile $H^{SSP}(\cdot)$ is obtained by producing a trip-allocation
vector which is as close as possible to the actions specified by the optimal
stochastic flow $\bff^*(\bfS_t)$.  The exact details are specified in 
Algorithm \ref{alg:matching_details}.
\end{definition}

Using the framework outlined in the above definitions we can re-state our main
approximate incentive compatibility theorem (Theorem \ref{thm:stochastic_ic}) as
follows:
\begin{theorem*}
Under Assumptions 1-3, the SSP strategy profile
$H^{SSP}$ is an asymptotic equilibrium with respect to the SSP pricing
policy $P^{SSP}$.
\end{theorem*}

The following theorem is a simple property of the trip-allocation vector
produced by the SSP mechanism, which follows from our assumption that
dispatch demand converges in probability to its mean.
\begin{theorem*}
Fix a time $t$, scenario $\omega_t$, and let $\bfS_t^k$ be a sequence of supply-location vectors in the support of the induced distribution over supply-location vectors with respect to the population-size parameter $k$, 
and let $\bar{\bfS}_t$ be the limiting supply-location vector.
Let $\bfd_t$ be the (random) dispatch-request vector, whose distribution
depends on the SSP pricing function $P^{SSP}$, and let 
$\bfh_k^{SSP}=H^{SSP}(t,\omega_t,k,\bfS_t^k,\bfd_t)$ be the resulting random trip-allocation
vector produced by the SSP mechanism. 

Then, under Assumptions 1-3, $\bfh_k^{SSP}$ converges in probability
to the optimal stochastic flow with respect to the limit of the supply location
vectors: 
$$
\lim_{k\to\infty} \bbP^k(\|\bfh_k^{SSP} - \bff^*(\bar{\bfS}_t)\|_1>\epsilon) = 0
$$
for all $\epsilon > 0$.
\end{theorem*}

The following is a conjecture about robustness of equilibria induced by
the SSP prices $P^{SSP}$:
\begin{theorem} \textbf{Conjecture}
\label{thm:general_eqlbm}
Fix a time $t$, scenario $\omega_t$, and let $\bfS_t^k$ be a sequence of supply-location vectors in the support of the induced distribution over supply-location vectors with respect to the population-size parameter $k$, 
and let $\bar{\bfS}_t$ be the limiting supply-location vector.
Let $\bfd_t$ be the (random) dispatch-request vector, whose distribution
depends on the SSP pricing function $P^{SSP}$.

Let $H$ be any non-degenerate strategy profile that is an asymptotic equilibrium with respect to
$P^{SSP}$, and let 
$\bfh_k=H(t,\omega_t,k,\bfS_t^k,\bfd_t)$ be the resulting random trip-allocation
vector produced $H$ under population size $k$. 

Then, under Assumptions 1-3, $\bfh_k$ converges in probability
to the optimal stochastic flow with respect to the limit of the supply location
vectors: 
$$
\lim_{k\to\infty} \bbP^k(\|\bfh_k - \bff^*(\bar{\bfS}_t)\|_1>\epsilon) = 0
$$
for all $\epsilon > 0$.
\end{theorem}

\Xomit{
\subsection{Towards a Proof of (Conjectured) Theorem \ref{thm:general_eqlbm}}
This section is a scratchpad for working out the proof of (the general statement of)
Theorem \ref{thm:general_eqlbm}.
Let us start by stating a number of simple lemmas that will be helpful for our 
analysis.

\begin{lemma}
Fix a time $t$, scenario $\omega_t$, and let $\bfS_t^k$ be a sequence of supply-location vectors in the support of the induced distribution over supply-location vectors with respect to the population-size parameter $k$, 
and let $\bar{\bfS}_t$ be the limiting supply-location vector.
Let $\bfd_t^k$ be the (random) dispatch-request vector, whose distribution
depends on the SSP pricing function $P^{SSP}$ (and, in turn, the output of the
pricing function depends on the supply-location vector $\bfS_t^k$).

Let $\bar{\bfd}_t$ be the dispatch request vector in the fluid limit with respect
to the limiting supply locations $\bar{\bfS}_t$.  The components of $\bar{\bfd}_t$
are specified by
$$
\bar{d}_{(r,\omega_t)} = \min\left(\bar{D}_{(r,\omega_t)}, f^*_{(r,\omega_t)}\right)
$$
for each active route $(r,\omega_t)\in\calA_{\omega_t}^t$, where $\bar{D}_{(r,\omega_t)}$
is the fluid price-inquiry demand for $(r,\omega_t)$ and
$f^*_{(r,\omega_t)}$ is the flow allocated to the route $(r,\omega_t)$ under the
optimal solution with respect to the limiting supply locations,
$\bff^*_{(r,\omega_t)}(\bar{\bfS}_t)$.

Then under the distribution induced by the SSP pricing function $P^{SSP}$,
the random dispatch request vector under population-size parameter $k$
converges in probability to the fluid dispatch request vector $\bar{\bfd}_t$:
$$
\lim_{k\to\infty} \bbP^k\left(\|\bfd_t^k - \bar{\bfd}_t\|_1>\epsilon) = 0
$$
for all $\epsilon > 0$.
\end{lemma}

\mccomment{\textbf{Scratchpad for a formal proof of Theorem \ref{thm:general_eqlbm}}}
\begin{proof}
We proceed via backwards induction and start in the final time period $t=T$.
\begin{itemize}
    \item 
\end{itemize}
\end{proof}}

}

\section{Optimal Centralized Solution for the Stochastic Fluid Model}
\label{sec:opt_central}
In this section we construct an optimization problem  to obtain the 
welfare-optimal movement of drivers for the stochastic fluid model,
ignoring strategic aspects of the problem and assuming that all drivers can be routed
by a centralized planner.
\pfedit{We refer to this problem as the {\it fluid optimization problem}.}
\mccomment{continue from here}

\pfedit{The fluid} optimization problem resembles a maximum-value flow problem over 
a ``stochastic flow network'' \pfedit{in which drivers move across locations and scenarios $\omega_1,\omega_2,\dots,\omega_T$ unfold as time progresses}.
\pfcomment{Will this make sense to readers?  Have we really defined this flow network?  Do we need to?}

The fluid optimization problem solves for the welfare-optimal trip specification from any market state $(\omega_t,\bfS_t)$.
\pfdelete{For any time $t$ and scenario $\omega_t\in\Omega_t$,} Let $\Phi_{\omega_t}:\calS\to\bbR$ be a function such that
$\Phi_{\omega_t}(\bfS_t)$ gives the optimal expected welfare
that can be achieved in the fluid model starting from the market state $(\omega_t,\bfS_t)$ at time $t$.

We \pfedit{define} \pfdelete{provide the formal definition of} $\Phi_{\omega_t}$ via backwards induction.  For the base case,
let $\Phi_{\omega_{T+1}}(\cdot)=0$.
Fix any time period $t$ and assume that $\Phi_{\omega_{t+1}}$ has already been defined for all
time $t+1$ scenarios $\omega_{t+1}$. We define $\Phi_{\omega_t}(\bfS_t)$ to be the value of the
optimization problem stated below in (\ref{eq:fluid_opt}):
\begin{align}
    \Phi_{\omega_t}(\bfS)\equiv \ \ \ \ \; \; \; 
    \sup_{\bff, \bfg} &\ \ \   \calW_{\omega_t}(\bff,\bfg)
                                & \label{eq:fluid_opt}\\
    \mbox{subject to} & \nonumber\\ 
    & f_{(\ell,d)} \geq 0 \ &\forall (\ell,d)\in\calL^2\label{eq:f_nonneg_}\\
    & g_{(\ell,d)} \geq 0 \ &\forall (\ell,d)\in\calL^2 \label{eq:g_nonneg_}\\
    & f_{(\ell,d)} \geq g_{(\ell,d)} \ &\forall (\ell,d)\in\calL^2  \label{eq:f_geq_g_}\\
    &     \sum_{d\in\calL} f_{(\ell,d)} =
        S_\ell  & \forall \ell\in\calL.\label{eq:flow_conservation_}
\end{align}
The decision variables in (\ref{eq:fluid_opt}) are the pair of vectors $(\bff,\bfg)$, where
$\bff=(f_{(\ell,d)} : (\ell,d)\in\calL^2)$ encodes the total trip volumes along each route
and $\bfg=(g_{(\ell,d)} : (\ell,d)\in\calL^2)$ encodes the dispatch trip volumes along each route.
We use the notation 
$$F^*_{\omega_t}(\bfS) = \{(\bff,\bfg) \mbox{ satisfying (\ref{eq:f_nonneg_} - \ref{eq:flow_conservation_})}, 
\calW_{\omega_t}(\bff,\bfg)=\Phi_{\omega_t}(\bfS)\}$$
to mean the set of optimal solutions with respect to $(\omega_t,\bfS_t)$.

The objective function has two components:  
$$
\calW_{\omega_t}(\bff,\bfg) = \calU_{\omega_t}(\bff,\bfg) + \calU_{\omega_t}^{> t}(\bff).
$$
The second component $\calU_{\omega_t}^{>t}(\bff)$ gives the expected
total welfare achievable over all future time periods starting from time $t+1$, as a function of the total
trip volumes $\bff$ taken in the current time period.  It is defined via
\begin{equation}
\label{eq:future_welfare}
\calU^{>t}_{\omega_t}(\bff)=\bbE\left[\Phi_{\omega_{t+1}}(\bar{\bfS}_{\omega_{t+1}}(\bff))\mid\omega_t\right].
\end{equation}
The expectation is taken over the scenario $\omega_{t+1}$
given the time $t$ scenario $\omega_t$. Recall that $\bar{\bfS}_{\omega_{t+1}}(\bff)$ is a deterministic function,
specified in equation (\ref{eq:fluid_state_transition}),
that gives the time $t+1$ supply location vector as a function of the total trip volumes $\bff$.

The first component of the objective function, $\calU_{\omega_t}(\bff,\bfg)$, gives the total welfare
generated by the trips $(\bff,\bfg)$ in time period $t$. $\calU_{\omega_t}(\bff,\bfg)$ is defined by
the following equation:
\begin{equation}
    \label{eq:immediate_reward}
\calU_{\omega_t}(\bff,\bfg) = \sum_{(\ell,d)\in\calL^2} U_{(\ell,d,\omega_t)}(g_{(\ell,d)}) 
                            - \sum_{(\ell,d)\in\calL^2} c_{(\ell,d,\omega_t)}f_{(\ell,d)} 
                            - \sum_{\ell} A(\bfg^T\bfone_\ell,\bff^T\bfone_\ell),
\end{equation}
where $U_{(\ell,d,\omega_t)}(g_{(\ell,d)})$ is the maximum utility generated by serving $g_{(\ell,d)}$ dispatch
trips, $-c_{(\ell,d,\omega_t)}f_{(\ell,d)}$ is the constant cost incurred by all drivers who drive from $\ell$ to $d$,
and $A(\bfg^T\bfone_\ell,\bff^T\bfone_\ell)$ is the add-passenger disutility incurred by all drivers from $\ell$
who serve a dispatch.

In writing the add-passenger disutility cost function, we adopt the convention that $\bfone_\ell$ is a vector
indexed by routes $(\ell,d)\in\calL^2$ that takes the value $1$ on components where the origin location is $\ell$
and $0$ otherwise.  With this convention, 
$\bfg^T\bfone_\ell$ specifies the total volume of drivers at $\ell$ who serve a dispatch, and
$\bff^T\bfone_\ell$ specifies the total volume of drivers at $\ell$ in total.

The minimum disutility we can incur by serving $\bfg^T\bfone_\ell$ dispatches is to 
select drivers positioned at $\ell$ whose add-passenger disutility falls
in the bottom $\frac{\bfg^T\bfone_\ell}{\bff^T\bfone_\ell}$ quantile of the disutility distribution.
Recall that the add-passenger disutility follows a $\mathrm{Uniform}(0,C)$ distribution in each time period. So the best quantile \pfedit{results} when drivers use the cutoff $x_{(\ell,d)}=(1-\frac{\bfg^T\bfone_\ell}{\bff^T\bfone_\ell})C$, resulting in an
idiosyncratic disutility cost function which is 0 when $\bff^T\bfone_\ell=0$, and otherwise can be expressed as:
\begin{align}
    A(\bfg^T\bfone_\ell,\bff^T\bfone_\ell) &= \bfg^T\bfone_\ell\left(\frac{C}{2}\frac{\bfg^T\bfone_\ell}{\bff^T\bfone_\ell}\right) 
    = \frac{C\left(\bfg^T\bfone_\ell\right)^2}{2\bff^T\bfone_\ell}. \label{eq:add_passenger_cost}
\end{align}

To define the function $U_{(\ell,d,\omega_t)}(g)$, which specifies the total rider welfare 
generated as a function of the dispatch volume $g$,
we first derive the rider-side price maximizing welfare along that route.
Recall that $\bar{D}_{(\ell,d,\omega_t)}$ riders arrive for a trip from $\ell$ to $d$ under the scenario $\omega_t$.
For the sake of notational simplicity we don't explicitly include a constraint that the dispatch-volume $g_{(\ell,d)}$ cannot
exceed rider demand $\bar{D}_{(\ell,d,\omega_t)}$. instead, we incorporate rider demand volume into the objective function.
If the dispatch volume is smaller than the total volume of riders, then the maximum utility is generated when the price selects for the upper
$\frac{g_{(o,d,\omega_t)}}{\bar{D}_{(o,d,\omega_t)}}$ quantile of potential riders to request a dispatch. 
If the dispatch volume is larger than the volume of riders then the price should select for all available riders.
Thus, the welfare-optimal price as a function of the dispatch-volume $g$ can be written as follows:
\begin{equation}
\label{eq:pricing_function}
    P_{(\ell,d,\omega_t)}(g) = \begin{cases}
        F_{(\ell,d,\omega_t)}^{-1}\left(1-\frac{g}{\bar{D}_{(\ell,d,\omega_t)}}\right) &
        \mbox{if }g\leq \bar{D}_{(\ell,d,\omega_t)},\\
        0 & \mbox{otherwise},
    \end{cases} 
\end{equation}
where $F_{(o,d,\omega_t)}$ is the distribution function for the rider value distribution of riders from $\ell$ to $d$ under $\omega_t$.
We are now ready to specify the reward function $U_{(\ell,d,\omega_t)}(g)$:
\begin{equation}
    \label{eq:reward_fn_1}
    U_{(\ell,d,\omega_t)}(g)= \begin{cases}
        \min(g,\bar{D}_{(\ell,d,\omega_t)})\bbE[V\mid V\geq P_{(\ell,d,\omega_t)}(g)] & \mbox{if }g\geq 0,\\
        V_{max}g & \mbox{if }g < 0,
\end{cases}
\end{equation}
where $V$ is a rider willingness-to-pay random variable with distribution function $F_{(\ell,d,\omega_t)}$,
and 
where $V_{max}$ is the smallest value for which $F_{(\ell,d,\omega_t)}(V_{max}) = 1$.
The following Lemma characterizes the derivative of each utility function $U_{(o,d,\omega_t)}(\cdot,\cdot)$:
\begin{lemma}
    \label{lem:reward_fn}
    Consider the reward function $U_{(\ell,d,\omega_t)}$
    associated with any route $(\ell,d)\in\calL^2$ and any scenario $\omega_t$.
    Assume the rider-value distribution $F_{(\ell,d,\omega_t)}$ satisfies Assumption \ref{assn:rider_value_dist}.  Then $U_{(\ell,d,\omega_t)}(g)$
    is concave in $g$, is differentiable at every $g\in\bbR$, and the derivative at each $g\in\bbR$ satisfies:
    $$
    \frac{d}{dg}U_{(\ell,d,\omega_t)}(g) = \begin{cases}
        P_{(\ell,d,\omega_t)}(g) &\mbox{if }g\geq 0,\\
        V_{max} &\mbox{if }g < 0.
    \end{cases}
    $$
    Moreover, the fluid optimization problem (\ref{eq:fluid_opt}) has a concave objective function for any market state.
\end{lemma}
We defer the proof of Lemma \ref{lem:reward_fn} to Appendix \ref{appdx:reward_fn_proof}.
Next, we state a Lemma characterizing the optimality conditions for the optimization problem (\ref{eq:fluid_opt}).
We defer the proof of Lemma \ref{lem:optimality_conditions} to Appendix \ref{appdx:optimality_conditions}.

\begin{lemma}
    \label{lem:optimality_conditions}
    Let $(\omega_t,\bfS_t)$ be any market state and let 
    $\bff,\bfg\in\bbR_+^{\calL^2}$ be any feasible solution to the fluid optimization
    problem (\ref{eq:fluid_opt}) with respect to $(\omega_t,\bfS_t)$.
    Then $(\bff,\bfg)$ is an optimal solution if and only if there exist dual variables
$\alpha_{(\ell,d)}, \beta_{(\ell,d)}, \gamma_{(\ell,d)}\geq 0$ for all $(\ell,d)\in\calL^2$,
and $\eta_\ell\in\bbR$ for all $\ell\in\calL$,
    for which the following conditions are satisfied:
    \begin{enumerate}
        \item Complementary Slackness:
$$
\forall (\ell,d)\in\calL^2,\ \ 
f_{(\ell,d)}\alpha_{(\ell,d)}=0,\ \ \ g_{(\ell,d)}\beta_{(\ell,d)} = 0,\ \ \ 
(g_{(\ell,d)}-f_{(\ell,d)})\gamma_{(\ell,d)} = 0.
$$
        \item Stationarity.  For all $(\ell,d)\in\calL^2$, if the volume of drivers $S_\ell$ at $\ell$ is larger than $0$
            then the following equations hold:
            \begin{align}
                P_{(\ell,d)}(g_{(\ell,d)}) -C\frac{\bfg^T\bfone_{\ell}}{\bff^T\bfone_{\ell}}
 &=
                        \gamma_{(\ell,d)} - \beta_{(\ell,d)} .\label{eq:dynamicopt_stationarity_g}\\
            -c_{(\ell,d)}  +\frac{C}{2}\left(\frac{\bfg^T\bfone_{\ell}}{\bff^T\bfone_{\ell}}\right)^2 &=
                \eta_{\ell} - \frac{\partial}{\partial f_{(\ell,d)}}\calU^{>t}_{\omega_t}(\bff) - \alpha_{(\ell,d)} - 
            \gamma_{(\ell,d)}\label{eq:dynamicopt_stationarity_f} .
            \end{align}
     \end{enumerate}
\end{lemma}
In equation (\ref{eq:dynamicopt_stationarity_f}) above, 
$\frac{\partial}{\partial f_{(\ell,d)}}\calU^{>t}_{\omega_t}(\bff)$
is the partial derivative of the future welfare function $\calU^{>t}_{\omega_t}(\bff)$ (see
equation \ref{eq:future_welfare}) with respect to the total trip volume along route $(\ell,d)$.
From the definition of the future welfare function, we have the following equality:
$$
\frac{\partial}{\partial f_{(\ell,d)}}\calU^{>t}_{\omega_t}(\bff)
= 
\bbE_{\omega_{t+1}}\left[\frac{\partial}{\partial f_{(\ell,d)}} \Phi_{\omega_{t+1}}(\bar{\bfS}_{\omega_{t+1}}(\bff))\mid\omega_t\right],
$$
where $\Phi_{\omega_{t+1}}(\cdot)$ is the optimal welfare achievable from scenario $\omega_{t+1}$ as a function of the
supply locations $\bar{\bfS}_{\omega_{t+1}}(\bff)$.

The following Lemma characterizes the partial derivatives of $\Phi_{\omega_t}(\bfS)$ in terms of dual variables.
\begin{lemma}
\label{lem:unique_dual}
Fix a time-scenario $\omega_t$ and let $\bfS=(S_\ell\geq 0 : \ell\in\calL)$ be any supply-location vector.  Pick any location $\ell$
for which the volume of supply at $\ell$ is nonzero under $\bfS$, i.e. $S_\ell > 0$.
\begin{enumerate}
\item \label{item:unique_dual} For the state-dependent optimization problem  with respect to $\bfS$ the value of any optimal dual 
    variable associated with the flow conservation constraint  for
    location $\ell$ is unique.  That is there exists a number $\eta_\ell^*$ such that $\eta_\ell=\eta_\ell^*$, where 
    $\eta_\ell$ is the $\ell$th component of $\bm{\eta}$ for any optimal dual variables 
    $(\bm{\alpha},\bm{\beta},\bm{\gamma},\bm{\eta})\in D^*(\bfS)$.
\item \label{item:dual_partial} The state dependent optimization function $\Phi_{\omega_t}(\cdot)$ is differentiable with respect to $S_\ell$ at the
supply location vector $\bfS$.  Moreover, the partial derivative is equal to the value of the optimal dual variable for the
flow conservation constraint at location $\ell$:
$$
\frac{\partial}{\partial S_\ell}\Phi_{\omega_t}(\bfS) = \eta_\ell^*.
$$
\item \label{item:cts_partial} The partial derivative $\frac{\partial}{\partial S_\ell}\Phi_{\omega_t}(\bfS)$ is continuous at $\bfS$.
\end{enumerate}
\end{lemma}

Lemma \ref{lem:unique_dual} concerns partial derivatives of the state-dependent optimization function assuming the volume
of drivers $S_\ell$ at location $\ell$ is larger than $0$.
Notice that, if the supply volume $S_\ell$ is strictly smaller than $0$, then the feasible region for the
state-dependent optimization problem is empty and the optimal value is $-\infty$.
Therefore, the state-dependent optimization function $\Phi_{\omega_t}(\cdot)$ is only finite for supply-location
vectors which are nonnegative in every component.

It will be useful for us to extend our understanding of the partial derivative $\frac{\partial}{\partial S_\ell}\Phi_{\omega_t}(\bfS)$ 
to supply-location vectors which lie on the boundary of the domain, i.e. where $S_\ell = 0$.
For supply-location vectors $\bfS$ where $S_\ell$ is $0$ we will take $\frac{\partial}{\partial S_\ell}\Phi_{\omega_t}(\bfS)$ to mean the
sequence of derivatives of $\bfS+h\bfone_\ell$ as $h$ goes to $0$ from above:
$$
\frac{\partial}{\partial S_\ell}\Phi_{\omega_t}(\bfS^+) = \lim_{h\downarrow 0}\frac{\partial}{\partial S_\ell}\Phi_{\omega_t}(\bfS+h\bfone_\ell),
$$
where $\bfone_\ell$ represents a vector indexed by locations $\calL$ with a $1$ in the $\ell$ component and $0$ everywhere else.
In a slight abuse of notation we will write
\begin{equation}
    \label{eq:partial_derivative_boundary}
    \frac{\partial}{\partial S_\ell}\Phi_{\omega_t}(\bfS)  = \begin{cases}
        \frac{\partial}{\partial S_\ell}\Phi_{\omega_t}(\bfS) & \mbox{if } S_\ell > 0,\\
        \frac{\partial}{\partial S_\ell}\Phi_{\omega_t}(\bfS^+) & \mbox{if }S_\ell = 0,
    \end{cases}
\end{equation}
for any supply-location vector $\bfS$ with nonnegative components.
The following Lemma states that the right derivative $\frac{\partial}{\partial S_\ell}\Phi_{\omega_t}(\bfS^+)$ is well-defined
for points on the boundary, and also that the partial derivative is continuous over all supply-location vectors with nonnegative components.
\begin{lemma}
    \label{lem:cts_derivative_boundary}
Let $\bfS$ be a supply-location vector with nonnegative components and assume $S_\ell=0$ for some location $\ell$.  
Then the right-derivative $\frac{\partial}{\partial S_\ell}\Phi_{\omega_t}(\bfS^+)$ is well-defined at $\bfS$.  Moreover, the
    partial derivative function $\frac{\partial}{\partial S_\ell}\Phi_{\omega_t}(\bfS)$, defined in (\ref{eq:partial_derivative_boundary}),
    is continuous over the set $\{\bfS\in\bbR^\calL : S_\ell \geq 0 \forall \ell\in\calL\}$.
\end{lemma}

We defer the proof of Lemma \ref{lem:unique_dual}  and Lemma \ref{lem:cts_derivative_boundary} to Appendix \ref{sec:opt_central_appdx}.

\Xomit{
Lemma \ref{lem:unique_dual} provides us a characterization of the continuation welfare function by applying
the chain rule:
\begin{align}
-\partial f_{(o,d,\omega_t)}\Phi_{\omega_t}(\bff) &= 
-\bbE_{\omega_{t+1}}\left[\partial f_{(o,d,\omega_t)} \calU^{>t}_{\omega_{t+1}}(\bfS_{\omega_{t+1}}(\bff))\mid\omega_t\right]\\
&=
-\bbE_{\omega_{t+1}}\left[\frac{\partial}{\partial S_d} \calU^{>t}_{\omega_{t+1}}(\bfS_{\omega_{t+1}}(\bff))\mid\omega_t\right]\\
&= 
-\bbE_{\omega_{t+1}}\left[\eta_{(d,\omega_{t+1})}\mid\omega_t\right],
\end{align}
where $\eta_{(d,\omega_{t+1})}$ is the optimal dual variable for hte $(d,\omega_{t+1})$ flow-conservation constraint,
for the state-dependent optimizationn problem with respect to $\omega_{t+1}$ and $\bfS_{\omega_{t+1}}(\bff)$.
Note this characterization is analogous to the optimality conditions we had in the static formulation (\ref{eq:stationarity_f}).

Finally, we conclude this section with the following lemma, showing that the state-dependent optimal welfare function
has continuous partial derivatives.
\begin{lemma}
\label{lem:cts_partials}
Fix any time-scenario $\omega_t$ and let $\bfS\in\bbR_+^\calL$ be any supply-location vector. Let $\ell\in\calL$ be
any location and assume there is nonzero supply-volume positioned at $\ell$ under the vector $\bfS$, i.e. assume $S_\ell>0$.
Then $\frac{\partial}{\partial S_\ell}\calU^{>t}_{\omega_t}(\bfS)$ is continuous at $\bfS$.
\end{lemma}
We defer the proof of Lemma \ref{lem:cts_partials} to Appendix \mccomment{insert}.

}

\section{The Stochastic Spatiotemporal Pricing Mechanism}
\label{sec:sspm}
In this section we describe our main algorithmic contribution, which we refer to as the
stochastic spatiotemporal pricing  (SSP) mechanism.
The SSP mechanism re-solves the fluid optimization problem based on the observed market state in each time period
and derives its prices and matching decisions from the computed optimum.

The SSP pricing policy is the same algorithm in both the fluid model and the two level model.  It is formally defined
below, in Definition \ref{def:sspm_pricing}.

\begin{definition}
\label{def:sspm_pricing}
At each time period $t$, the SSP pricing policy observes the market state $(\omega_t,\bfS_t)\in\calS_t$ and computes
the price along each route $(\ell,d)$ to be
$$
P_{(\ell,d)} = P_{(\ell,d,\omega_t)}(g^*_{(\ell,d)}) = F_{(\ell,d,\omega_t)}^{-1}\left(1-\frac{g^*_{(\ell,d)}}{\bar{D}_{(\ell,d,\omega_t)}}\right),
$$
where $(\bff^*,\bfg^*)\in F^*_{\omega_t}(\bfS_t)$ is an optimal solution for the associated fluid optimization problem.
\end{definition}

For brevity, we defer the exact definition of the SSP matching process to Appendix \ref{appdx:matching_process_defn}. For the arguments below, it is important to note only that the matching process satisfies 
two properties: 1) it attempts to serve all dispatches, and only fails to do so if too many drivers decline dispatches, and 2)
the only problem instances where the matching process results in dispatch volume $g_{(\ell,d)}$ strictly smaller than the optimal dispatch volume $g^*_{(\ell,d)}$
are those where drivers use a threshold $x_{(\ell,d)}$ strictly smaller than the acceptance threshold implied by the optimal solution.

\Xomit{
\subsection{The SSP Matching Process}
We now describe how the SSP matching process allocates dispatch requests to drivers.  

\subsubsection{The SSP Matching Process in the Fluid Model}
Recall that in the fluid model the matching process is a function $\mathrm{MP}_\ell(\cdot)$ for each location $\ell$,
which takes as input the market state $(\omega_t,\bfS_t)$, the disutility thresholds $\bfx_\ell$ used by drivers at $\ell$,
the relocation distribution $\bfe_\ell$ used by drivers at $\ell$, and produces trip volumes $(\bfh_\ell,\bfg_\ell) = \mathrm{MP}_\ell(\omega_t,\bfS_t,\bfx_\ell,\bfe_\ell)$
such that $\bfh_\ell=(h_{(\ell,d)}:d\in\calL)$ and $\bfg_\ell=(g_{(\ell,d)}:d\in\calL)$ specify the volume of relocation trips and dispatch
trips, respectively, on each outgoing route from $\ell$.
The dispatch volumes $\bfg_\ell$ also have to satisfy the capacity constraint (\ref{eq:fluid_matching_capacity_constraint}), which states that the total
volume of drivers who are allocated a dispatch can be at most the available volume of drivers at $\ell$.  (Recall the volume of drivers who need to be 
allocated a dispatch in order to see $g$ accepted dispatches is given by the function $Z(g,x)$ (\ref{eq:fluid_matching_dispatch_volume_fn}), 
where $x$ is the disutility threshold used by the drivers for the relevant destination.  The dispatch volume $Z(g,x)$ is inversely proportional
to the threshold $x$, meaning that larger dispatch volumes occur for smaller thresholds.

The matching process we use in the fluid model first computes an optimal solution $(\bff^*,\bfg^*)$ for the fluid optimization
problem with respect to $(\omega_t,\bfS_t)$, and it then computes the disutility threshold used by drivers under the computed optimum:
\begin{equation}
x^*_\ell = C\frac{\bfg^{*T}\bfone_\ell}{\bff^{*T}\bfone_\ell}.
\end{equation}
Note that the threshold $x^*_\ell$ is used by drivers for all destinations under the computed optimum.
Note that under the SSP pricing function, the optimal dispatch volume $g^*_{(\ell,d)}$ is also the volume of riders who
request a dispatch for each destination $d$.  Therefore, the dispatch volumes $g_{(\ell,d)}$ returned by the matching process can
be no larger than the optimal dispatch volumes $g^*_{(\ell,d)}$.  Ideally the matching process would like to dispatch all available
rider demand, but that might not be possible if the disutility thresholds contained in $\bfx_\ell$ are too small.
We define a matching process which distributes dispatches towards the different destinations such that, if the drivers
use a threshold $x_{(\ell,d)}$ for a destination $d$ which is no less than the optimal thresold $x^*_\ell$,
all dispatch demand is served towards that destination.

The matching process proceeds in two stages.  In the first stage, we partition drivers at $\ell$ into $|\calL|$ different
partitions, such that the partitition associated with each destination $d$ has volume $Z(g^*_{(\ell,d)}, x^*_\ell)$.
Notice that
$$
\sum_{d\in\calL}Z(g^*_{(\ell,d)}, x^*_\ell)= \sum_{d\in\calL}g^*_{(\ell,d)}\frac{C}{x^*_\ell} = \bff^{*T}\bfone_\ell = S_\ell,
$$
so this partition covers every driver at $\ell$.  Next, we allocate dispatches towards each destination $d$ to drivers in the partition
associated with $d$ until either no dispatch demand remains or until all drivers in that partition have been allocated a dispatch.
Let $g_{(\ell,d)}^{(1)}$ be the volume of drivers who accept a dispatch towards each destination $d$ from this first stage in the
matching process.  We have the equality
\begin{equation}
\label{eq:first_stage_volumes}
g_{(\ell,d)}^{(1)} = \begin{cases}
g^*_{(\ell,d)} & \mbox{if }x_{(\ell,d)} \geq x^*_\ell,\\
Z_{(\ell,d)}\frac{x_{(\ell,d)}}{C} & \mbox{if }x_{(\ell,d)} < x^*_\ell,
\end{cases}
\end{equation}
where we write 
\begin{equation}
\label{eq:optimal_partition_size}
Z_{(\ell,d)} = Z(g^*_{(\ell,d)},x^*_\ell)
\end{equation} 
is the size of the partition associated with destination $d$.
Let $Z^{(1)}_{(\ell,d)} = Z(g_{(\ell,d)}^{(1)},x_{(\ell,d)})$ be the volume of drivers from the partition associated with $d$ 
who were allocated a dispatch.
Note that the only case in which $Z^{(1)}_{(\ell,d)}$ is smaller than $Z_{(\ell,d)}$ is if the threshold $x_{(\ell,d)}$ is
larger than $x^*_\ell$, i.e. $Z^{(1)}_{(\ell,d)} < Z_{(\ell,d)}$ if and only if $x_{(\ell,d)} > x^*_\ell$.
Also note that $g_{(\ell,d)} < g^*_{(\ell,d)}$ if and only if $x_{(\ell,d)} < x^*_\ell$.

In the second stage of the matching process we go through the destinations with remaining dispatch demand and we try to dispatch all
demand for each destination, stopping when either all dispatch demand has been allocated or when all drivers have been allocated a dispatch.
Let $V_0$ be the volume of drivers who have not been allocated a dispatch at the end of the first round:
$$
V_0 = \sum_{d\in\calL} Z_{(\ell,d)} - Z^{(1)}_{(\ell,d)}.
$$
Let $d_1,\dots,d_L$ be a fixed ordering of all the locations in $\calL$.  For each $i=1,2,\dots,L$, let $g^{(2)}_{(\ell,d_i)}$ be the volume
of dispatches towards destination $d_i$ that get accepted in the second round of the matching process.  We have
\begin{equation}
\label{eq:second_stage_volumes}
g^{(2)}_{(\ell,d_i)} = \begin{cases}
r_{(\ell,d_i)} & \mbox{if }Z(r_{(\ell,d_i)},x_{(\ell,d_i)}) \leq V_{i-1},\\
V_{i-1}\frac{x_{(\ell,d_i)}}{C} & \mbox{ else},
\end{cases}
\end{equation}
where $r_{(\ell,d_i)} = g^*_{(\ell,d_i)} - g^{(1)}_{(\ell,d_i)} $ is the remaining volume of dispatch demand towards $d_i$,
and where $V_{i-1}$ is the volume of drivers who have not been allocated a dispatch at the end of the previous iteration:
$V_i = V_{i-1} - Z(g^{(2)}_{(\ell,d_i)}, x_{(\ell,d_i)}).$

\begin{definition}
\label{def:fluid_sspm_matching}
The SSP matching process for the fluid model uses the two-stage process defined above to produce dispatch volumes
$\bfg_\ell$ given a market state $(\omega_t,\bfS_t)$, a location $\ell$, and an arbitrary vector of disutility thresholds $\bfx_\ell$.
The resulting volume of dispatches along each route $(\ell,d)$ is given by
\begin{equation}
g_{(\ell,d)} = g^{(1)}_{(\ell,d)} + g^{(2)}_{(\ell,d)},
\end{equation}
where $g^{(1)}_{(\ell,d)}$ and $g^{(2)}_{(\ell,d)}$ are the volume of dispatches generated in the first (\ref{eq:first_stage_volumes}) and second (\ref{eq:second_stage_volumes})
stages, respectively.
\end{definition}

\subsubsection{The SSP Matching Process in the Two-Level Model}
The SSP matching process in the two-level model mirrors the matching process in the fluid model, but the specific dynamics are granular and stochastic.

We describe the matching process in the two-level model in terms of total counts of drivers instead of driver volumes.
Let $M_\ell = kS_\ell$ be the total number of drivers who are positioned at $\ell$ and let $\calM_\ell = \{1,2,\dots,M_\ell\}$ be an index set over these drivers.
Partition the drivers into $|\calL|$ separate partitions, $\calM_{(\ell,d)}\subseteq \calM_\ell$ for each destination $d\in\calL$ such that the relative size
of $\calM_{(\ell,d)}$ to $\calM_\ell$ is approximately the relative partition size $Z_{(\ell,d)}/S_\ell$ used in the fluid matching process.
The approximation error comes from the fact that $M_\ell Z_{(\ell,d)}/S_\ell$ is not an integer in general, but we can round the partition sizes up or
down as needed so that if $M_{(\ell,d)} = |\calM_{(\ell,d)}|$ is the number of drivers in the $d$ partition we have $|M_{(\ell,d)} - M_\ell Z_{(\ell,d)} / S_\ell|\leq 1$.
Once the partition sizes $M_{(\ell,d)}$ are determined, we assume the assignment of drivers to partitions happens uniformly at random.

Let $R_{(\ell,d)}$ be the number of riders who request a dispatch from $\ell$ to $d$.
In the first stage of the matching process we try to allocate all $R_{(\ell,d)}$ dispatch requests for $d$ to drivers in $\calM_{(\ell,d)}$.
Algorithm \ref{alg:dispatch} describes our process for allocating dispatches for a fixed destination to a set of drivers.
We run Algorithm \ref{alg:dispatch} separately for each destination $d$, passing in dispatch requests $R_{(\ell,d)}$ and the set of drivers $\calM_{(\ell,d)}$
as inputs.  Let $G^{(1)}_{(\ell,d)}$ and $\calM^{(1)}_{(\ell,d)}$ be the outputs, where $G^{(1)}_{(\ell,d)}$ counts how many dispatches were accepted
and $\calM^{(1)}_{(\ell,d)}$ is the set of drivers who have not yet been allocated a dispatch.
\begin{algorithm}
\caption{Algorithm for allocating dispatch requests to a single destination to a set of drivers.}
\label{alg:dispatch}
\begin{enumerate}
\item Inputs: $R$ is the number of dispatch requests, $\calM$ is a set of driver indices, $x_i$ is the disutility threshold used by each driver $i\in\calM$.
\item Initialize $\calN =\emptyset$ and $G=0$.
\item While $\calM\setminus \calN\neq \emptyset$ and $G<R$:
\begin{itemize}
\item Sample $i$ from $\calM\setminus\calN$ uniformly at random.
\item Add $i$ to $\calN$: $\calN \leftarrow \calN \cup \{i\}$.
\item Allocate $i$ a dispatch.  Sample acceptance decision $\delta\sim\mathrm{Bernoulli}\left(\frac{x^i}{C}\right)$.
\item Update number of accepted dispatches: $G \leftarrow G + \delta$.
\end{itemize}
\item Return accepted dispatches $G$ and set of remaining drivers $\calM \setminus \calN$.
\end{enumerate}
\end{algorithm}

In the second stage of the two-level model matching process we again go through the destinations one by one, and allocate dispatches until
either all dispatches have been accepted or all drivers have been allocated a dispatch.
Define $\calU^{(0)}_\ell = \calM_\ell \setminus \left(\cup_{d\in\calL} \calM^{(1)}_{(\ell,d)}\right)$ to be the initial set of undispatched
drivers at the start of the second stage.
Let $d_1,\dots,d_L$ be the same ordering of the locations that we used in the fluid matching process.  Then for each $j=1,2,\dots,L$ we use
Algorithm \ref{alg:dispatch} to dispatch the remaining $R_{(\ell,d)} - G_{(\ell,d)}^{(1)}$ requests to the set available undispatched drivers $\calU^{(j-1)}_\ell$.
Let $G^{(2)}_{(\ell,d)}$ be the resulting number of accepted dispatches, and let $\calU^{(j)}_\ell$ be the new set of undispatched drivers.

A detailed description of the matching process in the two-level model is provided in Algorithm \ref{alg:two_level_matching_process}
in Appendix \ref{sec:matching_process_appdx}.

}

\subsection{Incentive Compatibility of the Fluid Optimal Solution}
Our first result in this section shows that, under the SSP pricing mechanism, the optimal trips and acceptance thresholds
obtained from an optimal solution to the fluid optimization problem form an equilibrium.

Let $\Sigma^* = (\Sigma_1^*,\dots,\Sigma_T^*)$ be the fluid strateegy profile that maps market states $(\omega_t,\bfS_t)$ to
thresohlds $\bfx$ and relocation distributions which correspond to the same optimal solution $(\bff^*,\bfg^*)$ that the SSP mechanism
uses to set prices.  That is, when $(\bfx,\bfe) = \Sigma_t^*(\omega_t,\bfS_t)$, then, for every route $(\ell,d)\in\calL^2$, we mean
$$
x_{(\ell,d)}=\frac{\bfg^{*T}\bfone_\ell}{S_\ell}C
$$
and 
$$
e_{(\ell,d)} = \frac{f_{(\ell,d)} - g_{(\ell,d)}}{S_\ell - \bfg^{*T}\bfone_\ell}.
$$

We also define a variant of this strategy profile which follows the fluid optimal solutions for the stochastic two-level
model with population size $k$.  Let $\Pi^{(k)}$ be the strategy profile where drivers select their thresholds
and relocation destinations using the fluid optimal strategy $\Sigma^*$.  Specifically, at time $t$ with market state $(\omega_t,\bfS_t)$,
a driver $i$ positioned at location $\ell$ selects their action by first computing $(\bfx,\bfe) = \Sigma_t^*(\omega_t,\bfS_t)$  The driver uses
acceptance thresholds $\bfx_i = \bfx_\ell = (x_{(\ell,d)} : d\in\calL)$ as their threshold vector.  The drivers at each location $\ell$
collectively choose their relocation destination so that the fraction of drivers choosing each relocation destination is as close as possible
to the fraction prescribed by the fluid relocation distribution $\bfe_\ell=(e_{(\ell,d)}:d\in\calL)$. 

We start by proving a Lemma which characterizes the value function of the strategy $\Sigma^*$.
\begin{lemma}
    \label{lem:value_fn_equals_partial}
    Let $\calV_t(\ell,\omega_t,\bfS_t)$ be the value function associated with the strategy profile $\Sigma^*$, 
    for a location $\ell$ and market state $(\omega_t,\bfS_t)$. When the SSP mechanism is used to set prices in the stochastic fluid model,
    the value function satisfies
    \begin{equation}
        \label{eq:value_fn_equals_partial}
    \calV_t(\ell,\omega_t,\bfS_t) = \frac{\partial}{\partial S_\ell}\Phi_{\omega_t}(\bfS_t),
    \end{equation}
    where $\Phi_{\omega_t}(\bfS_t)$ is the state-dependent optimal welfare function.
\end{lemma}
\begin{proof}
We prove (\ref{eq:value_fn_equals_partial}) via backwards induction.  Fix a time period $t$ and assume that the value function at the next time period $t+1$ 
    satisfies equation (\ref{eq:value_fn_equals_partial}) for any time $t+1$ market state.

    Let $(\omega_t,\bfS_t)$ be any time $t$ market state, and let $(\bff^*,\bfg^*)\in F_{\omega_t}^*(\bfS_t)$ be the optimal fluid solution that the
    SSP mechanism uses to set prices, and that the strategy profile $\Sigma^*$ uses to determine its actions.

    Let $(\bm{\alpha},\bm{\beta},\bm{\gamma},\bfeta)$ be dual variables certifying the optimality of $(\bff^*,\bfg^*)$. 
    In Lemma \ref{lem:unique_dual}, we show that the dual variables $\bfeta$, which are associated with the flow-conservation constraint
    for each location $\ell$, are equal to the partial derivative of the state-dependent optimization function, i.e. 
    $\eta_\ell = \frac{\partial}{\partial S_\ell}\Phi_{\omega_t}(\bfS_t)$.
    Therefore, to finish the proof of equation (\ref{eq:value_fn_equals_partial}), it suffices to show that $\eta_\ell = \calV_t(\ell,\omega_t,\bfS_t)$.
    We proceed with the following equations, which start from the definition of the value function:
    \begin{align*}
        \calV_t(\ell,\omega_t,\bfS_t) &= \frac{1}{S_\ell}\left[\sum_{d\in\calL}(f^*_{(\ell,d)} - g^*_{(\ell,d)})\calQ_t(\ell,d,0) + 
        g^*_{(\ell,d)}\calQ_t(\ell,d,1,\frac{x_{(\ell,d)}}{2})\right]\\
    &= \frac{1}{S_\ell}\left[\sum_{d\in\calL}f^*_{(\ell,d)}\calQ_t(\ell,d,0) + 
        g^*_{(\ell,d)}\left(P_{(\ell,d)} - \frac{x_{(\ell,d)}}{2}\right)\right],
    \end{align*}
    where $\calQ_t(\ell,d,0)$ is the Q value associated with a relocation trip from $\ell$ to $d$, 
    $\calQ_t(\ell,d,1,\frac{x_{(\ell,d)}}{2})$ is the Q value associated with a dispatch trip from $\ell$ to $d$, with respect to the
    average pickup disutility $\frac{x_{(\ell,d)}}{2}$.  Recall that the disutility threshold $x_{(\ell,d)}$ used by the
    strategy profile $\Sigma^*$ is equal to $x_{(\ell,d)} = C\frac{\bfg^{*T}\bfone_\ell}{\bff^{*T}\bfone_\ell}$.
    Therefore the following equation holds, continuing from our earlier algebra:
    \begin{align*}
        \calV_t(\ell,\omega_t,\bfS_t) 
        &=\frac{1}{S_\ell}\left[\sum_{d\in\calL}f^*_{(\ell,d)}\calQ_t(\ell,d,0) + g^*_{(\ell,d)}\frac{C}{2}\frac{\bfg^{*T}\bfone_\ell}{\bff^{*T}\bfone_\ell} + 
        g^*_{(\ell,d)}\left(P_{(\ell,d)} -C \frac{\bfg^{*T}\bfone_\ell}{\bff^{*T}\bfone_\ell}\right)\right].
    \end{align*}

    The stationarity optimality condition (\ref{eq:dynamicopt_stationarity_g}) states 
    $P_{(\ell,d)} -C\frac{\bfg^{*T}\bfone_{\ell}}{\bff^{*T}\bfone_{\ell}}=\gamma_{(\ell,d)} - \beta_{(\ell,d)}$. 
    The complementary slackness conditions also state $g^*_{(\ell,d)} (\gamma_{(\ell,d)} - \beta_{(\ell,d)}) = f^*_{(\ell,d)}\gamma_{(\ell,d)}$,
    as well as $f^*_{(\ell,d)}\alpha_{(\ell,d)} = 0$.
    Therefore,
    \begin{align}
        \calV_t(\ell,\omega_t,\bfS_t) 
        &=\frac{1}{S_\ell}\left[\sum_{d\in\calL}f^*_{(\ell,d)}\calQ_t(\ell,d,0) + g^*_{(\ell,d)}\frac{C}{2}\frac{\bfg^{*T}\bfone_\ell}{\bff^{*T}\bfone_\ell} + 
            g^*_{(\ell,d)}\left(\gamma_{(\ell,d)} - \beta_{(\ell,d)}\right)\right]\nonumber\\
        &=\frac{1}{S_\ell}\left[\sum_{d\in\calL}f^*_{(\ell,d)}(\calQ_t(\ell,d,0)+\gamma_{(\ell,d)}+\alpha_{(\ell,d)}) 
        + g^*_{(\ell,d)}\frac{C}{2}\frac{\bfg^{*T}\bfone_\ell}{\bff^{*T}\bfone_\ell} 
            \right]\nonumber\\
        &=\frac{1}{S_\ell}\left[\sum_{d\in\calL}f^*_{(\ell,d)}(\calQ_t(\ell,d,0)+\gamma_{(\ell,d)}+\alpha_{(\ell,d)})\right]
         + \left(\frac{\bfg^{*T}\bfone_\ell}{S_\ell}\right)\frac{C}{2}\frac{\bfg^{*T}\bfone_\ell}{\bff^{*T}\bfone_\ell} \nonumber\\
        &=\frac{1}{S_\ell}\left[\sum_{d\in\calL}f^*_{(\ell,d)}(\calQ_t(\ell,d,0)+\gamma_{(\ell,d)}+\alpha_{(\ell,d)})\right]
        +\frac{C}{2}\left( \frac{\bfg^{*T}\bfone_\ell}{\bff^{*T}\bfone_\ell}\right)^2\nonumber\\
        &= \frac{1}{S_\ell}\left[\sum_{d\in\calL}f^*_{(\ell,d)}\left(\calQ_t(\ell,d,0)+\gamma_{(\ell,d)}+\alpha_{(\ell,d)} 
+\frac{C}{2}\left( \frac{\bfg^{*T}\bfone_\ell}{\bff^{*T}\bfone_\ell}\right)^2
        \right)\right] \label{eq:value_fn_expr}
    \end{align}
    
    The stationarity condition (\ref{eq:dynamicopt_stationarity_f}) states that, for any destination $d$ 
    $$
-c_{(\ell,d)}  +\frac{C}{2}\left(\frac{\bfg^T\bfone_{\ell}}{\bff^T\bfone_{\ell}}\right)^2 + \frac{\partial}{\partial f_{(\ell,d)}}\calU^{>t}_{\omega_t}(\bff^*)
+ \alpha_{(\ell,d)} + \gamma_{(\ell,d)}=
                \eta_{\ell}.
    $$
    Lemma \ref{lem:unique_dual} and our backwards induction assumption characterizes the partial derivative as
   \begin{align*} 
\frac{\partial}{\partial f_{(\ell,d)}}\calU^{>t}_{\omega_t}(\bff^*)
       & = 
\bbE_{\omega_{t+1}}\left[\frac{\partial}{\partial f_{(\ell,d)}} \Phi_{\omega_{t+1}}(\bar{\bfS}_{\omega_{t+1}}(\bff^*))\mid\omega_t\right]\\
       &= \bbE_{\omega_{t+1}}\left[\calV_{t+1}(d,\omega_{t+1},\bar{\bfS}_{\omega_{t+1}}(\bff^*))\mid\omega_t\right].
   \end{align*} 
   Therefore, the dual variable $\eta_\ell$ can be written as
   $$
   \eta_\ell = \calQ_t(\ell,d,0)  +\frac{C}{2}\left(\frac{\bfg^T\bfone_{\ell}}{\bff^T\bfone_{\ell}}\right)^2 + \alpha_{(\ell,d)} + \gamma_{(\ell,d)}.
   $$
   Plugging the above in to the expression (\ref{eq:value_fn_expr}), we obtain
   $$
    \calV_t(\ell,\omega_t,\bfS_t) = \frac{1}{S_\ell}\sum_{d\in\calL}f^*_{(\ell,d)}\eta_\ell = \eta_\ell,
   $$
   establishing our backwards induction hypothesis and finishing the proof of equation (\ref{eq:value_fn_equals_partial}).

\end{proof}

\begin{lemma}
    \label{lem:Q_values}
    Let $\calQ_t$ be the Q values associated with the strategy profile $\Sigma^*$. Let $t$ be any time period and $(\omega_t,\bfS_t)$ be any market state.
    Let $(\bff^*,\bfg^*)\in F_{\omega_t}^*(\bfS_t)$ be the optimal fluid solution that the
    SSP mechanism uses to set prices, and that the strategy profile $\Sigma^*$ uses to determine its actions. 
Let $(\bm{\alpha},\bm{\beta},\bm{\gamma},\bfeta)$ be dual variables certifying the optimality of $(\bff^*,\bfg^*)$.
Let $\ell\in\calL$ be any location.
    When the SSP mechanism is used to set prices, we have the following upper bound on the Q value of any relocation trip originating from $\ell$:
    \begin{equation}
        \label{eq:Q_value_ub}
        \eta_\ell - \frac{C}{2}\left(\frac{\bfg^{*T}\bfone_\ell}{\bff^{*T}\bfone_\ell}\right)^2 \geq \max_d \calQ_t(\ell,d,0).
    \end{equation}
    Moreover, if any drivers at $\ell$ take a relocation trip towards $d$, i.e. if $f^*_{(\ell,d)} > g^*_{(\ell,d)}$, then 
    \begin{equation}
        \label{eq:Q_value_equality}
\eta_\ell - \frac{C}{2}\left(\frac{\bfg^{*T}\bfone_\ell}{\bff^{*T}\bfone_\ell}\right)^2 = \calQ_t(\ell,d,0).
    \end{equation}
    Finally, if there are any drivers at $\ell$ who take a relocation trip, i.e. if $\bfg^{*T}\bfone_\ell < \bff^{*T}\bfone_\ell$, then
    for all destinations $d$ where $g^*_{(\ell,d)} > 0$ we have
    \begin{equation}
        \label{eq:Q_value_dispatch_equality}
    \calQ_t(\ell,d,1,x_{(\ell,d)}) = \max_{d'} \calQ_t(\ell,d',0) = \eta_\ell - \frac{C}{2}\left(\frac{\bfg^{*T}\bfone_\ell}{\bff^{*T}\bfone_\ell}\right)^2.
    \end{equation}
\end{lemma}
\begin{proof}
For any destination $d$, the stationarity condition (\ref{eq:dynamicopt_stationarity_f}) states that
$$
    \eta_\ell = \calQ_t(\ell,d,0) + \frac{C}{2}\left(\frac{\bfg^{*T}\bfone_\ell}{\bff^{*T}\bfone_\ell}\right)^2 + \alpha_{(\ell,d)} + \gamma_{(\ell,d)}.
$$
Rearranging,
$$
    \eta_\ell - \frac{C}{2}\left(\frac{\bfg^{*T}\bfone_\ell}{\bff^{*T}\bfone_\ell}\right)^2 = \calQ_t(\ell,d,0) + \alpha_{(\ell,d)} + \gamma_{(\ell,d)}.
$$
    The upper bound (\ref{eq:Q_value_ub}) follows from the fact that $\alpha_{(\ell,d)}$ and $\gamma_{(\ell,d)}$ are nonnegative.

    To establish (\ref{eq:Q_value_equality}), consider any destination $d\in\calL$ where $f^*_{(\ell,d)} > g^*_{(\ell,d)}$.
    Then by the complementary slackness conditions $f^*_{(\ell,d)}\alpha_{(\ell,d)} = 0$ we know that $\alpha_{(\ell,d)}=0$ must be satisfied,
    and by $(g^*_{(\ell,d)} - f^*_{(\ell,d)})\gamma_{(\ell,d)} = 0$ we know $\gamma_{(\ell,d)}=0$ must be satisfied.

    To establish (\ref{eq:Q_value_dispatch_equality}), first observe the condition $\bfg^{*T}\bfone_\ell < \bff^{*T}\bfone_\ell$
    guarantees the existence of at least one destination $d$ where $g^*_{(\ell,d)} < f^*_{(\ell,d)}$, so equation (\ref{eq:Q_value_equality}) guarantees
    $\max_{d'} \calQ_t(\ell,d',0) = \eta_\ell - \frac{C}{2}\left(\frac{\bfg^{*T}\bfone_\ell}{\bff^{*T}\bfone_\ell}\right)^2$.
    Next, consider any destination $d$ where $g^*_{(\ell,d)} > 0$.
    The stationarity condition (\ref{eq:dynamicopt_stationarity_g}) states 
    $$
    P_{(\ell,d)} - x_{(\ell,d)} + \beta_{(\ell,d)} = \gamma_{(\ell,d)},
    $$
    and since $g^*_{(\ell,d)} > 0$, complementary slackness provides us $\beta_{(\ell,d)} = 0$.
    Therefore, starting from the stationarity condition (\ref{eq:dynamicopt_stationarity_f}), we obtain
 $$
    \eta_\ell = \calQ_t(\ell,d,0) + \frac{C}{2}\left(\frac{\bfg^{*T}\bfone_\ell}{\bff^{*T}\bfone_\ell}\right)^2 + P_{(\ell,d)} - x_{(\ell,d)}.
$$   
    Observing that $\calQ_t(\ell,d,0) + P_{(\ell,d)} - x_{(\ell,d)} = \calQ_t(\ell,d,1,x_{(\ell,d)})$ establishes (\ref{eq:Q_value_dispatch_equality}).

\end{proof}

We are now ready to state and prove our main theorem, which states that the fluid optimal strategy profile $\Sigma^*$ is an
equilibrium under the SSP prices.

\begin{theorem}
    \label{thm:opt_ic}
The following statements are true:
    \begin{enumerate}
    \item When the SSP mechanism is used to set prices in the stochastic fluid model, the strategy profile $\Sigma^*$ is an exact equilibrium.
        \label{thm:item:opt_ic_fluid}
    \item There exist nonnegative sequences $(\epsilon_k:k\geq 1)$ and $(\delta_k:k\geq 1)$, both converging to $0$ as  $k\to\infty$, such that
        when the SSP mechanism is used to set prices in the stochastic two-level model with population-size $k$, then $\Pi^{(k)}$ is an
            $(\epsilon_k,\delta_k)$ equilibrium.
            \label{thm:item:opt_ic_stochastic}
    \end{enumerate}
\end{theorem}
\begin{proof}
    We present the proof of Theorem \ref{thm:opt_ic} part \ref{thm:item:opt_ic_fluid} below, 
    and defer the proof of part \ref{thm:item:opt_ic_stochastic}
    to Appendix \ref{appdx:two_level_approx_ic}.
Let $(\omega_t,\bfS_t)$ be any time $t$ market state, and let $(\bff^*,\bfg^*)\in F_{\omega_t}^*(\bfS_t)$ be the optimal fluid solution that the
    SSP mechanism uses to set prices, and that the strategy profile $\Sigma^*$ uses to determine its actions.
    Let $(\bm{\alpha},\bm{\beta},\bm{\gamma},\bfeta)$ be dual variables certifying the optimality of $(\bff^*,\bfg^*)$. 

   To show that $\Sigma^*$ is an equilibrium we have to show that no drivers have an incentive to deviate.

   First, we show that no driver has incentive to deviate from a relocation trip to another relocation trip. 
    Indeed, if there exists a destination $d$ where
    drivers are taking a relocation trip, i.e. where $f^*_{(\ell,d)} > g^*_{(\ell,d)}$, then Lemma \ref{lem:Q_values} shows
    $\calQ_t(\ell,d,0) = \max_{d'}\calQ_t(\ell,d',0).$

    Next, we show that no driver declines a dispatch that they would have preferred to take.
     Let $\ell$ be the origin  location and let $d$ be the destination that the driver declines to take.
    If a driver declines a dispatch then that means they are taking a relocation trip, so $\bfg^{*T}\bfone_\ell < \bff^{*T}\bfone_\ell$.
    Therefore, Lemma \ref{lem:Q_values} states $\calQ_t(\ell,d,1,x_{(\ell,d)}) = \max_{d'}\calQ_t(\ell,d',0)$.
    Any driver who declines a dispatch towards $d$ has add-passenger disutility $X$ larger than the threshold $x_{(\ell,d)}$, 
    so the Q value $\calQ_t(\ell,d,1,X)$ is smaller than the optimal relocation-trip utility that they do collect.

    Finally, we show that no driver who accepts a dispatch would prefer to take a relocation trip.  This also
    follows from Lemma \ref{lem:Q_values}, which shows $\calQ_t(\ell,d,1,x_{(\ell,d)}) \geq \max_{d'}\calQ_t(\ell,d',0)$
    always holds. Any driver who accepts a dispatch trip from $\ell$ to $d$ has add-passenger disutility $X$ smaller than the
    threshold $x_{(\ell,d)}$. Therefore, the utility they collect $\calQ_t(\ell,d,1,X)$ is larger than the optimal relocation-trip utility that
    they could collect.
\end{proof}

\subsection{Robustness of Equilibria Under the SSP Mechanism}

Theorem \ref{thm:robust_eqlbm_fluid}, below, states our welfare-robustness theorem for the SSP mechanism in the stochastic
fluid model.  Theorem \ref{thm:robust_eqlbm_fluid}'s proof is given in \S\ref{subsec:fluid_robustness_proof} with
additional details in Appendix \ref{appdx:approx_robustness_fluid_proof}.

\begin{theorem}[Welfare-robustness in the stochastic fluid model]
\label{thm:robust_eqlbm_fluid}
When the SSP mechanism is used to set prices in the stochastic fluid model, every equilibrium strategy profile achieves
optimal welfare, and every $\epsilon$-equilibrium strategy achieves $\epsilon'$-optimal welfare, where $\epsilon'$ goes to $0$ as $\epsilon$ goes to $0$.
\end{theorem}

Our welfare robustness theorem for the stochastic two-level model considers approximate equilibria whose approximation error vanishes as the population size
grows large.
We also restrict our attention to a subset of the state space in which the total \pfedit{(normalized)} volume of drivers in the network does not exceed a constant.
Let $\gamma > 0$ be a constant, and for each time period $t\in [T]$ define
$\calS_t(\gamma) = \{(\omega_t,\bfS_t)\in\Omega_t\times\calS : \sum_\ell S_\ell \leq \gamma\}$.
We assume that $\gamma$ is large enough so that the probability of the market state belonging to $\calS_t(\gamma)$ converges to $1$
as the population size grows large, regardless of the strategy profile. This is possible because there are only finitely many
points at which drivers can enter the network, and all the driver-entry random variables \pfedit{concentrate around} \pfdelete{converge to} their mean as the population size
grows, and these means grow linearly in $k$ (see Assumption \ref{assn:entry_distributions}).  Therefore, we can pick any 
$\gamma$ larger than the maximum fluid driver volume over all scenarios.

Let $\alpha_k\geq 0$ be a sequence of error terms converging to $0$ as $k\to\infty$, and for each $k\geq 1$
let $\calP^k$ be the set of $\alpha_k$-equilibrium strategy profiles for the two-level model with population size $k$,
over the state space $\calS_t(\gamma)$.

For any strategy profile $\Pi$, let $W_{\omega_t}(\bfS_t;\Pi,k)$ denote the normalized expected welfare achieved by $\Pi$ from the
market state $(\omega_t,\bfS_t)$ under the population size parameter $k$.  A formal definition of $W_{\omega_t}(\bfS_t;\Pi,k)$ is
given in Appendix \ref{appdx:two_level_welfare_def}.

Theorem \ref{thm:robust_eqlbm_two_level}, below, states our welfare-robustness theorem for the SSP mechanism in the stochastic
two-level model.  We defer the proof of Theorem \ref{thm:robust_eqlbm_two_level} 
to Appendix \ref{appdx:approx_robustness_two_level_proof}.

\begin{theorem}[Approximate welfare-robustness in the stochastic two-level model]
\label{thm:robust_eqlbm_two_level}
When the SSP mechanism is used to set prices in the stochastic two-level model, every approximate equilibrium
achieves approximately optimal welfare.

Specifically, there exists a sequence of error terms $\epsilon_k\geq 0$ converging to $0$ as $k\to\infty$ such that the following
is true: for every $k\geq 1$, every $\alpha_k$-approximate equilibrium $\Pi\in\calP^k$, and every market state $(\omega_t,\bfS_t)\in\calS(\gamma)$,
we have
$$
\Phi_{\omega_t}(\bfS_t) - W_{\omega_t}(\bfS_t;\Pi,k) \leq \epsilon_k.
$$
\end{theorem}

\subsection{Proof of Theorem \ref{thm:robust_eqlbm_fluid}}
\label{subsec:fluid_robustness_proof}
We begin by proving the first part of Theorem \ref{thm:robust_eqlbm_fluid}, stating that
equilibrium fluid strategies achieve optimal welfare.
Let $\Sigma$ be an equilibrium strategy profile for the fluid model, meaning that the
incentive compatibility conditions (\ref{eq:exact_fluid_ic}) hold.
Recall a fluid strategy profile is a sequence of functions
$\Sigma=(\Sigma_1,\dots,\Sigma_T)$ such that each $\Sigma_t$ maps market states $(\omega_t,\bfS_t)$
to a vector of disutility thresholds and relocation distributions: $(\bfx, \bfe) = \Sigma_t(\omega_t,\bfS_t)$.
The matching process takes as input the driver strategy $(\bfx,\bfe)$, the dispatch
volumes and the market state, and produces trip specifications $(\bff,\bfg)$.

For each $t$, let $\calV_t$ and $\calQ_t$ be the value function and the Q-value function
associated with $\Sigma$ at $t$. 

We prove Theorem \ref{thm:robust_eqlbm_fluid} by backward induction on $t$.
Fix a period $t\leq T$ and assume:
\begin{enumerate}
\item $\Sigma$ achieves optimal welfare from every market state at time $t+1$.
\item For any time $t+1$ state $(\omega_{t+1},\bfS_{t+1})\in\calS_{t+1}$, the value function for any location $\ell$ satisfies:
\begin{equation}
\label{eq:backwards_induction_assn}
\calV_{t+1}(\ell,\omega_{t+1},\bfS_{t+1}) = \frac{\partial}{\partial S_\ell}\Phi_{\omega_{t+1}}(\bfS_{t+1}).
\end{equation}
\end{enumerate}

For the rest of this section we will write $(\bfg',\bff',\bfx')$ to mean the trips and the disutility thresholds
used by the drivers in our arbitrary equilibrium, and we will use $(\bfg^*,\bff^*,\bfx^*)$ to mean the
optimal solution to the fluid optimization problem and its associated disutility thresholds.

\medskip\noindent \textbf{Where do the non-dispatched drivers go?}
We first show that the non-dispatched drivers are incentivized to accept welfare-optimal trips,
given any set of dispatch trips.  
To show this we consider
the optimization problem (\ref{eq:relocation_problem_}), which depends on the vector $\bfg'$
encoding the dispatch trips occurring under our equilibrium and solves for the corresponding welfare-optimal trips $\bff$.
\begin{align}
    \label{eq:relocation_problem_}
    \sup_{\bff} &\ \ \   \sum_{(\ell,d)\in\calL^2} -c_{(\ell,d)}f_{(\ell,d)} + \calU_{\omega_t}^{>t}(\bff)
 & 
\\
    \mbox{subject to} & \nonumber\\ 
    & f_{(\ell,d)} \geq g'_{(\ell,d)} \ &\forall (\ell,d)\in\calL^2\label{eq:relocation_inequality_constraint}\\
    &
        \sum_{d\in\calL} f_{(\ell,d)} =
        S_\ell 
    \ & \forall \ell\in\calL.\label{eq:relocation_equality_constraint}
\end{align}
where $\calU_{\omega_t}^{>t}(\bff)$ is the future welfare function (\ref{eq:future_welfare}).
The following Lemma shows $\bff'$ is an optimal solution for (\ref{eq:relocation_problem_}).
\begin{lemma}
    \label{lem:relocation_optimality}
    The total trip volumes $\bff'$ from our equilibrium is an optimal solution for the relocation
    problem \eqref{eq:relocation_problem_} with respect to the dispatch trips $\bfg'$.  Moreover, dual
    variables certifying the optimality of $\bff'$ are given by
    \begin{equation}
        \label{eq:relocation_lambda_def}
        \lambda_{(\ell,d)} = \max_{d'}\calQ_t(\ell,d',0) - \calQ_t(\ell,d,0),
    \end{equation}
    associated with the inequality constraint (\ref{eq:relocation_inequality_constraint}) for each $(\ell,d)$,
    and
    \begin{equation}
        \label{eq:relocation_eta_def}
        \eta_\ell = \max_{d'}\calQ_t(\ell,d',0)
    \end{equation}
    associated with the equality constraint (\ref{eq:relocation_equality_constraint}) for each $\ell$.
\end{lemma}
\begin{proof}
    First, the complementary slackness conditions follow immediately from incentive-compatibility
    properties.  In particular, if a nonzero volume of drivers take a relocation trip along $(\ell,d)$, i.e.
    if $f'_{(\ell,d)} - g'_{(\ell,d)} > 0$, then the incentive compatibility conditions (\ref{eq:exact_fluid_ic}) 
    state we must have
    $$
    \calQ_t(\ell,d,0) = \max_{d'}\calQ_t(\ell,d',0),
    $$
    hence $(f'_{(\ell,d)} - g'_{(\ell,d)})\lambda_{(\ell,d)} = 0$ is satisfied for every route $(\ell,d)$.

    It remains to check stationarity. 
    We take the negative objective function to convert (\ref{eq:relocation_problem_}) into
    a convex minimization problem, for compatibility with standard definitions of the Lagrangian and associated optimality
    conditions.    
    The Lagrangian for this convex minimization problem is:
    \begin{equation*}
        L(\bff;\bflambda,\bfeta) =
       c_{(\ell,d)} f_{(\ell,d)} 
        -\calU_{\omega_t}^{>t}(\bff) + \bflambda^T(\bfg'-\bff) + 
\sum_\ell \eta_\ell\left(\sum_d f_{(\ell,d)} - S_\ell\right).
    \end{equation*}

    The stationarity condition we must verify is that $0$ is the (sub-)gradient of $L(\bff';\bflambda,\bfeta)$.  Our induction hypothesis (\ref{eq:backwards_induction_assn}) yields 
\begin{equation*}
\frac{\partial}{\partial f'_{(\ell,d)}} \calU_{\omega_t}^{>t}(\bff') =  
\frac{\partial}{\partial f'_{(\ell,d)}}\bbE\left[\Phi_{\omega_{t+1}}(\bar{\bfS}_{\omega_{t+1}}(\bff'))\mid\omega_t\right]
= \bbE\left[
\frac{\partial}{\partial f'_{(\ell,d)}}
\Phi_{\omega_{t+1}}(\bar{\bfS}_{\omega_{t+1}}(\bff'))\mid\omega_t\right].
\end{equation*}
The derivative 
$\frac{\partial}{\partial f'_{(\ell,d)}}
\Phi_{\omega_{t+1}}(\bar{\bfS}_{\omega_{t+1}}(\bff'))$ is the directional derivative of 
$\Phi_{\omega_{t+1}}$ evaluated at 
$\bar{\bfS}_{\omega_{t+1}}(\bff')$
as we increase the supply at location $d$ (this is the effect on 
$\bar{\bfS}_{\omega_{t+1}}(\bff')$
of increasing the driver flow $f'_{(\ell,d)}$ from $\ell$ into $d$). 
By the induction hypothesis, this directional derivative is 
$\calV_{t+1}(d;\omega_{t+1}, \bar{\bfS}_{\omega_{t+1}}(\bff'))$.

Thus,
\begin{equation*}
\frac{\partial}{\partial f'_{(\ell,d)}} \calU_{\omega_t}^{>t}(\bff') 
= \bbE\left[\calV_{t+1}(d;\omega_{t+1}, \bar{\bfS}_{\omega_{t+1}}(\bff')) \mid \omega_t \right]\\
= \calQ_t(\ell,d,0) + c_{(\ell,d)}.
\end{equation*}
    We evaluate the partial derivative of $L(\bff';\bflambda,\bfeta)$ at each coordinate $f'_{(\ell,d)}$:
    \begin{align*}
        \frac{\partial}{\partial f_{(\ell,d)}} L(\bff';\bflambda,\bfeta) &= 
       c_{(\ell,d)} 
        -\frac{\partial}{\partial f_{(\ell,d)}} \calU_{\omega_t}^{>t}(\bff') - \lambda_{(\ell,d)} + \eta_\ell\\
        &= - \calQ(\ell,d,0) - \lambda_{(\ell,d)} + \max_{d'}\calQ(\ell,d,0)\\
        &= 0.
    \end{align*}
Therefore $\bff'$ is an optimal solution for the problem (\ref{eq:relocation_problem_}).
\end{proof}

\medskip\noindent \textbf{All equilibria serve all available dispatch demand.}
We proceed with the second step of our proof, which is to show that all equilibria that arise under the SSP prices must
serve all available dispatch demand.  Consider the optimum solution $(\bfg^*,\bff^*,\bfx^*)$ and an arbitrary equilibrium $(\bfg',\bff',\bfx')$ serves strictly fewer dispatch trips than the corresponding optimum at some route. Recall that $\bfg' \le \bfg^*$ on all routes, and that $g'_{(\ell,d)} < g^*_{(\ell,d)}$ on any route implies that $x'_{(\ell,d)} < x^*_{(\ell,d)}$. Based on our matching process of drivers to passengers, it is no loss  of generality to assume that drivers that in the optimum should have accepted the passenger for a route  $(\ell,d)$ were offered this drive.  To prove Theorem \ref{thm:robust_eqlbm_fluid} we will use Lemma \ref{lem:relocation_optimality} on an auxiliary networks that we construct next. 

For all drivers that serve dispatches in the equilibrium, assume they start in the next period at the destination of their current dispatch. 
For a location $\ell$ where some route $(\ell, d)$ serves strictly fewer dispatch trips in $\bfg'$ then in $\bfg'^*$, if $\ell$ has $k$ routes $(\ell, d)$ that had passengers to dispatch, we create $k$ copies of $\ell$, each associated with one the the routes $(\ell,d)$, denoted by $\ell^d$. We distribute the drivers from the equilibrium who started at location $\ell$ among the copies such that drivers who were offed dispatch $(\ell,d)$ will start at location $\ell^d$, Note if there are unserved 
dispatches at location $\ell$ than all drivers at $\ell$ were offered a dispatch. Each copy $\ell^d$ is connected to the same set of nodes for the next period as $\ell$ with the same cost, but each has one additional 
possible route: $(\ell^d,d')$ going to the same destination as the drive $d$, but with a different cost. Suppose the price offered to the drivers was $p_{\ell,d}$ and cutoff expected by our optimization is $x^*_{(\ell,d)}$, and the cost of the drive is $c_{(\ell,d)}$ then this additional route will have cost $c_{(\ell,d)}-p_{(\ell,d)}-x^*_{(\ell,d)}$. 

First we construct the optimal solution $(\bfg^{**},\bff^{**},\bfx^{**})$ in this network, which is the same as $(\bfg^*,\bff^*,\bfx^*)$ except drivers serving dispatches in the equilibrium start only in the next period at the destination of their dispatch, there are no dispatches available in period 1, so $\bfg_{(\ell,d)}^{**}=0$ for the first period and $f^{**}$ includes all additional dispatches that $\bff^*$ would serve along the routes with the new cost $c_{(\ell,d)}-p_{(\ell,d)}-x^*_{(\ell,d)}$, and $\bfx^{**}$ is only for later periods of the problem. We call this the linearized problem (where the first period values are linear). We first lemma this is an optimal solution to the linearized program, with the same dual variables as the original optimization problem. 
\begin{lemma}\label{lemma:lin-opt}
The flow $(\bfg^{**},\bff^{**},\bfx^{**})$ is an optimal solution to our modified convex program, with optimal value $OPT^{**}$. 
\end{lemma}

Second, for any route $(\ell,d)$ where not all dispatches are served, we can replace $x^*_{(\ell,d)}$ with its equilibrium cutoff $x'_{(\ell,d)}$, this makes the cost lower, and hence the value of the solution $(\bfg'',\bff'',\bfx'')$ is now strictly larger than $OPT''$.

\begin{lemma}\label{lemma:lin-larger-opt}
The optimal solution to our modified convex program, using $x'$ in places of $x^*$ is strictly larger by at least $\sum_{(\ell,d)}(g^*_{(\ell,d)}-g'_{(\ell,d)})(x^*_{(\ell,d)}-x'_{(\ell,d)})$ unless all dispatches are served.
\end{lemma}

Next consider the equilibrium solution $(\bfg',\bff',\bfx')$ on the modified network. To be precise, drivers who accepted dispatches again are starting in the next period at the destination of the dispatch. All other drives either were not offered dispatches or rejected their offer and choose an alternate drive instead. We will call this $(\bfg'',\bff'',\bfx'')$.  The resulting solution is an equilibrium of this network, and actually this is true both with $x^*$ and with $x'$ in our problem. 

\begin{lemma}
The flow $(\bfg'',\bff'',\bfx'')$ is an equilibrium on our modified network 
either with cost $c_{(\ell,d)}-p_{(\ell,d)}-x^*_{(\ell,d)}$ or with costs 
$c_{(\ell,d)}-p_{(\ell,d)}-x'_{(\ell,d)}$ of the drives with modified costs. 
\end{lemma}

Now by Lemma \ref{lem:relocation_optimality} this equilibrium solution is optimal with both version of the problem. Using $x^*$ to define cost we get the the equilibrium solution has value $OPT^{**}$. However, using $x'$ does not change the value of the equilibrium, while the optimum increases by Lemma \ref{lemma:lin-larger-opt} unless all dispatches are served.

\medskip\noindent \textbf{All equilibria are welfare optimal.}
We now finish the backwards induction proof.  
First, we claim that the thresholds $\bfx'$ from our arbitrary equilibrium have to equal the optimal thresholds $\bfx^*$ 
associated with the optimal solution $(\bff^*,\bfg^*)$.  This follows from the fact we just established, that $\bfg'=\bfg^*$.
Consider the case that all drivers at a location $\ell$ serve a dispatch.  Then no drivers can reject a dispatch, so it must be the
case that $x_{(\ell,d)}=x^*_\ell=C$ for all destinations $d$.  Otherwise, consider the case that some nonzero volume of drivers serve a relocation
trip originating from $\ell$.  Then from Lemma \ref{lem:relocation_optimality}, $\bff'$ and $\bff^*$ are both optimal solutions for the problem (\ref{eq:relocation_problem_}).
Let $\bflambda^*$ and $\bfeta^*$ be the dual variables (\ref{eq:relocation_lambda_def}) and (\ref{eq:relocation_eta_def}) associated with the flow $\bff^*$.
Let $d'$ be the destination where $f'_{(\ell,d')} > g'_{(\ell,d')}$. From the equality $0=\frac{\partial}{\partial f_{(\ell,d)}} L(\bff';\bflambda^*,\bfeta^*)$
we conclude 
$$
\max_d\calQ_t(\ell,d,0) = \calQ_t(\ell,d',0) = \calQ^*_t(\ell,d',0) \leq \max_d\calQ^*_t(\ell,d,0).
$$
Analogously, we can obtain the bound
$$
\max_d\calQ^*_t(\ell,d,0) = \calQ_t(\ell,d^*,0) = \calQ_t(\ell,d^*,0) \leq \max_d\calQ_t(\ell,d,0).
$$
Therefore the relocation utilities are the same under our arbitrary equilibrium and the optimal solution.  The incentive compatibility conditions (\ref{eq:exact_fluid_ic})
give 
$
P_{(\ell,d)} - x_{(\ell,d)} = \max_{d'}\calQ_t(\ell,d',0),
$
and we also have
$
P_{(\ell,d)} - x^*_{\ell} = \max_{d'}\calQ^*_t(\ell,d',0),
$
therefore $x_{(\ell,d)} = x^*_\ell$ follows.  

Therefore, the Q-values $\calQ_t(\ell,d,0)$ and $\calQ_t(\ell,d,1,X)$ associated with our arbitrary equilibrium take the same value
as the Q-values $\calQ^*_t(\ell,d,0)$ and $\calQ^*_t(\ell,d,1,X)$ associated with our optimal solution, which implies the value function $\calV_t(\ell;\omega_t,\bfS_t)$
is equal to the optimal value function $\calV^*_t(\ell;\omega_t,\bfS_t)$.  But we know the optimal value function is equal to $\frac{\partial}{\partial S_\ell}\Phi_{\omega_t}(\bfS_t)$,
establishing the second and final part of our backwards induction assumption.

We defer the proof of approximate welfare robustness in the fluid model to Appendix \ref{appdx:approx_robustness_fluid_proof}.




\Xomit{
\section{Approximate Incentive-Compatibility of the SSP Mechanism in the Stochastic Two-Level Model}
\label{sec:sspm_approx_ic}
In this section we analyze incentive-compatibility of the SSP mechanism in the stochastic two-level model.
Specifically, we state and prove a formal version of Theorem \ref{thm:robust_eqlbm_two_level}, asserting
that under the SSP mechanism, every approximate equilibrium in the stochastic two-level model achieves
approximately optimal welfare.  We restate Theorem \ref{thm:robust_eqlbm_two_level} below.
\begin{theorem*}[Theorem \ref{thm:robust_eqlbm_two_level}: Approximate welfare-robustness in the stochastic two-level model]
When the SSP mechanism is used to set prices in the stochastic two-level model, every approximate equilibrium
achieves approximately optimal welfare.
\end{theorem*}

\subsection{Preliminaries}
We begin this section by summarizing important properties about the SSP mechanism and strategic behaviour in the stochastic two-level model.
Specifically, we summarize the following properties:
\begin{enumerate}
\item The state-dependent optimization formulation (\ref{eq:dynamic_optimization_recursive}), from which the SSP prices and matches are derived,
has continuous partial derivatives as a function of the supply-volume input at each location, and the value of these partial derivatives corresponds
to the expected utility-to-go of a driver positioned at that location.
\item Clarify precisely what we mean by approximate equilibrium in the stochastic two-level model.
\item Re-state our welfare-robustness results for the stochastic fluid model from the previous section.
\item Prove that the optimal welfare from the stochastic fluid model is an upper bound on the optimal welfare in the stochastic two-level model.
\end{enumerate}

\section{Old version of section: Asymptotic Incentive-Compatibility in the Stochastic Two-Level Model}
\label{sec:stochastic_ic}
In the previous section, under the stochastic fluid model, we showed that welfare-optimal behavior is incentive compatible under both the static and dynamic mechanisms. The stochastic fluid model includes randomness from the top-level scenario that can affect the entire city, but not idiosyncratic randomness due to individual behavior. Here, we consider the stochastic two-lever model, which includes this idiosyncratic randomness, in the limit as the number of riders and drivers grows large.

Both the static and dynamic mechanisms remain asymptotically welfare optimal in this limit\footnote{The static mechanism must be amended to deal with infeasibility of planned flows created by randomness. This can be done by allocating the minimum of the number of passengers wishing to ride at the chosen price and the number of drivers planned for allocation to riders on that route. Similarly, a small change is needed for the dynamic mechanism, discussed below.} due to a straightforward argument using the law of large numbers and continuity of welfare in the dispatch-request volume and supply location vectors.
However, characterizing all equilibria 
is significantly more subtle.

In this section, we show that the dynamic mechanism is asymptotically incentive compatible, \pfedit{subgame perfect (that is, it remains incentive compatible starting from every state of the world),}
and all approximate equilibria of the mechanisms are approximately welfare optimal (Theorem~\ref{thm:stochastic_ic}). The static mechanism, however, is only globally approximately incentive compatible, and not subgame perfect, and can have equilibria with low social welfare. In Section~\ref{sec:example} we show that this is a serious problem via a simple example.
Indeed, the dynamic mechanism's recomputation of prices is necessary to provide the robustness of the mechanism.

We now proceed to define a version of the dynamic mechanism capable of dealing with infeasibilities created by idiosyncratic randomness in  
the stochastic two-level model, and then state a sequence of lemmas that culminate in our main result. 

First, in the context of the stochastic two-level model, we refer to the dynamic mechanism as the stochastic spatiotemporal pricing and matching (SSP) mechanism.
Just as in the stochastic fluid model, the SSP mechanism re-computes an optimal stochastic flow at each time $t$ based on the observed
supply-location vector $\bfS_t$ and the stochastic scenario $\omega_t$.
The main difference between the SSP mechanism in the two-level model and the dynamic mechanism in the fluid model is that
the trip-allocation counts (\ref{eq:trip_allocation_counts}) used in the fluid model are not necessarily feasible
trip-allocation counts in the two-level model.  There are two reasons why the fluid trip-allocation counts do not carry
over to the two-level setting. The first reason is that the optimal stochastic flow can have arbitrary fractional solutions,
whereas drivers and riders in the two-level model are indivisible units of size $1/k$, so some amount of rounding might be
necessary to account for the indivisibility of drivers and riders in this setting.  More importantly,  
dispatch-request
volume in the two-level model is stochastic, and the rider demand may not fully match 
the dispatch-trip
volume specified in equation (\ref{eq:trip_allocation_counts}) along any particular route.

As such, the SSP mechanism produces a matching vector $\bfg$ that only approximately matches the trip-suggestion
counts (\ref{eq:trip_allocation_counts}) obtained from the optimal stochastic flow.
We defer the exact details of how the SSP mechanism obtains the approximate matching vector $\bfg$ to Algorithm
\ref{alg:matching_details} in Appendix \ref{subsec:matching_details}.
\pfedit{Essentially,}
Algorithm \ref{alg:matching_details} produces the matching vector $\bfg$ \pfedit{by producing} 
a feasible matching vector that is as close as possible to the fluid trip-allocation counts (\ref{eq:trip_allocation_counts}) and 
\pfedit{then randomizes} among drivers when not all drivers can be assigned a passenger trip.

We state our formal definition of the SSP mechanism in the stochastic two-level model below.
\begin{definition}
    \label{def:sspm}
    The SSP mechanism for the stochastic \pfedit{two-level} model re-computes in each time period $t$ an optimal stochastic flow $\bff^*_t$ 
    for the stochastic maximum flow problem (\ref{eq:dynamic_optimization}), constructed with respect to the scenario $\omega_t$
    and the supply-location vector $\bfS_t$.  The SSP mechanism produces prices in time period $t$ using equation (\ref{eq:max_welfare_price})
    with respect to the re-computed flow $\bff^*_t$. 
    The SSP mechanism produces a matching vector $\bfg$ at each time period by calling Algorithm \ref{alg:matching_details}, which takes
    as arguments the optimal flow $\bff^*_t$ and the observed dispatch-request counts $\bfd=(d_{(r,\omega_t)} : (r,\omega_t)\in\calN_{\omega_t}^t)$.
\end{definition}

Recall that, for a driver positioned at an active location $(l,\omega_t)$, the expected utility-to-go 
collected by this driver, as a function of the supply locations $\bfS_t$, is defined by the recursive equation (\ref{eq:utility_recursion}), restated below:
$$
\calV^k_{(l,\omega_t)}(\bfS_t) = \bbE^k\left[\delta\rho_t(r,\omega_t) - c_{(r,\omega_t)} + \calV^k_{(l^+(r),\omega_{t+1})}(\bfS_{t+1})\right].
$$
The above expectation is taken over multiple sources of randomness:
\begin{itemize}
\item The random matching vector $\bfg$, which is produced by Algorithm \ref{alg:matching_details} and depends on the random dispatch-request counts
    $(d_{(r,\omega_t)} : (r,\omega_t)\in\calA_{\omega_t}^t)$.
\item The random trip-suggestion $(r,\omega_t,\delta)$ allocated to a driver positioned at $(l,\omega_t)$.  The distribution $\bfg(l,\omega_t)$
    from which this trip-suggestion is sampled is specified in equation (\ref{eq:trip_suggestion_dist}) and depends on the realized matching vector $\bfg$.
\item The random time $t+1$ scenario $\omega_{t+1}\sim\bbP(\cdot\mid\omega_t)$.
\item The time $t+1$ supply-location vector $\bfS_{t+1}$ which depends on
    the scenario $\omega_{t+1}$, the matching vector $\bfg$,
    and the random new-driver counts $M_{(l,\omega_{t+1})}$ at each active time $t+1$ LT-scenario $(l,\omega_{t+1})\in\calN_{\omega_{t+1}}^{t+1})$.
\end{itemize}

For the rest of this section we will analyze the utility that drivers collect by following the SSP mechanism,
and we will ultimately show that following trip-suggestions under the SSP mechanism is (asymptotically) incentive-compatible.
We proceed by stating a series of Lemmas which characterize driver-utility in the two-level model.

Our first lemma focuses on the expected immediate reward collected by a driver under the SSP mechanism.
Note that, if a driver is allocated a trip-suggestion $(r,\omega_t,\delta)$ by the dynamic mechanism in the fluid model, then
the immediate reward they collect is equal to $\rho_{(r,\omega_t)} - c_{(r,\omega_t)}$ regardless of whether or not the
trip-suggestion is a dispatch-trip or relocation-trip, because the trip-price $\rho_{(r,\omega_t)}$ will be $0$ if the
dynamic mechanism allocates a nonzero volume of relocation-trips along $(r,\omega_t)$.
However, in the stochastic two-level model, if the realized dispatch-request demand is less than expected, 
the SSP mechanism might be forced to allocate relocation trip-suggestions $(r,\omega_t,0)$ for which the price $\rho_{(r,\omega_t)}$ is nonzero
to a small number of drivers.
The following lemma shows that the probability of receiving such a dispatch vanishes as the population-size increases.
\begin{lemma}
    \label{lem:immediate_reward}
    Consider a driver positioned at an active location $(l,\omega_t)$ in the stochastic two-level model at a time period $t$ and scenario $\omega_t$. 
    Fix any supply-location vector $\bfS_t$ and let $(r,\omega_t,\delta)$ be the random trip-suggestion allocated to this driver under the SSP mechanism.
    For each value of the population-size parameter $k$, let
    $$
    p_k = \bbP^k\left(\delta\rho_{(r,\omega_t)}\neq \rho_{(r,\omega_t)}\right)
    $$
    be the probability that the driver is allocated a relocation trip-suggestion for which the trip-price is nonzero, in the large-market setting with size $k$.
    Then $p_k\to 0$ as $k\to\infty$.
\end{lemma}
The proof of Lemma \ref{lem:immediate_reward} is given in Appendix \ref{subsec:immediate_reward_proof}.

Our next Lemma characterizes to what extent the time $t+1$ supply-location vector $\bfS_{t+1}$ can differ from the time $t+1$ locations
that would have occurred in the fluid setting; let $\bfS_{t+1}^*$ denote this time $t+1$ fluid supply-location vector.
For each scenario $\omega_{t+1}$ the components of $\bfS^*_{t+1}$ are
specified by:
\begin{equation}
    \label{eq:fluid_supply_loc_vec}
    S^*_{(l,\omega_{t+1})} = \bar{M}_{(l,\omega_{t+1})} + \sum_{(r,\omega_t)\in A_{\omega_{t+1}}^-(l)} f^*_{(r,\omega_t)},
\end{equation}
and the components of $\bfS_{t+1}$ are specified by:
\begin{equation}
    \label{eq:stochastic_supply_loc_vec}
    S_{(l,\omega_{t+1})} = M_{(l,\omega_{t+1})} + \sum_{(r,\omega_t)\in A_{\omega_{t+1}}^-(l)} \left(g_{(r,\omega_t,0)} + g_{(r,\omega_t, 1)}\right).
\end{equation}
In equation (\ref{eq:fluid_supply_loc_vec}), $\bar{M}_{(l,\omega_{t+1})}$ is the expected volume of entering drivers at the LT-scenario $(l,\omega_{t+1})$, and the
second term is the sum of all drivers that the time $t$ optimal stochastic flow allocates to routes which terminate at $(l,\omega_{t+1})$.
In equation (\ref{eq:stochastic_supply_loc_vec}), $M_{(l,\omega_{t+1})}$ is the random volume of entering drivers at LT-scenario $(l,\omega_{t+1})$, and the
second term is the sum of all trip-volume that the SSP matching vector $\bfg$ allocates towards $(l,\omega_{t+1})$.

The following Lemma states, in the limit as the population-size $k$ grows, assuming that the market is operated using the SSP mechanism,
that the time $t+1$ supply-locations $\bfS_{t+1}$ approach the fluid supply-locations $\bfS_{t+1}^*$ with high probability.
\begin{lemma}
    \label{lem:future_reward}
Consider the stochastic two-level model at a time period $t$ and scenario $\omega_t$ in the large-market setting with
population size parameter $k$.  Fix any time $t$ supply-location vector $\bfS_t$.
    Let $\bfS^*_{t+1}$ be the fluid time $t+1$ supply-location vector, as defined in equation (\ref{eq:fluid_supply_loc_vec}),
    and let $\bfS_{t+1}$ be the stochastic time $t+1$ supply-location under the SSP mechanism, 
    as defined in equation (\ref{eq:stochastic_supply_loc_vec}).
    Then there exists a sequence $\epsilon_k\to0$ as $k\to\infty$ such that the following holds:
    \begin{equation}
        \label{eq:supply_loc_vec_concentration}
    \bbP^k\left(\|\bfS^*_{t+1} - \bfS_{t+1}\|_1 \leq \epsilon_k\right) \geq 1-\epsilon_k.
    \end{equation}
\end{lemma}
The proof of Lemma \ref{lem:future_reward} is given in Appendix \ref{subsec:future_reward_proof}.

The following Lemma asserts that, in the limit as the population-size $k$ goes to infinity, the expected utility-to-go
value for a driver positioned at $(l,\omega_t)$ is equal to the optimal dual variable associated with that location.
\begin{lemma}
    \label{lem:large_market_asymptotic_utility}
    Consider a driver positioned at an active location $(l,\omega_t)$ in the stochastic two-level model at a time period $t$ and scenario $\omega_t$. 
    Fix any supply-location vector $\bfS_t$ and let $\eta^*_{(l,\omega_t)}(\bfS_t)$ be the optimal dual variable associated with the
    $(l,\omega_t)$ constraint for the stochastic maximum flow problem (\ref{eq:dynamic_optimization}) over the  subnetwork induced by $\omega_t$ and
    the supply-location vector $\bfS_t$.
    Then, as the population size $k\to\infty$, the expected utility-to-go for a driver positioned at $(l,\omega_t)$ under the SSP mechanism
    converges to the optimal dual value $\eta^*_{(l,\omega_t)}(\bfS_t)$, i.e. the following holds:
    \begin{equation}
        \label{eq:large_market_asymptotic_utility}
    \lim_{k\to\infty}\calV^k_{(l,\omega_t)}(\bfS_t) = \eta^*_{(l,\omega_t)}(\bfS_t).
    \end{equation}
\end{lemma}
We defer the proof of Lemma \ref{lem:large_market_asymptotic_utility} to Appendix \ref{subsec:large_market_asymptotic_utility}.
The proof of Lemma \ref{lem:large_market_asymptotic_utility} follows from a backwards induction argument that uses Lemmas \ref{lem:immediate_reward}
and \ref{lem:future_reward}, in addition to another Lemma we state in the appendix establishing that the function $\bfS_t\mapsto\eta^*_{(l,\omega_t)}(\bfS_t)$
mapping supply-locations to optimal dual variables is continuous.

Finally, our main result in this section is a theorem asserting asymptotic incentive-compatibility in the limit as $k\to\infty$.
Our incentive-compatibility result includes a minimum-incentive-to-deviate parameter $\epsilon$.
We assume that a trip-suggestion is incentive-compatible if the maximum benefit that a driver can accrue by deviating from their
trip-suggestion is less than $\epsilon$. In this sense, the trip-suggestion is approximately incentive-compatible with additive error $\epsilon$. 
In Theorem \ref{thm:stochastic_ic} we show that we can pick a sequence of incentive-to-deviate parameters $\epsilon_k$ such that $\epsilon_k\to 0$ as $k\to\infty$,
and the probability of being allocated an $\epsilon_k$-incentive-compatible trip-suggestion goes to $1$ as $k\to\infty$.
\begin{theorem}
    \label{thm:stochastic_ic}
    Consider a driver positioned at an active location $(l,\omega_t)$ in the stochastic two-level model at a time period $t$ and scenario $\omega_t$ in the market using the SSP mechanism. 
    Then for any supply-location vector $\bfS_t$ there exists a sequence of error terms $\epsilon_k\to 0$ such that the probability the driver is allocated
    an $\epsilon_k$-incentive-compatible trip-suggestion under the SSP mechanism goes to $1$ as $k\to\infty$.
    
    Further, under the SSP mechanism,
    the expected social welfare resulting from 
    any sequence of $\epsilon_k$-incentive-compatible driver strategy profiles 
    (in which drivers accept passengers when they are otherwise utility-indifferent)
    approaches the optimal expected social welfare of the stochastic fluid model as $k \to \infty$.
\end{theorem}

The proof of Theorem \ref{thm:stochastic_ic} follows from a backwards induction argument invoking Lemmas \ref{lem:immediate_reward} and \ref{lem:large_market_asymptotic_utility}.
We defer the proof details to Appendix \ref{subsec:stochastic_ic} and the full version.

}
\section{The Value of Re-Solving}
\label{sec:example}
The SSP mechanism changes prices based on drivers' locations: in each period $t$, the fluid optimization problem \eqref{eq:fluid_opt} used to set prices is re-solved using drivers' locations $\bfS_t$. Here we show that re-solving is necessary, in the sense that 
Theorem~\ref{thm:robust_eqlbm_two_level}'s
robustness property vanishes without it.

To show this, we consider a variant of the SSP mechanism, called the {\it static mechanism}. This mechanism solves the optimization problem \eqref{eq:fluid_opt} once, for the initial market state, and re-uses this solution to compute prices following the same approach as the SSP mechanism.
To define the static mechanism formally, we first observe that solving \eqref{eq:fluid_opt} for the initial market state defines an optimal anticipated sequence of supply location vectors $\bfS_{\omega_{t}}^*$ and flows $\bff_{\omega_{t}}^*$, $\bfg_{\omega_{t}}^*$ indexed by time $t$ and scenario $\omega_t$.
These satisfy the forward recursion:
$\bff^*_{\omega_0} = 0$;
$\bfS_{\omega_{t+1}}^* = \bar{\bfS}_{\omega_{t+1}}(\bff^*_{\omega_t})$
via \eqref{eq:fluid_state_transition}
where $(\bff^*_{\omega_t},\bfg^*_{\omega_t})\in F^*_{\omega_t}(\bfS^*_{\omega_t})$ is an optimal solution to the fluid optimization problem.
Then, the price set by the static mechanism in scenario $\omega_t$ on route $\ell,d$ is $P_{(\ell,d,\omega_t)}(g^*_{(\omega_t,\ell,d)})$ as defined in Definition~\ref{def:sspm_pricing},
i.e., the price needed to have a flow $\bfg^*_{\omega_t}$ of riders requesting trips in the stochastic fluid model.

One can show that this static mechanism has a welfare-optimal equilibrium under the stochastic fluid model: this is the strategy profile implied by the solution to \eqref{eq:fluid_opt}. Moreover, there exists a sequence of approximate equilibria in the two-level model indexed by the population size $k$ that are asymptotically welfare-optimal: those corresponding to this same strategy profile.

Unfortunately, however, the static mechanism is not robust, in the sense that poor equilibria (both exact and approximate) can exist, in contrast with the SSP's Theorem~\ref{thm:robust_eqlbm_two_level}. Essentially, the issue is that prices do not react to deviations between the actual supply location vector and the one anticipated by solving \eqref{eq:fluid_opt}.

We demonstrate this with a simple two-time-period one-location one-scenario example.
In the first period there are $k$ drivers who join the market in the one location.
Each driver can exit the market and collect utility $E$, or stay 
for the second period and hope to serve a dispatch.
In the second period there are no new drivers who join, and the number of riders interested in taking a trip (from the one location to itself)
is $D\sim\mathrm{Binomial}(k,\frac12)$. The distribution of the rider value for taking a trip is $V\sim\mathrm{Uniform}(0,1)$, 
there is $c=0$ cost to a driver for serving a dispatch, \pfedit{and no add-passenger disutility ($C=0$).}

In the fluid model for this example, the welfare of having $S_2$ drivers available to serve dispatches at the beginning of the second period ($t=2$) is 
$$
\Phi_2(S_2) = \begin{cases}
    \frac{1}{4}k& \mbox{if }S_2 > k/2,\\
    S_2 - \frac{S_2^2}{k}    &
    \mbox{if } S_2 \leq k/2.
\end{cases}
$$
We can find the welfare-optimal $S_2$ by solving $\frac{d}{dS_2} \Phi_2(S_2)=E$.
Setting $E=\frac{1}{2}$, welfare optimality is obtained at $S_2 = \frac{k}{4}$, so $\frac{3k}{4}$ drivers should exit the market in the first period.
The trip-price set by the platform in this case is $P=\frac{1}{2}$, which correctly selects the $50$\%
of the $\frac{k}{2}$ price-inquiring riders with the highest value $V$.  In the fluid setting, all drivers collect utility $\frac{1}{2}$ by following
the welfare-optimal solution, and no driver has incentive to deviate under the static price $P=\frac{1}{2}$.

However, incentives break down if the static price $P=\frac{1}{2}$ is used in the two-level model.  From the perspective
of a driver in the first period, and relative to the fluid model, the utility of exiting at period $1$ remains $\frac{1}{2}$, but the utility
of staying is lower because receiving a dispatch is not guaranteed. While the 
probability of dispatch goes to $1$ as $k$ goes to $\infty$, and hence staying for the second period is approximately incentive compatible for the $k/4$ drivers, the decision to exit the market will always dominate the decision to stay in the market,
for all drivers and for all finite values of $k$, assuming the static price $P=\frac{1}{2}$ is used. If many drivers leave that results in significant welfare loss.

SSP's approach (using dynamic prices based on recomputing an optimal solution in the second period) solves this problem.
Adapting the price to the observed amount of driver volume $S_2$ and computing the optimal solution with respect
to the expected rider volume interested in taking a trip produces a trip-price $P=\frac{d}{dS_2}\Phi_2(S_2)$, where $S_2$ is now the observed volume of drivers that remain at the start of the second period.  
If $S_2$ is lower than the value of $\frac{k}{2}$ anticipated in the fluid solution, then $P$ will be larger than the anticipated price of $\frac12$.
From the perspective of
a driver in the first period, the utility of staying in the market in the first period is $\frac{d}{dS_2}\Phi_2(S_2)\bbP^k(\mathrm{dispatch} \mid S_2),$
where $\bbP^k(\mathrm{dispatch} \mid S_2)$ is the probability of receiving a dispatch in the second period and also increases as $S_2$ falls.
Since a driver's utility $\frac{d}{dS_2}\Phi_2(S_2)$ is increasing as $S_2$ decreases, it is no longer a dominant strategy, or even an equilibrium of the game for all drivers to exit the market in the first period.




\bibliographystyle{ACM-Reference-Format}
\bibliography{ec22}

\appendix

\section{Approximate Incentive-Compatibility of the Fluid Optimal Solution in the Two-Level Model}
\label{appdx:two_level_approx_ic}
In this section we prove part \ref{thm:item:opt_ic_stochastic} of Theorem \ref{thm:opt_ic}. We start by proving the following
Lemma, which shows that expected driver utilities in the two-level model are approximately equal to
driver utilities in the fluid model.
\begin{lemma}
\label{lem:two_level_utility}
    There exist nonnegative sequences $(\epsilon_k:k\geq 1)$ and $(\beta_k:k\geq 1)$, both converging to $0$ as  $k\to\infty$, such that
        when the SSP mechanism is used to set prices in the stochastic two-level model with population-size $k$, then 
        for any market state $(\omega_t,\bfS_t)\in\Omega_t\times\calS_t$, and any location $\ell$ with driver-volume larger than
        $\beta_k$, i.e.  $S_\ell \geq \beta_k$, we have the expected utility of drivers at $\ell$ is at most $\epsilon_k$ away
        from the fluid utility for drivers at $\ell$, i.e. 
        $$ 
        \left|\bbE[U_i^t\mid\ell_i^t=\ell] - \frac{\partial}{\partial S_\ell}\Phi_{\omega_t}(\bfS_t)\right| \leq \epsilon_k.
        $$  
\end{lemma}
\begin{proof}
Fix a time period $t$ and we assume via backwards induction that, at future time periods $t+1$, drivers who are positioned
at a location $\ell$ collect expected utility that is approximately the same as the fluid counterpart.
That is, assume there exist sequences $(\epsilon^{t+1}_k)$, $(\beta_k^{t+1})$ converging to $0$ as $k\to\infty$, such that 
for any time $t+1$ state $(\omega_{t+1},\bfS_{t+1})\in\Omega_{t+1}\times\calS_{t+1}(\gamma_{t+1})$ we have
that the location-specific value function for the two-level model uner $\Pi^{(k)}$ is within $\epsilon^{t+1}_k$ of the
value function for the corresponding two-level model, for any location $\ell$ where $S_\ell \geq \beta_k^{t+1}$.
That is,
\begin{equation}
\label{eq:two_level_utility_bi_assn}
|\bbE[U_i^{t+1}\mid\ell_i^t=\ell] - \frac{\partial}{\partial S_\ell}\Phi_{\omega_{t+1}}(\bfS_{t+1})|\leq \epsilon^{t+1}_k
\end{equation}
holds for every location $\ell$ where $S_\ell \geq \beta_k^{t+1}$,
where $i$ is a driver positioned at location $\ell$, and $U_i^{t+1}$ is a random variable specifying the utility collected
by the driver starting from time period $t+1$ onwards under the strategy profile $\Pi^{(k)}$.

Fix any state $(\omega_t,\bfS_t)$ and let $(\bfx,\bfe)=\Sigma_t^*(\omega_t,\bfS_t)$ be the disutility acceptance thresholds 
$\bfx=(\bfx_\ell:\ell\in\calL)$
and relocation destination distributions $\bfe=(\bfe_\ell:\ell\in\calL)$ selected by the fluid optimal strategy $\Sigma^*$.
Let $(\bff^*,\bfg^*)\in F_{\omega_t}^*(\bfS_t)$ be the optimal solution that the SSP mechanism uses to set prices.
Let $(\bar{\bff}_k ,\bar{\bfg}_k)$ be the fluid outcomes associated with the market state $(\omega_t,\bfS_t)$ and strategy-profile $\Pi^{(k)}$
(\ref{def:mp_fluid_outcome}).
Recall that, by definition, under the strategy profile $\Pi^{(k)}$, every driver positioned at $\ell$ uses $\bfx_\ell$
as their threshold vector, and the distribution of relocation destinations selected by drivers at $\ell$ is equal to a rounded
version of $\bfe_\ell$.
The fluid outcomes are deterministic functions of the common disutility threshold vector $\bfx_\ell$ and the relocation distribution used
by drivers at $\ell$
(see equations (\ref{eq:single_dest_fluid_dispatch}-\ref{eq:single_dest_fluid_undispatch}), and Definition \ref{def:matching_process}).
Since the relocation distribution used
by drivers at $\ell$ under $\Pi^{(k)}$ is a rounded version of $\bfe_\ell$, and the rounding error is on the order of $\frac{1}{k}$ for
where $k$ is the population-size parameter, it follows that 
$(\bar{\bff}_k,\bar{\bfg}_k)$ converges uniformly to $(\bff^*,\bfg^*)$ over all states $(\omega_t,\bfS_t)\in\Omega_t\times\calS_t(\gamma_t)$,
i.e.
$$\sup_{(\omega_t,\bfS_t)\in\Omega_t\times\calS_t(\gamma_t)} \|\bar{\bff}_k-\bff^*\| + \|\bar{\bfg}_k - \bfg^*\| \to 0,\ \ \mbox{as}\  k\to\infty.
$$

Next, concentration properties for the matching process tell us that the stochastic actions which occur under $\Pi^{(k)}$ converge to their
deterministic fluid counterparts. Specifically, let $(\bff_k,\bfg_k)$ be (stochastic) vectors which encode the actions taken under $\Pi^{(k)}$
with respect to a market state $(\omega_t,\bfS_t)\in\Omega_t\times\calS_t(\gamma_t)$.  Lemma \ref{lem:asymptotic_concentration} states
there exist nonnegative sequences $\alpha_k$ and $q_k$, both converging to $0$ as $k\to\infty$, such that 
\begin{equation}
\label{eq:uniform_trip_convergence}
\sup_{(\omega_t,\bfS_t)\in\Omega_t\times\calS_t(\gamma_t)} \bbP\left(\|\bar{\bff}_k-\bff_k\| + \|\bar{\bfg}_k - \bfg_k\| \leq \alpha_k\right) \geq
1-q_k.
\end{equation}
holds for all $k$.

Define the sequence $(\beta_k)_{k=1}^\infty$ by setting $\beta_k=\sqrt{\alpha_k}$ for every $k\geq 1$.  Since we know that $\alpha_k\to 0$ 
as $k\to\infty$, it follows that $\beta_k\to 0$ as $k\to\infty$, as required by our theorem statement.

Now consider a supply-location vector $(\omega_t,\bfS_t)\in\Omega_t\times\calS_t(\gamma_t)$, and consider a location $\ell$ with $S_\ell\geq\beta_k$.
Consider the expected utility collected by a driver $i$ positioned at $\ell$: 
$$
\bbE[U_i^t \mid \ell_i^t=\ell] = \sum_{d\in\calL\cup\{\emptyset\},\delta\in\{0,1\}}
\bbE[U_i^t \mid a_i^t=(\ell,d,\delta)]\bbP(a_i^t=(\ell,d,\delta)\mid\ell_i^t=\ell),
$$
where $a_i^t=(\ell,d,\delta,X)$ denote the action they take, where
$\delta\in\{0,1\}$ is an indicator specifying whether or not it is a dispatch trip, and $X\in [0,C]$ is their sampled
add-passenger disutility.
Under the backwards induction assumption (\ref{eq:two_level_utility_bi_assn}), the driver $i$ has expected utility at time $t$ given by
\begin{align*}
\bbE[U_i^t \mid a_i^t = (\ell,d,\delta)] &= \delta(P_{(\ell,d)} - c_{(\ell,d)} - X) + 
    \bbE[\frac{\partial}{\partial S_d}\Phi_{\omega_{t+1}}(\bfS_{t+1})\mid a_i^t] + \epsilon_k^{t+1}\\
&= \delta(P_{(\ell,d)} - c_{(\ell,d)} - X) + 
    \bbE[\frac{\partial}{\partial S_d}\Phi_{\omega_{t+1}}(\bfS_{t+1})] + \gamma_k+ \epsilon_k^{t+1}\\
&= \delta(P_{(\ell,d)} - c_{(\ell,d)} - X) + 
    \bbE[\frac{\partial}{\partial S_d}\Phi_{\omega_{t+1}}(\bar{\bfS}_{\omega_{t+1}}(\bff^*))] + \psi_k + \gamma_k+ \epsilon_k^{t+1}\\
\end{align*}
where $\epsilon_k^{t+1}$ is the error term bounding the difference between the time $t+1$ expected utility of agent $i$
and the partial derivative of the state-dependent optimization function, which exists by our backwards induction
assumption (\ref{eq:two_level_utility_bi_assn}), assuming, for now, the destination $d$ has sufficiently many drivers for
the backwards induction assumption to hold. The backwards induction assumption only holds if $S_d \geq \beta_k^{t+1}$, but
$\beta_k^{t+1}$ is vanishingly small as $k\to\infty$, so the proportion of drivers who drive towards destinations satisfying
this condition goes to one as $k\to\infty$.

$\gamma_k$ is the error we pay for going from the distribution of 
the time $t+1$ supply-locations $\bfS_{t+1}$
conditional on $a_i^t$ to the unconditional distribution on $\bfS_{t+1}$.  By Assumption \ref{assn:matching_process}
we know that there exists a constant $\gamma_k$ that bounds the difference between the conditional and unconditional
distribution of $\bfS_{t+1}$ for all initial states $(\omega_t,\bfS_t)$ and all strategy profiles, and that
$\gamma_k\to 0$ as the population size $k$ tends to $\infty$.

$\psi_k$ is an error term that bounds the difference between the expected partial derivative of the state-dependent optimization
function with respect to the stochastic time $t+1$ supply-location vector $\bfS_{t+1}$, and the fluid time $t+1$ supply-location
vector $\bar{\bfS}_{\omega_{t+1}}(\bff^*)$, which is a deterministic function of the time $t+1$ scenario $\omega_{t+1}$ and the
fluid optimal trips $\bff^*$.  
We know there exists a constant $\psi_k$ that bounds this difference, such that $\psi_k\to 0$ as $k\to\infty$,
 because of the uniform convergence described in
equation (\ref{eq:uniform_trip_convergence}), as well as the fact that the partial derivative function 
$\frac{\partial}{\partial S_d}\Phi_{\omega_{t+1}}(\cdot)$ is bounded and continuous over a compact domain.

Lemma \ref{lem:Q_values} gives us the following expression for Q-values in the fluid model:
$$
\calQ_t(\ell,d,1,x_{(\ell,d)}) = \frac{\partial}{\partial S_\ell}\Phi_{\omega_t}(\bfS_t) 
- \frac{C}{2}\left(\frac{\bfg^{*T}\bfone_\ell}{\bff^{*T}\bfone_\ell}\right)^2 = \max_{d'}\calQ_t(\ell,d',0),
$$
so therefore we have
\begin{align*}
\bbE[U_i^t \mid a_i^t = (\ell,d,\delta)] = 
\frac{\partial}{\partial S_\ell}\Phi_{\omega_t}(\bfS_t) 
- \frac{C}{2}\left(\frac{\bfg^{*T}\bfone_\ell}{\bff^{*T}\bfone_\ell}\right)^2
+ \delta(x_{(\ell,d)} - X) + \epsilon_k,
\end{align*}
where $\epsilon'_k=\psi_k + \gamma_k+ \epsilon_k^{t+1}$ is the sum of all the errors accrued by approximating the stochastic
utility-to-go with the fluid utility-to-go,
and $\delta(x_{(\ell,d)} - X)$ is the extra utility the driver collects when the trip is a dispatch trip ($\delta=1$) and
their add-passenger disutility $X$ is smaller than the threshold $x_{(\ell,d)}$.
Therefore, averaging over all trips we have
$$
\left|\bbE[U_i^t] - \frac{\partial}{\partial S_\ell}\Phi_{\omega_t}(\bfS_t)\right|
= \bbP(\delta = 1)\bbE[x_{(\ell,d)} - X\mid\delta = 1] - \frac{C}{2}\left(\frac{\bfg^{*T}\bfone_\ell}{\bff^{*T}\bfone_\ell}\right)^2 + \epsilon'_k.
$$

We turn to analyzing the $\bbP(\delta = 1)\bbE[x_{(\ell,d)} - X\mid\delta = 1]$ term.
Recall that the threshold $x_{(\ell,d)}$ is the same for every destination $d$ under the policy $\Pi^{(k)}$,
and this value is $C\frac{\bfg^{*T}\bfone_\ell}{\bff^{*T}\bfone_\ell}$. Also, conditioned on driver $i$ being allocated a dispatch
trip, we know that their add-passenger disutility $X$ is uniformly distributed between $0$ and $x_{(\ell,d)}$.  Therefore we have
$$
\bbE[x_{(\ell,d)} - X\mid\delta = 1] = \frac{C}{2}\frac{\bfg^{*T}\bfone_\ell}{\bff^{*T}\bfone_\ell}.
$$
Also, the probability of a dispatch trip $\bbP(\delta = 1)$ can be expressed in terms of the total number of dispatch trips:
$$
\bbP(\delta = 1) = \bbE\left[\frac{\bfg^T\bfone_\ell}{S_\ell}\right]
= \frac{\bfg^{*T}\bfone_\ell}{S_\ell} + \frac{\alpha_k}{S_\ell},
$$
where $\alpha_k$ is the error term bounding the convergence of $\bfg$ to $\bfg^*$.
Now we use the fact that $S_\ell\geq\beta_k=\sqrt{k}$ to conclude that $\frac{\alpha_k}{S_\ell}\leq\sqrt{\alpha_k}\to 0$ as $k\to\infty$.
So, we have
$$
\left|\bbP(\delta = 1)\bbE[x_{(\ell,d)} - X\mid\delta = 1] - \frac{C}{2}\left(\frac{\bfg^{*T}\bfone_\ell}{\bff^{*T}\bfone_\ell}\right)^2\right|
\leq \sqrt{\alpha_k}.
$$
Defining $\epsilon_k = \epsilon'_k + \sqrt{\alpha_k}$, which we know converges to $0$ as $k\to\infty$, we have shown
$$
\left|\bbE[U_i^t\mid\ell_i^t=\ell] - \frac{\partial}{\partial S_\ell}\Phi_{\omega_t}(\bfS_t)\right| \leq \epsilon_k,
$$
finishing the proof.
\end{proof}

We restate part \ref{thm:item:opt_ic_stochastic} of Theorem \ref{thm:opt_ic} below.
\begin{theorem*}
    There exist nonnegative sequences $(\epsilon_k:k\geq 1)$ and $(\delta_k:k\geq 1)$, both converging to $0$ as  $k\to\infty$, such that
        when the SSP mechanism is used to set prices in the stochastic two-level model with population-size $k$, then $\Pi^{(k)}$ is an
            $(\epsilon_k,\delta_k)$ equilibrium.
\end{theorem*}
\begin{proof}
To show that $\Pi^{(k)}$ is an $(\epsilon_k,\delta_k)$-approximate equilibrium, we have to show that from any market state 
$(\omega_t,\bfS_t)\in\Omega_t\times\calS_t(\gamma_t)$, the number of drivers who have 
at least $\epsilon_k$-conditional incentive to deviate from any market state is smaller than $\delta_k k$.
That is, if $\calM_t$ is the index set of drivers corresponding to the supply-location vector $\bfS_t$, and
$\calM_t(\epsilon_k;\omega_t,\bfS_t)$ is the set of drivers whose conditional incentive to deviate from $\Pi^{(k)}$ is
no larger than $\epsilon_k$, we have to show
\begin{equation}
\label{eq:eqlbm_condition}
|\calM_t\setminus \calM_t(\epsilon_k;\omega_t,\bfS_t)|\leq\delta_k k.
\end{equation}

    Define $\delta_k=|\calL|\beta_k$, where $(\beta_k)_{k=1}^\infty$ is the sequence from Lemma \ref{lem:two_level_utility}.
    Define $\epsilon_k$ to be the maximum incentive to deviate, over all market states $(\omega_t,\bfS_t)$,
    for a driver positioned at a location $\ell$ which satisfies
    the minimum driver volume condition described in Lemma \ref{lem:two_level_utility}, i.e. $S_\ell\geq\beta_k$.

    Note the inequalities
    $$
    |\calM_t\setminus \calM_t(\epsilon_k;\omega_t,\bfS_t)| \leq k \sum_{\ell\in\calL : S_\ell < \beta_k} S_\ell \leq k\delta_k,
    $$
    so the equilibrium condition (\ref{eq:eqlbm_condition}) is satisfied, and we already know $\delta_k\to 0$ holds.

    That the incentive to deviate term $\epsilon_k$ converges to $0$ follows from the fact that there is no incentive to deviate in the
    fluid model, and as $k\to\infty$ we have that the stochastic utility converges to the fluid utility, for 
    drivers at locations $\ell$ which satisfy the
    minimum driver volume condition $S_\ell\geq\beta_k$.
\end{proof}

\section{Approximate Welfare-Robustness Proof in the Fluid Model}
\label{appdx:approx_robustness_fluid_proof}
In this Appendix we prove the second part of the statement in Theorem \ref{thm:robust_eqlbm_fluid}, stating
that every $\epsilon$-equilibrium in the fluid model achieves approximately optimal welfare.

Let $\Sigma$ be an $\epsilon$-equilibrium for the fluid model under the SSP pricing and matching policy.
Let $\calV_t$ and $\calQ_t$ be the value function and Q-values associated with $\Sigma$.
We proceed via backwards induction on the time $t$, and make the following assumption about the future time period $t+1$:
There exists an error term $\epsilon_{t+1}$ (which converges to $0$ as $\epsilon\to 0$) such that
the following are true:
\begin{enumerate}
\item The welfare achieved by $\Sigma$ from any time $t+1$ market state is within $\epsilon_{t+1}$ from the optimum.
\item The value function for a location $\ell$ at any time $t+1$ state is within $\epsilon_{t+1}$ of the
partial derivative of the fluid optimization function: i.e. for any $\ell$ and $(\omega_{t+1},\bfS_{t+1})\in\calS_{t+1}$ we have
\begin{equation}
\label{eq:approx_backwards_induction_assn}
\left|\calV_{t+1}(\ell,\omega_{t+1},\bfS_{t+1}) - \frac{\partial}{\partial S_\ell}\Phi_{\omega_{t+1}}(\bfS_{t+1})\right| \leq \epsilon_{t+1}.
\end{equation}
\end{enumerate}

Fix a time period $t$ and let $(\omega_t,\bfS_t)\in\calS_t$ be any market state from time $t$.
Let $\bfx'$ be the disutility thresholds used by the drivers under $\Sigma$ at $(\omega_t,\bfS_t)$ and let
$(\bfg',\bff')$ be the vector of dispatch trips and total trips that result under $\Sigma$ and the SSP prices and matching process 
at $(\omega_t,\bfS_t)$.
Additionally, let $(\bfg^*,\bff^*)$ denote the optimal solution for the fluid optimization problem with respect to
$(\omega_t,\bfS_t)$ used by the SSP mechanism to set prices and allocate matches.

Recall an $\epsilon$-equilibrium strategy profile for the fluid model is characterized by the approximate incentive compatibility conditions
(\ref{eq:approx_fluid_ic}), which we restate here for clarity.
\begin{align}
    f'_{(\ell,d)} - g'_{(\ell,d)} > \epsilon& \implies \calQ_t(\ell,d,0) \geq \max_{d'\in\calL}\calQ_t(\ell,d',0) - \epsilon,\nonumber\\
            g'_{(\ell,d)} > \epsilon& \implies \calQ_t(\ell,d,1,x'_{(\ell,d)}) \geq \max_{d'\in\calL}\calQ_t(\ell,d',0) - \epsilon , \label{eq:approx_ic_appdx}  \\
            g'_{(\ell,d)} > \epsilon, S_\ell - \sum_{d'}g'_{(\ell,d')} > \epsilon &\implies \calQ_t(\ell,d,1,x'_{(\ell,d)}) \leq \max_{d'\in\calL}\calQ_t(\ell,d',0) + \epsilon.
\end{align}

Our proof mirrors the steps in Section \ref{subsec:fluid_robustness_proof}.

\medskip\noindent \textbf{Where do the non-dispatched drivers go?}
First, we show that the non-dispatched drivers, whose trips are specified by $\bff'-\bfg'$, take
approximately optimal trips given the dispatch trips $\bfg'$.
Recall the optimization problem (\ref{eq:relocation_problem_}), which depends on the vector $\bfg'$,
which we restate below:
\begin{align}
    \label{eq:relocation_problem_appdx}
    \sup_{\bff} &\ \ \   \sum_{(\ell,d)\in\calL^2} -c_{(\ell,d)}f_{(\ell,d)} + \calU_{\omega_t}^{>t}(\bff)
 & 
\\
    \mbox{subject to} & \nonumber\\ 
    & f_{(\ell,d)} \geq g'_{(\ell,d)} \ &\forall (\ell,d)\in\calL^2\label{eq:relocation_inequality_appdx}\\
    &
        \sum_{d\in\calL} f_{(\ell,d)} =
        S_\ell 
    \ & \forall \ell\in\calL.\label{eq:relocation_equality_appdx}
\end{align}
\begin{lemma}
    \label{lem:relocation_optimality_approx}
    Let $\epsilon'=\max(\epsilon,\epsilon_{t+1})$, where $\epsilon$ is the error term in our definition of approximate equilibrium,
    and $\epsilon_{t+1}$ is the error bound from our backwards induction assumption (\ref{eq:approx_backwards_induction_assn}).

    Then there exists a constant $C>0$ such that the total trip volumes $\bff'$ from our equilibrium is a $C\epsilon'$-optimal 
    solution for the relocation
    problem (\ref{eq:relocation_problem_appdx}) with respect to the dispatch trips $\bfg'$.  
    Moreover, the following dual variables form an $\epsilon'$-approximate Lagrange multiplier vector for $\bff'$ (in the sense
    of Definition \ref{def:lagrange_multiplier}):
    \begin{equation}
        \label{eq:relocation_lambda_approx}
        \lambda_{(\ell,d)}' = \max_{d'}\calQ_t(\ell,d',0) - \calQ_t(\ell,d,0),
    \end{equation}
    associated with the inequality constraint (\ref{eq:relocation_inequality_constraint}) for each $(\ell,d)$,
    and
    \begin{equation}
        \label{eq:relocation_eta_approx}
        \eta_\ell' = \max_{d'}\calQ_t(\ell,d',0)
    \end{equation}
    associated with the equality constraint (\ref{eq:relocation_equality_constraint}) for each $\ell$.
\end{lemma}
\begin{proof}
    We show that $(\bflambda',\bfeta')$ is an $\epsilon'$-approximate Lagrange multiplier vector for $\bff'$, in the sense
    of Definition \ref{def:lagrange_multiplier}, where $\epsilon'=\max(\epsilon,\epsilon_{t+1})$.

First, observe that approximate complementary slackness conditions follow from the approximate incentive compatibility
properties.
Indeed, if $f'_{(\ell,d)} - g'_{(\ell,d)} > \epsilon$ then (\ref{eq:approx_ic_appdx}) states $\lambda_{(\ell,d)} < \epsilon$.

Next, we check approximate stationarity. We work in terms of a convex cost function instead of a concave utility function.
    Define
    \begin{equation}
        \calC_{\omega_t}(\bff) = \sum_{(\ell,d)\in\calL^2} c_{(\ell,d)} f_{(\ell,d)} - \calU_{\omega_t}^{> t}(\bff)
    \end{equation}
    to be the cost function of the relocation trip variables $\bff$, i.e. the negative of the objective function in the relocation problem (\ref{eq:relocation_problem_appdx}).

    First, observe our backwards induction assumption \eqref{eq:approx_backwards_induction_assn}
yields the following equalities:
\begin{align}
    \frac{\partial}{\partial f_{(\ell,d)}} \calC_{\omega_t}(\bff') &= c_{(\ell,d)} - 
\frac{\partial}{\partial f_{(\ell,d)}}\bbE\left[\Phi_{\omega_{t+1}}(\bar{\bfS}_{\omega_{t+1}}(\bff'))\mid\omega_t\right]\nonumber\\
    &=c_{(\ell,d)} - \bbE\left[\frac{\partial}{\partial S_d }\Phi_{\omega_{t+1}}(\bar{\bfS}_{\omega_{t+1}}(\bff'))\mid\omega_t\right]\nonumber\\
    &= c_{(\ell,d)} - \bbE\left[\calV_{t+1}(d;\omega_{t+1}, \bar{\bfS}_{\omega_{t+1}}(\bff'))\right] + \delta\nonumber\\
    &= -\calQ_t(\ell,d,0)  + \delta,\label{eq:cost_fn_partial}
\end{align}
where $\delta$ is a constant satisfying $|\delta|\leq \epsilon_{t+1}$.

    The Lagrangian for (\ref{eq:relocation_problem_appdx}) is the following:
    \begin{equation*}
        L(\bff;\bflambda,\bfeta) = \calC_{\omega_t}(\bff) + \bflambda^T(\bfg'-\bff) + 
\sum_\ell \eta_\ell\left(\sum_d f_{(\ell,d)} - S_\ell\right).
    \end{equation*}
Evaluate the partial derivative of $L(\bff';\bflambda',\bfeta')$ at each coordinate $f'_{(\ell,d)}$:
    \begin{align*}
        \frac{\partial}{\partial f_{(\ell,d)}} L(\bff';\bflambda,\bfeta) &= 
        c_{(\ell,d)} -\frac{\partial}{\partial f_{(\ell,d)}} \calU_{\omega_t}^{>t}(\bff') - \lambda_{(\ell,d)} + \eta_\ell\\
        &= - \calQ(\ell,d,0) - \lambda_{(\ell,d)} + \max_{d'}\calQ(\ell,d,0) + \delta,
    \end{align*}
so $|\frac{\partial}{\partial f_{(\ell,d)}} L(\bff';\bflambda',\bfeta')|\leq \epsilon_{t+1}$.

    Therefore, $(\bflambda',\bfeta')$ is an $\epsilon'$-approximate Lagrange multiplier vector for $\bff'$, in the
    sense of Definition \ref{def:lagrange_multiplier}. By Lemma \ref{lem:approx_lagrange_optimality}, it follows that $\bff'$
    is an $\epsilon''$-optimal solution for (\ref{eq:relocation_problem_appdx}), where $\epsilon'' = C\epsilon'$ for some
    problem-independent constant $C$.

\end{proof}

\medskip\noindent \textbf{Approximate equilibria serve approximately all available dispatch demand.} The proof of this fact follows the outline along the same line as in Section~\ref{sec:sspm} using Lemma \ref{lem:relocation_optimality_approx}
in place of the exact version used there. We define an alternate network as was done there. By Lemma~\ref{lemma:lin-opt} the flow $(\bfg^{**},\bff^{**},\bfx^{**})$ is an optimal solution to our modified convex program, with optimal value $OPT^{**}$ and using costs $c_{(\ell,d)}-p_{(\ell,d)}-x^*_{(\ell,d)}$ on the special edges. 

Now consider the approximate equilibrium solution $(\bfg',\bff',\bfx')$. Different drivers may use different 
cutoffs for their dis-utility. We define $x'_{(\ell,d)}$ as 
lowest pick-up dis-utility by a driver who rejected a dispatch. This means that all drivers with $x_{(\ell,d)}< x'_{(\ell,d)}$ offered a dispatch $(\ell,d)$ accepted it, and by the equilibrium property, all drivers with $x_{(\ell,d)}\ge x'_{(\ell,d)}+\epsilon$ rejected the dispatch if offered.  

Now consider the same network using the alternate cost $c_{(\ell,d)}-p_{(\ell,d)}-x'_{(\ell,d)}$.

By Lemma ~\ref{lemma:lin-larger-opt} (using  $(\bfg^{**},\bff^{**},\bfx^{**})$ as a feasible solution), the optimum value with this new cost is now at least $\sum_{(\ell,d)}(g^*_{(\ell,d)}-g'_{(\ell,d)})(x^*_{(\ell,d)}-x'_{(\ell,d)}-\epsilon)$ larger.

Next consider the solution $(\bfg'',\bff'',\bfx'')$ constructed from the approximate equilibrium  as was done in  Section~\ref{sec:sspm}. We claim that this solution is an approximate equilibrium for the modified problem. 
\begin{lemma}
The flow $(\bfg'',\bff'',\bfx'')$ is an equilibrium on our modified network 
either with cost $c_{(\ell,d)}-p_{(\ell,d)}-x^*_{(\ell,d)}$ or with costs 
$c_{(\ell,d)}-p_{(\ell,d)}-x'_{(\ell,d)}$ of the drive\pfedit{r}s with modified costs. 
\end{lemma}

Now by Lemma \ref{lem:relocation_optimality_approx} this equilibrium solution is approximately optimal with both version of the problem. 

Using $x^*$ to define cost we get that the equilibrium solution has value close to $OPT^{**}$. Using $x'$ does not change the value of the equilibrium, while the optimum increases by Lemma \ref{lemma:lin-larger-opt} by at least $$\sum_{(\ell,d)}(g^*_{(\ell,d)}-g'_{(\ell,d)})(x^*_{(\ell,d)}-x'_{(\ell,d)}).$$
Since the shared equilibrium solution is approximately optimal for both problems, this gives an upper bound on this difference in terms of the optimality of the solution.

For the product $\sum_d (g^*_{(\ell,d)}-g'_{(\ell,d)})(x^*_{(\ell,d)}-x'_{(\ell,d)}$ to be small for a location $\ell$, we must have that for each destination, either $g^*_{(\ell,d)}-g'_{(\ell,d)}$ must be small or $x^*_{(\ell,d)}-x'_{(\ell,d)}$ is small. To be able to bound the difference between $g^*_{(\ell,d)}$ and $g'_{(\ell,d)}$, we need to show that $x^*_{(\ell,d)}\approx x'_{(\ell,d)}$ implies that $g^*_{(\ell,d)}\approx g'_{(\ell,d)}$. To see this, consider the subset of destinations  that have $x^*_{(\ell,d)}-x'_{(\ell,d)}\le \epsilon C$, and let $S^\epsilon_\ell$ denote the driver volume that is used to offer dispatches to one of these destinations from location $\ell$.

\begin{lemma}\label{lem:approx-equilib-approx-equal-g}
If an approximate equilibrium \pfedit{satisfies} $x'_{(\ell,d)}\ge  x^*_{(\ell,d)} -\epsilon C$ for a subset of destinations at a location $\ell$, and let $S_\ell^\epsilon$ be the set of drivers who would be offered rides to one of these locations in the optimum solution, then at most $\epsilon S^\epsilon_\ell$ riders requesting rides do not receive a ride to this subset of destinations.
\end{lemma}
\begin{proof}
The price is set so that we have $g^*_{(\ell,d)}$ riders that will accept the price offered. Our mechanisms offers the dispatch $(\ell,d)$ to at least $\frac{C}{x^*_{(\ell,d)}}g^*_{(\ell,d)}$ drivers. With the lower disutility cutoff $x'_{(\ell,d)}$, out of these dispatch offers, a fraction of $\frac{x^*_{(\ell,d)}-x'_{(\ell,d)}}{C}$ will reject the dispatch that would be accepted in the optimum. This is an upper bound on the riders remaining unserved at location $\ell$ with possible extra drivers, or other destinations where riders are already served, the mechanism may offer the rides to additional drivers. Summing these over the different routes starting at $\ell$, we see that at most $\epsilon S^\epsilon_\ell$ riders do not get a ride.
\end{proof}

\medskip\noindent \textbf{Approximate equilibria are approximately welfare optimal.}
We can now finish the backwards induction proof analogously to the proof for the exact case in Section~\ref{sec:sspm}. 
We start by showing that the thresholds $x'_{(\ell,d)}$ are approximately equal to the thresholds $x^*_{(\ell,d)}$.

\begin{lemma}
    \label{lem:approx_opt_threshold}
    For any location $\ell$ where the total driver volume $S_\ell$ is larger than $\sqrt{\epsilon}$, and for any route $(\ell,d)$ where the dispatch volume $g^*_{(\ell,d)}$ is larger then $\epsilon$, then the difference between the optimal threshold $x^*_{(\ell,d)}$
    and the threshold used by the drivers in an approximate equilibrium $x'_{(\ell,d)}$ is bounded by an error term $\epsilon''$, such that $\epsilon''$
    goes to $0$ as $\max(\epsilon,\epsilon_{t+1})$ goes to $0$.
\end{lemma}
\begin{proof}
Consider the following modification of the optimization problem (\ref{eq:relocation_problem_appdx}), where the pre-specified dispatch
trips correspond to the fluid optimal dispatch trips $\bfg^*$ rather than $\bfg'$:
\begin{align*}
    \sup_{\bff} &\ \ \   \sum_{(\ell,d)\in\calL^2} -c_{(\ell,d)}f_{(\ell,d)} + \calU_{\omega_t}^{>t}(\bff)
 & 
\\
    \mbox{subject to} & \nonumber\\ 
    & f_{(\ell,d)} \geq g^*_{(\ell,d)} \ &\forall (\ell,d)\in\calL^2\\
    &
        \sum_{d\in\calL} f_{(\ell,d)} =
        S_\ell 
    \ & \forall \ell\in\calL.
\end{align*}
We have established that $\bfg' \approx \bfg^*$.
Therefore, by Lemma \ref{lem:relocation_optimality_approx}, $\bff'$ is approximately optimal for the above optimization problem, and
$\bff^*$ is an exact optimum.
Let $(\bflambda^*,\bfeta^*)$ be the optimal dual variables associated with $\bff^*$.
By Lemma \ref{lem:approx_lagrange_multiplier}, it follows that $(\bflambda^*,\bfeta^*)$ is an $\epsilon'$-approximate Lagrange multiplier
vector for $\bff'$, in the sense of Definition \ref{def:lagrange_multiplier}, where $\epsilon'$ goes to $0$ as the suboptimality of $\bff'$ goes to $0$.
In particular, this means that the norm of the gradient of the mixed-solution Lagrangian is small: $\|\nabla_{\bff} L(\bff';\bflambda^*,\bfeta^*)\|_2\leq \epsilon'$,
where the Lagrangian is
$$
L(\bff;\bflambda,\bfeta) = \calC_{\omega_t}(\bff) + \bflambda^T(\bfg^* - \bff) + \sum_{\ell\in\calL}\eta_\ell\left(\sum_{d\in\calL}f_{(\ell,d)} - S_\ell\right).
$$
In particular, for any pair of origin and destination locations $(\ell,d)$ we have the bound
$$\left|\frac{\partial}{\partial f_{(\ell,d)}} L(\bff';\bflambda^*,\bfeta^*)\right|\leq \epsilon'.$$
    Evaluating the partial derivative, \pfcomment{need to fix sign errors here}\mccomment{Fixed}, we have
\begin{eqnarray*}
   \frac{\partial}{\partial f_{(\ell,d)}} L(\bff';\bflambda^*,\bfeta^*) & = & \frac{\partial}{\partial f_{(\ell,d)}} \calC_{\omega_t}(\bff') - \lambda^*_{(\ell,d)} + \eta^*_\ell \\
    & \approx & -\calQ'(\ell,d,0) - \lambda^*_{(\ell,d)} + \eta^*_\ell\\
    & = & -\calQ'(\ell,d,0) - \left(\max_{d'}\calQ^*(\ell,d',0) - \calQ^*(\ell,d,0) \right) +\max_{d'} \calQ^*(\ell,d',0)\\
    & = &\calQ^*(\ell,d,0) - \calQ'(\ell,d,0).
\end{eqnarray*}
    The first line uses the approximate equality established in equation (\ref{eq:cost_fn_partial}), which shows that, under our backwards induction assumption
    that the continuation utilties at a location are approximately equal to the partial derivative of the optimal welfare function with respect to driver supply at that
    location, the partial derivative of the cost function with respect to driver volume along a route is approximately equal to the negative utility of taking
    a relocation trip along that route.
    \pfcomment{In the equation block above, here you are using the induction hypothesis, and it is only approximately equal, right?}\mccomment{fixed}
Therefore,
$$
    \left|\calQ'(\ell,d,0) - \calQ^*(\ell,d,0)\right| \leq \epsilon'+\epsilon_{t+1},
$$
where $\calQ'(\ell,d,0)$ is the utility generated by a relocation trip from $\ell$ to $d$ under the approximate equilibrium $(\bff',\bfg',\bfx')$,
and $\calQ^*(\ell,d,0)$ is the utility generated by the same relocation trip under the exact equilibrium $(\bff^*,\bfg^*,\bfx^*)$.

    Now recall that, since the actions $(\bff',\bfg',\bfx')$ come from an approximate equilibrium, we know the following
    properties are satisfied:
    \begin{eqnarray}
        g'_{(\ell,d)} > \epsilon &\implies \calQ'(\ell,d,1,x'_{(\ell,d)}) \geq \max_{d'}\calQ'(\ell,d',0) - \epsilon\label{eq:approx_eqlbm_A}\\
        S_\ell - \sum_{d'}g'_{(\ell,d')} > \epsilon, g'_{(\ell,d)} > \epsilon & \implies \calQ'(\ell,d,1,x'_{(\ell,d)}) \leq \max_{d'}\calQ'(\ell,d',0) + \epsilon.
        \label{eq:approx_eqlbm_B}
    \end{eqnarray}
    The above properties formalize what we mean when we say that a driver, whose add-passenger disutility is exactly the threshold value $x'_{(\ell,d)}$,
    is approximately indifferent between serving a dispatch trip from $\ell$ to $d$ and serving a relocation trip to any destination.
    The first line (\ref{eq:approx_eqlbm_A}) says that, for a route $(\ell,d)$ where a non-negligible volume of drivers serve a dispatch trip, then the utility collected by a driver
    who serves a dispatch trip along $(\ell,d)$, and whose add-passenger disutility is exactly equal to the threshold $x'_{(\ell,d)}$, is (approximately)
    at least as large as the maximum relocation-trip utility achievable from the same origin location. 
    The second line (\ref{eq:approx_eqlbm_B}) says that, for a route $(\ell,d)$ where a non-negligible volume of drivers serve a dispatch trip,
    and where a non-negligible volume of drivers also serve a relocation trip, then the utility collected by a driver whose add-passenger
    disutility is exactly $x'_{(\ell,d)}$ who serves a dispatch-trip from $\ell$ to $d$ is (approximately) no larger than the maximum
    relocation-trip utility achievable from the same origin location.

\pfcomment{Need to make sure this is ok acknowledging that this exact IC thing only holds when there is a strictly positive relocation volume.}
\mccomment{Taken care of, but significantly increased length of the proof.  Would be good to split into separate lemmas eventually.}
    We proceed by analyzing two cases.  In the first case, suppose that the volume of drivers at $\ell$ who serve a relocation trip
    is no larger than $\epsilon$, i.e. $S_\ell - \sum_{d'}g'_{(\ell,d')}< \epsilon$.  

    In this case, since the realized dispatch trip volumes
    $g'_{(\ell,d)}$ cannot be larger than the optimal dispatch trip volumes $g^*_{(\ell,d)}$, we also have $S_\ell - \sum_{d'}g^*_{(\ell,d)} < \epsilon$.
    Recall that under the fluid optimal solution $(\bff^*,\bfg^*,\bfx^*)$, the thresholds along each route $(\ell,d)$ depend only on the origin location $\ell$,
    i.e. there is a threshold $x^*_\ell$ such that $x^*_{(\ell,d)} = x^*_\ell$, and the following equation holds:
    $$
    \sum_{d}g^*_{(\ell,d)} = S_\ell\frac{x^*_\ell}{C},
    $$
    where $\frac{x^*_\ell}{C}$ is the probability any driver from $\ell$ accepts a dispatch.
    Therefore,
    $$
    \epsilon \geq S_\ell - \sum_d g^*_{(\ell,d)} \geq S_\ell - S_\ell \frac{x^*_\ell}{C}.
    $$
    So,
    $$
    x^*_\ell \geq C - \epsilon\frac{C}{S_\ell} \geq C - \sqrt{\epsilon}C,
    $$
    where the final inequality follows from our assumption that $S_\ell\geq\sqrt{\epsilon}$.

    Now we want to compare the optimal threshold to the chosen driver thresholds $x'_{(\ell,d)}$.
    Observe the total dispatch demand volume is larger than if every destination used the minimum threshold $\min_{d}x'_{(\ell,d)}$:
    $$
    \sum_d g'_{(\ell,d)} \geq S_\ell \frac{\min_d x'_{(\ell,d)}}{C}.
    $$
    By the same logic as above, we have
    $$
    \min_d x'_{(\ell,d)} \geq C - \sqrt{\epsilon}C.
    $$
    Therefore we have $|x'_{(\ell,d)} - x^*_\ell| \leq \sqrt{\epsilon}C$ for any route $(\ell,d)$ where $g^*_{(\ell,d)} > \epsilon$. 
    Taking $\epsilon''=\sqrt{\epsilon}C$ establishes the claimed result in the case where approximately every driver serves a dispatch trip.

    In the next case, we consider the thresholds when a non-negligible fraction of drivers serve a relocation trip, i.e.
    where $S_\ell - \sum_d g'_{(\ell,d)} > \epsilon$.
    In this case, for any route $(\ell,d)$ where $g'_{(\ell,d)} > \epsilon$, the approximate incentive compatibility
    conditions (\ref{eq:approx_eqlbm_A}) and (\ref{eq:approx_eqlbm_B}) establish the following equality:
    $$
    P_{(\ell,d)} - x'_{(\ell,d)} + \calQ'(\ell,d,0) + \delta' = \max_{d'}\calQ'(\ell,d',0),
    $$
    where $\delta'$ is an error term smaller than $\epsilon$.
    Now, the exact incentive compatibility conditions on $(\bff^*,\bfg^*,\bfx^*)$ state the relationship
$$
P_{(\ell,d)} - x^*_{(\ell,d)} + \calQ^*(\ell,d,0) = \max_{d'}\calQ^*(\ell,d',0)
$$
holds. 
Combining the two equations, we have the difference between $x^*_{(\ell,d)}$ and $x'_{(\ell,d)}$ is bounded as follows:
$$
|x^*_{(\ell,d)} - x'_{(\ell,d)}| \leq \delta' + \left|\calQ'(\ell,d,0) - \calQ^*(\ell,d,0)\right| + \left|\max_{d'}\calQ'(\ell,d',0) - \max_{d'}\calQ^*(\ell,d,0)\right|  
    \leq \epsilon + 2(\epsilon'+\epsilon_{t+1}).
$$
Taking $\epsilon'' = \epsilon + 2(\epsilon'+\epsilon_{t+1})$ shows the bound
    $|x^*_{(\ell,d)} - x'_{(\ell,d)}|\leq \epsilon''$, finishing the proof of Lemma \ref{lem:approx_opt_threshold}.
\end{proof}

\pfcomment{Need to integrate this into the induction hypothesis.}\mccomment{Done, but still a little handwavy because I don't introduce explicit notation
to refer to the welfare generated by a strategy from a specific state.  Would be good to add this in a future draft.}
We now finish the proof, by showing that both our backwards induction assumptions hold at time period $t$.
First, we show that the $\epsilon$-equilibrium strategy profile produces actions which have total welfare at most $\epsilon_t$ away from
the optimal expected welfare, where $\epsilon_t$ goes to $0$ as $\epsilon$ goes to $0$.

We have already established that $\bff'$ are approximately optimal relocation trips with respect to the fluid dispatch trips $\bfg^*$.
Lemma \ref{lem:approx_opt_threshold} shows that $(\bff',\bfg',\bfx')$ is an approximately optimal fluid solution, when we additionally
include the welfare from dispatch trips in the objective function.  Indeed, the welfare from dispatch trips is a function of the dispatch
trip volumes and the add-passenger disutility thresholds.  We know that $\bfg'\approx \bfg^*$, and Lemma \ref{lem:approx_opt_threshold} establishes
$\bfx'\approx \bfx^*$, so the welfare generated by dispatch trips at time $t$ under $(\bff',\bfg',\bfx')$ is approximately equal to the
welfare generated by dispatch trips at time $t$ under $(\bff^*,\bfg^*,\bfx^*)$.
From the backwards induction assumption, we know that from any time $t+1$ state, the drivers will achieve welfare that is at most $\epsilon_{t+1}$
away from the optimal welfare from that state.
Therefore, it follows that the $\epsilon$-equilibrium strategy which produces actions
$(\bff',\bfg',\bfx')$ achieves approximately optimal social welfare at time $t$.

It remains to establish our second backwards induction assumption, i.e. we need to establish that the expected utility of a driver positioned at
a location $\ell$, under
the approximate equilibrium $(\bff',\bfg',\bfx')$, is approximately equal to the partial derivative of the state-dependent welfare function.
Recall $\calV_t(\ell,\omega_t,\bfS_t)$ denotes the expected value for a driver of being positioned at $\ell$, under scenario $\omega_t$, and
supply-location vector $\bfS_t$. We characterize the value of being positioned at $\ell$ as follows:
\begin{eqnarray*}
    \calV_t(\ell,\omega_t,\bfS_t) & = & \sum_{d\in\calL} \frac{g'_{(\ell,d)}}{S_\ell}\left(P_{(\ell,d)} - \frac{x'_{(\ell,d)}}{2} + \calQ'(\ell,d,0)\right)
    + \sum_{d\in\calL} \frac{f'_{(\ell,d)} - g'_{(\ell,d)}}{S_\ell} \calQ'(\ell,d,0)\\
    & \approx & \sum_{d\in\calL} \frac{g'_{(\ell,d)}}{S_\ell}\left(P_{(\ell,d)} - \frac{x'_{(\ell,d)}}{2} + \calQ'(\ell,d,0)\right)
    + \frac{S_\ell - \sum_d g'_{(\ell,d)}}{S_\ell}\left(\max_d \calQ'(\ell,d,0)\right)\\
    & \approx & \sum_{d\in\calL} \frac{g^*_{(\ell,d)}}{S_\ell}\left(P_{(\ell,d)} - \frac{x^*_{(\ell,d)}}{2} + \calQ^*(\ell,d,0)\right)
    + \frac{S_\ell - \sum_d g^*_{(\ell,d)}}{S_\ell}\left(\max_d \calQ^*(\ell,d,0)\right)\\
    & = & \eta_\ell^*\\
    & = & \frac{\partial}{\partial S_\ell}\Phi_{\omega_t}(\bfS_t).
\end{eqnarray*}
The first line is the definition of the expected utility for a driver positioned at $\ell$, 
the second line follows because all but a negligible fraction of drivers who serve a relocation trip will serve an optimal relocation trip,
the third line follows because we have established that the thresholds $\bfx'$, the dispatch trips $\bfg'$, and the relocation utilities $\calQ'$,
are all approximately equal to their exact-equilibrium counterparts, and the fourth and fifth lines follow from our earlier characterizations
of the dual variables for the state dependent optimization problem (Lemma \ref{lem:unique_dual}).

Therefore, there is an error term $\epsilon_t$ such that $\epsilon_t$ is an upper bound on the difference
$|\calV_t(\ell,\omega_t,\bfS_t) - \frac{\partial}{\partial S_\ell}\Phi_{\omega_t}(\bfS_t)|$, and such that $\epsilon_t$ goes to $0$
as $\max(\epsilon, \epsilon_{t+1})$ goes to $0$.  This establishes our backwards induction assumption, and therefore finishes the proof
of approximate welfare robustness in the fluid model.

\Xomit{
\mccomment{Old version below}
We can now finish the backwards induction proof analogously to the proof for the exact case in Section~\ref{sec:sspm}. We have established that $\bfg' \approx \bfg^*$. For the next step we need to consider the optimization problem  (\ref{eq:relocation_problem_appdx})  with both  $\bfg' $ as well as $\bfg^*$ as the initial flow. With $\bfg^*$ as the start the optimal solution to (\ref{eq:relocation_problem_appdx}) is the true optimal solution $(\bfg^*,\bff^*,\bfx^*)$ which is also an equilibrium, and the dual variables give the rider incentives. Starting with flow $\bfg'$ by Lemma \ref{lem:relocation_optimality_approx} we see that $(\bfg',\bff',\bfx')$ is an approximately optimal solution with the incentives giving approximations to dual variables. To be able to follow the proof outline used for exact equilibrium in Section~\ref{sec:sspm}, we need to connect the dual variables of the   (\ref{eq:relocation_problem_appdx}) with $\bfg^*$ and $\bfg'$ as the initial flow. 
\begin{lemma}\label{lem:approx_equal_utility}
For each location $\ell$ the expected utility of drivers at $\ell$ for the optimization problem (\ref{eq:relocation_problem_appdx}) with  $\bfg' $ and $\bfg^*$ are approximately equal.
\end{lemma}
\begin{proof}
 Recall from Lemma \ref{lem:unique_dual}, that the dual variables are continuous in the supply of drivers. The changes in the two problems correspond to two changes in the volume of drivers. In the problem defined by $\bfg'$, at each locations $\ell$ there are $\sum_d (\bfg^*_{(\ell,d)}- \bfg'_{(\ell,d)})$ more drivers to start with, and at each destination $d$ there are $\sum_\ell (\bfg^*_{(\ell,d)}- \bfg'_{(\ell,d)})$ more new drivers starting at location $d$ after the current period.
\end{proof}

We now outline how to finish the backwards induction proof.  
First, we claim that the thresholds $\bfx'$ from our arbitrary equilibrium have to approximately equal the optimal thresholds $\bfx^*$ 
associated with the optimal solution $(\bff^*,\bfg^*)$. As before, when $\bfx^*\approx C$, that is when all drivers are needed for rides, to serve approximately the same volume of traffic, we need to have  $\bfx'\approx C$ also. 

\etcomment{I appear to be loosing a factor proportional to the degree of  a node here, as all I know is that the sum of traffic lost is small, not direction that $\bfg'$ and $\bfg^*$ are so close route-by-route.} When $\bfx^*< C$, this will  follow from the facts we just established. By Lemma~\ref{lem:relocation_optimality_approx} the expected future utility $\calQ_t(\ell,d,0)$ of a driver at a location $\ell$ is approximately equal to this utility in the optimum solution of the optimization problem (\ref{eq:relocation_problem_appdx}). Now consider the optimization problem (\ref{eq:relocation_problem_appdx}) with initial flow $\bfg^*$, and the future utility of drivers $\calQ^*_t(\ell,d,0)$ at a location $\ell$ in this problem. The incentives are the dual variables of the optimization problem. By Lemma~\ref{lem:approx_equal_utility}, the two expected utilities $\calQ_t(\ell,d,0) \approx \calQ^*_t(\ell,d,0)$.

Having established that future utilities are approximately equal in the solutions $(\bfg',\bff',\bfx')$ and $(\bfg^*,\bff^*,\bfx^*)$ the rest of the proof follows along the same lines as was done in Section~\ref{sec:sspm}.
}

\Xomit{
\section{Welfare-Robustness and Convex Analysis}
In this section we establish convex optimization properties that are useful for our welfare robustness proof.
Let $\bff\in\bbR^{n^2}$ be a decision variable in $n^2$ dimensions and consider the problem of minimizing
$C(\bff)$ where $C:\bbR^{n^2}\to\bbR$ is a differentiable, convex function.
Suppose $\bff$ is constrained to be larger, pointwise, than some vector $\bfg_0\in\bbR^{n^2}$, and suppose there is a
linear equality constraint $A\bff = \bfS$, where $A$ is an $n\times n^2$ matrix and $\bfS\in\bbR^{n}$ is another
input vector.  Our convex optimization problem is stated below.
\begin{align}
    \label{eq:generic_cvxopt}
    \inf_{\bff\in\bbR^{n^2}} &\ \ \   C(\bff)
\\
    \mbox{subject to}\ \ \  
    & \bff \geq \bfg\label{eq:generic_inequality} \\
    &
        A\bff =
        \bfS \label{eq:generic_equality}
\end{align}
Assume the matrix $A$ satisfies the property $\|\bff\|_1 = \|A\bff\|_1$ for any $\bff\in\bbR^{n^2}$.

The Lagrangian associated with (\ref{eq:generic_cvxopt}) is the function
\begin{equation}
L(\bff;\bflambda,\bfeta) = C(\bff) + \bflambda^T(\bfg_0 - \bff) + \bfeta^T(A\bff - \bfS),
\end{equation}
where $\bflambda\in\bbR^{n^2}_+$ is the dual variable for the inequality constraint (\ref{eq:generic_inequality})
and, and $\bfeta\in\bbR^n$ is the dual variable for the equality constraint (\ref{eq:generic_equality}).

Strong duality for the convex optimization problem (\ref{eq:generic_cvxopt}) holds since all the
constraints are linear. 

Therefore the optimal value for (\ref{eq:generic_cvxopt}) is equal to the following 
unconstrained problems:
\begin{equation}
\label{eq:unconstrained_cvxopt}
\inf_{\bff\in\bbR^{n^2}}\sup_{\substack{\bflambda\in\bbR^{n^2}_+\\ \bfeta\in\bbR^n}} L(\bff;\bflambda,\bfeta) = 
 \sup_{\substack{\bflambda\in\bbR^{n^2}_+\\ \bfeta\in\bbR^n}} \inf_{\bff\in\bbR^{n^2}}L(\bff;\bflambda,\bfeta).
\end{equation}
Assume that the primal and dual problems attain optimal value, and let
$\bff^*$  and $(\bflambda^*,\bfeta^*)$ denote any primal and dual optimal solutions.

The Lemmas in this section are concerned with approximate primal and dual optima.  By approximate
primal and dual optima we mean any primal solution $\bar{\bff}\in\bbR^{n^2}$ and any dual solution $\bar{\bflambda}\in\bbR^{n^2}_+$ and $\bar{\bfeta}\in\bbR^{n}$,
for which there exists $\epsilon > 0$ such that the following properties hold:
\begin{itemize}
\item Approximate complementary slackness condition
\begin{equation}
\label{eq:approx_csc}
|\lambda_i (g_i - \bar{f}_i)| \leq \epsilon\ \ \forall i=1,2,\dots,n^2.
\end{equation}
\item Approximate stationarity
\begin{equation}
\label{eq:approx_stationarity}
\|\nabla C(\bar{\bff}) - \bar{\bflambda} + A^T\bar{\bfeta}\|_\infty \leq \epsilon.
\end{equation}
\item The primal solution $\bar{\bff}$ satisfies the equality constraint $A\bar{\bff} = \bfS$.
\end{itemize}

\begin{lemma}
Let $\bff$ and $(\bflambda,\bfeta)$ be any pair of approximate primal and dual optima with respect to an error term $\epsilon$.
Then $\bff$ is an $\epsilon'$-optimal solution for (\ref{eq:generic_cvxopt}), i.e.
\begin{equation}
C(\bff) - C(\bff^*) \leq \epsilon',
\end{equation}
where 
\begin{equation}
\epsilon' = n^2\epsilon + 2\|\bfS\|_1\epsilon .
\end{equation}
\end{lemma}
\begin{proof}
Observe the following inequalities hold:
\begin{equation}
\label{eq:initial_bound}
L(\bff^*,\bflambda^*,\bfeta^*) \geq L(\bff^*;\bflambda,\bfeta) \geq  L(\bff;\bflambda,\bfeta) + \nabla_{\bff}L(\bff;\bflambda,\bfeta)^T(\bff^* - \bff).
\end{equation}
The first inequality follows from optimality of $(\bflambda^*,\bfeta^*)$, and the second inequality follows from convexity of the function 
$\bff\mapsto L(\bff;\bflambda,\bfeta)$.

We proceed by lower bounding each term in the right hand side above. We start with the term $L(\bff;\bflambda,\bfeta)$:
\begin{align*}
L(\bff;\bflambda,\bfeta) &= C(\bff) + \bflambda^T(\bfg_0 - \bff) + \bfeta^T(A\bff - \bfS)\\
& \geq C(\bff) - \sum_{i=1}^{n^2} |\lambda_i(g_i - f_i)| \\
& \geq C(\bff) -n^2\epsilon,
\end{align*}
where the final inequality follows from the approximate complementary slackness condition.

To bound the next term, we start with the lower bound
$$
\nabla_{\bff}L(\bff;\bflambda,\bfeta)^T(\bff^* - \bff) \geq -|\nabla_{\bff}L(\bff;\bflambda,\bfeta)^T(\bff^* - \bff)|.
$$
Next, we apply Holder's inequality to conclude
$$
|\nabla_{\bff}L(\bff;\bflambda,\bfeta)^T(\bff^* - \bff)| \leq \|\nabla_{\bff}L(\bff;\bflambda,\bfeta)\|_\infty \|(\bff^* - \bff)\|_1\leq \epsilon\|\bff^* - \bff\|_1,
$$
where the final upper bound follows from our approximate stationarity assumption.
Finally, we use the fact that $\bff$ satisfies the equality constraint $A\bff=\bfS$ to upper bound $\|\bff^* - \bff\|_1$ by a constant:
$$
\|\bff^* - \bff\|_1 \leq \|\bff^*\|_1 + \|\bff\|_1 = \|A\bff^*\|_1 + \|A\bff\|_1 = 2\|\bfS\|_1.
$$
Therefore we have the lower bound
$$
\nabla_{\bff}L(\bff;\bflambda,\bfeta)^T(\bff^* - \bff) \geq -2\|\bfS\|_1\epsilon.
$$

Going back to the initial bound (\ref{eq:initial_bound}), we have
$$
C(\bff^*) \geq L(\bff;\bflambda,\bfeta) + \nabla_{\bff}L(\bff;\bflambda,\bfeta)^T(\bff^* - \bff) \geq C(\bff) - n^2\epsilon - 2\|\bfS\|_1\epsilon.
$$
Rearranging the above, we have established
$$
C(\bff) - C(\bff^*) \leq n^2\epsilon + 2\|\bfS\|_1\epsilon = \epsilon',
$$
as claimed.
\end{proof}

\begin{lemma}
\label{lem:mismatched_lagrange_lb}
Let $\bar{\bff}$ and $\bar{\bflambda},\bar{\bfeta}$ be the approximate primal and dual optima defined above (with error term $\epsilon$), and let
$\bff^*$ and $\bflambda^*,\bfeta^*$ be any optimal primal and dual optimal.
Then the following inequality holds:
\begin{equation}
L(\bar{\bff};\bar{\bflambda},\bar{\bfeta}) \leq L(\bff^*;\bar{\bflambda},\bar{\bfeta}) + 2\|\bfS\|_1\epsilon.
\end{equation}
\end{lemma}
\begin{proof}
By convexity we have
$$
L(\bff^*;\bar{\bflambda},\bar{\bfeta}) \geq  L(\bar{\bff};\bar{\bflambda},\bar{\bfeta}) + \nabla_{\bff}L(\bar{\bff};\bar{\bflambda},\bar{\bfeta})^T(\bff^* - \bar{\bff}).
$$
By Holder's inequality we have
$$
\nabla_{\bff}L(\bar{\bff};\bar{\bflambda},\bar{\bfeta})^T(\bff^* - \bar{\bff}) \leq \|\nabla_{\bff}L(\bar{\bff};\bar{\bflambda},\bar{\bfeta})\|_\infty\|\bff^*-\bar{\bff}\|_1.
$$
From feasibility of $\bff^*$ and $\bar{\bff}$ we have
$$
\|\bff^*-\bar{\bff}\|_1 \leq \|\bff^*\|_1+\|\bar{\bff}\|_1 = \|A\bff^*\|_1+\|A\bar{\bff}\|_1 = 2\|\bfS\|_1.
$$
From our approximate stationarity assumption we know
$$
\|\nabla_{\bff}L(\bar{\bff};\bar{\bflambda},\bar{\bfeta})\|_\infty \leq \epsilon.
$$
Combining the above,
\begin{align*}
L(\bar{\bff};\bar{\bflambda},\bar{\bfeta}) & \leq L(\bff^*;\bar{\bflambda},\bar{\bfeta}) + \nabla_{\bff}L(\bar{\bff};\bar{\bflambda},\bar{\bfeta})^T(\bar{\bff} - \bff^*)\\
& \leq L(\bff^*;\bar{\bflambda},\bar{\bfeta}) + \|\nabla_{\bff}L(\bar{\bff};\bar{\bflambda},\bar{\bfeta})\|_\infty\|\bff^*-\bar{\bff}\|_1\\
& \leq L(\bff^*;\bar{\bflambda},\bar{\bfeta}) + 2\|\bfS\|_1\epsilon,
\end{align*}
as claimed.
\end{proof}

\begin{lemma}
\label{lem:approx_p_d_lagrange_lb}
Let $\bar{\bff}$ and $\bar{\bflambda},\bar{\bfeta}$ be the approximate primal and dual optima defined above (with error term $\epsilon$), and let
$\bff^*$ and $\bflambda^*,\bfeta^*$ be any optimal primal and dual optimal.
Then the following inequality holds:
\begin{equation}
L(\bar{\bff};\bar{\bflambda},\bar{\bfeta}) \geq L(\bff^*;\bflambda^*,\bfeta^*) - n^2\epsilon.
\end{equation}
\end{lemma}
\begin{proof}
Observe the lower bound
\begin{align*}
\bar{\bflambda}^T(\bfg_0 - \bar{\bff}) &= \sum_{i=1}^{n^2} \bar{\lambda}_i(g_i - \bar{f}_i) \\
& \geq -\sum_{i=1}^{n^2} |\bar{\lambda}_i(g_i - \bar{f}_i)| \\
& \geq -n^2\epsilon
\end{align*}
follows from our assumption that $\bar{\bff}$ satisfies approximate complementary slackness with respect to $\bar{\bflambda}$.
Therefore,
\begin{align*}
L(\bar{\bff};\bar{\bflambda},\bar{\bfeta}) &\geq C(\bar{\bff}) - n^2\epsilon\\
& \geq C(\bff^*) - n^2\epsilon\\
& = L(\bff^*;\bflambda^*,\bfeta^*) - n^2\epsilon.
\end{align*}
\end{proof}

\begin{lemma}
Let $\bar{\bff}$ and $\bar{\bflambda},\bar{\bfeta}$ be the approximate primal and dual optima defined above (with error term $\epsilon$), and let
$\bff^*$ and $\bflambda^*,\bfeta^*$ be any optimal primal and dual optimal.
Then approximate complementary slackness holds between $\bff^*$ and $\bar{\bflambda}$.
That is, for any $i=1,2,\dots,n^2$ we have the bound
\begin{equation}
\label{eq:mismatched_approx_cs}
|\bar{\lambda}_i(g_i - f^*_i)| \leq \epsilon'
\end{equation}
where
$$
\epsilon' = (2\|\bfS\|_1 + n^2)\epsilon.
$$
\end{lemma}
\begin{proof}
We establish (\ref{eq:mismatched_approx_cs}) by showing that $L(\bff^*;\bar{\bflambda},\bar{\bfeta})$ is approximately sandwiched between
$L(\bff^*;\bflambda^*,\bfeta^*)$ and $L(\bar{\bff};\bar{\bflambda},\bar{\bfeta})$.
We obtain the upper bound by the following observations:
$$
L(\bff^*;\bflambda^*,\bfeta^*) = \sup_{\bflambda,\bfeta}L(\bff^*;\bflambda,\bfeta) \geq L(\bff^*;\bar{\bflambda},\bar{\bfeta}).
$$
The lower bound
$$
L(\bar{\bff};\bar{\bflambda},\bar{\bfeta}) - 2\|\bfS\|_1\epsilon\leq L(\bff^*;\bar{\bflambda},\bar{\bfeta}) .
$$
is established in Lemma \ref{lem:mismatched_lagrange_lb}.
Furthermore, Lemma \ref{lem:approx_p_d_lagrange_lb} establishes
$$
L(\bar{\bff};\bar{\bflambda},\bar{\bfeta}) \geq L(\bff^*;\bflambda^*,\bfeta^*) - n^2\epsilon,
$$
so we have the following lower and upper bounds:
\begin{equation}
\label{eq:lagrange_sandwich}
L(\bff^*;\bflambda^*,\bfeta^*) - \epsilon' \leq L(\bff^*;\bar{\bflambda},\bar{\bfeta}) \leq L(\bff^*;\bflambda^*,\bfeta^*),
\end{equation}
where $\epsilon' = (n^2 + 2\|\bfS\|_1)\epsilon$. Note also that $L(\bff^*;\bar{\bflambda},\bar{\bfeta})$ can be written as
$$
L(\bff^*;\bar{\bflambda},\bar{\bfeta}) = L(\bff^*;\bflambda^*,\bfeta^*) + \bar{\bflambda}^T(\bfg_0 - \bff^*).
$$
Therefore, subtracting $L(\bff^*;\bflambda^*,\bfeta^*)$ from each term in (\ref{eq:lagrange_sandwich}) yields the chain of inequalities
$$
-\epsilon'\leq \bar{\bflambda}^T(\bfg_0 - \bff^*) \leq 0.
$$
From nonnegativity of $\bar{\bflambda}$ and feasibility of $\bff^*$ we know that for any $i=1,2,\dots,n^2$
the term $\bar{\lambda}_i(g_i - f^*_i)$ is nonpositive, hence
$$
-\epsilon'\leq \bar{\bflambda}^T(\bfg_0 - \bff^*)\leq \bar{\lambda}_i(g_i - f^*_i)\leq 0,
$$
establishing the claim $|\bar{\lambda}_i(g_i-f^*_i)|\leq \epsilon'$ holds for all $i=1,2,\dots,n^2$.

\end{proof}

\begin{lemma}
Let $\bar{\bff}$ and $\bar{\bflambda},\bar{\bfeta}$ be the approximate primal and dual optima defined above (with error term $\epsilon$)
Let $\bff'$ be the solution to the following optimization problem:
\begin{equation}
\inf_{\bff} L(\bff;\bar{\bflambda},\bar{\bfeta}).
\end{equation}
Then
\begin{equation}
L(\bar{\bff};\bar{\bflambda},\bar{\bfeta}) \leq L(\bff';\bar{\bflambda},\bar{\bfeta}) + \epsilon',
\end{equation}
where $\epsilon'$ is an error term converging to $0$ as $\epsilon\to 0$.
\end{lemma}
\begin{proof}
By convexity we have
$$
L(\bar{\bff};\bar{\bflambda},\bar{\bfeta}) \leq L(\bff';\bar{\bflambda},\bar{\bfeta}) +\nabla L(\bar{\bff};\bar{\bflambda},\bar{\bfeta})^T(\bar{\bff}-\bff').
$$
By Holder's inequality and the triangle inequality we have
$$
\nabla L(\bar{\bff};\bar{\bflambda},\bar{\bfeta})^T(\bar{\bff}-\bff') \leq\|\nabla L(\bar{\bff};\bar{\bflambda},\bar{\bfeta})\|_\infty (\|\bar{\bff}\|_1 + \|\bff'\|_1). 
\leq \epsilon (\|\bfS\|_1 + \|\bff'\|_1).
$$
\mccomment{We would be finished if we could bound $\|\bff'\|_1$}.
\end{proof}

}

\section{Approximate Welfare-Robustness Proof in the Two Level Model}
\label{appdx:approx_robustness_two_level_proof}
\subsection{Expected Welfare of a Strategy Profile in the Two Level Model}
\label{appdx:two_level_welfare_def}
Let $W_i^t$ denote the welfare generated in the action driver by $i$ at time $t$.  The welfare $W_i^t$
differs from the driver reward $R_i^t$ when driver $i$ fulfils a dispatch in time $t$; in this case, the rider collects utility
equal to the difference between their value for the trip and the trip price.  Let $V_i^t$ be the value held by the rider whose dispatch
driver $i$ fulfils in time $t$, if any such rider exists.  The welfare term $W_i^t$ is defined as follows:
\begin{equation}
    W_i^t=\begin{cases}
    R_i^t + (V_i^t - P_{(\ell,d)}^t) & \mbox{if }a_i^t=(\ell,d,1),\\
        R_i^t & \mbox{otherwise.}
    \end{cases}
\end{equation}

The total welfare generated by the marketplace is then the sum over welfare terms $W_i^t$ for all drivers $i$ and time periods $t$.
In the context of the two-level model, where the number of riders and drivers scales with the population-size parameter $k$, we
normalize the expected welfare by dividing by $k$, so that expected welfare terms are comparable across different population sizes.

Let $W_{\omega_t}(\bfS_t;\Pi)$ be the expected welfare-to-go given a strategy profile $\Pi$, a population size $k$, as a function
of the market state $(\omega_t,\bfS_t)$:
\begin{equation}
    W_{\omega_t}(\bfS_t;\Pi,k) = \frac{1}{k}\bbE^\Pi\left[\sum_{i\in\calM}\sum_{\tau=t}^T W_i^\tau \mid (\omega_t,\bfS_t)\right].
\end{equation}
Notice the expected welfare term $ \Phi_{\omega_t}(\bfS_t;\Pi) $ implicitly depends on the population size parameter $k$, but we
omit this dependence from the notation for convenience.

Recall that $\Phi_{\omega_t}(\bfS_t)$ denotes the optimal value of the state-dependent optimization problem given the state $(\omega_t,\bfS_t)$,
which corresponds to the optimal welfare achievable given the state $(\omega_t,\bfS_t)$ in the stochastic fluid model.
The following Lemma states that the optimal fluid welfare is always an upper bound on the expected welfare in the two-level model.
\begin{lemma}
    Let $\Pi$ be any strategy profile and consider any time $t$ with state $(\omega_t,\bfS_t)$.
    The optimal welfare from the state $(\omega_t,\bfS_t)$ in the fluid model is always larger than the expected welfare generated by $\Pi$ in the
    two-level model, i.e.
    \begin{equation}
        \label{eq:fluid_welfare_ub}
    \Phi_{\omega_t}(\bfS_t) \geq W_{\omega_t}(\bfS_t;\Pi)
    \end{equation}
    always holds.
\end{lemma}
\subsection{Proof of Theorem \ref{thm:robust_eqlbm_two_level}}
In this section we summarize the key steps we take to prove Theorem \ref{thm:robust_eqlbm_two_level}.
We prove Theorem \ref{thm:robust_eqlbm_two_level} by backwards induction on the time period.
Fix a time period $t\in [T]$.  For each $k\geq 1$ assume there exists an error term $\epsilon_{t+1}(k)$, 
converging to $0$ as $k\to\infty$, such that the the following properties hold:
\begin{itemize}
\item For every market state $(\omega_{t+1},\bfS_{t+1})\in\calS_{t+1}(\gamma)$, and any approximate equilibrium $\Pi\in\calP^k$,
$$
\Phi_{\omega_{t+1}}(\bfS_{t+1}) - W_{\omega_{t+1}}(\bfS_{t+1};\Pi,k) \leq \epsilon_{t+1}(k).
$$
\item For any approximate equilibrium $\Pi\in\calP^k$, at time $t+1$, every driver has expected utility-to-go that is close to the partial derivative of the fluid optimization function.
        Specifically, for any market state $(\omega_{t+1},\bfS_{t+1})$ and any driver $i\in\calM_{t+1}$
        whose location is $\ell_i^{t+1}=\ell$,
        the following bound holds:
        \begin{equation} 
            \label{eq:two_level_robustness_bi_assn}
        \left| \bbE^{\Pi}[U_i^{t+1} \mid \ell_i^{t+1}=\ell, (\omega_{t+1},\bfS_{t+1})] - \frac{\partial}{\partial S_\ell}\Phi_{\omega_{t+1}}(\bfS_{t+1})\right| \leq \epsilon_{t+1}(k).
        \end{equation}
\end{itemize}

Our proof technique is to convert the stochastic actions taken by a strategy profile $\Pi$ to an approximate equilibrium fluid strategy.
We proceed via a series of lemmas.  The first Lemma shows that in an approximate equilibrium, drivers at the same location use approximately the same disutility
acceptance thresholds.
We provide the proof of Lemma \ref{lem:common_threshold} in Appendix \ref{appdx:common_threshold_proof_}

\begin{lemma}
    \label{lem:common_threshold}
    There exists an error function $\delta(k)\geq 0$, converging to $0$ as $k\to\infty$, such that the following is true:
    For any $k\geq 1$, any approximate equilibrium $\Pi\in\calP^k$, and any market state $(\omega_t,\bfS_t)\in\calS_t(\gamma)$,
    let $\bfx_i=(x_i^d : d\in\calL)$ be the disutility acceptance threshold used by each active driver $i\in\calM_t$.
    Then for each location $\ell$ there exists a disutility threshold vector 
    $\bfx_\ell = (x_{(\ell,d)} : d\in\calL)$ 
    such that the disutility threshold vector used by every driver positioned at $\ell$ is at most $\delta(k)$ away from 
    $\bfx_\ell$, i.e.
    \begin{equation}
        \max_{d\in\calL}|x_i^d - x_{(\ell,d)}| \leq \delta(k) , \ \ \ \ \forall i\in\calM_{t,\ell}.
    \end{equation}
\end{lemma}

Our next Lemma shows that the stochastic actions taken under equilibrium strategy profile $\Pi$ concentrate
towards the fluid actions which arise under the common disutility threshold established by Lemma \ref{lem:common_threshold}.

For any $k\geq 1$, $\Pi\in\calP^k$ and $(\omega_t,\bfS_t)\in\calS_t(\gamma)$, 
let $(\bff,\bfg)$ encode the stochastic trip volumes that occur on each route. 
Recall that $a_i^t$ is a random variable which corresponds to the action that each active driver $i\in\calM_t$ takes
in time period $t$.
Define $g_{(\ell,d)}$ by 
\begin{equation}
    \label{eq:two_level_g}
     g_{(\ell,d)} = \frac{1}{k} \sum_{i\in\calM_t} \bfone\left\{a_i^t = (\ell,d,1)\right\},\ \ \ \ \forall (\ell,d)\in\calL^2
\end{equation}
and $f_{(\ell,d)}$ by
\begin{equation}
    \label{eq:two_level_f}
    f_{(\ell,d)} = g_{(\ell,d)} + \frac{1}{k} \sum_{i\in\calM_t} \bfone\left\{a_i^t = (\ell,d,0)\right\}.\ \ \ \ \forall (\ell,d)\in\calL^2
\end{equation}
Note that $(\bff,\bfg)$
are random variables that depend on stochastic dispatch demand, stochastic add-passenger disutilities, as well as the randomness inherent
to the matching process.  

We want to compare $(\bff,\bfg)$ to trips that arise in a fluid version of the strategy $\Pi$.
Define a fluid strategy profile that produces a disutility threshold vector $\bfx$, where $\bfx$ is the common
disutility threshold vector whose existence is established by Lemma \ref{lem:common_threshold}, and define the relocation
distribution $\bfe$ using the relocation destinations selected by all drivers in the population:
For each destination $d$ define
$$e_{(\ell,d)} = 1/|\calM_{t,\ell}| \sum_{i\in\calM_{t,\ell}}\bfone\{r_i = d\}$$
where $r_i$ is the relocation destination selected by each driver $i$.
Let $\bar{\bff},\bar{\bfg}$ be the fluid trip volumes that occur under the market state $(\omega_t,\bfS_t)$ and
the fluid strategy $(\bfx,\bfe)$.

Recall that $\bfS_{\omega_{t+1}}(\bff)$ is the stochastic time $t+1$ supply location vector, given the
time $t$ trips $\bff$ and a time $t+1$ scenario $\omega_{t+1}$, and $\bar{\bfS}_{\omega_{t+1}}(\bar{\bff})$
is the deterministic time $t+1$ supply location vector arising in the fluid model.  The following Lemma
states that $\bfS_{\omega_{t+1}}(\bff)$ is close to $\bar{\bfS}_{\omega_{t+1}}(\bar{\bff})$ with high probability
for large $k$. We provide the proof of Lemma \ref{lem:supply_vec_concentration} in Appendix \ref{appdx:supply_vec_concentration_proof}.

\begin{lemma}
    \label{lem:supply_vec_concentration}
    There exists an error function $\kappa(k)\geq 0$, and another function $q(k)\geq 0$, both 
    converging to $0$ as $k\to\infty$, such that the following is true:
    For any $k\geq 1$, any approximate equilibrium $\Pi\in\calP^k$, and any market state $(\omega_t,\bfS_t)\in\calS_t(\gamma)$,
    let $(\bff,\bfg)$ be the stochastic trip volumes and $(\bar{\bff},\bar{\bfg})$ be the fluid trip volumes
    as defined above.
    Then following inequalities holds:
    \begin{align*}
        \bbP\left(\|\bfg - \bar{\bfg}\|_1 \geq \kappa(k)\right) &\leq q(k),\\
        \bbP\left(\|\bff - \bar{\bff}\|_1 \geq \kappa(k)\right) &\leq q(k),\\
        \bbP\left(\|\bfS_{\omega_{t+1}}(\bff) - \bar{\bfS}_{\omega_{t+1}}(\bar{\bff})\|_1 \geq \kappa(k)\right) &\leq q(k).
    \end{align*}
\end{lemma}

Our final lemma shows that $(\bar{\bff},\bar{\bfg},\bfx)$ corresponds to an approximate equilibrium in the fluid model.
We defer the proof of Lemma \ref{lem:fluid_eqlbm_from_two_level_eqlbm} to Appendix \ref{appdx:fluid_eqlbm_from_two_level_eqlbm}.

\begin{lemma}
    \label{lem:fluid_eqlbm_from_two_level_eqlbm}
    There exists an error function $\iota(k)\geq 0$, converging to $0$ as $k\to\infty$, such that the following is true:
    For any $k\geq 1$, any approximate equilibrium $\Pi\in\calP^k$, and any market state $(\omega_t,\bfS_t)\in\calS_t(\gamma)$,
    let  $(\bar{\bff},\bar{\bfg})$ be the fluid trip volumes associated with $\Pi$ under the market state $(\omega_t,\bfS_t)$,
    and let $\bfx$ be the common disutility threshold that drivers use under $\Pi$, established in Lemma \ref{lem:common_threshold}.
    Then $(\bar{\bff},\bar{\bfg},\bfx)$ corresponds to an $\iota(k)$-approximate equilibrium for the fluid model.
\end{lemma}

Lemma \ref{lem:supply_vec_concentration} and \ref{lem:fluid_eqlbm_from_two_level_eqlbm} are sufficient to
conclude the proof of Theorem \ref{thm:robust_eqlbm_two_level}.

From the second part of Theorem \ref{thm:robust_eqlbm_fluid} we know that there is an error function $\epsilon(k)\geq 0$ which
converges to $0$ as $k\to\infty$, such that the welfare achieved by $(\bar{\bff},\bar{\bfg},\bfx)$ in the fluid model
is at most $\epsilon(k)$ off from the optimal welfare.

Next, observe a consequence of Lemma \ref{lem:supply_vec_concentration} is the expected welfare under $\Pi$ converges
to the welfare of the associated fluid strategy, and the expected utility of drivers at $\ell$ under $\Pi$
converges to the expected utility of drivers at $\ell$ in the associated fluid strategy.  That is, there exists
an error function $\theta(k)\geq 0$ converging to $0$ as $k\to\infty$ such that
$$
|W_{\omega_t}(\bfS_t;\Pi,k) - \calW_{\omega_t}(\bar{\bff},\bar{\bfg},\bfx)| \leq \theta(k),
$$
and
$$
\left| \bbE^{\Pi}[U_i^{t} \mid \ell_i^{t}=\ell, (\omega_{t},\bfS_{t})] - \frac{\partial}{\partial S_\ell}\Phi_{\omega_{t}}(\bfS_{t})\right| \leq \theta(k)
$$
holds for every $k\geq 1$, $\Pi\in\calP^k$, $(\omega_t,\bfS_t)\in\calS_t(\gamma)$.

The above inequality establishes the second part of our backwards induction assumption.  
The first part our backwards induction assumption also follows:
$$
W_{\omega_t}(\bfS_t;\Pi,k) \geq \calW_{\omega_t}(\bar{\bff},\bar{\bfg},\bfx) - \theta(k) \geq \Phi_{\omega_t}(\bfS_t) - \epsilon(k) - \theta(k).
$$

\subsection{Proof of Lemma \ref{lem:common_threshold}}
\label{appdx:common_threshold_proof_}
Below is a restatement of Lemma \ref{lem:common_threshold}.
\begin{lemma*}
    There exists an error function $\delta(k)\geq 0$, converging to $0$ as $k\to\infty$, such that the following is true:
    For any $k\geq 1$, any approximate equilibrium $\Pi\in\calP^k$, and any market state $(\omega_t,\bfS_t)\in\calS_t(\gamma)$,
    let $\bfx_i=(x_i^d : d\in\calL)$ be the disutility acceptance threshold used by each active driver $i\in\calM_t$.
    Then for each location $\ell$ there exists a disutility threshold vector 
    $\bfx_\ell = (x_{(\ell,d)} : d\in\calL)$ 
    such that the disutility threshold vector used by every driver positioned at $\ell$ is at most $\delta(k)$ away from 
    $\bfx_\ell$, i.e.
    \begin{equation*}
        \max_{d\in\calL}|x_i^d - x_{(\ell,d)}| \leq \delta(k) , \ \ \ \ \forall i\in\calM_{t,\ell}.
    \end{equation*}
\end{lemma*}
\begin{proof}
    Fix $k\geq 1$, let $\Pi\in\calP^k$ be an $\alpha_k$-equilibrium, and let $(\omega_t,\bfS_t)\in\calS_t(\gamma)$ be
    any market state with less than $\gamma$ total driver volume.

    Consider any two drivers $i$ and $j$ positioned at the same location $\ell$, and suppose there is some destination $d$
    for which driver $i$ uses a lower acceptance threshold than driver $j$, i.e. $x^i_d < x^j_d$.
    Let $x\in (x_i^d, x_j^d)$ be a number in between the two thresholds and
    consider the following events:
    \begin{enumerate}
    \item Driver $i$ is allocated a dispatch towards $d$ and samples $X_i=x$ as their add-passenger disutility.
    \item Driver $j$ is allocated a dispatch towards $d$ and samples $X_j=x$ as their add-passenger disutility.
    \end{enumerate}
    In event 1, driver $i$ rejects the dispatch towards $d$ and drives empty towards their relocation destination $r_i$.
    The utility-to-go that driver $i$ collects conditioned on event 1 is thus equal to
    \begin{equation}
        \bbE[U^t_i\mid d_i=d, X_i=x] = -c_{(\ell,r_i)} +\bbE\left[U_i^{t+1}|\ell_i^{t+1}=r_i\right].
    \end{equation}
    From the definition of $\Pi$ being an approximate equilibrium, we know that conditioned on event 1 the alternate action
    of accepting the dispatch towards $d$ can only increase driver $i$'s utility to go by at most $\alpha_k$.
    This implies the inequality
    \begin{equation}
        \label{eq:event_1_first_bound}
        -c_{(\ell,r_i)} +\bbE\left[U_i^{t+1}|\ell_i^{t+1}=r_i\right] + \alpha_k \geq
        P_{(\ell,d)} - x - c_{(\ell,d)} + \bbE\left[U_i^{t+1}\mid \ell_i^{t+1}=d\right].
    \end{equation}
    Let $Z=-c_{(\ell,d')} + \alpha_k - P_{(\ell,d)} + x + c_{(\ell,d)}$ be a temporary variable to track the non-utility to go
    terms from the above expression, so we have
    \begin{equation}
    Z + \bbE\left[U_i^{t+1}|\ell_i^{t+1}=r_i\right] \geq \bbE\left[U_i^{t+1}\mid \ell_i^{t+1}=d\right].
    \end{equation}
    From the backwards induction assumption (\ref{eq:two_level_robustness_bi_assn}), the time $t+1$ utilities
    are approximately functions of the locations and market state:
    \begin{align}
        & Z + \bbE\left[\frac{\partial}{\partial S_{r_i}}\Phi_{\omega_{t+1}}(\bfS_{t+1})+ \epsilon_{t+1}(k)\mid \ell_i^{t+1}=r_i \right]\label{eq:start_line}
        \\
        \geq\ \ \ &Z + \bbE\left[U_i^{t+1}|\ell_i^{t+1}=r_i\right] \\
        \geq\ \ \  &\bbE\left[U_i^{t+1}\mid \ell_i^{t+1}=d\right]\\
        \geq\ \ \  & \bbE\left[\frac{\partial}{\partial S_{d}}\Phi_{\omega_{t+1}}(\bfS_{t+1})-\epsilon_{t+1}(k)\mid \ell_i^{t+1}=d \right].\label{eq:end_line}
    \end{align}
    In the above, $\bfS_{t+1}$ is the (stochastic) supply-location vector the time $t+1$, and $\epsilon(\omega_{t+1},\bfS_{t+1})$ is the
    error term provided in the backwards induction assumption Assumption \ref{assn:backwards_induction}.

    Next, Assumption \ref{assn:matching_process} states 
    that we can move from a conditional expectation to an unconditional expectation, at the cost of an error term $\beta_k$
    which converges to $0$ as $k\to\infty$:  \mccomment{This is a handwavy argument.  Come back and add details soon.}
    \begin{equation}
        \label{eq:event_1_final_bound}
    Z + \bbE\left[\frac{\partial}{\partial S_{r_i}}\Phi_{\omega_{t+1}}(\bfS_{t+1}) - 
    \frac{\partial}{\partial S_{d}}\Phi_{\omega_{t+1}}(\bfS_{t+1})
    \right] + 2 \epsilon_{t+1}(k) + \beta_k \geq 0.
    \end{equation} 

    In the second event, driver $j$ accepts the dispatch and drives a passenger towards $d$.  The utility to go that driver $j$ collects
    conditioned on event 2 is equal to
    \begin{equation*}
        \bbE[U_j^t\mid d_j=d, X_j=x] = P_{(\ell,d)} - x - c_{(\ell,d)} + \bbE\left[U_j^{t+1}\mid \ell_j^{t+1}=d\right].
    \end{equation*}
    From $\Pi$ being an approximate equilibrium, we know that the utility of accepting the dispatch towards $d$ is at most $\alpha_k$ short of
    the utility from any other action, in particular it is at most $\alpha_k$ short from the utility of taking a relocation
    trip towards driver $i$'s destination $r_i$.
    This implies the inequality
    \begin{equation*}
        P_{(\ell,d)} - x - c_{(\ell,d)} + \bbE\left[U_j^{t+1}\mid \ell_j^{t+1}=d\right]+ \alpha_k \geq 
        -c_{(\ell,r_i)} +\bbE\left[U_j^{t+1}|\ell_j^{t+1}=r_i\right].
            \end{equation*}
    Following the same steps as before,
    we deduce the bound
    \begin{equation}
        \label{eq:event_2_final_bound}
    Z' + \bbE\left[\frac{\partial}{\partial S_{d}}\Phi_{\omega_{t+1}}(\bfS_{t+1}) - 
        \frac{\partial}{\partial S_{r_i}}\Phi_{\omega_{t+1}}(\bfS_{t+1})  
        \right]  + 2\epsilon_{t+1}(k) + \beta_k \geq 0,
    \end{equation}
    where
    $Z' =  -Z + 2\alpha_k.
    $
    
    Now, recall the terms $Z$ in equation (\ref{eq:event_1_final_bound}) and $Z'$ in equation (\ref{eq:event_2_final_bound}) include an
    arbitrary disutility threshold $x$ in between $x^d_i$ and $x^d_j$.
    Let us now write $Z(x)$ and $Z'(x)$ to explicitly denote the dependence on $x$.  Next, consider what happens when we add
    the equations (\ref{eq:event_1_final_bound})
    using the threshold $x=x_i^d$ and (\ref{eq:event_2_final_bound}) using the threshold $x=x_j^d$.
    The partial derivative terms cancel, and we are left with
    \begin{equation*}
        Z(x_i^d) + Z'(x_j^d) + 4\epsilon_{t+1}(k) + 2\beta_k \geq 0.
    \end{equation*}
    Observe that $Z(x_i^d) + Z'(x_j^d) = x_i^d - x_j^d + 2\alpha_k$. Therefore, we have
    \begin{equation*}
        x_j^d - x_i^d \leq 2\alpha_k + 4\epsilon_{t+1}(k) + \beta_k.
    \end{equation*}

    Therefore, the Lemma holds by setting $\delta_k$ equal to the right hand side of the above equation
    and taking $\bfx_\ell$ to be the threshold vector $\bfx_i$ used by any driver located at $\ell$.
\end{proof}

\subsection{Proof of Lemma \ref{lem:supply_vec_concentration}}
\label{appdx:supply_vec_concentration_proof}
Below is a restatement of Lemma \ref{lem:supply_vec_concentration}.
\begin{lemma*}
    There exists an error function $\kappa(k)\geq 0$, and another function $q(k)\geq 0$, both 
    converging to $0$ as $k\to\infty$, such that the following is true:
    For any $k\geq 1$, any approximate equilibrium $\Pi\in\calP^k$, and any market state $(\omega_t,\bfS_t)\in\calS_t(\gamma)$,
    let $(\bff,\bfg)$ be the stochastic trip volumes and $(\bar{\bff},\bar{\bfg})$ be the fluid trip volumes
    as defined above.
    Then following inequalities holds:
    \begin{align*}
        \bbP\left(\|\bfg - \bar{\bfg}\|_1 \geq \kappa(k)\right) &\leq q(k),\\
        \bbP\left(\|\bff - \bar{\bff}\|_1 \geq \kappa(k)\right) &\leq q(k),\\
        \bbP\left(\|\bfS_{\omega_{t+1}}(\bff) - \bar{\bfS}_{\omega_{t+1}}(\bar{\bff})\|_1 \geq \kappa(k)\right) &\leq q(k).
    \end{align*}
\end{lemma*}
\begin{proof}
These bounds follow from the matching process concentration results discussed in Appendix \ref{sec:matching_process_appdx}.

From Lemma \ref{lem:matching_process_concentration} we have concentration functions $\epsilon_1(k)$, $q_1(k)$,
with $\epsilon_1(k)/k\to 0$ and $q_1(k)\to 0$ as $k\to 0$,
such that 
$$
\bbP\left(|G_{(\ell,d)} - \bar{G}_{(\ell,d)}| \geq \epsilon_1(k)\right) \leq q_1(k),
$$
where $G_{(\ell,d)}$ and $\bar{G}_{(\ell,d)}$ are the unscaled number of dispatches and fluid dispatches along 
the route $(\ell,d)$.
Take $\kappa_1(k) = \epsilon_1(k)|\calL|^2/k$, $p_1(k) = |\calL|^2 q_1(k)$ and observe:
\begin{align*}
\bbP\left(\|\bfg - \bar{\bfg}\|_1 \geq \kappa_1(k)\right) &\leq \sum_{(\ell,d)}\bbP(|g_{(\ell,d)} - \bar{g}_{(\ell,d)}| \leq \kappa_1(k)/|\calL|^2) \\
&=\sum_{(\ell,d)}\bbP(|G_{(\ell,d)} - \bar{G}_{(\ell,d)}| \leq \epsilon_1(k))\\
& \leq |\calL|^2 q_1(k) = p_1(k)
\end{align*}

From Lemma \ref{lem:asymptotic_concentration} we have concentration functions $\epsilon_2(k)$, $q_2(k)$,
with $\epsilon_2(k)/k\to 0$ and $q_2(k)\to 0$ as $k\to 0$, such that 
$$
\bbP\left(|H_{(\ell,d)} - \bar{H}_{(\ell,d)}| \geq \epsilon_2(k)\right) \leq q_2(k),
$$
where $H_{(\ell,d)}$ is the number of relocation trips along $(\ell,d)$.
Take $\kappa_2(k) =\kappa_1(k) +  \epsilon_2(k)|\calL|^2/k$ and $p_2(k) =p_1(k) +  |\calL|^2 q_2(k) $ and observe:
\begin{align*}
\bbP\left(\|\bff - \bar{\bff}\|_1 \geq \kappa_2(k)\right) &\leq 
\bbP \left(\|\bff - \bar{\bff}\|_1 \geq \kappa_2(k)\mid \|\bfg - \bar{\bfg}\|_1 \leq \kappa_1(k)\right)
+ \bbP\left(\|\bfg - \bar{\bfg}\|_1 \geq \kappa_1(k)\right)\\
& \leq \sum_{(\ell,d)} \bbP \left(|H_{(\ell,d)} - \bar{H}_{(\ell,d)}| \geq \epsilon_2(k)\right) + p_1(k)\\
& \leq |\calL|^2 q_2(k) + p_1(k) = p_2(k).
\end{align*}

Finally, we observe that
$$\|\bfS_{\omega_{t+1}}(\bff) - \bar{\bfS}_{\omega_{t+1}}(\bar{\bff})\|_1\leq \|\bff - \bar{\bff}\|_1 + 
\|\bfM_{\omega_{t+1}} - \bar{\bfM}_{\omega_{t+1}}\|_1
$$
where $\bfM_{\omega_{t+1}}$ is the vector counting driver-entry at time $t+1$ under scenario $\omega_{t+1}$.
By Assumption \ref{assn:entry_distributions} we have concentration functions $\epsilon_3(k)$ and $q_3(k)$
which bound the convergence of $\bfM_{\omega_{t+1}}$ to $\bar{\bfM}_{\omega_{t+1}}$.
We define $\kappa_3(k) = \kappa_2(k) + \epsilon_3(k)|\calL|^2/k$ and $p_3(k) = p_2(k) + |\calL|q_3(k)$.

\end{proof}

\subsection{Proof of Lemma \ref{lem:fluid_eqlbm_from_two_level_eqlbm}}
\label{appdx:fluid_eqlbm_from_two_level_eqlbm}

Below is a restatement of Lemma \ref{lem:fluid_eqlbm_from_two_level_eqlbm}.
\begin{lemma*}
    There exists an error function $\iota(k)\geq 0$, converging to $0$ as $k\to\infty$, such that the following is true:
    For any $k\geq 1$, any approximate equilibrium $\Pi\in\calP^k$, and any market state $(\omega_t,\bfS_t)\in\calS_t(\gamma)$,
    let  $(\bar{\bff},\bar{\bfg})$ be the fluid trip volumes associated with $\Pi$ under the market state $(\omega_t,\bfS_t)$,
    and let $\bfx$ be the common disutility threshold that drivers use under $\Pi$, established in Lemma \ref{lem:common_threshold}.
    Then $(\bar{\bff},\bar{\bfg},\bfx)$ corresponds to an $\iota(k)$-approximate equilibrium for the fluid model.
\end{lemma*}
\begin{proof}
We outline the proof for Lemma \ref{lem:fluid_eqlbm_from_two_level_eqlbm}.  For any $k\geq 1$ and $\alpha_k$-equilibrium $\Pi$, 
let $\Sigma$ be the fluid strategy we associate with $\Pi$, which maps market states $(\omega_t,\bfS_t)$ to a disutility threshold vector $\bfx$ that is
approximately common to all drivers (see Lemma \ref{lem:common_threshold}), and a relocation distribution $\bfe$ which is derived from the population distribution
of relocation destinations.

\Xomit{
This tells us that, for any market state $(\omega_{t+1},\bfS_{t+1})$, the utility-to-go collected by a driver at $\ell$ is
approximately equal to $\frac{\partial}{\partial S_\ell}\Phi_{\omega_{t+1}}(\bfS_{t+1})$.
This is also approximately equal to the utility to go of a driver at the same location and state in the two level model under $\Pi$,
from our backwards induction assumption (\ref{eq:two_level_robustness_bi_assn}).
}

To show that $\Sigma$ is an approximate equilibrium \pfedit{under the fluid model}, it is sufficient to show that the continuation utility under the two-level model with strategy $\Pi$ \pfedit{(which is approximately incentive compatible)}
is close to the continuation utility in the fluid model with strategy $\Sigma$.
This can be accomplished by providing a bound 
$$
\left|\bbE\left[\frac{\partial}{\partial S_\ell}\Phi_{\omega_{t+1}}(\bfS_{t+1}(\bff)) - \frac{\partial}{\partial S_\ell}\Phi_{\omega_{t+1}}(\bar{\bfS}_{t+1}(\bar{\bff}))\right]\right| \le \epsilon_k,
$$
that holds uniformly in the market state and location $\ell$, where the constant
$\epsilon_k$ goes to $0$ as $k\to\infty$.
Approximate incentive compatibility conditions on $\Sigma$ then follow from \pfedit{approximate} incentive compatibility conditions on $\Pi$.

To provide this bound, recall that Lemma \ref{lem:unique_dual} establishes that the partial derivative $\frac{\partial}{\partial S_\ell}\Phi_{\omega_{t+1}}(\cdot)$ exists and is continuous.  
The space of supply-location vectors $\bfS$ satisfying $\sum_\ell S_\ell \leq \gamma$ is compact, and a continuous function over a compact set is 
bounded.  Let $U$ be the maximum value.

    Moreover, a continuous function over a compact set is uniformly continuous.  Therefore, for every $\epsilon>0$ there exists a $\delta(\epsilon) > 0$ such that,
    for any $\bfS_1,\bfS_2$ in $\calS(\gamma)$ satisfying $\|\bfS_1-\bfS_2\|\leq\delta$, we have $\|
\frac{\partial}{\partial S_\ell}\Phi_{\omega_{t+1}}(\bfS_1) - \frac{\partial}{\partial S_\ell}\Phi_{\omega_{t+1}}(\bfS_2)
\|\leq\epsilon$.

For any $\delta>0$, define
$$
    \epsilon(\delta) = \max\left(
    \|\frac{\partial}{\partial S_\ell}\Phi_{\omega_{t+1}}(\bfS_1) - \frac{\partial}{\partial S_\ell}\Phi_{\omega_{t+1}}(\bfS_2)
    \| : \bfS_1,\bfS_2\in\calS(\gamma), \|\bfS_1-\bfS_2\|\leq\delta\right).
$$
    Because the partial derivative $\frac{\partial}{\partial S_\ell}\Phi_{\omega_{t+1}}(\cdot)$ is uniformly continuous, the error term $\epsilon(\delta)$
    goes to $0$ as $\delta$ goes to $0$.

Next, we use the fact from Lemma \ref{lem:supply_vec_concentration} that $\bfS_{\omega_{t+1}}(\bff) $ concentrates towards $\bar{\bfS}_{\omega_{t+1}}(\bff) $.
Let $\kappa(k)$, $q(k)$ be functions such that 
$$
        \bbP\left(\|\bfS_{\omega_{t+1}}(\bff) - \bar{\bfS}_{\omega_{t+1}}(\bar{\bff})\|_1 \geq \kappa(k)\right) \leq q(k).
$$
    For simplicity, write $\bfS = \bfS_{\omega_{t+1}}(\bff)$, $\bar{\bfS} = \bar{\bfS}_{\omega_{t+1}}(\bar{\bff})$, and
    $F(\bfS) = \frac{\partial}{\partial S_\ell}\Phi_{\omega_{t+1}}(\bfS)$.
We have:
\begin{align*}
    \bbE[F(\bfS) - F(\bar{\bfS})]| & \leq  \bbE[F(\bfS) - F(\bar{\bfS})\mid |\bfS-\bar{\bfS}|\leq \kappa_k]|  +U q_k\\
    & \leq \epsilon(\kappa_k) + U(1-q_k).
\end{align*}
    Therefore the term we needed to bound is uniformly bounded by a term $\epsilon(\kappa_k) + Uq_k$ which converges to $0$ as $k\to\infty$, so incentive compatibility conditions on $\Sigma$ follow
from incentive compatibility conditions on $\Pi$.
\end{proof}


\Xomit{
\section{Stochastic Fluid Problem}
\subsection{Proof of Lemma \ref{lem:reward_fn}}
\label{appdx:reward_fn_proof}
\subsection{Proof of Lemma \ref{lem:disutility_cost_fn}}
\label{appdx:disutility_cost_fn_proof}
\subsection{Proof of Lemma \ref{lem:optimality_conditions}}
\label{appdx:optimality_conditions_proof}
\subsection{Optimality Conditions}
\mccomment{this section proves Lemma \ref{lem:optimality_conditions}. I'll rewrite this section in theorem-proof style soon, and I'll move it 
to the appendix.}
To obtain Lagrangian optimality conditions for the above problem, we need to associated dual variables with each
of the constraints in the optimization problem (\ref{eq:static_opt_formulation}).
Before doing so, let us rewrite each of the constraints (\ref{eq:f_nonneg_constraint} - \ref{eq:flow_constraint})
in a way that doesn't change the problem structure but that will lead to more interpretable optimality conditions.
We apply two transformations: first, we rewrite each of the inequality constraints so that $0$ is the upper bound, and 
second, we multiply both sides of each constraint by the probability of the scenario $\omega_t$ occurring, where $\omega_t$
is the scenario component for the relevant arc or node with which the constraint is associated.
The resulting constraints, and the dual variables we associate with them, are stated as follows:
\begin{itemize}
    \item Constraint (\ref{eq:f_nonneg_constraint}) becomes $-\bbP(\omega_t)f_{(o,d,\omega_t)} \leq 0$. 
            We associate a dual variable $\alpha_{(o,d,\omega_t)}$ for all $(o,d,\omega_t)\in\calA$.
    \item Constraint (\ref{eq:g_nonneg_constraint}) becomes $-\bbP(\omega_t)g_{(o,d,\omega_t)} \leq 0$. 
            We associate a dual variable $\beta_{(o,d,\omega_t)}$ for all $(o,d,\omega_t)\in\calA$.
    \item Constraint (\ref{eq:g_lt_f_constraint}) becomes $\bbP(\omega_t)\left[g_{(o,d,\omega_t)} -f_{(o,d,\omega_t)}\right]\leq 0$. 
            We associate a dual variable $\gamma_{(o,d,\omega_t)}$ for all $(o,d,\omega_t)\in\calA$.
    \item Constraint (\ref{eq:flow_constraint}) becomes 
        $$
        \begin{rcases}
            \bbP(\omega_t)\sum_{d\in\calL} f_{(\ell,d,\omega_t)} =
                \bbP(\omega_t)\bar{M}_{(\ell,\omega_t)} & \mbox{if }t=1\\
                \bbP(\omega_t)\left[\sum_{d\in\calL} f_{(\ell,d,\omega_t)} 
                - \sum_{o\in\calL} f_{(o,\ell,\omega_{t-1})}\right] =
                \bbP(\omega_t)\bar{M}_{(\ell,\omega_t)} & \mbox{if }t > 1
        \end{rcases} \forall (\ell,\omega_t)\in\calN.
        $$
        We associate a dual variable $\eta_{(\ell,\omega_t)}$ for all $(\ell,\omega_t)\in\calN$.
\end{itemize}
Note that the constraints written above are equivalent ways of writing the constraints as they originally appear 
in the optimization problem (\ref{eq:static_opt_formulation}).  We write the constraints in this manner to so that the resulting
Lagrangian optimality conditions have a simple and meaningful interpretation.

We obtain the Lagrangian for our problem by taking the negative of $W(\bff,\bfg)$, since our formulation above is a concave problem,
and incorporating the constraints into the objective as follows:
\begin{align}
    L(\bff,\bfg;\bm{\alpha},\bm{\beta},\bm{\gamma},\bm{\eta}) =  -W(\bff,\bfg) &
    - \sum_{\substack{(o,d,\omega_t)\\ \in\calA}}\bbP(\omega_t)f_{(o,d,\omega_t)}\alpha_{(o,d,\omega_t)} \nonumber\\
    &
    - \sum_{\substack{(o,d,\omega_t)\\ \in\calA}}\bbP(\omega_t)g_{(o,d,\omega_t)}\beta_{(o,d,\omega_t)}\nonumber\\
    & + \sum_{\substack{(o,d,\omega_t)\\ \in\calA}}\bbP(\omega_t)(g_{(o,d,\omega_t)} - f_{(o,d,\omega_t)})\gamma_{(o,d,\omega_t)}\nonumber\\
    &
    + 
    \sum_{(\ell,\omega_t)\in\calN}\bbP(\omega_t)\left(B_{(\ell,\omega_t)}(\bff)-\bar{M}_{(\ell,\omega_t)}\right)\eta_{(\ell,\omega_t)}.
\end{align}
In the final line of the above equation, we use the notation $B_{(\ell,\omega_t)}(\bff)$ to mean the difference between outgoing flow and incoming
flow at the node $(\ell,\omega_t)$ under the decision variable vector $\bff$.  In other words, $B_{(\ell,\omega_t)}(\bff)$ is equal
to the left hand side of the flow-conservation constraint (\ref{eq:flow_constraint}).

An optimal solution $(\bff,\bfg)$ is characterized by the existence of dual variables $(\bm{\alpha},\bm{\beta},\bm{\gamma},\bm{\eta})$
for which the complementary slackness conditions are satisfied, and for which the following stationarity conditions hold:
\begin{align}
    \frac{\partial}{\partial g_{(o,d,\omega_t)}}L(\bff,\bfg;\bm{\alpha},\bm{\beta},\bm{\gamma},\bm{\eta}) &= 0 \ \ \ \forall (o,d,\omega_t)\in\calA,
    \label{eq:L_deriv_g}\\
    \frac{\partial}{\partial f_{(o,d,\omega_t)}}L(\bff,\bfg;\bm{\alpha},\bm{\beta},\bm{\gamma},\bm{\eta}) &= 0 \ \ \  \forall (o,d,\omega_t)\in\calA.
    \label{eq:L_deriv_f}
\end{align}

To evaluate each of the above derivatives, we start by evaluating the derivative of $W(\bff,\bfg)$ with respect to each decision variable.
We start by evaluating the derivative of $W(\bff,\bfg)$ with respect to $g_{(o,d,\omega_t)}$:
\begin{align}
    \frac{\partial}{\partial g_{(o,d,\omega_t)}} W(\bff,\bfg) &= \bbP(\omega_t)\left[
        \frac{\partial}{\partial g_{(o,d,\omega_t)}} U_{(o,d,\omega_t)}(f_{(o,d,\omega_t)},g_{(o,d,\omega_t)}) - 
        \frac{\partial}{\partial g_{(o,d,\omega_t)}} A\left(\bfg^T\bfone_{(o,\omega_t)},\bff^T\bfone_{(o,\omega_t)}\right)
        \right]\nonumber\\
    &= \bbP(\omega_t)\left[\rho_{(o,d,\omega_t)}(g_{(o,d,\omega_t)}) 
                            -C\frac{\bfg^T\bfone_{(o,\omega_t)}}{\bff^T\bfone_{(o,\omega_t)}}
                        \right].
\end{align}
Next, to obtain an interpretable characterization of the derivative (\ref{eq:L_deriv_g}), 
for each arc $(o,d,\omega_t)\in\calA$ we divide both sides of the expression by the scenario-probability $\bbP(\omega_t)$ to
obtain the equivalent equation
\begin{equation}
    \frac{1}{\bbP(\omega_t)} \frac{\partial}{\partial g_{(o,d,\omega_t)}}L(\bff,\bfg;\bm{\alpha},\bm{\beta},\bm{\gamma},\bm{\eta})
    = 0 \ \ \ \forall (o,d,\omega_t)\in\calA.
\end{equation}
We expand the left side above:
\begin{align}
    \frac{1}{\bbP(\omega_t)} \frac{\partial}{\partial g_{(o,d,\omega_t)}}L(\bff,\bfg;\bm{\alpha},\bm{\beta},\bm{\gamma},\bm{\eta}) &=   
-\rho_{(o,d,\omega_t)}(g_{(o,d,\omega_t)}) 
                            +C\frac{\bfg^T\bfone_{(o,\omega_t)}}{\bff^T\bfone_{(o,\omega_t)}}
                         - \beta_{(o,d,\omega_t)} + \gamma_{(o,d,\omega_t)}.
\end{align}

We take the same steps to evaluate the derivative of the Lagrangian with respect to the decision
variables $f_{(o,d,\omega_t)}$.  We start by evaluating the derivative of the objective function $W(\bff,\bfg)$ 
with respect to $f_{(o,d,\omega_t)}$:
\begin{align}
    \frac{\partial}{\partial f_{(o,d,\omega_t)}} W(\bff,\bfg) &= \bbP(\omega_t)\left[
        \frac{\partial}{\partial f_{(o,d,\omega_t)}} U_{(o,d,\omega_t)}(f_{(o,d,\omega_t)},g_{(o,d,\omega_t)}) - 
        \frac{\partial}{\partial f_{(o,d,\omega_t)}} A\left(\bfg^T\bfone_{(o,\omega_t)},\bff^T\bfone_{(o,\omega_t)}\right)
        \right]\nonumber\\
    &= \bbP(\omega_t)\left[-c_{(o,d,\omega_t)} 
                            +\frac{C}{2}\left(\frac{\bfg^T\bfone_{(o,\omega_t)}}{\bff^T\bfone_{(o,\omega_t)}}\right)^2
                        \right]. \label{eq:W_deriv_f}
\end{align}
To obtain an interpretable characterization of the derivative of the Lagrangian with respect to each $f_{(o,d,\omega_t)}$,
we again divide the expression (\ref{eq:L_deriv_f}) by the inverse of the scenario probability $\bbP(\omega_t)$, to obtain 
the equivalent equation 
\begin{equation}
    \frac{1}{\bbP(\omega_t)} \frac{\partial}{\partial f_{(o,d,\omega_t)}}L(\bff,\bfg;\bm{\alpha},\bm{\beta},\bm{\gamma},\bm{\eta})
    = 0 \ \ \ \forall (o,d,\omega_t)\in\calA. \label{eq:L_deriv_f_scaled}
\end{equation}

Before expanding the above expression, let's consider how the decision variable $f_{(o,d,\omega_t)}$ appears in the flow-conservation
constraints $\bbP(\omega_t)B_{(\ell,\omega_t)}(\bff)=\bbP(\omega_t)\bar{M}_{(\ell,\omega_t)}$.
In the flow-conservation constraint associated with the node $(o,\omega_t)$ the variable $f_{(o,d,\omega_t)}$ appears once in the function 
$B_{(o,\omega_t)}(\bff)$, with a positive
coefficient.  In addition, for every time-$t+1$ scenario $\omega_{t+1}\in\Omega_{t+1}(\omega_t)$ that is in the subtree induced by
$\omega_t$, the variable $f_{(o,d,\omega_t)}$ appears once in the function $B_{(d,\omega_{t+1})}(\bff)$ with a negative coefficient.
The variable $f_{(o,d,\omega_t)}$ does not appear in the flow-conservation constraints associated with any other node.
Let us compute the ($\frac{1}{\bbP(\omega_t)}$-weighted) derivative of these flow-conservation constraints with respect to the 
variable $f_{(o,d,\omega_t)}$ as follows:
\begin{align}
    \frac{1}{\bbP(\omega_t)}\frac{\partial}{\partial f_{(o,d,\omega_t)}} 
    \left( \sum_{\substack{(\ell,\omega_t)\\\in\calN}}\bbP(\omega_t)\left(B_{(\ell,\omega_t)}(\bff)-\bar{M}_{(\ell,\omega_t)}\right)\eta_{(\ell,\omega_t)}\right)
    &= \eta_{(o,\omega_t)} - 
        \sum_{\substack{\omega_{t+1}\in\\\Omega_{t+1}(\omega_t)}} \frac{\bbP(\omega_{t+1})}{\bbP(\omega_t)} \eta_{(d,\omega_{t+1})}\nonumber\\
    &= \eta_{(o,\omega_t)} - \bbE_{\omega_{t+1}}\left[\eta_{(d,\omega_{t+1})}\mid\omega_t\right].\label{eq:flow_conservation_deriv}
\end{align}
The last line follows from the identity $$\frac{\bbP(\omega_{t+1})}{\bbP(\omega_t)} = \bbP(\omega_{t+1}\mid\omega_t).$$
The expression (\ref{eq:flow_conservation_deriv}) is why we applied the funny scaling to the constraints at the start of this section.
This expression will let us eventually conclude that each dual variable $\eta_{(\ell,\omega_t)}$ 
can be interpreted as the utility-to-go, under the optimal solution,
of a driver positioned at the node $(\ell,\omega_t)$.  

Combining the derivatives (\ref{eq:W_deriv_f}) and (\ref{eq:flow_conservation_deriv}), we expand the Lagrangian derivative
(\ref{eq:L_deriv_f_scaled}):
\begin{align}
    \frac{1}{\bbP(\omega_t)} \frac{\partial}{\partial f_{(o,d,\omega_t)}}L(\bff,\bfg;\bm{\alpha},\bm{\beta},\bm{\gamma},\bm{\eta}) & = 
    c_{(o,d,\omega_t)}  -\frac{C}{2}\left(\frac{\bfg^T\bfone_{(o,\omega_t)}}{\bff^T\bfone_{(o,\omega_t)}}\right)^2 - \alpha_{(o,d,\omega_t)}
    - \gamma_{(o,d,\omega_t)}\nonumber\\
    & \ \ \ \ \ + \eta_{(o,\omega_t)} 
                        - \bbE_{\omega_{t+1}}\left[\eta_{(d,\omega_{t+1})}\mid\omega_t\right].
\end{align}

\subsubsection{Proof of Lemma \ref{lem:matching_concentration}}
\label{appdx:matching_concentration_proof}
We now prove Lemma \ref{lem:matching_concentration}, restated below.
\begin{lemma*}
Suppose that all drivers positioned at $\ell$ use the same acceptance probability $p_d$ for each destination $d$.
Let $k,C\geq 0$ be any constants such that the number of drivers at $\ell$ and the number of riders requesting a dispatch towards
any destination $d$ are all upper bounded by $kC$.  Then there exists an error term $\epsilon(k,C)$ and a probability term $q(k,C)$
such that 
\begin{equation}
    \bbP\left(\frac{1}{k}\sum_{j=1}^L|Y^{(j)} - \bar{Y}^{(j)}|\geq \epsilon(k)\right) \leq q(k).
\end{equation}
    Moreover, the error term $\epsilon(k)$ and the probability term $q(k)$ both go to $0$ as $k$ goes to $\infty$, assuming $C$ is held fixed.
\end{lemma*}

Let $R_j = R_{d_j}$ and let $p_j = p_{d_j}$.
Let $j^*\in\{1,2,\dots,L\}$ be the largest index such that $\bar{Y}^{(j^*)}$ is nonzero.

\begin{lemma}
Fix any index $j=1,2,\dots,L$. Let $k\geq 1$ be a constant. In the event that
\begin{equation}
\label{eq:dispatch_sizes_condition}
    \left|\sum_{j'=1}^{j-1}M^{(j')} - \bar{M}^{(j')}\right| \leq \epsilon
\end{equation}
holds, for some error term $\epsilon> 0$, then
$$
    \left|M^{(j)} - \bar{M}^{(j)}\right| \leq 24\sqrt{k} + \epsilon
$$
holds with high probability $q_j$, where
    $$
    q_j = \max\left(2\exp\left(-p_j^3k/4R_j\right), \exp\left(-8\sqrt{k}p_j\right)\right).
    $$
\end{lemma}
\begin{proof}
    Let $X=\left|M^{(j)} - \bar{M}^{(j)}\right|$ be the random variable we are interested in bounding.  
    Observe that $X$ can be expanded as:
    $$X  = \left| \min\left(Z_j, M - \sum_{j'=1}^{j-1} M^{(j)}\right) - \min\left(R_j/p_j, M - \sum_{j'=1}^{j-1}\bar{M}^{(j)}\right)\right|,$$
    where $Z_j$ has a negative binomial distribution with parameters $R_j$ and $p_j$.
    Define $$q_j = \max\left(2\exp\left(-p_j^3k/4R_j\right), \exp\left(-8\sqrt{k}p_j\right)\right).$$
    By Lemma \ref{lem:nb_universal_concentration} we have
    $$
    |Z_j - R_j/p_j| \leq 24\sqrt{k}
    $$
    holds with probability at least $q_j$. When this inequality holds, we have the following:
    \begin{align*}
        X &\leq 24\sqrt{k} + \left| \min\left(R_j/p_j, M - \sum_{j'=1}^{j-1} M^{(j)}\right) - \min\left(R_j/p_j, M - \sum_{j'=1}^{j-1}\bar{M}^{(j)}\right)\right|\\
        & \leq 24\sqrt{k} + \left| \left(M - \sum_{j'=1}^{j-1} M^{(j)}\right) - \left(M - \sum_{j'=1}^{j-1}\bar{M}^{(j)}\right)\right|\\
        & \leq 24\sqrt{k} + \epsilon,
    \end{align*} 
    as claimed.
\end{proof}

\begin{lemma}
    \label{lem:nb_chernoff_concentration}
    Let $Z$ have a negative Binomial distribution with parameters $R, p$. Let $\epsilon > 0 $ be any constant.
    Then we have the bounds
    \begin{equation}
        \bbP\left(Z \geq (1 + \epsilon)\bbE[Z]\right)\leq \exp\left(-\epsilon^2 R / 2(1+\epsilon)\right)\ \mbox{ and } \ 
        \bbP\left(Z \geq (1 - \epsilon)\bbE[Z]\right)\leq \exp\left(-\epsilon^2 R / 2(1-\epsilon)\right).
    \end{equation}
    Further, if $\epsilon < 1/2$ then there exists an absolute constant $C$ such that 
    \begin{equation}
        \bbP\left(|Z-\bbE[Z]| \geq \epsilon\bbE[Z]\right)\leq 2\exp\left(-C\epsilon^2 R\right).
    \end{equation}
\end{lemma}
\begin{proof}
    Follows from Chernoff bound.  See, e.g. \url{https://www.cs.ubc.ca/~nickhar/W15/Lecture4Notes.pdf}.
\end{proof}

\begin{lemma}
\label{lem:nb_absolute_concentration}
Let $Z$ be a negative Binomial random variable with parameters $R$ and $p$ and let $\epsilon>0$.
Assume  $\epsilon < 1/4$, and that $\epsilon\bbE[Z]\geq 2$.
Then we have
\begin{equation}
\label{eq:nb_absolute_concentration}
\bbP\left(\left|Z - \bbE[Z]\right| > \epsilon\bbE[Z]\right) \leq 2\exp\left(-\epsilon^2 Rp/4\right).
\end{equation}
\end{lemma}
\begin{proof}
Notice the event $\left|Z - \bbE[Z]\right| > \epsilon\bbE[Z]$ can only happen if $Z>(1+\epsilon)\bbE[Z]$ or $Z<(1-\epsilon)\bbE[Z]$.
We analyze each of these events separately.

Consider the event  $Z>(1+\epsilon)\bbE[Z]$. This can only happen if we see fewer than $R$ successful experiments after 
a total of $N=\lfloor(1+\epsilon)\bbE[Z]\rfloor$ trials. From our assumption that $\epsilon\bbE[Z]\geq 2$ it follows that taking this
floor is equivalent to decreasing $\epsilon$ by a multiplicative factor of at most $1/2$.  Therefore we can write $N=(1+\epsilon')\bbE[Z]$
where $\epsilon\geq\epsilon' \geq \epsilon/2$.

Let $X$ be a Binomial random variable with $N$ trials and $p$ success probability. Then we have
\begin{align*}
\bbP\left(Z > (1+\epsilon)\bbE[Z]\right) & = \bbP\left(X \leq R - 1\right)\\
&= \bbP\left(\bbE[X] - X \geq \bbE[X] - R + 1\right)\\
&\leq \bbP\left(\bbE[X] - X \geq Np - R\right)\\
&\leq \bbP\left(\bbE[X] - X \geq (1+\epsilon')R - R\right)\\
&\leq \bbP\left(\bbE[X] - X \geq \epsilon'R\right)\\
&\leq \bbP\left(\bbE[X] - X \geq \epsilon R/2\right)\\
&\leq \exp\left(-2\frac{(\epsilon R/2)^2}{N}\right), 
\end{align*}
where the last line follows by Hoeffding's inequality.
Expanding the expression inside the paranthesis, we obtain
$$
-2\frac{(\epsilon R/2)^2}{N} = -\frac{\epsilon^2Rp}{2(1+\epsilon')}.
$$
Now, observe the equality 
$$
\frac{-1}{1+\epsilon'} < -1/2
$$
follows from our assumption $\epsilon < 1$ and the fact that $\epsilon'\leq \epsilon$.
Therefore, we have
$$
\bbP\left(Z > (1+\epsilon)\bbE[Z]\right) \leq \exp\left(-\epsilon^2Rp/4\right).
$$

We analyze the second event via a symmetric argument. Let $Y$ be a Binomial random variable with parameters
$M=\lfloor(1-\epsilon)\bbE[Z]\rfloor$ and $p$, and observe that the probability of the event $Z<(1-\epsilon)\bbE[Z]$ is the same as observing at least $R$
successes in $M$ trials.  From the assumption that $\epsilon\bbE[Z]>2$ it follows we can write $M=(1-\epsilon'')\bbE[Z]$ for some $\epsilon \leq \epsilon'' \leq 2\epsilon $.
Therefore we obtain the following probability bounds:
\begin{align*}
\bbP\left(Z\leq \frac{(1-\epsilon) R}{p}\right) & = \bbP(Y \geq R)\\
&= \bbP\left(Y - \bbE[Y] \geq R - \bbE[Y]\right)\\
&= \bbP\left(Y - \bbE[Y] \geq R - (1-\epsilon'')R\right)\\
&= \bbP\left(Y - \bbE[Y] \geq \epsilon''R\right)\\
&\leq \bbP\left(Y - \bbE[Y] \geq \epsilon R\right)\\
&\leq \exp\left(-2\frac{(\epsilon R)^2}{M}\right)\\
&= \exp\left(-2\frac{\epsilon^2 Rp}{1-\epsilon''}\right),\\
\end{align*}
where the last line follows by Hoeffding's inequality.

From our assumption that $\epsilon < \frac{1}{4}$ it follows that $\epsilon'' < 1/2$, from which we obtain the inequality
$$
\frac{-2}{1-\epsilon''} < -1/4,
$$
from which 
$$
\bbP\left(Z\leq \frac{(1-\epsilon) R}{p}\right) \leq 
\exp\left(-2\frac{\epsilon^2 Rp}{1-\epsilon''}\right)\leq 
\exp\left(-\epsilon^2 Rp/4\right)
$$
follows.
\end{proof}

\begin{lemma}
\label{lem:nb_universal_concentration}
Let $Z$ be a negative Binomial random variable with parameters $R$ and $p$ and let $k\geq 1$ be any constant. 
Let $q$ be defined by
\begin{equation}
q=\max\left(2\exp\left(-p^3k/4R\right), \exp\left(-8\sqrt{k}p\right)\right).
\end{equation}
Then the following is true:
\begin{equation}
\bbP\left(\left|Z - \bbE[Z]\right| \geq 24\sqrt{k}\right) \leq q.
\end{equation}
\end{lemma}
\begin{proof}
We consider two cases.  

For the first case, assume that $\bbE[Z]> 8\sqrt{k}$ holds and set $\epsilon = \frac{2\sqrt{k}}{\bbE[Z]}$. We can check that the conditions of Lemma \ref{lem:nb_absolute_concentration} hold,
namely $\epsilon < 1/4$ and $\epsilon\bbE[Z] \geq 2$, so it follows that
\begin{align*}
\bbP\left(|Z - \bbE[Z]| > 2\sqrt{k}\right) &=\bbP\left(|Z - \bbE[Z]| > \epsilon\bbE[Z]\right)\\
& \leq 2\exp(-\epsilon^2Rp/4)\\
&= 2\exp(-kRp/4\bbE[Z]^2)\\
&= 2\exp(-kp^3/4R).\\
\end{align*}
Therefore, when $\bbE[Z]> 8\sqrt{k}$, we have
$$
\bbP\left(|Z - \bbE[Z]| > 24\sqrt{k}\right) 
\leq \bbP\left(|Z - \bbE[Z]| > 2\sqrt{k}\right) \leq 2\exp(-kp^3/4R) \leq q.
$$

For the second case, assume that $\bbE[Z] \leq 8\sqrt{k} $. Consider the following:
\begin{align*}
\bbP\left(|Z-\bbE[Z]| \geq 24\sqrt{k}\right) & \leq \bbP\left(Z + \bbE[Z] \geq 24\sqrt{k}\right)\\
& \leq \bbP\left(Z \geq 16\sqrt{k}\right),
\end{align*}
where the final line follows because $8\sqrt{k} \geq \bbE[Z]$.
Let $X$ be a Binomial random variable over $ 16\sqrt{k}$ trials and $p$ success probability.
We thus have
\begin{align*}
\bbP\left(Z \geq 16\sqrt{k}\right) &= \bbP\left(X \leq R - 1\right)\\
& \leq \bbP\left(X -  \bbE[X] \leq R - \bbE[X]\right)\\
& = \bbP\left(X -  16\sqrt{k}p \leq R - 16\sqrt{k}p\right).
\end{align*}
From our assumption that $\bbE[Z]\leq 8\sqrt{k}$ we have that $R \leq 8\sqrt{k}p$ because $\bbE[Z] = \frac{R}{p}$.
\begin{align*}
\dots & \leq \bbP\left(X -  16\sqrt{k}p \leq  - 8\sqrt{k}p\right)\\
& \leq \exp\left(-2(8\sqrt{k}p)^2/ (16\sqrt{k}p)\right)\\
& = \exp\left(-8\sqrt{k}p\right),
\end{align*}
where the last line follows by Hoeffding's inequality.
Therefore, when $\bbE[Z] \leq 8\sqrt{k} $, we have
$$
\bbP\left(|Z-\bbE[Z]| \geq 24\sqrt{k}\right) \leq \exp\left(-8\sqrt{k}p\right) \leq q.
$$

\end{proof}

}

\section{Partial Derivatives of the State-Dependent Optimization Problem}
\label{sec:opt_central_appdx}
In this section we prove a number of results about the fluid optimization problem.

\subsection{Proof of Lemma \ref{lem:reward_fn}}
\label{appdx:reward_fn_proof}
\begin{lemma*}
Consider the reward function $U_{(\ell,d,\omega_t)}$
    associated with any route $(\ell,d)\in\calL^2$ and any scenario $\omega_t$.
    Assume the rider-value distribution $F_{(\ell,d,\omega_t)}$ satisfies Assumption \ref{assn:rider_value_dist}.  Then $U_{(\ell,d,\omega_t)}(g)$
    is concave in $g$, is differentiable at every $g>0$, and the derivative at each $g>0$ satisfies:
    $$
    \frac{d}{dg}U_{(\ell,d,\omega_t)}(g) =P_{(\ell,d,\omega_t)}(g).    $$
    Moreover, the fluid optimization problem (\ref{eq:fluid_opt}) has a concave objective function for any market state.
\end{lemma*}
\begin{proof}
    Let us write the utility function as
    $U(f) = U_{(\ell,d,\omega_t)}(f) ,$
First, observe that the utility function can be equivalently written as
the following equation
\begin{equation}
\label{eq:reward_fn_nice_form}
    U(g) =  \bar{D}\bbE\left[V\bfone_{\left\{V\geq F^{-1}\left(1-\frac{g\wedge\bar{D}}{\bar{D}}\right)\right\}}\right],
\end{equation}
    where we write $\bar{D}$ in place of  $\bar{D}_{(\ell,d,\omega_t)}$,
    $F$ to mean $F_{(\ell,d,\omega_t)}$,
     and $f\wedge\bar{D}$  to mean $\min(f,\bar{D})$.
     This characterization is justified by the following series of equalities:
\begin{align*}
\bbE\left[V\bfone_{\left\{V\geq F^{-1}\left(1-\frac{g\wedge \bar{D}}{\bar{D}}\right)\right\}}\right]
& = \bbE\left[V\mid V\geq F^{-1}\left(1-\frac{g\wedge \bar{D}}{\bar{D}}\right)\right]
 \bbP\left(V\geq F^{-1}\left(1-\frac{g\wedge \bar{D}}{\bar{D}}\right)\right)\\
&=\bbE\left[V\mid V\geq F^{-1}\left(1-\frac{g\wedge \bar{D}}{\bar{D}}\right)\right]
 \left(1 - F\left(F^{-1}\left(1-\frac{g\wedge \bar{D}}{\bar{D}}\right)\right) \right)\\
& = \bbE\left[V\mid V\geq F^{-1}\left(1-\frac{g\wedge \bar{D}}{\bar{D}}\right)\right]
\frac{g\wedge \bar{D}}{\bar{D}}.
\end{align*}

Next, recall that if $X$ is a uniform $[0,1]$ random variable then $F^{-1}(X)$ is
a random variable with distribution function $F$.  
Using the fact that if $X$ is uniform $[0,1]$ then so too is $1-X$,
from the characterization (\ref{eq:reward_fn_nice_form}) we have the
following equalities:
\begin{align*}
U(g) &= \bar{D} \int_0^1 F^{-1}(u)
        \bfone_{\left\{F^{-1}(u)\geq F^{-1}\left(1-\frac{g\wedge \bar{D}}{\bar{D}}\right)\right\}}du\\
 & = \bar{D} \int_0^1 F^{-1}(1-u)
        \bfone_{\left\{F^{-1}(1-u)\geq F^{-1}\left(1-\frac{g\wedge \bar{D}}{\bar{D}}\right)\right\}}du\\
 & = \bar{D} \int_0^1 F^{-1}(1-u)
        \bfone_{\left\{\frac{g\wedge \bar{D}}{\bar{D}}\geq u\right\}}du\\
 & = \bar{D} \int_0^{\frac{g\wedge \bar{D}}{\bar{D}}} F^{-1}(1-u)du.
\end{align*}
Now, fix any $0<g<\bar{D}$ and consider the above expression.
The assumption $g<\bar{D}$ means that the integral upper bound $\frac{g\wedge \bar{D}}{\bar{D}}$
is simply $\frac{g}{\bar{D}}$, and from the assumption that $F^{-1}$ satisfies
Assumption \ref{assn:rider_value_dist} we know that $F^{-1}(\cdot)$ is continuous.
Hence, by the fundamental theorem of calculus the function $U(g)$ is
differentiable at $0<g<\bar{D}$, and using the chain rule we compute the derivative
to be:
\begin{equation*}
\frac{d}{d g}U(g)= 
\bar{D} F^{-1}\left(1-\frac{g}{\bar{D}}\right)\frac{1}{\bar{D}} = F^{-1}\left(1-\frac{g}{\bar{D}}\right).
\end{equation*}
Further, note that $U(g)$ is constant for $g\geq \bar{D}$, from which 
we conclude that $\frac{d}{d g}U(g)=0$ for $g>\bar{D}$.
Finally, we must establish existence of the derivative at the point $g=\bar{D}$.
To this end, it suffices to show that limit of the partial derivatives from below $g=\bar{D}$
and from above $g=\bar{D}$ are equal. This fact follows from the second assertion in Assumption
\ref{assn:rider_value_dist} which states that $F^{-1}(0)=0$:
\[
\lim_{g\uparrow \bar{D}} \frac{d }{d g}U(g)=
\lim_{g\uparrow \bar{D}}F^{-1}\left(1-\frac{g}{\bar{D}}\right)=0
=\lim_{g\downarrow \bar{D}} \frac{d}{d g}U(g).
\]

Thus, we have established the derivative $\frac{d}{d g} U(g)$
exists for all $g>0$ and is equal to $P_{(\ell,d,\omega_t)}(g).$
Concavity of $U(\cdot)$ follows from the observation that this derivative is non-increasing
in $g$.

To show that (\ref{eq:fluid_opt}) we has a concave objective we have to show that the remaining terms in
$\calW_{\omega_t}(\bff,\bfg)$ are also concave.  Well, clearly the linear costs are concave.  And finally,
the add-passenger disutility cost function $-A(\bfg^T\bfone_\ell,\bff^T\bfone_\ell)$ is convex,
from the equation (\ref{eq:add_passenger_cost}) and Lemma \ref{lem:convex_add_passenger}.
\end{proof}

\begin{lemma}
\label{lem:convex_add_passenger}
    Define $$f(x,y) = \begin{cases}
        \frac{y^2}{x} & \mbox{if }x > 0,\\
        0 & \mbox{if }x=0.
    \end{cases}$$
    Then $f$ is convex over the domain $x,y\geq 0$, $y\leq x$.
\end{lemma}
\begin{proof}
 Let 
$(x_1,y_1)$ and $(x_2,y_2)$ be two points in the domain
of $f$, and let $\lambda\in [0,1]$.  We need to check
$$
f(\lambda (x_1,y_1) + (1-\lambda)(x_2,y_2)) \leq 
\lambda f(x_1,y_1) + (1-\lambda)f(x_2,y_2).
$$
Let's start with the case where $x_1$ and $x_2$ are both
nonzero. In this case, the inequality we have to check
is given by
$$
\frac{\left(\lambda y_1 + (1-\lambda) y_2\right)^2}
{\lambda x_1 + (1-\lambda)x_2} \leq 
\lambda\frac{y_1^2}{x_1} + (1-\lambda)\frac{y_2^2}{x_2}.
$$
We will verify this inequality by applying the 
Cauchy-Schwarz inequality. Define the following values:
$$
u_1 = \frac{\lambda y_1}{\sqrt{\lambda x_1}}\ \ 
u_2 = \frac{(1-\lambda)y_2}{\sqrt{(1-\lambda)x_2}}\ \ 
v_1 = \sqrt{\lambda x_1}\ \ 
v_2 = \sqrt{(1-\lambda)x_2}.
$$
The Cauchy-Schwarz inequality says $u^Tv\leq \|u\|\|v\|$.
Observe the following equalities:
\begin{align*}
(u^Tv)^2 &= \left(\lambda y_1 + (1-\lambda) y_2\right)^2,\\
\|v\|^2 &= \lambda x_1 + (1-\lambda)x_2,\\
\|u\|^2 &= 
\lambda\frac{y_1^2}{x_1} + (1-\lambda)\frac{y_2^2}{x_2}.
\end{align*}
Rearranging C-S we have 
$$
\frac{(u^Tv)^2}{\|v\|^2}\leq \|u\|^2,
$$
which implies the desired inequality holds, and hence $f$
is convex whenever $x_1$ and $x_2$ are both nonzero.

The case where $x_1$ and $x_2$ are both zero is immediate.
It remains to check the case where $x_1$ is nonzero and
$x_2$ is $0$.  In this case, $y_2$ must also be zero, because
of the constraint $y \leq x$. Therefore we have to verify
$$
f(\lambda (x_1,y_1) ) \leq \lambda f(x_1,y_1).
$$
By inspection we see that $f(\lambda (x,y))=\lambda f(x,y)$
is always satisfied, so $f$ is convex in this case as well.
\end{proof}

\subsection{State-Dependent Optimization Problem}
For clarity we restate the state-dependent optimization problem below. Fix a time period $t$ and a scenario $\omega_t$.
The state-dependent optimization problem depends on a supply-location vector $\bfS=(S_\ell : \ell\in\calL)$ where each
component $S_\ell\geq 0$ specifies the volume of active drivers at location $\ell$.
Active drivers at a location $\ell$ consist of drivers who took a trip destined towards $\ell$ at the previous time period $t-1$,
as well as new drivers who enter the market at location $\ell$ in the current time period $t$.
The state-dependent optimization problem solves for the welfare-optimal trips in the current time period in the stochastic fluid model,
given the market state specified by the scenario $\omega_t$ and the supply-location vector $\bfS$.
We write $\Phi_{\omega_t}(\bfS)$ to denote the value of the state-dependent optimization problem under the scenario $\omega_t$
as a function of the supply-location vector $\bfS$.
The function $\Phi_{\omega_t}(\bfS)$ is formally defined as the value of the following optimization problem:

\begin{align}
    \Phi_{\omega_t}(\bfS)\equiv \ \ \ \ \; \; \; 
    \sup_{\bff, \bfg} &\ \ \   \calU_{\omega_t}(\bff,\bfg) + 
                               \bbE_{\omega_t}\left[\Phi_{\omega_{t+1}}(\bfS_{\omega_{t+1}}(\bff))\right] & \label{eq:state_dep_opt}\\
    \mbox{subject to} & \nonumber\\ 
    & f_{(\ell,d)} \geq 0 \ &\forall (\ell,d)\in\calL^2\label{eq:f_nonneg}\\
    & g_{(\ell,d)} \geq 0 \ &\forall (\ell,d)\in\calL^2 \label{eq:g_nonneg}\\
    & f_{(\ell,d)} \geq g_{(\ell,d)} \ &\forall (\ell,d)\in\calL^2  \label{eq:f_geq_g}\\
    &     \sum_{d\in\calL} f_{(\ell,d)} =
        S_\ell  & \forall \ell\in\calL\label{eq:flow_conservation}
\end{align}

The decision variable $\bff=(f_{(\ell,d)} : (\ell,d)\in\calL^2)$ has components $f_{(\ell,d)}$ which specify the total trip volume 
along each route $(\ell,d)$. By total trip volume we mean$f_{(\ell,d)}$ specifies the sum of the relocation-trip volume and the dispatch trip-volume.
The decision variable $\bfg=(g_{(\ell,d)} : (\ell,d)\in\calL^2)$ has components $g_{(\ell,d)}$ which specify the dispatch trip
volume along each route $(\ell,d)$.

The objective function is the sum of two functions: $\calU_{\omega_t}(\bff,\bfg)$ specifies the welfare collected in the current time period
$t$, and $\bbE_{\omega_t}\left[\Phi_{\omega_{t+1}}(\bfS_{\omega_{t+1}}(\bff))\right]$ specifies the welfare to be collected in future time periods.

When $t=T$ is the final time period we just take $\bbE_{\omega_t}\left[\Phi_{\omega_{t+1}}(\bfS_{\omega_{t+1}}(\bff))\right]$ to be $0$.
When $t<T$, we define $\bfS_{\omega_{t+1}}(\bff)$ to be the supply-location vector arising at time $t+1$ 
under the trip volumes specified by $\bff$
and the future scenario $\omega_{t+1}$.   The expectation 
$\bbE_{\omega_t}\left[\cdot\right]$ is taken over all time $t+1$ scenarios given the time $t$ scenario $\omega_t$.
The function $\bfS_{\omega_{t+1}}(\bff)$ follows the convention that upply-location vectors include new drivers who enter the
market in the relevant time period.  Let us write $S_{\omega_{t+1},\ell}(\bff)$ for the component of the supply-location
vector $\bfS_{\omega_{t+1}}(\bff)$ corresponding to location $\ell$. $S_{\omega_{t+1},\ell}(\bff)$ is defined
formally by the following equation
\begin{equation}
    S_{\omega_{t+1},\ell}(\bff) = \frac{1}{k}M_{\omega_{t+1},\ell} + \sum_{o\in\calL} f_{(o,\ell)},
\end{equation}
where $M_{\omega_{t+1},\ell}$ is the volume of new drivers who enter the market at location $\ell$ under the scenario
$\omega_{t+1}$ and the sum is over all routes whose destination location is $\ell$.

The utility collected in the current time period is the difference between the rider value we generate by serving dispatches
and the disutility that drivers incur.  Since we assume the price is a transfer from riders to drivers the price does not appear
explicitly in the objective function.  The function $\calU_{\omega_t}(\bff,\bfg)$ is formally defined by the following equation:
\begin{equation}
    \label{eq:immediate_reward}
\calU_{\omega_t}(\bff,\bfg) = \sum_{(\ell,d)\in\calL^2} U_{(\ell,d,\omega_t)}(g_{(\ell,d)}) 
                            -\left( \sum_{(\ell,d)\in\calL^2} c_{(\ell,d)}f_{(\ell,d)} + \sum_{\ell} A(\bfg^T\bfone_\ell,\bff^T\bfone_\ell)\right).
\end{equation}
The function $U_{(\ell,d,\omega_t)}(g_{(\ell,d)})$ specifies the total rider value generated as a function of dispatch-trip volume along the route
$(\ell,d)$, the function $A(\bfg^T\bfone_\ell,\bff^T\bfone_\ell)$ specifies the total add-passenger disutility incurred by drivers located
at $\ell$, as a function of the volume of dispatch trips originating from $\ell$, $\bfg^T\bfone_\ell$, and the total volume of available drivers
located at $\ell$, $\bff^T\bfone_\ell$.

We take $\bfone_\ell$ to be an indicator vector indexed by pairs of locations, where the value corresponding to each 
$(\ell',d)\in\calL^2$ is $1$ if $\ell'=\ell$ and $0$ otherwise.  With this convention, the quantities $\bfg^T\bfone_\ell$
and $\bff^T\bfone_\ell$ specify the volume of dispatch trips originating from $\ell$ and the total volume of trips originating from $\ell$, 
respectively:
\begin{align*}
    \bfg^T\bfone_\ell &= \sum_{d\in\calL}g_{(\ell,d)},\\
    \bff^T\bfone_\ell &= \sum_{d\in\calL}f_{(\ell,d)}.
\end{align*}
Assuming $\bff$ satisfies the flow-conservation constraint (\ref{eq:flow_conservation}), the total trip volume originating 
from $\ell$ is equal to the total volume of supply positioned at $\ell$:
$$
\bff^T\bfone_\ell = S_\ell.
$$

\subsection{Optimality Conditions}
\label{appdx:optimality_conditions}
We now derive the Lagrangian optimality conditions for the state-dependent optimization problem (\ref{eq:state_dep_opt}).
For succinctness, we use the following notation for the objective function:
\begin{equation}
    \label{eq:state_dep_objective}
\calW_{\omega_t}(\bff,\bfg) = \calU_{\omega_t}(\bff,\bfg) + 
                               \bbE_{\omega_t}\left[\Phi_{\omega_{t+1}}(\bfS_{\omega_{t+1}}(\bff))\right].
\end{equation}
We begin by converting the optimization problem to a convex minimization problem where all inequality constraints have an upper bound of $0$:
\begin{align}
    -\Phi_{\omega_t}(\bfS)\equiv \ \ \ \ \; \; \; 
    \inf_{\bff, \bfg} &\ \ \   -\calW_{\omega_t}(\bff,\bfg) & \label{eq:state_dep_opt_cvx}\\
    \mbox{subject to} & \nonumber\\ 
    & -f_{(\ell,d)} \leq 0 \ &\forall (\ell,d)\in\calL^2\label{eq:f_nonneg_cvx}\\
    & -g_{(\ell,d)} \leq 0 \ &\forall (\ell,d)\in\calL^2 \label{eq:g_nonneg_cvx}\\
    &  g_{(\ell,d)}- f_{(\ell,d)}\leq 0 \ &\forall (\ell,d)\in\calL^2  \label{eq:f_geq_g_cvx}\\
    &     \sum_{d\in\calL} f_{(\ell,d)} - S_\ell = 0
         & \forall \ell\in\calL\label{eq:flow_conservation_cvx}
\end{align}
We associate dual variables $\alpha_{(\ell,d)}$, $\beta_{(\ell,d)}$, $\gamma_{(\ell,d)}$ and $\eta_\ell$ with each of the constraints
(\ref{eq:f_nonneg_cvx}), (\ref{eq:g_nonneg_cvx}), (\ref{eq:f_geq_g_cvx}), (\ref{eq:flow_conservation_cvx}), respectively.
We will write $\bm{\alpha},\bm{\beta},\bm{\gamma},\bm{\eta}$ to indicate the vector of dual variables.

Since all we have done is changed the sign and direction of the objective function and algebraically rearranged the inequality
constraint functions, the optimization problems (\ref{eq:state_dep_opt_cvx}) and (\ref{eq:state_dep_opt}) have the same set of 
optimal solutions.

We now obtain the Lagrangian function for the optimization problem (\ref{eq:state_dep_opt_cvx}):
\begin{align*}
    L(\bff,\bfg;\bm{\alpha},\bm{\beta},\bm{\gamma},\bm{\eta}) &= -\calW_{\omega_t}(\bff,\bfg) + 
    \sum_{(\ell,d)} \left[ \gamma_{(\ell,d)}(g_{(\ell,d)} - f_{(\ell,d)}) 
                            - \alpha_{(\ell,d)}f_{(\ell,d)}
                            -\beta_{(\ell,d)}g_{(\ell,d)} \right] \\
        &\hspace{2.5cm}+ \sum_\ell \eta_\ell\left(\sum_d f_{(\ell,d)} - S_\ell\right).
\end{align*}

Because all of the constraints for the problem (\ref{eq:state_dep_opt_cvx}) are linear, and the primal
problem (\ref{eq:state_dep_opt}) has a finite optimal solution, we know
strong duality holds \cite{borwein_lewis}. Therefore, a feasible solution $(\bff,\bfg)$ is optimal 
if and only if there exist feasible dual variables $\bm{\alpha},\bm{\beta},\bm{\gamma},\bm{\eta}$ for which
the stationarity conditions and the complementary slackness conditions hold.
For dual feasibility to hold the variables associated with inequality constraints must be nonnegative,
that is the following inequalities must hold pointwise:
$$
\bm{\alpha}\geq 0,\ \bm{\beta} \geq 0,\ \bm{\gamma}\geq 0.
$$
The complementary slackness conditions are satisfied when the following equations hold for all origin-destination pairs $(\ell,d)\in\calL^2$:
$$
\alpha_{(\ell,d)}f_{(\ell,d)} = 0,\ \beta_{(\ell,d)}g_{(\ell,d)} = 0,\ \gamma_{(\ell,d)}(g_{(\ell,d)}-f_{(\ell,d)}) = 0,
$$
that is the dual variables associated with inequality constraints must be $0$ unless the corresponding
inequality constraint is tight at the primal solution.

Finally, the stationarity conditions are satisfied when the primal solution $(\bff,\bfg)$ are a stationary
point of the Lagrangian function when the dual variables are held fixed.
Notice that when we hold the dual variables fixed the Lagrangian is a convex function of the primal solution,
so a primal solution $(\bff,\bfg)$ is a stationary point if and only if $0$ is a subgradient of the Lagrangian
at $(\bff,\bfg)$. We use the notation $\partial L(\bff,\bfg;\bm{\alpha},\bm{\beta},\bm{\gamma},\bm{\eta})$ to refer to the
subgradient of the Lagrangian where the dual variables $\bm{\alpha},\bm{\beta},\bm{\gamma},\bm{\eta}$ are held fixed.
The subgradient condition for a primal solution $(\bff,\bfg)$ to be a stationary point can thus be expressed as follows
$$
0\in\partial L(\bff,\bfg;\bm{\alpha},\bm{\beta},\bm{\gamma},\bm{\eta}).
$$
We work in terms of the subgradient because the objective function $\calW_{\omega_t}(\bff,\bfg)$ is not differentiable
at coordinates where $f_{(\ell,d)}=0$ or $g_{(\ell,d)}=0$.
However, we can use the following property (see Theorem 3.1.8 in \cite{borwein_lewis}) about general convex functions to obtain a stationarity condition in terms
of the partial derivatives for the nonzero coordinates of $\bff$ and $\bfg$:
\begin{lemma*}
Let $h:\bbR^m\to\bbR^n$ be a convex function and consider any point $x\in\bbR^m$ in its domain.
Let $\partial h(x_0)$ be the subdifferential of $h$ at $x$ and assume the partial derivative $\frac{\partial}{\partial x_j}h(x)$
exists for some coordinate $j$. Then the $j$th component of every subgradient in the subdifferential of $h$ at $x$
is equal to the partial derivative of $h$ at $x$. That is, for every $\phi\in\partial h(x)$ the equality
$\phi_j = \frac{\partial}{\partial x_j}h(x)$ holds.
\end{lemma*}

Therefore, for any pair of locations $(\ell,d)$ where the objective function is differentiable with respect to $f_{(\ell,d)}$,
the stationarity conditions require the following equality hold:
\begin{equation}
\label{eq:stationarity_nonzero_f}
\frac{\partial}{\partial f_{(\ell,d)}} \calW_{\omega_t}(\bff,\bfg) = \eta_\ell - \alpha_{(\ell,d)} - \gamma_{(\ell,d)}.
\end{equation}
And, for any pair of locations $(\ell,d)$ where the objective function is differentiable with respect to $g_{(\ell,d)}$,
the stationarity conditions require the following equality hold:
\begin{equation}
\label{eq:stationarity_nonzero_g}
\frac{\partial}{\partial g_{(\ell,d)}} \calW_{\omega_t}(\bff,\bfg) = \gamma_{(\ell,d)} - \beta_{(\ell,d)}.
\end{equation}

\subsection{Statement of Lemma \ref{lem:unique_dual}}
For the rest of this document we will focus on properties of the state-dependent optimization
function (\ref{eq:state_dep_opt}). Recall the function $\Phi_{\omega_t}(\bfS)$ gives the optimal value
of the state-dependent optimization problem with respect to the time-scenario $\omega_t$ as a function
of a supply-location vector $\bfS$.  In this section we will show that the partial derivative
$\frac{\partial}{\partial S_\ell}\Phi_{\omega_t}(\bfS)$ exists for every location $\ell$ with a nonzero
volume of drivers under $\bfS$.

First, let us introduce notation to refer to optimal primal and dual solutions of the state-dependent
optimization problem. Let
$$
F^*(\bfS) = \left\{(\bff,\bfg) : \calW_{\omega_t}(\bff,\bfg) = \Phi_{\omega_t}(\bfS),\ (\bff,\bfg) \mbox{ is feasible for (\ref{eq:state_dep_opt}) with respect to }\bfS\right\}
$$ 
denote the set of primal optimal solutions as a function of the supply-location vector $\bfS$.
Let
$$
D^*(\bfS) = \left\{(\bm{\alpha},\bm{\beta},\bm{\gamma},\bm{\eta}) : \exists (\bff^*,\bfg^*)\in F^*(\bfS) \mbox{ such that }
                                                                    0 \in \partial L(\bff^*,\bfg^*;\bm{\alpha},\bm{\beta},\bm{\gamma},\bm{\eta}),\ 
 \bm{\alpha}\geq 0,\bm{\beta}\geq 0,\bm{\gamma}\geq 0\right\}
$$
denote the set of dual optimal solutions as a function of the supply-location vector $\bfS$.

Our main result in this section is the following Lemma, which characterizes important properties
about partial derivatives of the state-dependent optimization function.
Below is a restatement of Lemma \ref{lem:unique_dual}.
\begin{lemma*}
Fix a time-scenario $\omega_t$ and let $\bfS=(S_\ell\geq 0 : \ell\in\calL)$ be any supply-location vector.  Pick any location $\ell$
for which the volume of supply at $\ell$ is nonzero under $\bfS$, i.e. $S_\ell > 0$.
\begin{enumerate}
\item \label{item:unique_dual} For the state-dependent optimization problem (\ref{eq:state_dep_opt}) with respect to $\bfS$ the value of any optimal dual 
    variable associated with the flow conservation constraint (\ref{eq:flow_conservation}) for
    location $\ell$ is unique.  That is there exists a number $\eta_\ell^*$ such that $\eta_\ell=\eta_\ell^*$, where 
    $\eta_\ell$ is the $\ell$th component of $\bm{\eta}$ for any optimal dual variables 
    $(\bm{\alpha},\bm{\beta},\bm{\gamma},\bm{\eta})\in D^*(\bfS)$.
\item \label{item:dual_partial} The state dependent optimization function $\Phi_{\omega_t}(\cdot)$ is differentiable with respect to $S_\ell$ at the
supply location vector $\bfS$.  Moreover, the partial derivative is equal to the value of the optimal dual variable for the
flow conservation constraint at location $\ell$:
$$
\frac{\partial}{\partial S_\ell}\Phi_{\omega_t}(\bfS) = \eta_\ell^*.
$$
\item \label{item:cts_partial} The partial derivative $\frac{\partial}{\partial S_\ell}\Phi_{\omega_t}(\bfS)$ is continuous at $\bfS$.
\end{enumerate}
\end{lemma*}

Below, in Section \ref{sec:partials_proof} we prove Lemma \ref{lem:unique_dual}.

\subsection{Proof of Lemma \ref{lem:unique_dual}}
\label{sec:partials_proof}
We prove Lemma \ref{lem:unique_dual} by backwards induction on the time $t$. For the rest of this section we hold fixed
a time-scenario $\omega_t$, a supply-location vector $\bfS=(S_\ell:\ell\in\calL)$, and we fix a location $\ell\in\calL$
for which the volume of supply at $\ell$ under $\bfS$ is nonzero, i.e. $S_\ell > 0$.

Our backwards induction hypothesis states that the conclusion of Lemma \ref{lem:unique_dual} holds for all supply-location
vectors at all time $t+1$ scenarios. For clarity we formally state our backwards induction hypothesis in Assumption \ref{assn:backwards_induction}.
\begin{assumption}
\label{assn:backwards_induction}
When $t=T$ is the final time period then we make no assumption.  When $t<T$, let $\omega_{t+1}$ be any time $t+1$ scenario,
let $\bfS'$ be any supply location vector, and let $\ell'$ be any location for which the volume of supply at $\ell'$ under $\bfS'$
is nonzero. Then Lemma \ref{lem:unique_dual} parts \ref{item:unique_dual}, \ref{item:dual_partial}, and \ref{item:cts_partial} are true,
with respect to $\omega_{t+1}$, $\bfS'$ and $\ell'$. 
\end{assumption}

In the following subsections we prove parts \ref{item:unique_dual}, \ref{item:dual_partial}, and \ref{item:cts_partial} of Lemma
\ref{lem:unique_dual} assuming the backwards induction hypothesis.

\subsubsection{Proof of Part \ref{item:unique_dual}}
\label{appdx:unique_dual_proof}
Lemma \ref{lem:unique_dual} Part \ref{item:unique_dual} claims that the optimal dual variable associated with the flow-conservation
constraint for location $\ell$ is unique.  We prove this claim by invoking Lemma \ref{lem:csc_any_primal_dual} which states
that the Lagrangian optimality conditions for the state-dependent optimization problem hold between any pair of primal
and dual optima.

Specifically, let $(\bm{\alpha}, \bm{\beta}, \bm{\gamma}, \bm{\eta}), (\bm{\alpha'}, \bm{\beta'}, \bm{\gamma'}, \bm{\eta'})\in D^*(\bfS)$
be any pair of dual optima and let $(\bff,\bfg)\in F^*(\bfS)$ be any primal optimum.
Lemma \ref{lem:csc_any_primal_dual} states that the stationarity conditions and complementary slackness conditions hold between
the primal optimum $(\bff,\bfg)$ and both dual optima  
$(\bm{\alpha}, \bm{\beta}, \bm{\gamma}, \bm{\eta}), (\bm{\alpha'}, \bm{\beta'}, \bm{\gamma'}, \bm{\eta'})$.

We first consider the stationarity optimality conditions. From the assumption that location $\ell$ has nonzero supply-volume under $\bfS$,
there must be a destination $d\in\calL$ for which a nonzero volume of drivers traverse from $\ell$ to $d$ under any feasible solution. In particular,
consider a location $d$ for which the $f_{(\ell,d)}$ component of the optimal solution $(\bff,\bfg)$ is nonzero.

Observe that under the backwards induction hypothesis, the objective function $\calW_{\omega_t}(\bff,\bfg)$ is differentiable with respect
to $f_{(\ell,d)}$ at the primal optimum $(\bff,\bfg)$.  Recall the objective function $\calW_{\omega_t}(\bff,\bfg)$ is the sum of the
current reward $\calU_{\omega_t}(\bff,\bfg)$ and the future reward
$\bbE_{\omega_t}\left[\Phi_{\omega_{t+1}}(\bfS_{\omega_{t+1}}(\bff))\right]$.
That the current reward $\calU_{\omega_t}(\bff,\bfg)$ is differentiable with respect to any nonzero component of $\bff$ follows from
Lemma \ref{lem:reward_fn}.
That the future reward is differentiable with respect to $f_{(\ell,d)}$ follows from the backward induction hypothesis.
Specifically, for any time $t+1$ scenario $\omega_{t+1}$, there will be nonzero supply-volume at location $d$ under the resulting time $t+1$
supply-location vector, since there is a nonzero volume of drivers driving from $\ell$ to $d$. 
Therefore the state-dependent optimization function $\Phi_{\omega_{t+1}}(\bfS_{\omega_{t+1}}(\bff))$ is differentiable with respect to 
the volume of supply at location $d$. Therefore, it follows from the chain rule that the partial derivative of the future reward exists
and can be written as follows:
$$\frac{\partial}{\partial f_{(\ell,d)}}\bbE_{\omega_t}\left[\Phi_{\omega_{t+1}}(\bfS_{\omega_{t+1}}(\bff))\right] =
\bbE_{\omega_t}\left[\frac{\partial}{\partial S_{d}}\Phi_{\omega_{t+1}}(\bfS_{\omega_{t+1}}(\bff))\right] 
 .$$

Having established differentiability of the objective function with respect to the $f_{(\ell,d)}$ variable, let's return to the
stationarity optimality conditions.
It follows from equation (\ref{eq:stationarity_nonzero_f})\mccomment{Maybe equation (\ref{eq:stationarity_nonzero_f})
and (\ref{eq:stationarity_nonzero_g}) should be stated as a Lemma} that the stationarity conditions imply the following equality:
$$
\frac{\partial}{\partial f_{(\ell,d)}}\calW_{\omega_t}(\bff,\bfg) = \eta_\ell - \alpha_{(\ell,d)} - \gamma_{(\ell,d)}.
$$
The above equation gives us a useful characterization of the dual variable $\eta_\ell$, whenever the supply-volume $S_\ell$ is
greater than $0$:
\begin{align}
    \eta_\ell& = \frac{1}{S_\ell} \sum_{d\in\calL} f_{(\ell,d)}\left[\frac{\partial}{\partial f_{(\ell,d)}}\calW_{\omega_t}(\bff,\bfg)  + \alpha_{(\ell,d)}
                                                                                                                    + \gamma_{(\ell,d)}\right]\nonumber\\
    & = \frac{1}{S_\ell} \sum_{d\in\calL}\left[ f_{(\ell,d)}\frac{\partial}{\partial f_{(\ell,d)}}\calW_{\omega_t}(\bff,\bfg)
    + g_{(\ell,d)}\gamma_{(\ell,d)}\right]\nonumber\\
    & = \frac{1}{S_\ell} \sum_{d\in\calL}\left[ f_{(\ell,d)}\frac{\partial}{\partial f_{(\ell,d)}}\calW_{\omega_t}(\bff,\bfg)
    + g_{(\ell,d)}\frac{\partial}{\partial g_{(\ell,d)}}\calW_{\omega_t}(\bff,\bfg)\right]\label{eq:eta_expression}
\end{align}

For the remainder of this proof we consider two cases: in one case a nonzero volume of drivers traversing $(\ell,d)$ have no
passenger, i.e. $g_{(\ell,d)} < f_{(\ell,d)}$; in the other case, we have all drivers traversing $(\ell,d)$ are carrying
a passenger, i.e. $g_{(\ell,d)} = f_{(\ell,d)}$.
In the first case, since the constraint
$g_{(\ell,d)} \leq f_{(\ell,d)}$ is strict it follows from the complementary slackness conditions that dual variable
associated with the constraint, $\gamma_{(\ell,d)}$, is $0$. Similarly, since $f_{(\ell,d)}$ is nonzero, the dual variable associated with the
nonnegativity constraint on $f_{(\ell,d)}$, that is $\alpha_{(\ell,d)}$, is $0$.
Therefore, the stationarity condition simplifies to the following:
$$
\frac{\partial}{\partial f_{(\ell,d)}}\calW_{\omega_t}(\bff,\bfg) = \eta_\ell.
$$
Since Lemma \ref{lem:csc_any_primal_dual} states the optimality conditions hold between any pair of primal and dual optima, we can apply
the same line of reasoning to our other dual solution and conclude
$$
\frac{\partial}{\partial f_{(\ell,d)}}\calW_{\omega_t}(\bff,\bfg) = \eta'_\ell,
$$
from which $\eta_\ell=\eta'_\ell$ follows.

In the second case where all drivers along $(\ell,d)$ have a passenger the dual variable $\gamma_{(\ell,d)}$ need not be $0$, but
equation (\ref{eq:stationarity_nonzero_g}) gives us the following characterization:
$$
\frac{\partial}{\partial g_{(\ell,d)}}\calW_{\omega_t}(\bff,\bfg) = \gamma_{(\ell,d)} - \beta_{(\ell,d)}.
$$
Note that the objective function only depends on $\bfg$ for the current reward $\calU_{\omega_t}(\bff,\bfg)$, and this is differentiable
with respect to any nonzero component $g_{(\ell,d)}$.  Further, $\beta_{(\ell,d)}$ is the dual variable associated with the nonnegativity
constraint on $g_{(\ell,d)}$, and from the assumption that $g_{(\ell,d)}=f_{(\ell,d)}$ and $f_{(\ell,d)}>0$ the complementary slackness
conditions imply that $\beta_{(\ell,d)}$ is $0$.
Therefore we obtain the equality
$$
\frac{\partial}{\partial f_{(\ell,d)}}\calW_{\omega_t}(\bff,\bfg) + \frac{\partial}{\partial g_{(\ell,d)}}\calW_{\omega_t}(\bff,\bfg) = \eta_\ell
$$
and
$$
\frac{\partial}{\partial f_{(\ell,d)}}\calW_{\omega_t}(\bff,\bfg) + \frac{\partial}{\partial g_{(\ell,d)}}\calW_{\omega_t}(\bff,\bfg) = \eta'_\ell
$$
from which $\eta'_\ell=\eta_\ell$ follows.

\subsubsection{Proof of Part \ref{item:dual_partial}}
We give a high-level outline of the proof for part \ref{item:dual_partial}.
We start by using Lemma \ref{lem:duality_gap_value_fn}, which considers the \emph{value function} associated with an optimization problem,
which gives the optimal value
of an optimization problem as a function of the constraint vector.
Lemma \ref{lem:duality_gap_value_fn} shows that the set of optimal
dual variables for the optimization problem at a particular constraint vector is the same as the set of negative subgradients
for the value function at that constraint vector.

The negative state-dependent optimization function $-\Phi_{\omega_t}(\bfS)$ is similar to the value function considered by Lemma
\ref{lem:duality_gap_value_fn}, except the state-dependent optimization problem has a mix of equality constraints and inequality constraints
whereas the optimization problem considered in Lemma \ref{lem:duality_gap_value_fn} only explicitly includes inequality constraints,
and the supply-location vector $\bfS$ that $-\Phi_{\omega_t}(\bfS)$ takes as an argument only varies the bounds for the equality constraints.

To use the result of Lemma \ref{lem:duality_gap_value_fn} in the context of our state-dependent optimization function we 
first rewrite the state-depent optimization problem solely in terms of inequality constraints, where each equality
constraint is replaced by two inequality constraints pointing in opposite directions.  When the state-dependent optimization
problem is written in this way it has the same structure as the optimization problem considered in Lemma \ref{lem:duality_gap_value_fn},
so we can associate a ``value function'' with the problem in the same manner as Lemma \ref{lem:duality_gap_value_fn}, and then the 
state-dependent optimization problem $-\Phi_{\omega_t}(\bfS)$ is equivalent to this value function applied to a linear
transformation of the supply-location vector $\bfS$.

Finally we invoke Lemma \ref{lem:subgradient_chain_rule} which gives a version of the chain-rule that applies to subgradients.
Invoking Lemma \ref{lem:subgradient_chain_rule} tells us that a vector $\bm{\phi}$ is a subgradient of $(-\Phi_{\omega_t})(\bfS)$
if and only iff $\bm{\phi}=-\bm{\eta}$, where $\bm{\eta}$ is the restriction of any optimal dual variable 
$(\bm{\alpha}, \bm{\beta}, \bm{\gamma}, \bm{\eta})\in D^*(\bfS)$ to the components associated with the 
flow-conservation equality constraints.

The conclusion of part \ref{item:dual_partial} follows from the results of part \ref{item:unique_dual}, which states that
there is a unique optimal dual variable for the flow-conservation constraint associated with location $\ell$. It follows that
there is a unique value for the $\ell$th component of 
any subgradient for the negative state-dependent optimization function evaluated at $\bfS$. We know that a function
is differentiable at a point when the subderivative of that function 
at that point is unique \mccomment{cite}.  Therefore, the partial derivative of the negative state-dependent value
function with respect to the $\ell$th component of the input vector exists and is equal to the negative dual variable
associated with the $\ell$th flow-conservation constraint.  Taking negatives on both sides of the equality,
we conclude $\frac{\partial}{\partial S_\ell}\Phi_{\omega_t}(\bfS)$ exists and is equal to $\eta_\ell^*$, as claimed.

\subsubsection{Proof of Part \ref{item:cts_partial}}
The final result left to establish for Lemma \ref{lem:unique_dual} is that 
the state-dependent optimization function has continuous partial derivatives at any location where the 
supply-location vector is nonzero.
We prove this result by showing that, for any sequence of supply-location vectors converging to $\bfS$,
the corresponding sequence of partial derivatives with respect to location $\ell$ converges to the partial
derivative evaluated at $\bfS$.

Formally, let $(\bfS_k:k=1,2,\dots)$ be a sequence of supply-location vectors converging to $\bfS$,
and assume without loss of generality that the $\ell$th component of each iterate $\bfS_k$ is nonzero.
Having already established parts \ref{item:unique_dual} and \ref{item:dual_partial} of Lemma \ref{lem:unique_dual},
we know the following:
\begin{itemize}
\item For the state-dependent optimization problem with respect to each supply-location vector $\bfS_k$
    there is a unique optimal dual variable associated with the flow-conservation constraint for location $\ell$.
\item The state-dependent optimization function evaluated at $\bfS_k$ is partially differentiable in the direction $\ell$, 
and the value of the partial derivative is equal to the optimal dual variable for the location $\ell$ flow-conservation constraint.
\end{itemize}
Let $\eta_\ell^*(\bfS_k)$ denote the optimal dual variable for the location $\ell$ flow-conservation constraint with respect
to $\bfS_k$ and let $\eta_\ell^*(\bfS)$ denote the same optimal dual variable with respect to $\bfS$.
We will show that the partial derivative of the state-dependent optimization function is continuous at $\bfS$
by showing that the sequence of optimal dual variables $(\eta_\ell^*(\bfS_k):k=1,2,\dots)$ converges to $\eta_\ell^*(\bfS)$,
i.e.
\begin{equation}
\label{eq:dual_sequence_converges}
\lim_{k\to\infty}\eta_\ell^*(\bfS_k) = \eta_\ell^*(\bfS).
\end{equation}

Our approach for establishing the equality in equation (\ref{eq:dual_sequence_converges}) is to use the
Lagrangian optimality conditions to obtain an equivalent expression in terms of primal solutions.
To obtain this equivalent expression that works in the space of primal solutions, we construct a function
that takes as input a primal optimal solution and produces the value of the optimal dual variable for location
$\ell$ as the output.  Define $\calS=\left\{\bfS_k : k=1,2,\dots\right\}\cup\left\{\bfS\right\}$ to be the 
set of all supply-location vectors in our sequence and the limiting supply-location vector to which they converge,
and define
$$
\calF = \bigcup_{\bfS'\in\calS} F^*(\bfS')
$$
to be the set of all primal solutions that are optimal for some supply-location vector in $\calS$.

We write $E_\ell:\calF\to\bbR$ to denote our function that recovers the optimal dual variable associated with
location $\ell$ from a primal optimal solution.  For a primal optimal solution $(\bff^*,\bfg^*)\in\calF$, 
the exact definition of $E_\ell(\bff^*,\bfg^*)$ will reflect the optimality conditions associated with a particular
route $(\ell,d)$.  The choice of the destination location $d$ will depend on which components of $\bff^*$ are nonzero.
Specifically, order the locations in $\calL$ as $d_1,d_2,\dots,d_n$ where $n=|\calL|$, and let $i(\bff^*)=i$ be the
smallest index in $\{1,2,\dots,n\}$ such that $f^*_{(\ell,d_{i})}$ is nonzero. Note that by construction every supply-location vector
in $\calS$ has nonzero volume on location $\ell$, so every primal optimal solution $(\bff^*,\bfg^*)\in\calF$ always has at least
one destination $d$ for which $f^*_{(\ell,d)}$ is nonzero; in particular, the index $i(\bff^*)$ is always well-defined.

Now, consider any sequence of optimal solutions $(\bff^*_k,\bfg^*_k)\in F^*(\bfS_k)$ for $k\geq 1$, and observe the sequence $(\bff^*_k,\bfg^*_k)$ is bounded, in
particular there is a convergent subsequence $((\bff^*_k(i),\bfg^*_k(i)) : i\geq 1)$. Let $\bff^*,\bfg^*$ be the limit point of this subsequence, and observe that
$(\bff^*,\bfg^*)\in F^*(\bfS)$.  Define the function $E_\ell(\bff,\bfg)$ as
$$
E_\ell(\bff,\bfg) = \frac{1}{S_\ell} \sum_{d\in\calL}\left[ f_{(\ell,d)}\frac{\partial}{\partial f_{(\ell,d)}}\calW_{\omega_t}(\bff,\bfg)
    + g_{(\ell,d)}\frac{\partial}{\partial g_{(\ell,d)}}\calW_{\omega_t}(\bff,\bfg)\right]
$$
From the equation (\ref{eq:eta_expression}),  we know that $\eta_\ell^*(\bfS')$ is equal to $E_\ell(\bff^*,\bfg^*)$, if 
$(\bff^*,\bfg^*)\in F^*(\bfS')$ for any $\bfS'\in \calF$.

Observe that for $i$ large enough, the convergent subsequence $((\bff^*_{k(i)},\bfg^*_{k(i)}) : i\geq 1)$ will be nonzero on the same components as $(\bff^*,\bfg^*)$.
Therefore 
$$\lim_{i\to\infty}\eta_\ell^*(\bfS_{k(i)}) = \lim_{i\to\infty} E_\ell(\bff^*_{k(i)},\bfg^*_{k(i)}) = E_\ell(\bff^*,\bfg^*) = \eta_\ell^*(\bfS).$$
Also observe that the sequence of dual variables $(\eta_\ell^*(\bfS_k):k\geq 1)$ is the same sequence as 
$(E_\ell(\bff^*_k,\bfg^*_k) : k\geq 1)$. 
From the above equation it follows that every limit point is equal to $\eta_\ell^*(\bfS)$.  Since the sequence $(\eta_\ell^*(\bfS_k):k\geq 1)$ is bounded
and since there is a single limit point, it follows the sequence converges: $\lim_{k\to\infty}\eta_\ell^*(\bfS_k)=\eta_\ell^*(\bfS)$, finishing the proof.

\subsubsection{Additional Lemma for Proof of Lemma \ref{lem:unique_dual}}

\begin{lemma}
\label{lem:csc_any_primal_dual}
For any convex optimization problem the stationarity and complementary slackness conditions hold between any pair of primal and dual optima.
In particular, for the state-dependent optimization problem (\ref{eq:state_dep_opt}) with respect to any scenario $\omega_t$
and supply location vector $\bfS$, if $(\bff,\bfg)\in F^*(\bfS)$ is any primal optimum and 
$(\bm{\alpha},\bm{\beta},\bm{\gamma},\bm{\eta})\in D^*(\bfS)$ is any dual optimum then the stationarity conditions hold, i.e. 
$$
0\in\partial L(\bff,\bfg;\bm{\alpha},\bm{\beta},\bm{\gamma},\bm{\eta}),
$$
and the complementary slackness conditions hold, i.e.
$$
\alpha_{(\ell,d)}f_{(\ell,d)} = 0,\ \beta_{(\ell,d)}g_{(\ell,d)} = 0,\ \gamma_{(\ell,d)}(g_{(\ell,d)}-f_{(\ell,d)}) = 0.
$$
\end{lemma}

For the following lemma, consider the optimization problem
\begin{equation}
\label{eq:generic_opt}
\inf_{x\in\bbR^m}\left\{f(x) \mid g(x)\leq 0\right\},
\end{equation}
where $f:\bbR^m\to\bbR$ is our objective function and $g:\bbR^m\to\bbR^n$ is our constraint function.
We assume that $f$ and $g_1,\dots,g_n$ are convex functions, where $g_j(x)$ is the $j$th component function
of the multivariate constraint function $g$.
The Lagrangian function $L:\bbR^m\times\bbR^n_+\to\bbR$ is defined by
$$
L(x,\lambda) = f(x) + \lambda^Tg(x).
$$
The dual function $\Gamma:\bbR^n_+\to\bbR$ is defined by
$$
\Gamma(\lambda) = \inf_{x\in\bbR^m} L(x;\lambda).
$$
The \emph{value function} associated with the mathematical program (\ref{eq:generic_opt}) describes how the optimal
value changes as we perturb the constraint vector away from $0$.  Formally, it is a function
$v:\bbR^n\to\bbR$ defined by the equation
\begin{equation}
v(b) = \inf_{x\in\bbR^m}\left\{f(x)\mid g(x)\leq b\right\}.
\end{equation}
The problem (\ref{eq:generic_opt}) is said to have zero duality gap when strong duality holds,
i.e. when the primal optimum is equal to the dual optimum, as described by the following equation:
$$
 \inf_{x\in\bbR^m}\left\{f(x) \mid g(x)\leq 0\right\} =
 \sup_{\lambda\in\bbR^m_+}\Gamma(\lambda).
 $$
Any $\lambda^*\in\bbR_m^*$ which achieves the optimum on the right side of the above equation is said to be
an optimal dual solution.
The following lemma appears as Corollary 4.3.6 in \cite{borwein_lewis}.
\begin{lemma}
\label{lem:duality_gap_value_fn}
The mathematical program (\ref{eq:generic_opt}) has zero duality gap if and only if the value function $v$
is lower semicontinuous at $0$.  In this case the set of dual optimal solutions is $-\partial v(0)$.
\end{lemma}

In order to apply the result of Lemma \ref{lem:duality_gap_value_fn} we also make use of the following result
from \cite{borwein_lewis}, which provides a chain rule for subdifferentials of convex functions composed
with linear functions.
\begin{lemma}
\label{lem:subgradient_chain_rule}
Let $f:\bbR^m\to\bbR$ be a convex function and let $A\in\bbR^{m\times n}$ be a matrix. Then the following equality
is satisfied for $x\in\bbR^n$:
$$
\partial(f\circ A)(x) = A^T\partial f(Ax).
$$
\end{lemma}

\subsection{Proof of Lemma \ref{lem:cts_derivative_boundary}}
We re-state Lemma \ref{lem:cts_derivative_boundary} below.
\begin{lemma*}
Let $\bfS$ be a supply-location vector with nonnegative components and assume $S_\ell=0$ for some location $\ell$.  
Then the right-derivative $\frac{\partial}{\partial S_\ell}\Phi_{\omega_t}(\bfS^+)$ is well-defined at $\bfS$.  Moreover, the
    partial derivative function $\frac{\partial}{\partial S_\ell}\Phi_{\omega_t}(\bfS)$, defined in (\ref{eq:partial_derivative_boundary}),
    is continuous over the set $\{\bfS\in\bbR^\calL : S_\ell \geq 0 \forall \ell\in\calL\}$.

Also, in the case where $S_\ell=0$, there exists an optimal dual solution such that the dual variable
$\eta_\ell$
associated with the $\ell$th flow-conservation constraint is equal to the right derivative $\frac{\partial}{\partial S_\ell}\Phi_{\omega_t}(\bfS^+)$.
\end{lemma*}

We prove Lemma \ref{lem:cts_derivative_boundary} by characterizing optimal primal solutions to the state-dependent
optimization problem, in the regime where there is an infinitesimal volume of drivers at $\ell$.

We start by defining the continuation utilities associated with a primal solution.  Let $\omega_t$ be any scenario,
and let $\bff = (f_{(\ell,d)}\geq 0 : (\ell,d)\in\calL^2)$ be any flow vector.  For a destination $d\in\calL$ define the  continuation
utility associated with $d$ and $\bff$ to be
\begin{equation}
\label{eq:continuation_utility}
U_d(\bff) = \bbE\left[\frac{\partial}{\partial S_d}\Phi_{\omega_{t+1}}(\bar{\bfS}_{\omega_{t+1}}(\bff))\mid\omega_t\right].
\end{equation}
The following lemma states that the continuation utilities associated with optimal solutions to the state-dependent optimization
problem all take the same value.
\begin{lemma}
\label{lem:continuation_utilities}
Let $\omega_t$  be any scenario and let $\bfS$ be any feasible supply-location vector.  Let $(\bff_i,\bfg_i)\in F^*_{\omega_t}(\bfS)$, 
for $i=1,2$, be 
any optimal solutions to the state-dependent optimization problem with respect to $(\omega_t,\bfS_t)$.  Then the continuation utilities
under $\bff_1$ and $\bff_2$ are the same, i.e. $U_d(\bff_1) = U_d(\bff_2)$ for any choice of destination $d$.
\end{lemma}
Lemma \ref{lem:continuation_utilities} follows from the optimality conditions, and the fact that complementary slackness
holds between any pair of primal and dual optima.

Our next Lemma characterizes optimal solutions for the state-dependent optimization problem 
in the regime where there is an infinitesimal volume of drivers at $\ell$.

\begin{lemma}
\label{lem:infinitesimal_volume_optimum}
Let $\omega_t$ be any scenario and let $\bfS$ be any feasible supply-location vector.  Assume $S_\ell=0$ for some location $\ell$.
Let $(\bfS_n)_{n=1}^\infty$ be a sequence of feasible supply-location vectors converging to $\bfS$, such that each element of the
sequence has a nonzero volume of drivers at $\ell$, i.e. $S^n_\ell > 0$ for all $n$.
Let $(\bff^*_n,\bfg^*_n)\in F^*_{\omega_t}(\bfS_n)$ be an optimal primal solution for each $n$, and define
$$
\bff_{\ell}^n = \frac{1}{S^n_\ell}\left(f^*_{n,(\ell,d)} : d\in\calL\right)\ \ \ 
\bfg_{\ell}^n = \frac{1}{S^n_\ell}\left(g^*_{n,(\ell,d)} : d\in\calL\right)
$$
to be the restriction of $(\bff^*_n,\bfg^*_n)$ to components that correspond to trips originating from $\ell$, divided by the
volume of drivers at $\ell$ under the $n$th iterate in the sequence.  
Then every limit point of the sequence $(\bff_{\ell,n},\bfg_{\ell,n})_{n=1}^\infty$ is an optimal solution to the following
optimization problem:
\begin{align}
\sup\ \ \ & \sum_{d\in\calL} V_d g_d + \sum_{d\in\calL} (U_d-c_{(\ell,d)}) f_d - A\left(\sum_{d\in\calL}g_d,\sum_{d\in\calL}f_d\right)
\label{eq:linearized_infinitesimal_optimization}\\
\mbox{such that}\ \ \ & 0 \leq g_d \leq f_d\ \  \forall d\in\calL,\\
& \sum_{d\in\calL}f_d = 1.\label{eq:infinitesimal_opt_flow_conservation}
\end{align}
In the above, $V_d$ is to the maximum rider value held by riders requesting a trip from $\ell$ to $d$ under $\omega_t$,
$U_d$ is the continuation utility associated with each destination $d\in\calL$ under an optimal solution for the limiting 
supply-location vector $\bfS$, and $A(\cdot,\cdot)$ is the add-passenger disutility cost function $A(g,f) = \frac{C}{2}\frac{g^2}{f}$.

Moreover, the optimization problem (\ref{eq:linearized_infinitesimal_optimization}) has a unique optimal dual variable $\eta_\ell$
associated with the constraint (\ref{eq:infinitesimal_opt_flow_conservation}), and the value of this dual variable is equal
to the right-derivative limit $\frac{\partial}{\partial S_\ell}\Phi_{\omega_t}(\bfS^+)$.
\end{lemma}
\begin{proof}
We give a high-level outline for the proof of Lemma \ref{lem:infinitesimal_volume_optimum}.
First, a backwards induction argument lets us assume that the objective function of the original state-dependent optimization problem
$\calW_{\omega_t}(\bff,\bfg)$ has continuous right derivatives on the boundary of the feasible region (the backwards induction assumption
applies to the future-period reward function $\Phi_{\omega_{t+1}}(\bfS_{\omega_{t+1}}(\bff))$ which appears as a summand in the objective
$\calW_{\omega_t}(\bff,\bfg)$).

Next, by considering convergent subsequences, we can assume without loss of generality that the sequence of optimal
solutions $(\bff^*_n,\bfg^*_n)$, $n=1,2,\dots$, converges to some limit $(\bff^*,\bfg^*)$, and by continuity of the objective function it 
follows that the limit point is optimal with respect to the limiting supply-location vector $(\bff^*,\bfg^*)\in F^*_{\omega_t}(\bfS)$.

Also by considering convergent subsequences, we can assume without loss of generality that the scaled sequence of points
$$
\bff_{\ell}^n = \frac{1}{S^n_\ell}\left(f^*_{n,(\ell,d)} : d\in\calL\right)\ \ \ 
\bfg_{\ell}^n = \frac{1}{S^n_\ell}\left(g^*_{n,(\ell,d)} : d\in\calL\right)
$$
converges to some limit $(\bar{\bff}_\ell,\bar{\bfg}_\ell)$, and by virtue of the scaling it follows that this limit is a feasible
solution for the linearized optimization problem (\ref{eq:linearized_infinitesimal_optimization}).

Next, we consider the optimality conditions associated with each iterate in our sequence of supply-location vectors.
Let $\eta_\ell^n$ be the dual variable associated with the $\ell$th flow-conservation constraint, for the $n$th supply-location
vector in our sequence.  Since every iterate in our sequence has nonzero volume of drivers at $\ell$, the dual variable $\eta_\ell^n$
is unique.  Further, the characterization (\ref{eq:eta_expression}) of this dual variable yields the following expression:
$$
\eta_\ell^n = \frac{1}{S_\ell^n} \sum_{d\in\calL}\left[ f^n_{(\ell,d)}\frac{\partial}{\partial f_{(\ell,d)}}\calW_{\omega_t}(\bff^*_n,\bfg^*_n)
    + g^n_{(\ell,d)}\frac{\partial}{\partial g_{(\ell,d)}}\calW_{\omega_t}(\bff^*_n,\bfg^*_n)\right].
$$
Notice that the partial derivatives of the objective function have the following expressions:
$$
\frac{\partial}{\partial f_{(\ell,d)}}\calW_{\omega_t}(\bff,\bfg) = -c_{(\ell,d)} + \frac{\partial}{\partial f_{(\ell,d)}} \calU_{\omega_t}^{>t}(\bff)
- \frac{\partial}{\partial f_{(\ell,d)}} A(\bfg^T\bfone_\ell,\bff^T\bfone_\ell),
$$
and
$$
\frac{\partial}{\partial g_{(\ell,d)}}\calW_{\omega_t}(\bff,\bfg) = \frac{d}{dg}U_{(\ell,d,\omega_t)}(g_{(\ell,d)}) - \frac{\partial}{\partial g_{(\ell,d)}} A(\bfg^T\bfone_\ell,\bff^T\bfone_\ell).
$$
Also, notice that the partial derivatives of the add-passenger disutility function have the following expressions:
$$
\frac{\partial}{\partial f_{(\ell,d)}} A(\bfg^T\bfone_\ell,\bff^T\bfone_\ell) = -\frac{C}{2}\left(\frac{\bfg^T\bfone_\ell}{\bff^T\bfone_\ell}\right)^2\ \ \mbox{and}\ \ 
\frac{\partial}{\partial g_{(\ell,d)}} A(\bfg^T\bfone_\ell,\bff^T\bfone_\ell) = C\frac{\bfg^T\bfone_\ell}{\bff^T\bfone_\ell}.
$$
Therefore, the partial derivatives of $A(\cdot,\cdot)$ are invariant to both of its arguments being scaled by the same multiple. In particular,
we have the equality
$$
\frac{\partial}{\partial f_{(\ell,d)}} A(\bfg_n^{*T}\bfone_\ell,\bff_n^{*T}\bfone_\ell) = 
\frac{\partial}{\partial f_{(\ell,d)}} A(\frac{\bfg_n^{*T}\bfone_\ell}{S_\ell^n},\frac{\bff_n^{*T}\bfone_\ell}{S_\ell^n})
\ \ \mbox{and} \ \ 
\frac{\partial}{\partial g_{(\ell,d)}} A(\bfg_n^{*T}\bfone_\ell,\bff_n^{*T}\bfone_\ell) = 
\frac{\partial}{\partial g_{(\ell,d)}} A(\frac{\bfg_n^{*T}\bfone_\ell}{S_\ell^n},\frac{\bff_n^{*T}\bfone_\ell}{S_\ell^n}).
$$

We can rewrite our expression for the optimal dual variable $\eta_\ell^n$ as follows:
\begin{align*}
\eta_\ell^n = & \sum_{d\in\calL} \frac{f^n_{(\ell,d)}}{S_\ell^n}
					\left(\frac{\partial}{\partial f_{(\ell,d)}}\calU_{\omega_t}^{>t}(\bff^*_n) - c_{(\ell,d)} - 
\frac{\partial}{\partial f_{(\ell,d)}} A(\frac{\bfg_n^{*T}\bfone_\ell}{S_\ell^n},\frac{\bff_n^{*T}\bfone_\ell}{S_\ell^n})\right)\\
    &+ \sum_{d\in\calL}  \frac{f^n_{(\ell,d)}}{S_\ell^n}\left(
\frac{d}{dg}U_{(\ell,d,\omega_t)}(g^n_{(\ell,d)}) - \frac{\partial}{\partial g_{(\ell,d)}} A(\frac{\bfg_n^{*T}\bfone_\ell}{S_\ell^n},\frac{\bff_n^{*T}\bfone_\ell}{S_\ell^n})\right)
\end{align*}

We know that $(\bff^*_n,\bfg^*_n)$ converges as $n\to\infty$, as does $\frac{f^n_{(\ell,d)}}{S_\ell^n}$ and $\frac{g^n_{(\ell,d)}}{S_\ell^n}$.
Therefore the sequence of optimal dual variables converges to the following limit.
\begin{align*}
\lim_{n\to\infty} \eta_\ell^n = &\sum_{d\in\calL}\bar{f}_{(\ell,d)}\left(\frac{\partial}{\partial f_{(\ell,d)}}\calU_{\omega_t}^{>t}(\bff^*) - c_{(\ell,d)}
- \frac{\partial}{\partial f_{(\ell,d)}} A(\bar{\bfg}^T\bfone_\ell,\bar{\bff}^T\bfone_\ell)\right)\\
& + \sum_{d\in\calL}\bar{g}_{(\ell,d)}\left(\frac{d}{dg}U_{(\ell,d,\omega_t)}(g^*_{(\ell,d)}) - \frac{\partial}{\partial g_{(\ell,d)}} A(\bar{\bfg}^T\bfone_\ell,\bar{\bff}^T\bfone_\ell)\right).
\end{align*}
However, since $\bfg^*$ is feasible for the limiting supply-location vector $\bfS$, we know 
$$\frac{d}{dg}U_{(\ell,d,\omega_t)}(g^*_{(\ell,d)})=\frac{d}{dg}U_{(\ell,d,\omega_t)}(0)=V_d,$$
where $V_d$ is the maximum rider value for riders requesting from $\ell$ to $d$ under $\omega_t$.
Also, since $(\bff^*,\bfg^*)$ is an optimal solution with respect to $\bfS$, we have the partial derivative
$\frac{\partial}{\partial f_{(\ell,d)}}\calU_{\omega_t}^{>t}(\bff^*)$ is equal to the optimal continuation utility $U_d$ associated with $d$.
Therefore, the limit of the sequence of dual variables, $\eta_\ell=\lim_{n\to\infty}\eta_\ell^n$, is equal to
\begin{align*}
\eta_\ell = &\sum_{d\in\calL}\bar{f}_{(\ell,d)}\left(U_d - c_{(\ell,d)}
- \frac{\partial}{\partial f_{(\ell,d)}} A(\bar{\bfg}^T\bfone_\ell,\bar{\bff}^T\bfone_\ell)\right)\\
& + \sum_{d\in\calL}\bar{g}_{(\ell,d)}\left(V_d - \frac{\partial}{\partial g_{(\ell,d)}} A(\bar{\bfg}^T\bfone_\ell,\bar{\bff}^T\bfone_\ell)\right).
\end{align*}
Optimality of $(\bar{\bff},\bar{\bfg})$ and $\eta_\ell$ follow by using the above characterization of $\eta_\ell$ to show that
 $(\bar{\bff},\bar{\bfg})$ satisfies optimality conditions for the optimization problem (\ref{eq:linearized_infinitesimal_optimization}).
Finally, uniqueness of the dual variable $\eta_\ell$ follows from the same argument we used in Part one (appendix \ref{appdx:unique_dual_proof})
 of the proof for Lemma 
\ref{lem:unique_dual}.
\end{proof}

Lemma \ref{lem:infinitesimal_volume_optimum} shows that the right-hand derivative limit $\frac{\partial}{\partial S_\ell}\Phi_{\omega_t}(\bfS^+)$
is well-defined when $S_\ell=0$, and that every sequence of partial derivatives $\frac{\partial}{\partial S_\ell}\Phi_{\omega_t}(\bfS_n)$ with
$S_\ell^n>0$ converges to the same limit.  To finish showing that the partial derivative function
To finish proving Lemma \ref{lem:cts_derivative_boundary} it suffices to show that sequences
of right-hand derivatives $\frac{\partial}{\partial S_\ell}\Phi_{\omega_t}(\bfS_n^+)$ converge, for supply-location vectors
$\bfS_n^+$ on the boundary of the feasible space, i.e. with $S_\ell^n=0$.

Lemma \ref{lem:boundary_partials_limit} follows from analyzing the optimization problem (\ref{eq:linearized_infinitesimal_optimization})
using  the same logic as part three of our proof of Lemma \ref{lem:unique_dual}.
\begin{lemma}
\label{lem:boundary_partials_limit}
Let $(\bfS_n)_{n=1}^\infty$ be a sequence of supply location vectors which converge to $\bfS$, all of which have $0$ driver volume at $\ell$, i.e.
$S_\ell^n=0$ for all $n$.  Then $\frac{\partial}{\partial S_\ell}\Phi_{\omega_t}(\bfS_n^+)$ converges to 
$\frac{\partial}{\partial S_\ell}\Phi_{\omega_t}(\bfS)$ as $n\to\infty$.
\end{lemma}

\section{Matching Process Details}
\label{sec:matching_process_appdx}
We view the matching process as a generic procedure for allocating trips to available drivers.
In general, we assume there is a stochastic matching process for the two-level model as well
as a deterministic matching process for the fluid model.
Our analysis holds for any matching process that satisfies three properties, stated informally below:
\begin{assumption}
    \label{assn:matching_process}
    \begin{enumerate}
        \item The random trip-volumes produced by the stochastic two-level model matching process converge to
            their corresponding deterministic fluid trip-volumes as the population size parameter grows to infinity.
        \item In the deterministic fluid matching process,  the only way for the trip volume produced by the matching process along a
            route to be smaller than the optimal trip volume along that route is if the drivers are using an
            acceptance threshold smaller than the optimal acceptance threshold.
        \item In the stochastic two-level model matching process, conditioning on the action taken by a single driver has
            negligible effect on the overall distribution of aggregate trip counts in the limit as the population size grows to infinity.
            Specifically,
            we assume there exists a sequence $(\beta_k)_{k=1}^\infty$ converging to $0$ as $k\to\infty$ such that the conditional distribution
            $\bbP(\bfS_{t+1}\mid a_i^t)$ is at most $\beta_k$ different from the unconditional distribution $\bbP(\bfS_{t+1})$,
            i.e. $|\bbP(\bfS_{t+1}\mid a_i^t) - \bbP(\bfS_{t+1})|\leq\beta_k$.
            We assume the sequence $(\beta_k)_{k=1}^\infty$ works for all initial states $(\omega_t,\bfS_t)$ and all driver strategy profiles.
    \end{enumerate}

\end{assumption}

For completeness, we define one example of a matching process that the platform can use, and show that it satisfies properties 1 and 2
listed above.  We conjecture that this process also satisfies 3, but have not yet verified this.

\subsection{Example Matching Process Definition}
\label{appdx:matching_process_defn}
The SSP matching process definition differs slightly between the fluid model and the two-level model, because granular
rider and driver decisions which affect the dynamics of the matching process are stochastic in the two level model but deterministic in the fluid model.

In both cases the SSP matching process makes use of a subroutine which takes a collection of drivers and a collection of riders
all heading towards the same destination, and allocates dispatches towards that destination until either no drivers or riders remain. 

\begin{definition}
    \label{def:single_destination_dispatch_subroutine}
The {single destination dispatch subroutine} in the fluid model is a procedure that takes as input a rider volume $\bar{R}$, 
driver volume $\bar{M}$, and a single disutility threshold $x$.  The output is a number $\bar{G}(\bar{R},\bar{M},x)$ specifying
the volume of dispatches that were accepted, and a number $\bar{U}(\bar{R},\bar{M},x)$ specifying the volume of drivers who were
not allocated a dispatch in the process.
The function definitions are stated below.
\begin{align}
    \label{eq:single_dest_fluid_dispatch}
\bar{G}(\bar{R},\bar{M},x) &= \min\left(Z(\bar{R},x), \bar{M}\right) \frac{x}{C},\\
    \bar{U}(\bar{R},\bar{M},x) &= \bar{M} - \min\left(Z(\bar{R},x), \bar{M}\right).\label{eq:single_dest_fluid_undispatch}
\end{align}
(Recall $Z(\bar{R},x)$ is the volume of drivers (\ref{eq:fluid_matching_dispatch_volume_fn}) who need to be allocated a dispatch in order
to see $\bar{R}$ accepted dispatches given a disutility threshold $x$).

In the two level model it is a procedure that takes as input a number of riders $R$, a number of drivers $M$, and a choice
of disutility thresholds $x_1,\dots,x_M$ for each driver.  The procedure allocates dispatches to riders until all rides
have been served or no drivers remain.  The output is a number of drivers $G$ who accepted a dispatch, and a set $U\subseteq [M]$
of driver labels who were not allocated a dispatch in the process.
The stochastic dynamics governing $G$ and $U$ are stated in Algorithm \ref{alg:single_destination_dispatch}.
\end{definition}

\begin{algorithm}
\caption{The Single Destination Dispatch Subroutine in the Two Level Model}
\label{alg:single_destination_dispatch}
\begin{enumerate}
\item Input: A number of dispatch requests $R$, a number of drivers $M$, a choice of disutility threshold $x_i$ for each $i=1,2,\dots,M$.
\item Randomly permute the driver labels: select a permutation $\pi:[M]\to[M]$ uniformly at random and define new labels $j=\pi(i)$.
\item Initialize $G\leftarrow 0$, $U\leftarrow \{1,2,\dots,M\}$.
\item For $j=1,2,\dots,M$:
\begin{itemize}
\item Allocate a dispatch to driver $j$. 
\item Sample the accept/reject decision $\delta_j\sim\mathrm{Ber}(x_j/C)$.
\item Record the decision: $R\leftarrow R - \delta_j$, $G \leftarrow G + \delta_j$.
\item Remove $j$ from $U$: $U\leftarrow U\setminus\{j\}$.
\item If $R=0$: go to step \ref{item:return_step}.
\end{itemize}
\item Return $G$, $U$. \label{item:return_step}
\end{enumerate}
\end{algorithm}

The SSP matching process, in both the fluid model and the two level model, uses the single destination dispatch subroutine in two
separate stages.  In the first stage, drivers are subdivided into groups, where there is one group for each destination,
and group sizes are determined by the dispatch volumes
and the disutility threshold from the optimal solution.  The single destination subroutine is then used to allocate dispatches
for each destination to drivers in the group associated with that destination.  This is the first stage of the matching process.
If any drivers remain undispatched after the first stage, the second stage goes through the dispatch destinations one by one and it uses
the single destination dispatch subroutine to allocate all remaining demand for that destination to all remaining drivers.

Notice that the disutility threshold associated with the optimal solution $(\bff^*,\bfg^*)$ is the same for every destination
$d$:
\begin{equation}
x^*_\ell = C\frac{\bfg^{*T}\bfone_\ell}{\bff^{*T}\bfone_\ell}.
\end{equation}
Let 
\begin{equation}\label{eq:opt_partition_size}Z_{(\ell,d)} = Z(g^*_{(\ell,d)},x^*_\ell)=g^*_{(\ell,d)}\frac{x^*_\ell}{C}\end{equation} be the volume of drivers we need to allocate a dispatch toward $d$ in order to
see $g^*_{(\ell,d)}$ accepted dispatches under the threshold $x^*_\ell$.  Observe
$$
\sum_{d\in\calL}Z_{(\ell,d)} = \sum_{d\in\calL}Z(g^*_{(\ell,d)}, x^*_\ell)= \sum_{d\in\calL}g^*_{(\ell,d)}\frac{C}{x^*_\ell} = \bff^{*T}\bfone_\ell = S_\ell,
$$
so the fractions $Z_{(\ell,d)} / S_\ell$ sum to $1$ over all $d$.
These fractions are used to determine the partition sizes in the first stage of the matching process.

\begin{algorithm}
\caption{The Matching Process in the Two Level Model}
\label{alg:matching_process}
\begin{enumerate}
    \item Input: A location $\ell$, a number of dispatch requests $R_{(\ell,d)}$ for each destination $d$, 
        a number of drivers $M_\ell$, an add-passenger threshold vector $\bfx_i = (x^i_{(\ell,d)} : d\in\calL)$ for each driver $i=1,2,\dots,M_\ell$,
        the fluid optimal actions $(\bff^*,\bfg^*,\bfx^*)$.
    \item Stage one:
        \begin{enumerate}
            \item Compute partition sizes $Z^*_{(\ell,d)}$ for each destination $d$, using equation (\ref{eq:opt_partition_size}) with
                the fluid optimal trip volume $g^*_{(\ell,d)}$ and threshold $x^*_\ell$.
            \item Partition the $M_\ell$ drivers into groups of size $M_{(\ell,d)}$, where each $M_{(\ell,d)}$ is rounded up
                or down from $Z^*_{(\ell,d)}$.  
            \item Use the single-destination dispatch subroutine to allocate the $R_{(\ell,d)}$ dispatch requests to the $M_{(\ell,d)}$ drivers,
                for each destination $d$.
            \item Record the output from the single-destination dispatch subroutine: 
                Let $G^{(1)}_{(\ell,d)}$ be the number of dispatch trips accepted and $U^{(1)}_{(\ell,d)}$ the number of drivers who were not allocated
                a trip.
        \end{enumerate}
\item Stage two:
\begin{enumerate}
    \item Let $R^{(2)}_{(\ell,d)} = R_{(\ell,d)} - G^{(1)}_{(\ell,d)}$ be the number of riders who have not been matched to a driver at the end
        of the first stage.
    \item Let $M^{(2)}_\ell = \sum_d U^{(1)}_{(\ell,d)}$ be the number of drivers who were not allocated a dispatch at the end of the first stage.
    \item Pick an ordering of the destinations $d_1,d_2,\dots,d_{|\calL|}$.  For each destination $d=d_1,d_2,\dots,d_{|\calL|}$:
        \begin{itemize}
            \item Use the single-destination dispatch subroutine to allocate the $R^{(2)}_{(\ell,d)}$ dispatch requests to the $M^{(2)}_\ell$
                remaining drivers.
            \item Record the output from the single-destination dispatch subroutine: let $G^{(2)}_{(\ell,d)}$ be the number of dispatch trips accepted,
                and $U^{(2)}_{(\ell,d)}$ be the number of drivers who 
                remain undispatched.  
            \item Update the number of remaining drivers: set $M^{(2)}_\ell = U^{(2)}_{(\ell,d)}$.
        \end{itemize}
\end{enumerate}
\end{enumerate}
\end{algorithm}

\begin{definition}
    \label{def:matching_process}
The matching process in the fluid model takes as input a market state $(\omega_t,\bfS_t)$, a volume of requests $\bar{R}_d$ for each destination $d$,
a location $\ell$,
a volume of drivers $\bar{M}_\ell$, and a disutility threshold vector $\bfx_\ell = (x_{(\ell,d)} : d\in\calL)$.  It proceeds in two stages:
\begin{enumerate}
\item In the first stage, it partitions the $M_\ell$ drivers into groups of size $\bar{M}_{(\ell,d)} = \frac{Z_{(\ell,d)}}{S_\ell} \bar{M}_\ell$ for each $d$.
It uses the single destination dispatch subroutine for each $d$ to allocate the $\bar{R}_d$ dispatches to the $\bar{M}_{(\ell,d)}$ drivers.
The first stage produces $\bar{G}_{d}^{(1)}=\bar{G}(\bar{R}_d,\bar{M}_{(\ell,d)}, x_{(\ell,d)})$ accepted dispatches towards each destination $d$, and $\bar{U}_d = \bar{U}(\bar{R}_d,\bar{M}_{(\ell,d)}, x_{(\ell,d)})$ drivers
remain unallocated from each group $d$.
\item In the second stage the initial volume of drivers who were not allocated in the first stage is equal to $\bar{U}^{(0)} = \sum_{d\in\calL}\bar{U}_d$. 
The matching process orders the locations $d_1,\dots,d_L$ and it goes through the destinations and uses the single destination dispatch subroutine to 
allocate all remaining dispatches to the pool of unallocated drivers, until either all demand has been served or all drivers have been allocated. 
Specifically, for each destination $i=1,2,\dots,L$, it runs the single destination dispatch subroutine on $\bar{R}_d - \bar{G}_d^{(1)}$, $\bar{U}^{(i-1)}$, $x_{(\ell,d_i)}$,
and it records $\bar{G}^{(2)} = \bar{G}(\bar{R}_d - \bar{G}_d^{(1)}, \bar{U}^{(i-1)}, x_{(\ell,d_i)})$ accepted dispatches and 
$\bar{U}^{(i)} = \bar{U}(\bar{R}_d - \bar{G}_d^{(1)}, \bar{U}^{(i-1)}, x_{(\ell,d_i)})$ remaining \pfdelete{pool of} unallocated drivers.
\end{enumerate}

The matching process in the two level model takes as input a market state $(\omega_t,\bfS_t)$, a number of requests $R_d$ for each destination $d$,
a location $\ell$, a number of drivers $M_\ell$, and a disutility threshold vector $\bfx_i = (x_d^i : d\in\calL)$ for each driver $i=1,2,\dots,M_\ell$.  It proceeds in two stages:
\begin{enumerate}
\item In the first stage, it partitions the $M_\ell$ drivers into groups of size $M_{(\ell,d)}$, which are either rounded up or down from $\frac{Z_{(\ell,d)}}{S_\ell} M_\ell$, for each $d$.
The allocation of drivers to groups happens uniformly at random.
It uses the single destination dispatch subroutine for each $d$ to allocate the $R_d$ dispatches to the $M_{(\ell,d)}$ drivers.
The first stage produces $G_{d}^{(1)}$ accepted dispatches towards each destination $d$, and $U_d$ is the set of driver indices which
remain unallocated from each group $d$.
\item In the second stage the initial volume of drivers who were not allocated in the first stage is equal to $U^{(0)} = \cup_{d\in\calL}U_d$. 
The matching process orders the locations $d_1,\dots,d_L$ and it goes through the destinations and uses the single destination dispatch subroutine to 
allocates all remaining dispatches to the pool of unallocated drivers, until either all demand has been served or all drivers have been allocated. 
Specifically, for each destination $i=1,2,\dots,L$, it runs the single destination dispatch subroutine on $R_d - G_d^{(1)}$, $U^{(i-1)}$, and the thresholds $x_d^i$ for $i\in U^{(i-1)}$.
The output is $G^{(2)}$ accepted dispatches and $U^{(i)}$ is the index set of unallocated drivers.
\end{enumerate}
\end{definition}
\begin{lemma}
Let $(\omega_t,\bfS_t)$ be any market state, let $(\bff^*,\bfg^*)\in F^*_{\omega_t}(\bfS_t)$ be a solution to the fluid optimization problem.  Consider
the fluid matching process which allocates all dispatch demand $(g^*_{(\ell,d)} : d\in\calL)$ originating from $\ell$ to all $S_\ell$ drivers positioned at $\ell$. 
Let $\bfx = (x_{(\ell,d)} : d\in\calL)$ be any disutility threshold vector used by drivers at $\ell$ and let $\bfg_\ell = (g_{(\ell,d)} : d\in\calL)$ be the
output of the matching process.  Then $g_{(\ell,d)} < g^*_{(\ell,d)}$ implies $x_{(\ell,d)} < x^*_\ell$.
\end{lemma}

\subsection{Matching Process Concentration Properties}
In this section of the appendix we establish that the matching process satisfies good concentration properties as the
population size goes to infinity.
For ease of use in the analysis of our main algorithm, we establish concentration inequalities that hold uniformly across
all relevant market states.
Informally, the relevant market states are those in which approximately every agent uses approximately the same add-passenger disutilities.  In addition, we
require that the total number of drivers is no larger than a multiple of the population-size parameter.

To describe the relevant market states to which our concentration inequalities apply, fix a time period $t$ and scenario $\omega_t$.
Let $\calM$ be an index set of all active drivers in the marketplace, and let $\calM_\ell$ be the subset of drivers who are positioned
at each location $\ell$.
For a driver $i\in\calM_\ell$, we use $\bfx_i=(x^i_{(\ell,d)}:d\in\calL)$ \pfcomment{notation conflict with below} to denote the add-passenger disutility threshold vector selected by driver $i$,
and we use $r_i\in\calL$ to denote the relocation destination selected by driver $i$, which is the destination towards which driver $i$ will
drive empty if they do not accept a dispatch trip.
We use the term \emph{driver-state} to mean the specification of add-passenger disutility threshold vector $\bfx_i$ and relocation destination $r_i$
for each driver $i\in\calM_\ell$ for each location $\ell$.
At a location $\ell$, we will use $\bfr_\ell = (r_{(\ell,d)} : d\in\calL)$ to mean the distribution of relocation destinations used by drivers positioned at
$\ell$.  Each component $r_{(\ell,d)}$ is the probability a randomly selected driver from $\ell$ would have selected $d$ as their relocation destination:
$$
r_{(\ell,d)} = \frac{\sum_{i\in\calM_\ell} \bfone\left\{r_i = d\right\}}{|\calM_\ell|}.
$$
To simplify notation we will use $\calM$ to refer to the set of driver indices, as well as their choice of disutility threshold and relocation destinations.

For each location $\ell$ let $\bfx_\ell = (x_{(\ell,d)} : d\in\calL)$ denote a common add-passenger threshold vector, potentially used
by drivers at $\ell$.  For an error term $\epsilon > 0$, define
$$
\calM_\ell(\epsilon,\bfx_\ell) = \{i\in\calM_\ell : \|\bfx_i - \bfx_\ell\|_\infty < \epsilon\}
$$
to be the subset of drivers positioned at $\ell$ whose threshold vector is no more than $\epsilon$ away from $\bfx_\ell$ in any component.
Let $\bfx = (\bfx_\ell : \ell\in\calL)$ denote a common  threshold vector for each location.  For $\epsilon > 0$, define
$$
\calM(\epsilon,\bfx) = \bigcup_{\ell\in\calL}\calM_\ell(\epsilon,\bfx_\ell).
$$

\begin{definition}
    \label{def:permissible_driver_state}
    Let $\gamma$ be any constant, and let $(\epsilon_k:k\geq 1)$ and $(\delta_k:k\geq 1)$ be nonnegative sequences 
    which converge to $0$ as $k\to\infty$. For any population-size $k\geq 1$ and any driver state $\calM$, we say
    $\calM$ is permissible with respect to $\gamma,\epsilon_k,\delta_k$ if the following conditions are satisfied:
    \begin{enumerate}
        \item The total number of drivers is no larger than $\gamma k$, i.e $|\calM|\leq \gamma k$.
        \item There exists a common disutility threshold vector $\bfx = (\bfx_\ell : \ell\in\calL)$ such that the number of drivers who
            use a disutility threshold vector that is further than $\epsilon_k$ from $\bfx$ is vanishingly small, relative to $k$:
            $$\frac{|\calM\setminus\calM(\epsilon_k,\bfx)|}{k}\leq \delta_k.$$
    \end{enumerate}
\end{definition}

The concentration inequalities we provide in this section show that,  when the matching process is applied to a
driver state that is permissible with respect to $\gamma,\epsilon_k,\delta_k$, then with high probability, the difference between the stochastic
output of the matching process and the corresponding fluid output is small.

We now describe what we mean by the fluid outcome associated with a particular driver state.
Given population-size parameter $k\geq 1$, a driver state $\calM$ determines the supply-location vector $\bfS_t$ by, in each component $\ell$,
taking the ratio between the total number of drivers at $\ell$ and $k$:
$$
S_\ell = \frac{|\calM_\ell|}{k}.
$$
The market state $(\omega_t,\bfS_t)$ then determines the prices, $P_{(\ell,d)}$ for $(\ell,d)\in\calL^2$, set by the SSP mechanism.
The prices then determine the expected number of riders who request a dispatch:
$$
\bar{R}_{(\ell,d)} = \bbE[D^k_{(\ell,d)}](1- F_{(\ell,d)}(P_{(\ell,d)})),
$$
where $D^k_{(\ell,d)}$ is the (stochastic) number of riders who are potentially interested in a dispatch from $\ell$ to $d$.

\begin{definition}
    \label{def:mp_fluid_outcome}
    Fix a population-size parameter $k$ and let $\calM$ be a driver-state that is permissible with respect to parameters $(\gamma,\epsilon_k,\delta_k)$.
    Let $\bfx$ be the common disutility threshold vector used by approximately all drivers in $\calM$ (which exists from the second condition in the definition
    of permissible driver state, Definition \ref{def:permissible_driver_state}).  Let $\bfr = (\bfr_\ell : \ell\in\calL)$ denote the relocation distributions
    used by the population of drivers across each location.

    The fluid outcome associated with $\calM$, $k$, are, for each route $(\ell,d)$, 
    the dispatch trip volumes $\bar{G}_{(\ell,d)}$ and total trip volumes $\bar{F}_{(\ell,d)}$, which
    result from using the fluid matching process to allocate the dispatch demand volumes $\bar{R}_{(\ell,d)}$, along each route $(\ell,d)$,
    assuming $|\calM_\ell|$ drivers are positioned
    at each location $\ell$, and the drivers at $\ell$ use disutility threshold vector $\bfx_\ell$ and relocation-trip distribution $\bfr_\ell$.

    We also define $\bar{f}_{(\ell,d)} = \frac{\bar{F}_{(\ell,d)}}{k}$ and $\bar{g}_{(\ell,d)} = \frac{\bar{G}_{(\ell,d)}}{k}$ to be
    the fluid outcome, normalized by the population size $k$.
\end{definition}

To summarize, a permissible driver state $\calM$ and a population-size parameter $k$ induce both a deterministic fluid trip specification,
denoted by dispatch trip volumes $\bar{G}_{(\ell,d)}$ and total trip volumes $\bar{F}_{(\ell,d)}$, for each route $(\ell,d)$,
and stochastic trip specifications, denoted by dispatch trip volumes $G_{(\ell,d)}$ and total trip volumes $F_{(\ell,d)}$, for each $(\ell,d)$.
The stochastic procedure governing $G_{(\ell,d)}$ and $F_{(\ell,d)}$ is described in Algorithm \ref{alg:matching_process}, and the
deterministic procedure governing $\bar{G}_{(\ell,d)}$ and $\bar{F}_{(\ell,d)}$ is described in Definition \ref{def:matching_process}.
We will use $\bfF$, $\bfG$, $\bar{\bfF}$, $\bar{\bfG}$, to mean the corresponding vectors of trip counts (the vectors
are indexed by routes $(\ell,d)\in\calL^2$).

Also, define $H_{(\ell,d)} = F_{(\ell,d)} - G_{(\ell,d)}$ to mean the (stochastic) total number of relocation trips along $(\ell,d)$, and
define $\bar{H}_{(\ell,d)} = \bar{F}_{(\ell,d)} - \bar{G}_{(\ell,d)}$ to mean the deterministic fluid number of relocation trips along $(\ell,d)$.
Let $\bfH$ be the vector with components $H_{(\ell,d)}$ for each $(\ell,d)\in\calL^2$ and let $\bar{\bfH}$ be the vector with components
$\bar{H}_{(\ell,d)}$ for each $(\ell,d)\in\calL^2$.

For each $k\geq 1$, let $\calD_k$ be the set of driver states $\calM$ that are permissible with respect to
$(\gamma,\epsilon_k,\delta_k)$ when the population-size parameter is $k$.
The main concentration lemma that we prove in this section is stated below:
\begin{lemma}
    \label{lem:asymptotic_concentration}
    There exist nonnegative sequences $(\alpha_k : k\geq 1)$ and $(q_k : k\geq 1)$, both converging to $0$ as $k\to\infty$,
    such that the following equation is true for every $k$:
    \begin{equation}
        \label{eq:asymptotic_concentration}
        \sup_{\calM\in\calD_k}
        \bbP\left(\frac{1}{k}
            \left(\|\bfH - \bar{\bfH}\|_1 + \|\bfG - \bar{\bfG}\|_1\right)\leq \alpha_k
            \right) \geq 1-q_k.
    \end{equation}
    In the above equation, it is understood that the trip specifications $\bfH$, $\bfG$, $\bar{\bfH}$, $\bar{\bfG}$ 
    are those which arise from the driver state $\calM$ and the population-size parameter $k$.
\end{lemma}
\begin{proof}
    We give a brief summary of the proof of Lemma \ref{lem:asymptotic_concentration}, the details of which are contained in the Lemmas below.

    Observe that the outcome of the matching process, i.e. the vector $\bfG$, is the sum of two vectors $\bfG=\bfG_1+\bfG_2$ where $\bfG_1$ encodes
    the output from the first stage of the matching process and $\bfG_2$ encodes the output from the second stage.
    We analyze the convergence of $\bfG_1$ and $\bfG_2$ separately.

    Lemma \ref{lem:single_destination_concentration} provides an asymptotic concentration result for the output of the single-destination
    dispatch procedure with deterministic inputs, which we use to deduce that $\bfG_1$ converges asymptotically to $\bar{\bfG}_1$. 
    Next, Lemma \ref{lem:single_dest_concentration_perturbed} provides asymptotic concentration for the single-destination
    dispatch procedure with inputs that have small stochastic perturbations.  These small stochastic perturbations correspond to the second
    stage of the matching process, and are used to show that $\bfG_2$ converges to $\bar{\bfG}_2$. 

    Finally, Lemma \ref{lem:relocation_trip_concentration} shows that the remaining undispatched drivers, i.e. the trips encoded by $\bfH$, converge
    to the deterministic fluid approximation $\bar{\bfH}$.
\end{proof}

Before proving Lemma \ref{lem:asymptotic_concentration}, we provide a number of Lemmas that help us analyze the different
components of the matching process.
The following Lemma follows from standard concenteration inequalities for sub-Gaussian random variables. \mccomment{insert reference}
\begin{lemma}
\label{lem:binomial_concentration}
Let $\calI$ be an arbitrary index set and let $\gamma > 0$ be a constant. For each $k\geq 1$ and $i\in\calI$, let $X_{k,i}$ be a Binomial random variable 
and let $R_{k,i}$ be a constant no larger than $\gamma k$.  Let $Z_{k,i}= \min(R_{k,i}, X_{k,i})$ and let $\bar{Z}_{k,i} = \min(R_{k,i}, \bbE[X_{k,i}])$.
Then there exists concentration functions $\epsilon(k)$, $q(k)$ such that
\begin{equation}
\label{eq:binomial_concentration}
\sup_{i\in\calI} \bbP\left(|Z_{k,i} - \bar{Z}_{k,i}| \geq \epsilon(k)\right) \leq q(k).
\end{equation}
\end{lemma}

\begin{lemma}
\label{lem:neg_bin_concentration}
Moreover, for each $k$ and $i$ let $Y_{k,i}$ have a negative binomial distribution, let $M_{k,i}$ be a constant no larger than $\gamma k$. 
Define $Z_{k,i} = \min(Y_{k,i},M_{k,i})$ and $\bar{Z}_{k,i} = \min(\bbE[Y_{k,i}], M_{k,i})$.
Then there exist concentration functions $\epsilon(k)$, $q(k)$ such that
\begin{equation}
\label{eq:neg_bin_concentration}
\sup_{i\in\calI} \bbP\left(|Z_{k,i} - \bar{Z}_{k,i}| \geq \epsilon(k)\right) \leq q(k).
\end{equation}
\end{lemma}

Next, we analyze asymptotic convergence of the single destination dispatch
subroutine.  

For a population-size parameter value $k$, consider the single destination dispatch subroutine with $R$ riders and $M$ drivers, 
both of which are smaller than $\gamma k$. 
Assume that, except for a subset of size at most $k\delta_k$ drivers,
 each driver has probability of accepting a dispatch no more than $\epsilon_k$ away from some constant $p$.

Let $G$ and $U$ be random variables counting the number of accepted dispatches, and the number of undispatched drivers,
respectively.  Let $\bar{G}$ and $\bar{U}$ be the volume of accepted dispatches and undispatched drivers from the fluid
matching process, with $R$ riders, $M$ drivers, and acceptance probability $p$.

\begin{lemma}
    \label{lem:single_destination_concentration}
    There exists sequences of nonnegative numbers, $(\alpha_k : k\geq 1)$ and $(q_k:k\geq 1)$, both of which converge to $0$
    as $k\to\infty$, such that the following statement is true for every $k$:
    \begin{eqnarray}
        \sup_{M\leq \gamma k,\ R\leq \gamma k,\ p\in [0,1]}\bbP\left(\frac{1}{k}|G - \bar{G}| \leq \alpha_k\right) \geq 1-q_k\label{eq:single_destination_concentration_G}\\
        \sup_{M\leq \gamma k,\ R\leq \gamma k,\ p\in [0,1]}\bbP\left(\frac{1}{k}|U - \bar{U}| \leq \alpha_k\right) \geq 1-q_k.\label{eq:single_destination_concentration_U}
    \end{eqnarray}
In the above equations, it is understood that $G$ and $\bar{G}$ are the stochastic and fluid number of accepted dispatches from the
    single-destination dispatch subroutine with $M$ drivers, $R$ riders, and, except for a subset of size at most $k\delta_k$ drivers,
    drivers use an acceptance probability
    within $\epsilon_k$ of $p$. Similarly, it is understood that $U$ and $\bar{U}$ are the stochastic and fluid number of remaining undispatched drivers.
    \end{lemma}
We defer the proof of Lemma \ref{lem:single_destination_concentration} to Appendix \ref{subsec:single_destination_concentration_proof}.

Our next Lemma proves asymptotic convergence when the single-destination dispatch subroutine is called twice, where the number of riders and drivers remaining
unmatched in the first call are used as input for the second call to the procedure.
We consider a situation where the input parameters for the single-destination dispatch subroutine are stochastically perturbed by a random variable
which satisfies asymptotic concentration properties.
We show the conclusion of Lemma \ref{lem:single_destination_concentration} still hold despite this stochastic perturbation.

Specifically, for each $k\geq 1$, let $\calI_k$ be a set of tuples of random variables $(X,Y)\in\calI_k$, with deterministic fluid approximations
$(\bar{X},\bar{Y})$, such that the following concentration property is satisfied:
\begin{eqnarray}
    \sup_{(X,Y)\in\calI_k}\bbP\left(\frac{1}{k}\left(|X-\bar{X}| + |Y - \bar{Y}|\right) \leq \alpha_k\right) \geq 1-q_k,\label{eq:asymptotic_concentration_XY}\\
\end{eqnarray}
where $(\alpha_k:k\geq 1)$ and $(q_k : k\geq 1)$ are nonnegative sequences which converge to $0$ as $k\to\infty$.

Our next Lemma analyzes convergence of the single-destination dispatch subroutine when the initial number of drivers and riders are perturbed by subtracting
$X$ and $Y$.  
Let $G$ and $U$ be the stochastic output from when the single-destination dispatch subroutine when $M-X$ is the initial number of drivers
and $R-Y$ is the initial number of drivers. Let $\bar{G}$ and $\bar{U}$ be the fluid number of drivers when the initial driver volume is $M-\bar{X}$
and the initial rider volume is $R-\bar{Y}$.
We show that $G$ and $U$ converge asymptotically to $\bar{G}$ and $\bar{U}$, assuming that $X$ and $Y$ satisfy the concentration property
(\ref{eq:asymptotic_concentration_XY}).

\begin{lemma}
    \label{lem:single_dest_concentration_perturbed}
    For each $k\geq 1$, let $\calI_k$ be a set of tuples of nonnegative random variables $(X,Y)\in\calI_k$, with deterministic fluid approximations $(\bar{X},\bar{Y})$,
    which satisfy the asymptotic concentration property (\ref{eq:asymptotic_concentration_XY}).
    Let $(M,R,p)$ be any constants satisfying $M\leq \gamma k$, $R\leq \gamma k$, and $p\in [0,1]$. 
    Let $G$ be the number of dispatches and $U$ the number of remaining drivers, when the single-destination dispatch subroutine (\ref{alg:single_destination_dispatch})
    is used to allocate $R-Y$ dispatch requests to $M-X$ drivers, assuming that, except for a subset of size at most $k\delta_k$ drivers,
    drivers use an acceptance probability
    within $\epsilon_k$ of $p$. 
    Let $\bar{G}$ and $\bar{U}$ be the output of the fluid subroutine when $R-\bar{Y}$ riders are allocated to $M-\bar{X}$ drivers.
    Then $G$ converges asymptotically to $\bar{G}$ and $U$ converges asymptotically to $\bar{U}$, in the sense that the following equation
    holds for all $k\geq 1$:
    \begin{eqnarray}
        \sup_{M\leq \gamma k,\ R\leq \gamma k,\ p\in [0,1],\ (X,Y)\in\calI_k}\bbP\left(\frac{1}{k}|G - \bar{G}| \leq \beta_k\right) \geq 1-p_k
        \label{eq:single_dest_perturbed_G}\\
        \sup_{M\leq \gamma k,\ R\leq \gamma k,\ p\in [0,1],\ (X,Y)\in\calI_k}\bbP\left(\frac{1}{k}|U - \bar{U}| \leq \beta_k\right) \geq 1-p_k,
        \label{eq:single_dest_perturbed_U}
    \end{eqnarray}
    where $(\beta_k:k\geq 1)$ and $(p_k : k\geq 1)$ are nonnegative sequences which converge to $0$ as $k\to\infty$.
\end{lemma}
Our final intermediate Lemma analyzes the asymptotic convergence of the relocation trips taken by the drivers.
Recall that, in the fluid matching process, the volume of relocation trips towards each location is proportional
to the volume of drivers who selected that destination as their relocation destination.
Specifically, the fluid volume of relocation trips from $\ell$ to $d$ is determined by the following equation:
$$
\bar{H}_{(\ell,d)} = \left(M_\ell - \sum_{d'}\bar{G}_{(\ell,d')}\right)\left(\frac{\sum_{i\in\calM_\ell}\bfone\left\{r_i=d\right\}}{|\calM_\ell|}\right).
$$
The factor on the left, i.e. $M_\ell - \sum_{d'}\bar{G}_{(\ell,d')}$, counts the volume of supply that does not serve a dispatch in the fluid model (recall
$M_\ell = |\calM_\ell|$ is the unnormalized volume of drivers at $\ell$),
and the factor on the right, $\frac{\sum_{i\in\calM_\ell}\bfone\left\{r_i=d\right\}}{|\calM_\ell|}$, counts the proportion of drivers positioned at $\ell$ who
choose $d$ as their relocation destination.

The stochastic number of relocation trips along each route $(\ell,d)$ is defined as the total number of undispatched drivers positioned at $\ell$ 
who chose $d$ as their relocation destination.  For a driver state $\calM$, let $\calM^R_\ell \subseteq \calM_\ell$ be the (stochastic) subset of drivers who
take a relocation trip.  The number of relocation trips $H_{(\ell,d)}$ is defined by
$$
H_{(\ell,d)} = \sum_{i\in\calM^R_\ell}\bfone\left\{r_i=d\right\}.
$$

Let us also use the notation $\bar{H}_\ell$ to mean the total volume of fluid relocation trips, and $H_\ell$ to mean the stochastic total number
of relocation trips in the two level model:
$$
\bar{H}_\ell = \sum_d \bar{H}_{(\ell,d)}\ \ \mbox{and} \ \ H_\ell = \sum_d H_{(\ell,d)}.
$$

The following Lemma shows that the distribution of relocation trips converges asymptotically to the fluid distribution of relocation trips.

\begin{lemma}
    \label{lem:relocation_trip_concentration}
Suppose that the total number of relocation trips converges to the fluid volume of relocation trips, as $k\to\infty$, for all admissible driver states.
    That is, assume there exists sequences $(\alpha_k:k\geq 1)$ and $(q_k : k\geq 1)$ such that $\lim_{k\to\infty}\alpha_k = \lim_{k\to\infty}q_k = 0$,
    for which the following inequality holds for every $k$ and every $\ell$
    $$
    \sup_{\calM\in\calD_k}\bbP\left(\frac{1}{k}|H_\ell - \bar{H}_\ell| \leq \alpha_k\right) \geq 1-q_k.
    $$
    Then the relocation trip volumes along each individual route converge to their fluid approximations, i.e. there 
    exist sequences $(\beta_k:k\geq 1)$ and $(p_k : k\geq 1)$ such that $\lim_{k\to\infty}\beta_k = \lim_{k\to\infty}p_k = 0$,
    for which the following inequality holds for every $k$:
    $$
    \sup_{\calM\in\calD_k}\bbP\left(\frac{1}{k}\|\bfH - \bar{\bfH}\|_1 \leq \beta_k\right) \geq 1-p_k.
    $$
\end{lemma}
Note that in the above equation, $\bfH$ and $\bar{\bfH}$ are vectors with components corresponding to $H_{(\ell,d)}$ and $\bar{H}_{(\ell,d)}$
    for each route $(\ell,d)$.
    Therefore, the difference inside the probability is equal to the following:
    $$
    \|\bfH - \bar{\bfH}\|_1 = \sum_{(\ell,d)\in\calL^2} |H_{(\ell,d)} - \bar{H}_{(\ell,d)}|.
    $$
    We defer the proof of Lemma \ref{lem:relocation_trip_concentration} to Appendix \ref{subsec:relocation_trip_concentration_proof}

\subsection{Proof of Lemma \ref{lem:single_destination_concentration}}
\label{subsec:single_destination_concentration_proof}
We restate Lemma \ref{lem:single_destination_concentration} below.
\begin{lemma*}
    There exists sequences of nonnegative numbers, $(\alpha_k : k\geq 1)$ and $(q_k:k\geq 1)$, both of which converge to $0$
    as $k\to\infty$, such that the following statement is true for every $k$:
    \begin{eqnarray}
        \sup_{M\leq \gamma k,\ R\leq \gamma k,\ p\in [0,1]}\bbP\left(\frac{1}{k}|G - \bar{G}| \leq \alpha_k\right) \geq 1-q_k\label{eq:single_destination_concentration_G}\\
        \sup_{M\leq \gamma k,\ R\leq \gamma k,\ p\in [0,1]}\bbP\left(\frac{1}{k}|U - \bar{U}| \leq \alpha_k\right) \geq 1-q_k.\label{eq:single_destination_concentration_U}
    \end{eqnarray}
    In the above equations, it is understood that $G$ and $\bar{G}$ are the stochastic and fluid number of accepted dispatches from the
    single-destination dispatch subroutine with $M$ drivers, $R$ riders, and, except for a subset of size at most $k\delta_k$ drivers,
    drivers use an acceptance probability
    within $\epsilon_k$ of $p$. Similarly, it is understood that $U$ and $\bar{U}$ are the stochastic and fluid number of remaining undispatched drivers.
\end{lemma*}
\begin{proof}
    We start with a proof of equation (\ref{eq:single_destination_concentration_G}), which states that $G$ converges to $\bar{G}$ as $k\to\infty$. 

    Fix $k\geq 1$ and any $M\leq \gamma k,\ R\leq \gamma k,\ p\in [0,1]$.
    Write the number of drivers $M$ as $M=M' + M''$, such that $M'$ is the number of drivers whose acceptance probability 
    is within $\epsilon_k$ of $p$, and $M''$ is the number of drivers whose acceptance probability is further from $p$ than $\epsilon_k$.
    By assumption we have $M''\leq \delta_k k$.

    Consider the following modification of the single-destination dispatch subroutine parameters, which is designed to slightly underestimate
    the total number of dispatches produced the matching process.  Assume that the dispatches are only allocated to the $M'$ drivers whose  
    acceptance probability is within $\epsilon_k$ of $p$,  and assume that all $M'$ drivers
    exactly use acceptance threshold $p-\epsilon_k$. Let the number of riders stay $R$.
    Let $G_L$ be the number of accepted dispatches from this version of the single-destination dispatch subroutine.
    Also, define $$\bar{G}_L = \min(M'(p-\epsilon_k), R)$$ to be the fluid output from this version of the single-destination
    dispatch subroutine.

    Also consider the following modification, which is desgined to slightly overestimate the total number of dispatches.
    Assume that all the $M''$ drivers agree to serve a dispatch before the single-destination dispatch subroutine is called,
    so that the remaining number of riders is $R'=R - \min(R,M'')$ and the remaining number of drivers is $M'$.
    Also assume that all $M'$ drivers exactly use threshold value $p+\epsilon_k$.
    Let $G_U$ be the number of accepted dispatches from this process, i.e. $G_U$ is equal to $\min(R,M'')$ plus the stochastic number of dispatches
    that occur when $R'$ riders are matched to $M'$ drivers using the single-destination dispatch subroutine, assuming all $M'$ drivers have
    acceptance probability exactly equal to $p+\epsilon_k$.
    Define $$\bar{G}_U = \min(R,M'') + \min(R', M'(p+\epsilon_k))$$ to be the fluid output from this version of the dispatch
    subroutine.

    Notice that when all drivers use the same acceptance probability, the resulting number of dispatch trips is equal in distribution to the minimum
    of the number of drivers and a Binomial distribution parameterized by the number of drivers and the common acceptance probability.
    Therefore, by Lemma \ref{lem:binomial_concentration}, we have the following bounds:
    \begin{eqnarray}
        \sup_{M\leq \gamma k,\ R\leq \gamma k,\ p\in [0,1]}\bbP\left(\frac{1}{k}|G_U - \bar{G}_U| \leq \alpha'_k\right) \geq 1-q'_k\nonumber\\
        \sup_{M\leq \gamma k,\ R\leq \gamma k,\ p\in [0,1]}\bbP\left(\frac{1}{k}|G_L - \bar{G}_L| \leq \alpha'_k\right) \geq 1-q'_k,\nonumber
    \end{eqnarray}
    where $(\alpha'_k : k\geq 1)$ and $(q'_k:k\geq 1)$ are sequences that converge to $0$ as $k\to\infty$.
    Also, Notice that for any parameter values $M\leq \gamma k,\ R\leq \gamma k,\ p\in [0,1]$, we have the upper bound
    \begin{equation}
        \label{eq:uniform_ub}
    \bar{G}_U - \bar{G}_L \leq M'' + M'2\epsilon_k \leq k(\delta_k + 2\gamma\epsilon_k).
\end{equation}

    By construction, we have that $G$ stochastically dominates $G_L$, and $G_U$ stochastically dominates $G$.  That is, for any $g\geq 0$, we have the
    following:
    $$
    \bbP(G\leq g) \leq \bbP(G_L \leq g),
    $$
    and
    $$
    \bbP(G_U\leq g) \leq \bbP(G \leq g).
    $$
Therefore we obtain the following bounds, for any $\epsilon > 0$:
\begin{align}
\bbP\left(|G - \bar{G}| \geq \epsilon\right) & \leq \bbP\left(G - \bar{G} \geq \epsilon\right) + \bbP\left(\bar{G} - G \geq \epsilon\right) \nonumber\\
& \leq \bbP\left(G_U - \bar{G} \geq \epsilon\right) + \bbP\left(\bar{G} - G_L \geq \epsilon\right)\nonumber\\
& \leq \bbP\left(|G_U - \bar{G}_U| \geq \epsilon + \bar{G} - \bar{G}_{U}\right)\nonumber \\
& \ \ \ \ + \bbP\left(|\bar{G}_{L} - G_{L}| \geq  \epsilon+\bar{G}_{L} - \bar{G} \right). \label{eq:G_concentration_}
\end{align}
Finally, define
$$
    \alpha_k = \alpha'_k + (\delta_k + 2\gamma\epsilon_k).
$$
    By equation (\ref{eq:G_concentration_}) we have the bound
    \begin{align}
        \bbP\left(\frac{1}{k}|G - \bar{G}|\geq \alpha_k\right) & \leq \bbP\left(|G_U - \bar{G}_U| \geq k\alpha_k + \bar{G} - \bar{G}_{U}\right)\nonumber \\
        & \ \ \ \ + \bbP\left(|\bar{G}_{L} - G_{L}| \geq  k\alpha_k+\bar{G}_{L} - \bar{G} \right)\label{eq:G_concentration__}.
    \end{align}
    Now, from equation (\ref{eq:uniform_ub}), we have
    $$
k\alpha_k + \bar{G} - \bar{G}_{U}  \geq k \alpha_k',
    $$
    and similarly
    $$
k\alpha_k+\bar{G}_{L} - \bar{G} \geq k\alpha_k'.
    $$
    Therefore, continuing from (\ref{eq:G_concentration__}), we have
    \begin{align}
        \bbP\left(\frac{1}{k}|G - \bar{G}|\geq \alpha_k\right) & \leq \bbP\left(|G_U - \bar{G}_U| \geq k\alpha'_k\right) + 
        \bbP\left(|\bar{G}_{L} - G_{L}| \geq  k\alpha'_k\right)\nonumber\\
        & \leq q'_k + q'_k.
    \end{align}
    Taking $q_k=2q'_k$ finishes the proof.

The proof of $U$ converging to $\bar{U}$ is analogous to the above argument. We first define $U_{L}$ and $U_{U}$
to mean the random number of undispatched drivers assuming all drivers use the acceptance probability $p-\epsilon_k$ and $p+\epsilon_k$, respectively,
and we observe show that $U_{U}$ stochastically dominates $U$ which in turn stochastically dominates $U_{L}$.
Stochastic dominance lets us bound the convergence of $U$ in terms of the convergence of $U_{U}$ and $U_{L}$.
We then observe that each $U_{L}$ and $U_{U}$ is equal in distribution to the minimum of a constant and a negative binomial distribution,
so Lemma \ref{lem:neg_bin_concentration} gives us large-population convergence.

\end{proof}

\subsection{Proof of Lemma \ref{lem:single_dest_concentration_perturbed}}
\label{subsec:single_dest_concentration_perturbed}
To prove Lemma \ref{lem:single_dest_concentration_perturbed} we first state and prove the following Lemma.
\begin{lemma}
\label{lem:lipschitz_asymptotic_convergence}
For each $k\geq 1$, let $\calZ_k$ be a set of random variables $Z\in\calZ_k$ with deterministic fluid approximations $\bar{Z}$.
Assume that $Z$ converges asymptotically to $\bar{Z}$, in the sense that there exists sequences
$(\alpha_k : k\geq 1)$ and $(q_k : k\geq 1)$, both converging to $0$ as $k\to\infty$, such that
the following holds for every $k\geq 1$:
$$
\sup_{Z\in\calZ_k}\bbP\left(\frac{1}{k}|Z - \bar{Z}| \leq \alpha_k\right)\geq 1-q_k.
$$
Let $f$ be a Lipschitz continuous function with Lipschitz constant $L$.  Then $f(Z)$ converges asymptotically
to $f(\bar{Z})$, in the sense that the following equation holds:
$$
\sup_{Z\in\calZ_k}\bbP\left(\frac{1}{k}|f(Z) - f(\bar{Z})| \leq L\alpha_k\right)\geq 1-q_k.
$$
\end{lemma}
\begin{proof}
Observe that, if $\frac{1}{k}|Z - \bar{Z}| \leq \alpha_k$ is true, then we have
$$
\frac{1}{k}|f(Z) - f(\bar{Z})| \leq \frac{1}{k} L |Z - \bar{Z}| \leq L\alpha_k.
$$
Therefore we have
$$
\sup_{Z\in\calZ_k}\bbP\left(\frac{1}{k}|f(Z) - f(\bar{Z})| \leq L\alpha_k\right)\geq \sup_{Z\in\calZ_k}\bbP\left(\frac{1}{k}|Z - \bar{Z}| \leq \alpha_k\right)\geq 1-q_k,
$$
as claimed.

\end{proof}

We restate Lemma \ref{lem:single_dest_concentration_perturbed} below.
\begin{lemma*}
    For each $k\geq 1$, let $\calI_k$ be a set of tuples of nonnegative random variables $(X,Y)\in\calI_k$, with deterministic fluid approximations $(\bar{X},\bar{Y})$,
    which satisfy the asymptotic concentration property (\ref{eq:asymptotic_concentration_XY}).
    Let $(M,R,p)$ be any constants satisfying $M\leq \gamma k$, $R\leq \gamma k$, and $p\in [0,1]$. 
    Let $G$ be the number of dispatches and $U$ the number of remaining drivers, when the single-destination dispatch subroutine (\ref{alg:single_destination_dispatch})
    is used to allocate $R-Y$ dispatch requests to $M-X$ drivers, assuming that, except for a subset of size at most $k\delta_k$ drivers,
    drivers use an acceptance probability
    within $\epsilon_k$ of $p$. 
    Let $\bar{G}$ and $\bar{U}$ be the output of the fluid subroutine when $R-\bar{Y}$ riders are allocated to $M-\bar{X}$ drivers.
    Then $G$ converges asymptotically to $\bar{G}$ and $U$ converges asymptotically to $\bar{U}$, in the sense that the following equation
    holds for all $k\geq 1$:
    \begin{eqnarray}
        \sup_{M\leq \gamma k,\ R\leq \gamma k,\ p\in [0,1],\ (X,Y)\in\calI_k}\bbP\left(\frac{1}{k}|G - \bar{G}| \leq \beta_k\right) \geq 1-p_k
        \label{eq:single_dest_perturbed_G_}\\
        \sup_{M\leq \gamma k,\ R\leq \gamma k,\ p\in [0,1],\ (X,Y)\in\calI_k}\bbP\left(\frac{1}{k}|U - \bar{U}| \leq \beta_k\right) \geq 1-p_k,
        \label{eq:single_dest_perturbed_U_}
    \end{eqnarray}
    where $(\beta_k:k\geq 1)$ and $(p_k : k\geq 1)$ are nonnegative sequences which converge to $0$ as $k\to\infty$.
\end{lemma*}
\begin{proof}
Consider the single-destination dispatch subroutine when the population-size parameter is $k\geq 1$, with $M\leq \gamma k$ drivers, $R\leq \gamma k$
riders, $p\in [0,1]$ common acceptance probability, and let $(X,Y)\in\calI_k$.
Let $G'$ and $U'$ be the output of the fluid matching process with $R-Y$ riders and $M-X$ drivers.
By Lemma \ref{lem:single_destination_concentration} we have that $G$ and $U$ converge asymptotically to $G'$ and $U'$, in the sense
that there exists $(\alpha'_k : k\geq 1)$, $(q'_k : k\geq 1)$, both converging to $0$ as $k\to\infty$, such that the following
holds for every $k$:
    \begin{eqnarray*}
        \sup_{M\leq \gamma k,\ R\leq \gamma k,\ p\in [0,1],\ (X,Y)\in\calI_k}\bbP\left(\frac{1}{k}|G - G'| \leq \alpha'_k\right) \geq 1-q'_k
    \\
        \sup_{M\leq \gamma k,\ R\leq \gamma k,\ p\in [0,1],\ (X,Y)\in\calI_k}\bbP\left(\frac{1}{k}|U - U'| \leq \alpha'_k\right) \geq 1-q'_k.
    \end{eqnarray*}

We now claim that $G'$ and $U'$ converge asymptotically to $\bar{G}$ and $\bar{U}$.  Recall the output of the fluid single-destination dispatch
subroutine is defined (Definition \ref{def:single_destination_dispatch_subroutine}) 
in terms of deterministic functions $\bar{G}(\cdot)$ and $\bar{U}(\cdot)$, so we have
\begin{eqnarray*}
G' = \bar{G}(R - Y, M - X, p) & \bar{G} = \bar{G}(R - \bar{Y}, M - \bar{X}, p)\\
U' = \bar{U}(R - Y, M - X, p) & \bar{U} = \bar{U}(R - \bar{Y}, M - \bar{X}, p).
\end{eqnarray*}
Observe that the functions $\bar{G}(\cdot)$ and $\bar{U}(\cdot)$ are both Lipschitz continuous, so by Lemma \ref{lem:lipschitz_asymptotic_convergence} we have
$G'$ converges asymptotically to $\bar{G}$, i.e. the following equation holds for every $k\geq 1$
$$
\sup_{M\leq \gamma k,\ R\leq \gamma k,\ p\in [0,1],\ (X,Y)\in\calI_k}\bbP\left(\frac{1}{k}|G' - \bar{G}| \leq L\alpha_k\right) \geq 1-q_k,
$$
where $\alpha_k$ and $q_k$ are the error term and probability term from the convergence of $(X,Y)$ to $(\bar{X},\bar{Y})$ (see equation (\ref{eq:asymptotic_concentration_XY})),
and $L$ is the Lipschitz constant for the function $\bar{G}(\cdot)$.
Taking $\beta_k = \alpha'_k + L\alpha_k$ and $p_k = q'_k + q_k$ finishes the proof.
\end{proof}

\subsection{Proof of Lemma \ref{lem:relocation_trip_concentration}}
\label{subsec:relocation_trip_concentration_proof}
We restate Lemma \ref{lem:relocation_trip_concentration} below.

\begin{lemma*}
Suppose that the total number of relocation trips converges to the fluid volume of relocation trips, as $k\to\infty$, for all admissible driver states.
    That is, assume there exists sequences $(\alpha_k:k\geq 1)$ and $(q_k : k\geq 1)$ such that $\lim_{k\to\infty}\alpha_k = \lim_{k\to\infty}q_k = 0$,
    for which the following inequality holds for every $k$ and every $\ell$
    $$
    \sup_{\calM\in\calD_k}\bbP\left(\frac{1}{k}|H_\ell - \bar{H}_\ell| \leq \alpha_k\right) \geq 1-q_k.
    $$
    Then the relocation trip volumes along each individual route converge to their fluid approximations, i.e. there 
    exist sequences $(\beta_k:k\geq 1)$ and $(p_k : k\geq 1)$ such that $\lim_{k\to\infty}\beta_k = \lim_{k\to\infty}p_k = 0$,
    for which the following inequality holds for every $k$:
    $$
    \sup_{\calM\in\calD_k}\bbP\left(\frac{1}{k}\|\bfH - \bar{\bfH}\|_1 \leq \beta_k\right) \geq 1-p_k.
    $$
    
\end{lemma*}
\begin{proof}
It suffices to prove the following is true for each route $(\ell,d)$:
$$
    \sup_{\calM\in\calD_k}\bbP\left(\frac{1}{k}\|H_{(\ell,d)} - \bar{H}_{(\ell,d)}\|_1 \leq \beta_k\right) \geq 1-p_k.
$$
If the above inequality holds for each route $(\ell,d)$, then the claimed inequality follows from a union bound over all locations.

Fix $k\geq 1$, an admissible driver state $\calM\in\calD_k$, and a location $\ell$.
Let $\calM_\ell^R\subseteq\calM$ be the stochastic subset of drivers who serve a relocation trip.  By definition, we have $|\calM_\ell^R|$
is the random variable $H_\ell$, and by assumption we have that $H_\ell$ converges to a deterministic fluid approximation $\bar{H}_\ell$.

The probability any individual driver falls in this subset depends on the exact threshold vector that the driver has selected, as well as their
realized add-passenger disutility.
In particular, if two drivers use the exact same threshold vector, they have the same probability of going non-dispatched.
If two drivers use approximately the same threshold, and we condition on the event that their sampled add-passenger disutilities are 
bounded away from the region where the different thresholds would lead to different decisions, then drivers still have the
same probability of going non-dispatched.

Let $\bfx_\ell$ be the common disutility threshold vector which is approximately used by approximately all drivers at $\ell$.
Define
$$
\calX(\epsilon_k,\bfx_\ell) = [0,C]\setminus\left(\bigcup_{d\in\calL}[x_{(\ell,d)}-\epsilon_k,x_{(\ell,d)}+\epsilon_k] \right)
$$
to be the subset of feasible disutility thresholds $[0,C]$ where a small band $[x_{(\ell,d)}-\epsilon_k,x_{(\ell,d)}+\epsilon_k]$ centered at each
threshold $x_{(\ell,d)}$ is removed.

Define
$$
\calM_\ell' = \calM_\ell(\epsilon_k,\bfx_\ell)\cap\{i\in\calM_\ell : X_i\in\calX(\epsilon_k,\bfx_\ell)\}
$$
to be the subset of drivers who approximately use the threshold $\bfx_\ell$ and whose sampled disutilities lie in $\calX(\epsilon_k,\bfx_\ell)$.
Notice that the cardinality $|\calM_\ell'|$ has Binomial distribution with parameters $|\calM_\ell(\epsilon_k,\bfx_\ell)|$ and $1-2|\calL|\epsilon_k$,
so $|\calM_\ell'|$ converges to $|\calM_\ell|$ as $k\to\infty$.

Define
$$
H'_{(\ell,d)} = \left|\{i\in\calM_\ell : r_i=d\}\cap \calM_\ell'\cap\calM_\ell^R\right|.
$$
Notice that the number of trips from $\ell$ to $d$ is bounded by
$$
 |\calM_\ell\setminus\calM_\ell'| + H'_{(\ell,d)} \geq  H_{(\ell,d)} \geq H'_{(\ell,d)}.
$$

Finally, notice that the distribution of $H'_{(\ell,d)}$ is equivalent to sampling $|\calM_\ell^R|$ balls, without replacement,
from a bag with $|\calM'_\ell|$ balls, where each ball is associated with a destination, and counting how many balls are associated with the destination $d$.
Concentration inequalities for sampling without replacement \mccomment{cite} show that $H'_{(\ell,d)}$ converges asymptotically to
$
\bar{H}_{(\ell,d)},
$
and this is sufficient to prove that $H_{(\ell,d)}$ converges to $\bar{H}_{(\ell,d)}$ asymptotically, because the difference between $H'_{(\ell,d)}$
and $H_{(\ell,d)}$ vanishes asymptotically.

\mccomment{add citation to \url{https://projecteuclid.org/journals/annals-of-statistics/volume-2/issue-1/Probability-Inequalities-for-the-Sum-in-Sampling-without-Replacement/10.1214/aos/1176342611.full}}
\end{proof}

\Xomit{
\subsection{Analysis of the Two Level Model Matching Process}

Recall $\calS(\gamma) = \left\{(\omega_t,\bfS_t)\in\calS : \sum_{\ell\in\calL}S_\ell \leq \gamma\right\}$ is the set of market states
where
the total volume of all active drivers does not exceed the constant $\gamma$.
We will define a number of variables in this section which depend on the realized market state $(\omega_t,\bfS_t)\in\calS(\gamma)$, but for simplicity
our notation will suppress this dependence.

Let $(\bff^*,\bfg^*)$ be an optimal solution for the fluid optimization problem with respect to $(\omega_t,\bfS_t)$.  Consider the matching
process at a location $\ell$ and assume there are $M_\ell$ active drivers at $\ell$. Recall the first stage of the matching process partitions
drivers into partitions associated with each possible destination, where the optimal solution $(\bff^*,\bfg^*)$ is used to choose
the size of each partition.  In particular, the proportion of the number of drivers $M_{(\ell,d)}$ in the partition associated with a destination $d$
to the total number of drivers $M_\ell$ is approximately equal to the fraction $Z(g^*_{(\ell,d)},x^*_\ell) / S_\ell$, which is the expected volume
of dispatches that need to be allocated towards $d$ in order to see $g^*_{(\ell,d)}$ acceptances, assuming drivers use the optimal threshold $x^*_\ell$.
In particular, the partition sizes $M_{(\ell,d)}$ satisfy $\sum_{d\in\calL}M_{(\ell,d)} = M_\ell$ and $|M_{(\ell,d)} - M_\ell Z(g^*_{(\ell,d)},x^*_\ell) / S_\ell| \leq 1$.

The first stage of the matching process randomly assigns drivers to each partition, and then runs the single destination dispatch subroutine for each destination, trying to
dispatch $R_{(\ell,d)}$ requests to a pool of $M_{(\ell,d)}$ drivers.
Let $G_{(\ell,d)}$ and $U_{(\ell,d)}$ be the random number of accepted dispatches and unallocated drivers, respectively, from the single destination dispatch subroutine for 
the destination $d$. Let $\bar{G}_{(\ell,d)}$ and $\bar{U}_{(\ell,d)}$ be the fluid volume of accepted dispatches and unallocated drivers.

Note that $G_{(\ell,d)}$ and $U_{(\ell,d)}$ are random variables which depend on the population size parameter $k$ and a market state $(\omega_t,\bfS_t)\in\calS(\gamma)$,
but we omit this dependence for clarity.
The following Lemma provides our concentration result for the first stage of the matching process.

\begin{lemma}
\label{lem:first_stage_concentration}
Suppose the drivers use a strategy profile such that the disutility threshold vector selected by each driver is at most $\delta(k)$ from some common
disutility threshold vector, where $\delta:\bbN\to\bbR_{\geq 0}$ is an error function converging to $0$ as $k\to\infty$.
Then $G_{(\ell,d)}$ converges towards $\bar{G}_{(\ell,d)}$ and $U_{(\ell,d)}$ converges towards $\bar{U}_{(\ell,d)}$ in the large population limit, both over the state space $\calS(\gamma)$.
\end{lemma}
\begin{proof}
This claim is a direct consequence of Lemma \ref{lem:single_destination_concentration}.  Our assumption that drivers all approximately use the same disutility threshold translates 
to each driver's acceptance probability being within $\delta(k)/C$ of a common probability, where $C$ is the maximum add-passenger disutility.

Our assumption that $(\omega_t,\bfS_t)\in\calS(\gamma)$ means that $M_\ell =S_\ell k\leq \gamma k$, so the condition $M_{(\ell,d)} \leq \gamma k$ is satisfied.

Finally, we have to show that $R_{(\ell,d)}\leq \gamma k$ is satisfied. Recall that $R_{(\ell,d)}$ follows a Binomial distribution with parameters $D_{(\ell,d)}$ and
a probability $p$ which depends on the price. From Assumption \ref{assn:entry_distributions} we have that $D_{(\ell,d)}$ converges to its expectation in the large
population limit. Let $\epsilon(k)$, $q(k)$ be the concentration functions associated with this convergence. We also know the expectation $\bbE[D_{(\ell,d)}]$ grows linearly in $k$.
Assuming $\gamma$ is larger than $\bbE[D_{(\ell,d)}]/k$, for large enough $k$ we have $\epsilon(k) \leq \gamma k - \bbE[D_{(\ell,d)}]$, so
$$
\bbP(D_{(\ell,d)} \geq \gamma k) \leq \bbP(|D_{(\ell,d)} - \bbE[D_{(\ell,d)}]| \geq \epsilon(k)) \leq q(k).
$$
Therefore $R_{(\ell,d)} \leq D_{(\ell,d)} \leq \gamma k$ is satisfied with a probability approaching $1$ as $k\to\infty$, so invoking
Lemma \ref{lem:single_destination_concentration} concludes our proof.

\end{proof}

The second stage of the matching process depends on how many drivers remain unallocated and how many riders remain undispatched after the end of the first stage.
The following Lemma states that we still get convergence in the large population limit despite this dependence.
\begin{lemma}
\label{lem:second_stage_concentration}
For each $k\geq 1$ and $s\in \calS(\gamma)$, let $R_{k,s}$ and $M_{k,s}$ be random variables with mean $\bar{R}_{k,s}$ and $\bar{M}_{k,s}$, and assume
$R_k$ (resp. $M_k$) converges to $\bar{R}_k$ (resp. $\bar{M}_k$) in the large population limit, uniformly across $\calS(\gamma)$.
Consider the single destination dispatch procedure with $R_{k,s}$ riders, $M_{k,s}$ drivers, and assume all drivers use an acceptance probability that
is within $\delta(k)$ of a common acceptance probability, where $\delta(k)\to 0$ as $k\to \infty$.
Let $G_{k,s}$ and $U_{k,s}$ be the resulting number of accepted dispatches and unallocated drivers, respectively.  Let $\bar{G}_{k,s}$ and $\bar{U}_{k,s}$
be the fluid volume of accepted dispatches and unallocated drivers.  Then $G_k$ converges to $\bar{G}_k$ and $U_k$ converges to $\bar{U}_k$ in the large population limit,
uniformly across $\calS(\gamma)$.
\end{lemma}
\begin{proof}
Let $\epsilon(k)$, $q(k)$ be concentration functions which certify the convergence of $R_k$ to $\bar{R}_k$ and $M_k$ to $\bar{M}_k$.
Let $\calI_k$ be an index set which contains tuples $(s,R,M)$, where $s\in\calS(\gamma)$ is any state, $R\in [\bar{R}_{k,s} - \epsilon(k), \bar{R}_{k,s}+\epsilon(k)]$,
and $M\in [\bar{M}_{k,s} - \epsilon(k), \bar{M}_{k,s}+\epsilon(k)]$.
Then Lemma \ref{lem:single_destination_concentration} gives the convergence of $G_k$ to $\bar{G}_k$ and $U_k$ to $\bar{U}_k$ over $\calI_k$.
This is sufficient to conclude that $G_k$ and $U_k$ converge over $\calS(\gamma)$, since the probability of $(s,R_{k,s},M_{k,s})$ belonging to $\calI_k$
is at most $1-2q(k)$, which converges to $1$ as $k\to\infty$.
\end{proof}

The following Lemma provides our ultimate concentration bound for the two level matching process.
\begin{lemma}
\label{lem:matching_process_concentration}
Consider the number of dispatches that are allocated towards a destination $d$ in the two level matching process.
Assume drivers use an acceptance threshold that
is within $\delta(k)$ of a common acceptance threshold, where $\delta(k)\to 0$ as $k\to \infty$.

 Let $G^{(1)}_{(\ell,d)}$ be the number of accepted dispatches in the first stage, let $G^{(2)}_{(\ell,d)}$ be the number of accepted dispatches in the second
stage, and let $G_{(\ell,d)} = G^{(1)}_{(\ell,d)} + G^{(2)}_{(\ell,d)}$ be the total number of accepted dispatches.  Let $\bar{G}_{(\ell,d)}$ be the
fluid volume of accepted dispatches.  Then $G_{(\ell,d)}$ converges to $\bar{G}_{(\ell,d)}$ in the large population limit, uniformly across $\calS(\gamma)$.
\end{lemma}
\begin{proof}
The fluid volume of accepted dispatches is the sum of the fluid volume of accepted dispatches in the first stage $\bar{G}^{(1)}_{(\ell,d)}$ and the second stage $\bar{G}^{(2)}_{(\ell,d)}$.
Lemma \ref{lem:first_stage_concentration} gives that $G^{(1)}_{(\ell,d)}$ converges to $\bar{G}^{(1)}_{(\ell,d)}$.

The second stage uses the single destination dispatch subroutine to allocate all remaining dispatch demand to all remaining unallocated drivers.
The number of remaining dispatch requests towards $d$ is given by $R_{(\ell,d)} - G^{(1)}_{(\ell,d)}$, which we know converges to
$\bar{R}_{(\ell,d)} - \bar{G}^{(1)}_{(\ell,d)}$.
The number of unallocated drivers depends on the output of all previous calls to the single destination dispatch subroutine.
But Lemma \ref{lem:first_stage_concentration} tells us the number of unallocated drivers from the first stage subroutine converges to its fluid value,
and Lemma \ref{lem:second_stage_concentration} gives us the same result for the unallocated drivers after the subroutine is called in the second stage.
Therefore the number of unallocated drivers from when the second stage of the matching process allocates dispatches towards $d$ converges to its fluid value.
Lemma \ref{lem:second_stage_concentration} then gives that $G^{(2)}_{(\ell,d)}$ converges to $\bar{G}^{(2)}_{(\ell,d)}$, concluding the proof.

\end{proof}
}
\Xomit{
\subsection{Analysis of the Relocation Distribution}

In this section we analyze the destinations of drivers who don't accept a dispatch.
Consider the matching process at a location $\ell$.
First, we remark that the \emph{number} of drivers who don't serve a dispatch is readily understood.
For each $k\geq 1$ and each $s\in\calS(\gamma)$, let $G_{k,s}$ be the number of dispatches accepted by drivers at $\ell$ and let $U_{k,s}$ be the
number of drivers who don't accept a dispatch.
By Lemma \ref{lem:matching_process_concentration} we have that $G_k$ converges to the fluid dispatch volumes $\bar{G}_k$ uniformly
across $\calS(\gamma)$, and we have that $U_{k,s} = k S_\ell - G_{k,s}$, so it follows that $U_{k,s}$ converges to the fluid
volume $k S_\ell - \bar{G}_{k,s}$ uniformly across $\calS(\gamma)$.

However, we must also characterize where these undispatched drivers are driving.  For each $k\geq 1$ and $s\in\calS(\gamma)$,
let $\calM_{k,s}$ be an index set of all drivers at $\ell$, and let $\calU_{k,s}\subseteq \calM_{k,s}$ be the (random) subset of drivers
who don't serve a dispatch.  For each $i\in\calM_{k,s}$ let $r_i\in\calL$ be the relocation destination selected by each driver, and let $$E_{k,s,d} = \sum_{i\in\calM_{k,s}} \bfone\{r_i = d\}$$
be the number of drivers who would relocate to $d$ if they are not allocated a dispatch.

Also, let $\bfx_i = (x_d^i:d\in\calL)$ be the disutility threshold vector used by driver $i$,
and assume there is a common disutility threshold vector $\bfx = (x_d : d\in\calL)$ such that $\|\bfx_i - \bfx\|\leq \delta(k)$, where
$\delta(k)\to 0$ as $k\to\infty$.  Define $$\calX(\bfx,\delta(k)) = [0,C]\setminus\left(\cup_{d\in\calL} [x_d - \delta(k), x_d + \delta(k)]\right)$$
to be the subset of the disutility threshold space $[0,C]$ where a small band $[x_d - \delta(k), x_d + \delta(k)]$ centered at each threshold $x_d$ is removed.

Notice that all drivers make the same accept/reject decision as each other if they sample an add-passenger disutility that lies in $\calX(\bfx,\delta(k))$.  Indeed, from our assumption, 
the only way that two drivers could make a different accept/reject decision given the same add-passenger disutility $X$ is if $X$ lies within $\delta(k)$ of one of the thresohlds $x_d$.

Therefore, we have the property that, conditioned on the add-passenger disutility lying in $\calX(\bfx,\delta(k))$, every driver has the same probability of going undispatched.
If $X_i$ is the add-passenger disutility of each driver $i$, we have:
$$
\forall i,j\in\calM_{k,s}\ \ \bbP(i\in\calU_{k,s}\mid X_i\in\calX(\bfx,\delta(k))) = \bbP(j\in\calU_{k,s}\mid X_j\in\calX(\bfx,\delta(k))).
$$
Moreover, the event $X_i\in\calX(\bfx,\delta(k))$ follows an independent and identical bernoulli distribution for each driver $i$:
$$
\forall i \bbP(X_i\in\calX(\bfx,\delta(k))) =p_{k,s}\leq \frac{2|\calL|\delta(k)}{C}.
$$
Let $B_{k,s}$ be a random variable counting the number of drivers whose add-passenger disutility does not fall in $\calX(\bfx,\delta(k))$.
We have $B_{k,s}\sim\mathrm{Binomial}(M_{k,s},p_{k,s})$, so by Lemma \ref{lem:binomial_concentration} we have $B_{k,s}$ converges
in the large population limit to to $M_{k,s}p_{k,s}$ uniformly over $s\in\calS(\gamma)$.

Now, let $\calM'_{k,s}$ be the subset of drivers whose disutility $X_i$ does lie in $\calX(\bfx,\delta(k))$, 
for each destination
$d$ let $$E'_{k,s,d} = \sum_{i\in\calM'_{k,s}} \bfone\{r_i = d\}$$ be the total number of drivers who would relocate
towards $d$, and let
$$\calU'_{k,s}=\calU_{k,s}\cap \calM'_{k,s}$$ be the random subset of undispatched drivers whose add-passenger disutility lies in $\calX(\bfx,\delta(k))$.
Our earlier observation that drivers always make the same decision if their disutility lies in $\calX(\bfx,\delta(k))$ can be restated as follows:
$$\bbP(i\in\calU'_{k,s}\mid i\in\calM'_{k,s}) = \frac{|\calU'_{k,s}|}{|\calM'_{k,s}|}.$$
Therefore, we can think of the distribution of $\calU'_{k,s}$ conditional on  $\calM'_{k,s}$ as the result of sampling without replacement from a
bag with $E_{k,s,d}$ balls of colour $d$ for each destination $d$.

We invoke the following Lemma, which follows from standard concentration inequalities regarding sampling without replacement.
\mccomment{add citation to \url{https://projecteuclid.org/journals/annals-of-statistics/volume-2/issue-1/Probability-Inequalities-for-the-Sum-in-Sampling-without-Replacement/10.1214/aos/1176342611.full}}
\begin{lemma}
    \label{lem:sampling_sans_replacement}
    Let $\calI$ be an index set of problem instances, and for each $k\geq 1$, $i\in\calI$, let $U_{k,i}$ and $E_{k,i,d}$ for $d\in\calL$ be constants no larger than $\gamma k$.
    Consider a random procedure where we sample $U_{k,i}$ balls from a bag without replacement, where the bag contains $E_{k,i,d}$ balls of type $d$ for each $d\in\calL$.
    Let $N_{k,i,d}$ be a random variable specifying how many balls of type $d$ that we sample, and let $\bar{N}_{k,i,d}$ be the
    expected number of type $d$ balls that we sample. Write $\bfN_{k,i} = (N_{k,i,d} : d\in\calL)$ and analogously define $\bar{\bfN}_{k,i}$.

    Then there exists an error function $\epsilon(k)$ such that $\epsilon(k) \to 0$ as $k\to\infty$ and another function $q(k)\geq 0$ which converges to $0$ as $k\to\infty$
    such that the following holds for every $k\geq 1$:
    \begin{equation}
        \label{eq:sampling_sans_replacement_bound}
        \sup_{i\in\calI, d\in\calL}\bbP\left(\frac{1}{k}|N_{k,i,d} - \bar{N}_{k,i,d}| \geq \epsilon(k)\right) \leq q(k).
    \end{equation}
\end{lemma}
\begin{proof}
    We briefly outline the proof of how to establish (\ref{eq:sampling_sans_replacement_bound}).
    Classic concentration inequalities for sampling without replacement establish that $(N - \bar{N}) / \bar{N}$ converges to $0$ at
    a rate of $\sqrt{\bar{N}}$, where $N$ is any $N_{k,i,d}$.  From our assumption that all the problem inputs are bounded by $\gamma k$,
    we have that $(N - \bar{N}) / k$ is always smaller than a constant times $(N - \bar{N}) / \bar{N}$.  
    We can use these concentration inequalities to show that $(N - \bar{N}) / k$ converges to $0$ at a rate that depends on $\bar{N}$:
    this gives us a meaningful bound if $\bar{N}$ is large enough relative to $k$.  

    If $\bar{N}$ is small relative to $k$ we take a different
    approach altogether. If $\bar{N}$ is small relative to $k$ that means $\bar{N}/k$ is almost $0$, so we just have to show that $N/k$ is small with
    high probability.

    So specifically, for each $k$ we take the following steps: We use concentration inequalities to find $\epsilon_{k,1}$, $q_{k,1}$ that work for all problem instances
    where $\bar{N}_{k,i,d} \geq c k^{3/4}$ for some constant $c$, and we use much weaker tail bounds to find $\epsilon_{k,2}$, $q_{k,2}$ that work for
    problem instances where $\bar{N}_{k,i,d} < c k ^{3/4}$.  We finish the proof by taking $\epsilon_k = \max(\epsilon_{k,1}, \epsilon_{k,2})$ and $q_k = \max(q_{k,1},q_{k,2})$.
\end{proof}

\begin{lemma}
    \label{lem:relocation_concentration}
    For any $k\geq 1$ let $\Pi\in\calP^k$ be any $\alpha_k$-equilibrium.  Fix a location $\ell$.  For each market state $s\in\calS_t(\gamma)$ ,
    let $H_{k,s,d}$ be the number of relocation trips which occur towards $d$ from $\ell$ if $s$ is the market state at time $t$.
    Let $\bar{H}_{k,s,d}$ be the fluid volume of relocation trips towards $d$ for the fluid policy we associate with $\Pi$.
    Then there exists $\epsilon_k\geq 0$ and $q_k\geq 0$, both converging to $0$ as $k\to\infty$, such that, for all $k\geq 1$,
$$
    \sup_{s\in\calS_t(\gamma)}\bbP\left(\frac{1}{k}|H_{k,s,d} - \bar{H}_{k,s,d}|\geq \epsilon_k\right) \leq q_k.
$$
\end{lemma}
\begin{proof}
    We give an outline of the proof for Lemma \ref{lem:relocation_concentration}.

    First, we show that as $k\to\infty$, the probability that any problem instance $s\in\calS_t(\gamma)$ is ``well behaved'' goes to $1$.
    By well behaved we require that the total number of riders $D_{(\ell,d)}$ is close to its mean (which happens with high probability
    from Assumption \ref{assn:entry_distributions}). We also require that the number of riders whose add passenger disutility falls outside the interval $\calX(\bfx,\delta_k)$
    is nearly zero (this follows from the discussion earlier in this appendix, and simple binomial concentration properties).
    Finally, we also require that the volume of allocated dispatches is close to the fluid volume of allocated dispatches (this follows from Lemma \ref{lem:matching_process_concentration}).

    It follows that the number of drivers that are undispatched is close to the fluid volume of undispatched drivers, and also the majority of unallocated drivers (i.e. those drivers
    whose disutility do fall in to $\calX(\bfx,\delta_k)$) are sampled uniformly at random from the population, so Lemma \ref{lem:sampling_sans_replacement} can be invoked
    to obtain the desired concentration bound.
\end{proof}
}

\Xomit{
\section{Notation Reference}
In this appendix we provide a reference for the commonly used notation throughout the paper.

\begin{itemize}
\item \textbf{Stochastic Network Structure:}
\begin{itemize}
\item $t\in \{1,2,\dots,T\}$ to denotes a time period.
\item $\omega_t\in\Omega_t$ denotes a realization of stochastic scenario in time period $t$.
\item For time periods $\tau\geq t$ and a scenario $\omega_t\in\Omega_t$, the set $\Omega_\tau(\omega_t)$ denotes
the subset of time $\tau$ scenario $\Omega_\tau$ which have nonzero probability of occuring in time $\tau$
conditioned on the time $t$ scenario $\omega_t$.
\item $\calN$ denotes the set of all location-time scenarios (LT-scenarios) in the network.  
    An LT-scenario consists of a location, a time period, and a scenario.
    We use the shorthand notation $(l,\omega_t)\in\calN$ to denote an LT-scenario, where $l=(\ell,t)$
    denotes the location-time pair and $\omega_t$ denotes the scenario.  We frequently refer to LT-scenarios as locations.
\item $\calA$ denotes the set of all origin-destination-time scenarios (ODT-scenarios) in the network) in the network.
    An ODT-scenario consists of an origin location, a destination location, a time period, and a scenario.
    We use the shorthand notation $(r,\omega_t)\in\calA$ to denote an ODT-scenario, where $\omega_t$ denotes the
    scenario and $r=(o,d,t)$
    denotes the origin location $o$, the destination location $d$, and the time period $t$.  
We frequently refer to ODT-scenarios as routes.
\item For any LT-scenario $(l,\omega_t)\in\calN$:
\begin{itemize}
    \item $A^+_{\omega_t}(l)$ is the set of ODT-scenarios which originate from $(l,\omega_t)$.
    \item $A^-_{\omega_t}(l)$ is the set of ODT-scenarios which terminate at $(l,\omega_t)$.
\end{itemize}
\item $\calN_{\omega_t}^t$ and $\calA_{\omega_t}^t$ denote the set of active LT-scenarios and active ODT-scenarios
    with respect to the time $t$ and scenario $\omega_t$. An LT-scenario or ODT-scenario is considered active if
    its time and scenario components equal $t$ and $\omega_t$.  Formally:
\begin{align*}
\calN_{\omega_t}^t &= \left\{(l,\tau,\omega_\tau)\in\calN : \tau = t \ \mbox{and} \ \omega_\tau=\omega_t\right\}, \\ 
\calA_{\omega_t}^t &= \left\{(o,d,\tau,\omega_\tau)\in\calN : \tau = t \ \mbox{and} \ \omega_\tau=\omega_t\right\}.
\end{align*}
\item $\calN_{\omega_t}^{\geq t}$ and $\calA_{\omega_t}^{\geq t}$ denote the set of LT-scenarios and ODT-scenarios
which have nonzero probability of becoming active, given time period $t$ and scenario $\omega_t$. An LT-scenario or ODT-scenario
has nonzero probability of becoming active if it does not occur in an earlier time period than $t$, and if the scenario 
is compatible with $\omega_t$. Formally:
\begin{align*}
\calN_{\omega_t}^{\geq t} &= \left\{(l,\tau,\omega_\tau)\in\calN : \tau \geq t \ \mbox{and} \ \omega_\tau\in\Omega_\tau(\omega_t)\right\}, \\ 
\calA_{\omega_t}^{\geq t} &= \left\{(o,d,\tau,\omega_\tau)\in\calN : \tau \geq t \ \mbox{and} \ \omega_\tau\in\Omega_\tau(\omega_t)\right\}.
\end{align*}
\end{itemize}
\item \textbf{Rider and Driver Dynamics:}
\begin{itemize}
\item $D_{(r,\omega_t)}$ is a random variable governing the amount of price-inquiry demand which occurs for the route $(r,\omega_t)$.
$G_{(r,\omega_t)}^k$ is the distribution from which $D_{(r,\omega_t)}$ is sampled in the large-market setting with population size $k$.
\item $F_{(r,\omega_t)}$ is the distribution governing the private value for a dispatch held by a rider who submits a price-inquiry request for the route $(r,\omega_t)$.
\item $d_{(r,\omega_t)}$ is a random variable indicating the total volume of dispatch requests which occur for the route $(r,\omega_t)$.
\item $M_{(l,\omega_t)}$ is a random variable governing the amount of new-driver supply which enters the market at time $t$ positioned at location $(l,\omega_t)$.
 $Q^k_{(l,\omega_t)}$ is the distribution from which $M_{(l,\omega_t)}$ is sampled in the large-market setting with size $k$.
\item $S_{(l,\omega_t)}$ is a variable specifying the total volume of available driver-supply at the LT-scenario $(l,\omega_t)$.
The value $S_{(l,\omega_t)}$ includes the new-driver volume specified by $M_{(l,\omega_t)}$ as well as the total volume of drivers who
arrive at $(l,\omega_t)$ after completing their trip from the previous time period.
\item $\bbP^k(\cdot)$ and $\bbE^k[\cdot]$ are used to indicate a probability and expected value, respectively, in the large-market setting with size $k$.
\item For a time period $t$ and scenario $\omega_t$, the supply-location vector gives supply-counts $S_{(l,\omega_t)}$ at each active location $(l,\omega_t)$:
$$\bfS_t = (S_{(l,\omega_t)} : (r,\omega_t)\in\calN_{\omega_t}^t).$$
\item For every random variable $X$, we use $\bar{X}$ to denote its mean.
\end{itemize}
\item \textbf{The Optimal Stochastic Flow:}
\begin{itemize}
\item $U_{(r,\omega_t)}(f)$ is the reward function specifying the total welfare generated by allocating $f$ units of flow
along route $(r,\omega_t)$ in the stochastic fluid model.
\item $\bff^*$ is used to denote an optimal primal solution to the dynamic opimization problem (\ref{eq:dynamic_optimization}).
\item $\bfeta^*$ is used to denote an optimal dual solution.
\end{itemize}
\item \textbf{Pricing, Matching, and Driver Utility:}
\begin{itemize}
\item $\rho_{(r,\omega_t)}$ is the price that the platform produces for dispatches on the route $(r,\omega_t)$.
\item The tuple $(r,\omega_t,\delta)$ specifies an action-suggestion that the platform can allocate to a driver.  The
first two components specify a route $(r,\omega_t)$, and the third component is an indicator $\delta\in\{0,1\}$ specifying
whether or not the action is a dispatch-trip (in which case $\delta=1$) or a relocation-trip (in which case $\delta=0$).
\item $A^+(l,\omega_t)\times\{0,1\}$ specifies the set of all possible action-suggestions that can be allocated to a driver
positioned at $(l,\omega_t)$.
\item 
$\bfg$ is the platform-produced matching vector. The components of $\bfg$ range over all possible action-suggestions under a specified
time and scenario, i.e.
$$
\bfg = (g_{(r,\omega_t,\delta)} : (r,\omega_t)\in\calA_{\omega_t}^t,\ \delta\in\{0,1\}).
$$
Each value $g_{(r,\omega_t,\delta)}$ specifies the total volume of drivers to which the the platform will allocate the action-suggestion $(r,\omega_t,\delta)$.
\item $\bfg(l,\omega_t)$ denotes the distribution over action-suggestions allocated to a driver positioned at $(l,\omega_t)$ under the matching vector $\bfg$.
\item $\calV_{(l,\omega_t)}^k(\bfS_t)$ is the expected utility-to-go that a driver positioned at $(l,\omega_t)$ collects in the two-level model with population-size $k$,
    as a function of the supply-location vector $\bfS_t$.  This utility-to-go depends on the pricing and matching mechanism used by the platform.
\item $\calV_{(l,\omega_t)}(\bfS_t)$ is the expected utility-to-go that a driver positioned at $(l,\omega_t)$ collects in the fluid model,
    as a function of the supply-location vector $\bfS_t$. 
\end{itemize}
\Xomit{
\item \mccomment{\textbf{Notation Discussion}}
\begin{itemize}
\item Problematic notation from Lemma 4 is the following expression:
$$\bbP^k\left(\delta\rho_{(r,\omega_t)}\neq \rho_{(r,\omega_t)}\right) 
$$
I mean for this to indicate, for a driver positioned at $(l,\omega_t)$, the probability of receiving an action-suggestion
$(r,\omega_t,\delta)$ for which the condition  $\delta\rho_{(r,\omega_t)}\neq\rho_{(r,\omega_t)}$ holds, under the following
conditions/assumptions about the state of the world:
\begin{itemize}
\item The new-driver counts $\bfM_t$ for the time period have already been sampled, so the total supply $S_{(l,\omega_t)}$ positioned at $(l,\omega_t)$ is a constant.
\item For each route $(r,\omega_t)$ the price $\rho_{(r,\omega_t)}$ has already been set by the platform 
\item The first source of randomness is the matching vector $\bfg$, which itself depends on the random number of rider-side dispatch requests on each route (i.e. depends on
$\bfd$).
\item The second source of randomness is the random action-allocation $(r,\omega_t,\delta)$, sampled from the distribution $\bfg(l,\omega_t)$ induced by $\bfg$ 
of actions allocated to drivers positioned at $(l,\omega_t)$. 
\end{itemize}
\item Maybe a better notation could be written along the lines of:
\begin{align*}
\bbP^k\left(\delta\rho_{(r,\omega_t)}\neq \rho_{(r,\omega_t)}\right) & = 
\sum_{\bfg}\bbP_{(r,\omega_t,\delta)\sim\bfg(l,\omega_t)}\left(\delta\rho_{(r,\omega_t)}\neq \rho_{(r,\omega_t)}\mid\bfg\right)\bbP^k(\bfg)\\
&=\bbE^k_{\bfg}\left[\bbP_{(r,\omega_t,\delta)\sim\bfg(l,\omega_t)}\left(\delta\rho_{(r,\omega_t)}\neq \rho_{(r,\omega_t)}\mid\bfg\right)\right]
\end{align*}
Ok this doesn't look phenomenal, but at least the two sources of randomness are explicit... Let's brainstorm how to improve?
\end{itemize}
\item Some comments:
\begin{itemize}
\item Distribution of $\bfg$ depends on network-size $k$, since distribution of $d$ depends on $k$, but once $\bfg$ is sampled
the distribution over action-allocation $(r,\omega_t,\delta)$ depends only on $\bfg(l,\omega_t)$.
\item Lemma 4 is used in the proof of theorem 2 to show that the following expression goes to $0$ as $k\to\infty$:
$$\bbE^k\left[\delta\rho_t(r,\omega_t) - \rho_t(r,\omega_t)\right].$$
\item In fact, it's worth calling to attention the following expression for expected utility-to-go:
    \mccomment{Would be really great to align on reasonable notation for this.}
    \begin{equation}
        \label{eq:utility_recursion}
    \calV^k_{(l,\omega_t)}(\bfS_t) = \bbE^k\left[\delta\rho_t(r,\omega_t) - c_{(r,\omega_t)} + \calV^k_{(l^+(r),\omega_{t+1})}(\bfS_{t+1})\right].
    \end{equation}
    The above expectation is taken over multiple sources of randomness:
    \begin{itemize}
    \item The random matching vector $\bfg$, which is produced by Algorithm \ref{alg:matching_details} and depends on the random dispatch-request vector
        $\bfd$.
    \item The random action allocated to a driver positioned at $(l,\omega_t)$, denoted by $(r,\omega_t,\delta)\sim \bfg(l,\omega_t)$.
    \item The random time $t+1$ scenario $\omega_{t+1}\sim\bbP(\cdot\mid\omega_t)$.
    \item The time $t+1$ supply-location vector $\bfS_{t+1}$, which is random because it depends on the time $t+1$ new-driver counts $\bfM_{t+1}$.
        The components of $\bfS_{t+1}$ are specified in equation (\ref{eq:stochastic_supply_loc_vec}).
    \end{itemize}
\end{itemize}
}

\end{itemize}

\section{Details from Section \ref{sec:model}}
\label{sec:model_appdx}
The following lemma states that a similar concentration property holds for the dispatch-request counts $d_{(r,\omega_t)}$
in the large-market setting, under any fixed price $\rho_{(r,\omega_t)}$ which does not depend on $D_{(r,\omega_t)}$.
\begin{lemma}
    \label{lem:dispatch_request_concentration}
    For any route $(r,\omega_t)\in\calA$ and any trip-price $\rho_{(r,\omega_t)}$ let $d_{(r,\omega_t)}$ be the random dispatch-request
    volume for $(r,\omega_t)$ under the price $\rho_{(r,\omega_t)}$.
    Note that $d_{(r,\omega_t)}$ has distribution specified by equation (\ref{eq:d_t_binomial_dist}) with expected value
    $$\bar{d}_{(r,\omega_t)} = \bar{D}_{(r,\omega_t)}(1-F_{(r,\omega_t)}(\rho_{(r,\omega_t)})).$$
    Then there exists a sequence of error terms $\epsilon^k_{(r,\omega_t)}$
    converging to $0$ as $k\to\infty$ such that the following properties hold:
    \begin{equation}
        \label{eq:large_market_dispatch_concentration}
        \bbE^k\left[|d_{(r,\omega_t)} - \bar{d}_{(r,\omega_t)}|\right] \leq \epsilon^k_{(r,\omega_t)},
    \end{equation}
    and
    \begin{equation}
        \label{eq:larg_market_dispatch_concentration_whp}
        \bbP^k(|d_{(r,\omega_t)} - \bar{d}_{(r,\omega_t)}|\geq\epsilon^k_{(r,\omega_t)})\leq\epsilon^k_{(r,\omega_t)}.
    \end{equation}
\end{lemma}
\etcomment{where is this comment going? getting deleted?}\mccomment{I had left it as a comment so I wouldn't forget to come back and edit it -- edits are now included.}
The above Lemma follows from basic concentration properties of the Binomial distribution and our concentration
assumptions on $D_{(r,\omega_t)}$ stated in Assumption \ref{assn:large_market}.  We omit the details here for succinctness.

\section{Proofs from Section \ref{sec:opt_central}}
\label{sec:opt_central_appdx}

\section{Proofs from Section \ref{sec:fluid_ic}}
\label{sec:fluid_ic_proofs}
\subsection{Proof of Theorem \ref{thm:fluid_ic}}
\begin{theorem*}
   The following statements are true in the stochastic fluid model:
    \begin{enumerate}
        \item When prices and matching decisions are produced by the static mechanism, then no individual driver
            is incentivized to deviate from their trip-allocation, assuming that no drivers have deviated from their
            trip-allocations in previous time periods.
        \item When prices and matching decisions are produced by the dynamic mechanism, then no individual driver
            is incentivized to deviate from their trip-allocation at any time period, regardless of the trips taken in earlier time periods.
    \end{enumerate}
\end{theorem*}
\begin{proof}
    Recall that $\calV_{(l,\omega_t)}(\bfS_t)$ indicates the expected utility-to-go of a driver located at $(l,\omega_t)$,
    assuming that drivers always follow the actions they are allocated.
    The proof follows by showing that this utility-to-go is exactly equal to the optimal dual variable $\eta^*_{(l,\omega_t)}$.

    Let's start with the static mechanism. Let $\bff^*=(f^*_{(r,\omega_t)} : (r,\omega_t)\in\calA)$ be the time $0$ optimal flow
    computed by the static mechanism and consider any time period $t$ and scenario $\omega_t$.
    The assumption that no drivers have deviated in earlier time periods is sufficient for the matching vector $\bfg$ computed at time $t$ 
    from the time 0 flow $\bff^*$ to be feasible with respect to the spatial distribution of drivers at time $t$.
    Thus, following any trip-suggestion $(r,\omega_t,\delta)$ produces utility
    $$
    \delta\rho_{(r,\omega_t)} - c_{(r,\omega_t)} + \bbE\left[\calV_{(l^+(r),\omega_{t+1})}(\bfS_{t+1})\right].
    $$
    In addition, the mechanism only allocates a relocation trip-suggestion (i.e. $\delta=0$) if all demand along that route
    is being served and the corresponding price is $0$, so $\delta\rho_{(r,\omega_t)}=\rho_{(r,\omega_t)}$ is always satisfied
    for trip-suggestions $(r,\omega_t,\delta)$ allocated by the platform.  Assuming, via backwards induction, that 
    $\calV_{(l^+(r),\omega_{t+1})}(\bfS_{t+1}) = \eta^*_{(l+(r),\omega_{t+1})}$, it follows by complementary slackness that
    the utility of any trip-suggestion $(r,\omega_t,\delta)$ is equal to
    $$
    \eta^*_{(l,\omega_t)} = \rho_{(r,\omega_t)} - c_{(r,\omega_t)} + \bbE\left[\eta^*_{(l^+(r),\omega_{t+1})}\right].
    $$
    Since $\eta^*_{(l,\omega_t)}$ is the utility of following any action suggestion produced by the static mechanism,
    we can take an expectation over $(r,\omega_t,\delta)\sim\bfg(l,\omega_t)$ for free and conclude
    $$
    \calV_{(l,\omega_t)}(\bfS_t) = \bbE_{(r,\omega_t,\delta)\sim\bfg(l,\omega_t)}\left[\rho_{(r,\omega_t)} - c_{(r,\omega_t)} + \bbE\left[\eta^*_{(l^+(r),\omega_{t+1})}\right]\right] = \eta^*_{(l,\omega_t)}.
    $$
    That following this trip-suggestion is incentive-compatible now follows from the optimality conditions (\ref{eq:kkt_conditions}), which assert
    that
    $$
    \eta^*_{(l,\omega_t)} \geq \rho_{(r,\omega_t)} - c_{(r,\omega_t)} + \bbE\left[\eta^*_{(l^+(r),\omega_{t+1})}\right]
    $$
    holds for any outgoing route $(r,\omega_t)\in A^+_{\omega_t}(l)$.

    The same argument establishes incentive-compatibility under the dynamic mechanism, except there is no need for the assumption that
    drivers have not deviated in early time periods because the re-computation guarantees the matching vector produced by the mechanism
    will always be feasible with respect to the observed spatial distribution of drivers.
    Moreover, Lemma \ref{lem:marginal_utility} establishes that the function $\bfS\mapsto\eta^*_{(l,\omega_{t+1})}(\bfS)$ mapping supply-location vectors
    to optimal dual variables is a continuous function, so an infinitesimal driver deviating does not affect the dual value $\eta^*_{(l^+(r),\omega_{t+1})}$,
    even under the dynamic mechanism.
\end{proof}


\subsection{Proof of Theorem \ref{thm:robust_fluid_welfare}}
\mccomment{DEPRECATED.  this appendix is subsumed by above appendix.}
\begin{theorem*} (Welfare-Robustness in the stochastic fluid model)
    Any driver strategy-profile that is incentive-compatible under prices charged by the dynamic mechanism is welfare-optimal.
\end{theorem*}
\begin{proof}
    Let $\bfg = (g_{(r,\omega_t,\delta)})$ be a driver strategy-profile that is incentive-compatible with respect to the SSP prices $\rho^*_{(r,\omega_t)}$.
    Incentive-compatibility of $\bfg$ means that, if $\bfg$ allocates nonzero drivers a relocation-trip along $(r,\omega_t)$, then all available
    price-inquiry demand for that route is being served, and so without loss of generality we can omit the relocation-trip / dispatch-trip
    distinction made by the $\delta\in\{0,1\}$ term in the index $(r,\omega_t,\delta)$, and just use $g_{(r,\omega_t)} = g_{(r,\omega_t,0)} + g_{(r,\omega_t,1)}$
    to refer to the driver flow along each active route $(r,\omega_t)$.

    Let $\Phi(\bfg)$ be the expected welfare-to-go over future time periods as a function of the current-timeperiod actions specified
    by $\bfg$.  Formally we have
    $$
    \Phi(\bfg) = \bbE\left[\calW_{\omega_{t+1}}(\bfS_{t+1})\right],
    $$
    where $\bfS_{t+1}$ is the time $t+1$ supply-location vector that arises under the actions specified by $\bfg$.
\end{proof}

\subsection{Proof of Theorem \ref{thm:robust_fluid_welfare}}
\begin{theorem*}
Let $\Pi$ be any driver strategy profile that is a subgame-perfect equilibrium with respect
    to the dynamic mechanism.
    Then $\Pi$ is welfare-maximizing, and the utility-to-go of a driver is equal to the corresponding 
    corresponding dual variable from the dynamic formulation.
\end{theorem*}
\begin{proof}
    It suffices to consider the following single-timeperiod problem. Let $n$ be the number of locations,
    let $S_i$ be the number of active drivers at each location $i$, and let $U_{ij}(x_{ij})$ be the utility generated by
    Let $x=(x_{ij} : 1\leq i,j\leq n)$ be a decision variable indicating the amount of drivers we send from location $i$
    to location $j$. Let $\Phi(x)$ be the utility-to-go in future time periods as a function of the flow $x$ in the current time
    period.
    The dynamic optimization problem (\ref{eq:dynamic_optimization}) is equivalent to
\begin{equation}
    \label{eq:original_problem}
    \sup_x \left( \sum_{i,j=1}^n U_{ij}(x_{ij}) + \Phi(x) : \mbox{ subject to } 
    \forall i S_i=\sum_{j=1}^nx_{ij},
    \ x_{ij}\geq 0 \forall i,j
\right).
\end{equation}

The dynamic mechanism computes an optimal solution $x^*$ and set the price along the route from $i$ to $j$
as
$$
    \rho_{i,j}^* = \frac{d}{dx_{ij}}U_{ij}(x_{ij}^*).
$$
The optimality conditions state there exists a dual variable $\eta_i^*$ associated with the equality constraint for each location $i$,
    and dual variables $\alpha_{ij}^*\geq 0$, associated with each inequality constraint, satisfying
$$
    \eta_i^* = \rho_{ij}^* + \frac{\partial}{\partial x_{ij}}\Phi(x^*) + \alpha_{ij}^*
\ \ 
\forall i,j
$$
as well as the complementary slackness conditions
$$
    \alpha_{ij}^*x_{ij}^* = 0.
$$

Let $d_ij$ be the
    amount of rider demand on each edge $ij$ resulting from our price $\rho_{ij}^*$.
    Note $d_{ij}$ is the number of passenger-trips that occur on edge $i$
    under $x^*$, and $\max(0,x_{ij}^* - d_{ij})$ is the number of relocation-trips that occur on edge $ij$ under $x^*$. Thus we can represent $x^*$ using the passenger-distinguishing parameterization
by defining
$$
x_{(ij,\delta)}^* = \begin{cases}
    d_{ij}, & \mbox{if }\delta = 1,\\
    \max(0,d_{ij} - x_{ij}^*), & \mbox{if }\delta = 0.
\end{cases}
$$
for each $(ij,\delta)\in [n]^2\times\{0,1\}$.

    After we set the prices $\rho^* = (\rho_{ij}^*: 1\leq i,j\leq n)$,  we find a new feasible solution $x$ which is an equilibrium under $\rho^*$.  

    \mccomment{ . . . change here . . .}

    This means there exists a number $V_i$ for which the following equation holds
$$
V = \rho_i^* + \frac{\partial}{\partial x_i}\Phi(x) + \beta_i \ \ 
\forall i,
$$
where $\beta_i$ is defined as
$$
\beta_i = \begin{cases}
0 & \mbox{ if }x_i > 0,\\
V - \rho_i^* - \frac{\partial}{\partial x_i}\Phi(x) & \mbox{ else}.
\end{cases}
$$
We also assume that $\beta_i\geq 0$ is always true.

Our incentive compatibility condition on $x$ is almost exactly the
optimality condition we had on $x^*$.  The difference is that the
$\rho_i^*$ appearing in the incentive compatibility conditions does
not depend on $x$, whereas when $\rho_i^*$ appears in the optimality
conditions it does depend on $x^*$.

If we had alternative values, $\bar{V}$ and $\bar{\beta}_i\geq 0$, satisfying
$$
\bar{V} = \frac{d}{dx_r}U_i(x_r) + \frac{\partial}{\partial x_i}\Phi(x) + \bar{\beta}_i
$$
then we could conclude $x$ was an optimal solution for our original optimization problem.  
We can almost achieve this by setting $\bar{V}=V$ and setting
$$\bar{\beta}_i=\beta_i + \frac{d}{dx_r}U_i(x_i) - \rho_i^*\ \ \forall i$$
but we need to argue this does not 
violate the $\bar{\beta}_i\geq 0$ condition, or the complementary slackness condition.

\medskip
\noindent \textbf{Capacitated Linearized Problem.}
Let's introduce a parameterization of our edges that distinguishes
between passenger-trips and relocation-trips.  For each edge
$i=1,2,\dots,n$, let $(i,1)$ indicate the copy of edge $i$ that
includes a passenger, and let $(i,0)$ indicate the copy of edge $i$
that does not include a passenger.
We can represent an optimal solution for the original optimization
problem (\ref{eq:original_problem}) in terms of this 
passenger-distinguishing parameterization as follows.  

Let $x^*=(x_1^*,\dots,x_n^*)$ be an optimal solution for
(\ref{eq:original_problem}) using the original parameterization,
let $\rho=(\rho_1,\dots,\rho_n)$ be the prices set by our mechanism, and let $d_1,\dots,d_n$ be the
amount of rider demand for each edge resulting from our prices.
Note $d_i$ is the number of passenger-trips that occur on edge $i$
under $x^*$, and $\max(0,x_i^* - d_i)$ is the number of relocation-trips that occur on edge $i$ under $x^*$. Thus we can represent $x^*$ using the passenger-distinguishing parameterization
by defining
$$
x_{(i,\delta)}^* = \begin{cases}
d_i, & \mbox{if }\delta = 1,\\
\max(0,d_i - x_i^*), & \mbox{if }\delta = 0.
\end{cases}
$$
for each $(i,\delta)\in [n]\times\{0,1\}$.

Consider the following linearized version of the optimization
problem (\ref{eq:original_problem}):
$$
\sup\left(\sum_{i=1}^n \rho_i y_{(i,1)} + \Phi(y) : \sum_{i,\delta} y_{(i,\delta)} = S,\ \forall (i,\delta) \ y_{(i,\delta)} \geq 0, \ \forall i\ y_{(i,1)}\leq x_{(i,1)}^* \right).
$$
For a feasible solution $y$ to be 
optimal we get the optimality
conditions
$$
\eta = \rho_i + 
\frac{\partial}{\partial y_{(i,1)}}\Phi(y)
+ \alpha_{(i,1)} - \gamma_i,
$$
and
$$
\eta = \frac{\partial}{\partial y_{(i,0)}}\Phi(y)
+ \alpha_{(i,0)},
$$
with complementary slackness
conditions:
$$
y_{(i,1)}\alpha_{(i,1)} = 0,\ \ 
y_{(i,0)}\alpha_{(i,0)} = 0,\ \ 
(x_{(i,1)}^* - y_{(i,1)})\gamma_i = 0.
$$
Here, $\alpha_{(i,\delta)}$ is a dual
variable associated with the $y_{(i,\delta)}\geq 0$ constraint,
$\eta$ is a dual variable associated with
the flow-conservation constraint,
$\gamma_i$ is a dual variable associated
with the capacity constraint.
It is also worth noting that the
welfare-to-go function $\Phi$ only
depends on the sum $y_{(i,0)}+y_{(i,1)}$ for each
edge $i$, so the partial derivatives $\frac{\partial}{\partial y_{(i,0)}}\Phi(y)$ and
$\frac{\partial}{\partial y_{(i,1)}}\Phi(y)$ are equal.
Note that our original optimum $x^*$ easily satisfies these conditions.

\medskip
\noindent
\textbf{Optimality of Arbitrary Equilibrium for the Capacitated Linearized Problem.}
Let $y=(y_{(i,\delta)})$ be a 
flow corresponding to an 
equilibrium among the drivers.
This means there is no profitable deviation: no driver can increase
their welfare by taking a relocation trip along any edge, or
a passenger trip along an edge with excess passenger capacity.
We will show that $y$ is optimal by deriving, from the equilibrium condition,
 dual variables for which the optimality conditions are satisfied.

Let $U_{(i,\delta)}$ be the utility generated on edge $(i,\delta)$ under
the equilibrium $y$.  The dual variables are obtained as follows:
Let $\eta$ be the minimum utility generated on any edge $(i,\delta)$ where
a nonzero number of drivers go,
let $\gamma_i=max(U_{(i,1)} - \eta, 0)$ for each $i$, and let
$\alpha_{(i,\delta)} = \max(\eta - U_{(i,\delta)}, 0)$ for each $(i,\delta)$.

The stationarity conditions  between $y$ and $(\alpha,\eta,\gamma)$
holds by definition . Complementary slackness follows for the following reasons:
\begin{itemize}
\item The utility $U_{(i,1)}$ only exceeds $\eta$ when it is a passenger-trip edge and when that edge is at capacity, which means $\gamma_i$ is only nonzero when $y{(i,1)} = x_{(i,1)}^*$. Note the utility $U_{(i,0)}$, for any edge $(i,0)$ encoding a relocation-trip, will never exceed $\eta$, otherwise some drivers would have a profitable deviation to take that edge.
\item The utility $U_{(i,\delta)}$ is smaller than $\eta$ only when no drivers
traverse $(i,\delta)$, hence $\alpha_{(i,\delta)}$ is only nonzero when 
$y_{(i,\delta)}$ is zero.
\end{itemize}
Therefore the dual variables $(\alpha,\eta,\gamma)$ certify the optimality
of $y$ for the capacitated linearized problem.

\medskip
\noindent
\textbf{Optimality of an Arbitrary Equilibrium for the Original Problem.}
Note that complementary slackness holds between any primal optimum and any dual optimum, so we know that $y$ satisfies 
the complementary slackness conditions (for the capacitated linearized problem)
with respect to dual variables $(\alpha^*,\eta^*,\gamma^*)$ associated with the original optimum $x^*$.  We will show this is enough to conclude that $y$ is optimal for the original problem.

First, we claim that all rider-demand 
arising under the prices $\rho$ must receive a dispatch under the equilibrium $y$.
To see this, note the following:
\begin{itemize}
    \item The utility $U_{(i,\delta)}$,
for any edge $(i,\delta)$ that drivers actually take under $y$, is equal to the  dual value $\eta^*$.
    \item 
    The utility $U_{(i,1)}$ for any passenger-trip edge $(i,1)$ with nonzero capacity is equal to $\eta^*$, since
    $y$ satisfies complementary slackness
    with respect to $(\alpha^*,\eta^*,\gamma^*)$ and since nonzero capacity on edge $(i,1)$ means
    $\alpha_i^*=0$.
    \item The total capacity for passenger-trip edges is $\sum_{i}x^*_{(i,1)}\leq S$. 
\end{itemize}
Since the utility collected by the driver along each passenger-trip edge $(i,1)$ with nonzero capacity is $\eta^*$, and since we assume drivers are incentivized to drive with a passenger (all else being equal), then the equilibrium $y$ must use all available passenger-trip edge capacity, otherwise some drivers would be incentivized to deviate to one of the available passenger-trip edges.

Let $y'=(y'_i : i=1,\dots,n)$ be a reparameterization of the equilibrium $y$ that drops the passenger-trip vs. relocation-trip distinction, so each
component is defined as $y'_i = y_{(i,0)} + y_{(i,1)}$.  We will show that $y'$ is optimal for the original problem by establishing the optimality conditions hold between $y'$ and $(\alpha^*,\eta^*)$.
To show that the optimality conditions hold for the original problem, we have to establish the following equality for every edge $i$: \mccomment{whoops, my $U_{(i,\delta)}$ driver-utility notation from above clashes with the $U_i(y'_i)$ utility reward function in this context.  Will fix in the main writeup.}
$$
\eta^* = \frac{d}{dy'_i}U_i(y'_i) + \Phi(y') + \alpha_i^*,
$$
and we have to show the complementary slackness conditions $y_i'\alpha_i^*=0$ hold.
From $y$ satisfying optimality conditions for the linearized problem, we already know the the complementary conditions $y_i'\alpha_i^*=0$ hold for each $i$, and we have the following equality for each $i$:
$$
\eta^* = \rho_i + \Phi(y') + \alpha_i^*.
$$

What's left is to show that $\frac{d}{dy'_i}U_i(y'_i) = \rho_i$ holds for each $i$.
This equality follows from the fact that all rider demand arising under the prices $\rho$ receives a dispatch under the equilibrium $y$.

For an edge $i$ with nonzero trip price $\rho_i>0$, note that the equilibrium $y$ will never send a relocation-trip along edge $i$, because we have already established that all trips under $y$ generate utility $\eta^*$ for the driver, and it is impossible for a passenger-trip and a relocation-trip using the same edge to generate the same utility when the price is nonzero.  Thus, the 
flow value $y_i'$ from the equilibrium solution $y'$
is the same as the flow value $x_i^*$ from the original optimum $x^*$, so 
$\frac{d}{dy'_i}U_i(y'_i) = \rho_i =\frac{d}{dx^*_i}U_i(x^*_i) $ holds in this case.

For an edge $i$ where the trip price $\rho_i$ is zero, we know that the volume of dispatches $x^*_{(i,1)}$ allocated to this edge under the original optimum is equal to the total volume of (price-inquiry) rider demand for this edge.  Since the equilibrium $y$ serves all rider demand arising under our prices we know that the flow value $y_i'$ is at least $x^*_{(i,1)}$.  We also know the
utility reward function $U_i(\cdot)$ is constant for flow values which exceed the total amount of rider demand.  Thus $\frac{d}{dy'_i}U_i(y'_i)=0=\rho_i$ is true in this case as well.

Therefore $y'$ and $(\alpha^*,\eta^*)$ together satisfy the optimality conditions for the original optimization problem, 
establishing that an arbitrary equilibrium under the SSP pricing mechanism achieves optimal welfare.

\end{proof}

\section{Proofs and Details from Section \ref{sec:stochastic_ic}}
\subsection{SSP Matching Vector Details}
\label{subsec:matching_details}
The details specifying how the SSP mechanism obtains a feasible matching vector $\bfg$ from an optimal
stochastic flow $\bff^*$ are given in Algorithm \ref{alg:matching_details}
\begin{algorithm}[t]
\SetAlgoNoLine \small
        \textbf{Input:} Time $t$, scenario $\omega_t$, 
        optimal stochastic flow $\bff^*$, supply-location vector $\bfS_t$, observed dispatch request counts $\bfd$, 
        expected dispatch request counts $\bar{\bfd}$, population-size parameter $k$.\newline
        \textbf{Output:} Matching vector $\bfg$.
        \begin{enumerate}
            \item For each active route $(r,\omega_t)\in\calA_{\omega_t}^t$, define the candidate matching vector $\bar{\bfg}$ by:
                \begin{equation}
                    \label{eq:candidate_match_vec}
                    \bar{g}_{(r,\omega_t,1)} = d_{(r,\omega_t)} \ \ \ \mbox{and} \ \ \ 
                    \bar{g}_{(r,\omega_t,0)} = \frac{1}{k}\lfloor k(f^*_{(r,\omega_t)} - \bar{d}_{(r,\omega_t)})\rfloor.
                \end{equation}
                Note that the candidate number of dispatch allocations is just equal to the total volume of realized
                dispatch requests for $(r,\omega_t)$.  The candidate number of relocation-trip suggestions is just equal to the number of
                relocation suggestions under $\bff^*$, rounded down to the nearest multiple of $1/k$.
            \item For each active location $(l,\omega_t)\in\calN_{\omega_t}^t$ count the total volume of  actions
                allocated to drivers positioned at $(l,\omega_t)$ under the candidate matching vector:
                \begin{equation}
                    L_{(l,\omega_t)} = \sum_{(r,\omega_t)\in A^+_{\omega_t}(l)}\left(\bar{g}_{(r,\omega_t,0)} + \bar{g}_{(r,\omega_t,1)}\right).
                \end{equation}
            \item If  $L_{(l,\omega_t)} \geq S_{(l,\omega_t)}$:
                \begin{itemize}
                    \item In this case we have excess demand. Our matching procedure chooses
                        to not serve some amount of excess demand.  For each outgoing route $(r,\omega_t)\in A^+_{\omega_t}(l)$ choose values $\delta_{(r,\omega_t,1)}\geq 0$
                        satisfying the following properties:
                        \begin{itemize}
                            \item Each $\delta_{(r,\omega_t,1)}$ is a multiple of $1/k$.
                            \item If $\delta_{(r,\omega_t,1)}$ is nonzero then $\bar{g}_{(r,\omega_t,1)} - \delta_{(r,\omega_t,1)} \geq \bar{d}_{(r,\omega_t)}$.
                            \item $\sum_{(r,\omega_t)\in A^+_{\omega_t}(l)} \delta_{(r,\omega_t,1)} = L_{(l,\omega_t)} - S_{(l,\omega_t)}$.
                        \end{itemize}
                    \item Define the values $\gamma_{(r,\omega_t,0)} = 0$ for all $(r,\omega_t)\in A^+_{\omega_t}(l)$.
                \end{itemize}
            \item If $L_{(l,\omega_t)} < S_{(l,\omega_t)}$:
                \begin{itemize}
                    \item In this case we have excess supply.  The matching procedure we describe here selects outgoing routes which can be allocated to 
                        excess drivers as relocation-trips. For each $(r,\omega_t)\in A^+_{\omega_t}(l)$ choose values $\gamma_{(r,\omega_t,0)}$ satisfying the following:
                        \begin{itemize}
                            \item Each $\gamma_{(r,\omega_t,0)}$ is a multiple of $1/k$.
                            \item Each $\gamma_{(r,\omega_t,0)}$ is nonzero only on routes $(r,\omega_t)$ for which $f^*_{(r,\omega_t)}$ is nonzero.
                            \item $\sum_{(r,\omega_t)\in A^+_{\omega_t}(l)} \gamma_{(r,\omega_t,0)} = S_{(l,\omega_t)} - L_{(l,\omega_t)}$.
                        \end{itemize}
                    \item Define the values $\delta_{(r,\omega_t,1)} = 0$ for all $(r,\omega_t)\in A^+_{\omega_t}(l)$.                
                \end{itemize}
            \item After steps 2 through 4 have been carried out for all active locations $(l,\omega_t)$, obtain the final matching vector
                as follows for each $(r,\omega_t)\in \calA_{\omega_t}^t$:
                \begin{equation}
                    g_{(r,\omega_t,0)} = \bar{g}_{(r,\omega_t,0)} + \gamma_{(r,\omega_t,0)} \ \ \ \mbox{and} \ \ \ 
                    g_{(r,\omega_t,1)} = \bar{g}_{(r,\omega_t,1)} - \delta_{(r,\omega_t,1)}.
                \end{equation}
            \item Return $\bfg$.
        \end{enumerate}
\caption{Details of how $\bff^*$ and $\bfd$ are converted into relocation-trip counts and passenger-trip counts $\bfg$,
    for the stochastic two-level model in the large-market setting with population-size parameter $k$. 
    }
    \label{alg:matching_details}
\end{algorithm}
\subsection{Proof of Lemma \ref{lem:immediate_reward}}
\label{subsec:immediate_reward_proof}
\begin{lemma*}
    Consider a driver positioned at an active location $(l,\omega_t)$ in the stochastic two-level model at a time period $t$ and scenario $\omega_t$. 
    Fix any supply-location vector $\bfS_t$ and let $(r,\omega_t,\delta)$ be the random trip-suggestion allocated to this driver under the SSP mechanism.
    For each value of the population-size parameter $k$, let
    $$
    p_k = \bbP^k\left(\delta\rho_{(r,\omega_t)}\neq \rho_{(r,\omega_t)}\right)
    $$
    be the probability that the driver is allocated a relocation trip-suggestion for which the trip-price is nonzero, in the large-market setting with size $k$.
    Then $p_k\to 0$ as $k\to\infty$.
\end{lemma*}
\begin{proof}
In the steps taken to construct the matching vector $\bfg$ in Algorithm \ref{alg:matching_details}, we begin with a candidate
matching vector which tries to allocate as many actions as possible that are in accordance with the fluid optimal solution $\bff^*$.
The total volume of actions allocated under this candidate vector is
    $$L = \sum_{(r,\omega_t)\in A^+_{\omega_t}(l)} \left(d_{(r,\omega_t)} + \frac{1}{k}\lfloor k(f^*_{(r,\omega_t)} - \bar{d}_{(r,\omega_t)})\rfloor\right),$$
    where $d_{(r,\omega_t)}$ is the realized number of dispatch requests along route $(r,\omega_t)$, $\bar{d}_{(r,\omega_t)}$ is the expected number
    of dispatch requests along $(r,\omega_t)$, and the expression $\frac{1}{k}\lfloor k(f^*_{(r,\omega_t)} - \bar{d}_{(r,\omega_t)})\rfloor$ is equal to
    the number of relocation-trips allocated to $(r,\omega_t)$ under the optimal solution $\bff^*$, rounded down to the nearest multiple of driver-size $1/k$.

If the number of candidate matches $L$ exceeds the number of active drivers at $(l,\omega_t)$, i.e. if $L \geq S_{(l,\omega_t)}$, then
components of the candidate vector are arbitrarily decreased until the number of action allocations equals the number of available drivers,
and every driver located at $(l,\omega_t)$ receives an action allocation from the candidate matching vector.  Since the candidate matching vector
    only contains action allocations which are also produced by the fluid optimal solution $\bff^*$, every action allocation $(r,\omega_t,\delta)$
    in the support of $\bfg(l,\omega_t)$ satisfies $\delta\rho_t(r,\omega_t) = \rho_t(r,\omega_t)$.

If the number of candidate matches $L$ is less than the number of active drivers, i.e. if $L < S_{(l,\omega_t)}$, then $S_{(l,\omega_t)} - L$ drivers
are arbitrarily allocated relocation-trips which do not necessarily get allocated under the fluid optimal solutions.
    Thus, conditioning on a particular value of $L$ we have the upper bound
    $$
\bbP^k\left(\delta\rho_t(r,\omega_t)\neq \rho_t(r,\omega_t)\mid L\right) \leq \frac{(L - S_{(l,\omega_t)})_+}{S_{(l,\omega_t)}},
    $$
    and decomposing the event $\delta\rho_t(r,\omega_t)\neq\rho_t(r,\omega_t)$ over all possible values of $L$, we have the following bound
    on our probability of interest \mccomment{The first line below is correct, right?}\etcomment{do you mean $\bbP^k(...|L)$? without the $L$ its only an inequality}:
    \begin{align*}
    \bbP^k\left(\delta\rho_t(r,\omega_t)\neq \rho_t(r,\omega_t)\right) & = \bbE^k\left[\bbP^k\left(\delta\rho_t(r,\omega_t)\neq \rho_t(r,\omega_t)\mid L\right)\right]\\
    & \leq \bbE^k\left[\frac{( S_{(l,\omega_t)} - L )_+}{S_{(l,\omega_t)}}\right],
    \end{align*}
    where the expectation in both lines is taken over the randomness in $L$, which is equivalent to the randomness in $\bfd$.

    Using the following lower bound on $L$, which drops the inclusion of the floor term,
    $$
    L \geq \sum_{(r,\omega_t)\in A^+_{\omega_t}(l)}d_{(r,\omega_t)} + f^*_{(r,\omega_t)} - \bar{d}_{(r,\omega_t)} - \frac{1}{k},
    $$
    and using the flow-conservation property
    $$
    S_{(l,\omega_t)} = \sum_{(r,\omega_t)\in A^+_{\omega_t}(l)} f^*_{(r,\omega_t)},
    $$
    we deduce the following upper bounds
    \begin{align*}
        \frac{(S_{(l,\omega_t)} - L)_+}{S_{(l,\omega_t)}} &\leq \frac{1}{S_{(l,\omega_t)}}\left(\sum_{(r,\omega_t)\in A^+_{\omega_t}(l)}(\bar{d}_{(r,\omega_t)} - d_{(r,\omega_t)} + \frac{1}{k}) \right)_+\\
        & \leq \frac{1}{S_{(l,\omega_t)}}\sum_{(r,\omega_t)\in A^+_{\omega_t}(l)}\left( \bar{d}_{(r,\omega_t)} - d_{(r,\omega_t)}\right)_+ + \frac{1}{k}.
    \end{align*}
    Thus, we can bound our original probability of interest as follows:
    $$\bbP^k\left(\delta\rho_t(r,\omega_t)\neq \rho_t(r,\omega_t)\right) \leq \frac{1}{S_{(l,\omega_t)}}\sum_{(r,\omega_t)\in A^+_{\omega_t}(l)}\bbE^k\left[\left( \bar{d}_{(r,\omega_t)} - d_{(r,\omega_t)}\right)_+\right] + \frac{|A^+_{\omega_t}(l)|}{kS_{(l,\omega_t)}}.$$
    Finally, using the fact that $\left( \bar{d}_{(r,\omega_t)} - d_{(r,\omega_t)}\right)_+\leq | \bar{d}_{(r,\omega_t)} - d_{(r,\omega_t)}|$,
    Lemma \ref{lem:dispatch_request_concentration} gives us the bound
    \begin{align*}
        \bbP^k\left(\delta\rho_t(r,\omega_t)\neq \rho_t(r,\omega_t)\right) &\leq \frac{1}{S_{(l,\omega_t)}}\sum_{(r,\omega_t)\in A^+_{\omega_t}(l)}\epsilon^k_{(r,\omega_t)}+ \frac{|A^+_{\omega_t}(l)|}{kS_{(l,\omega_t)}}\\
        & = \frac{\epsilon_k}{S_{(l,\omega_t)}}+ \frac{|A^+_{\omega_t}(l)|}{kS_{(l,\omega_t)}},
    \end{align*}
    as desired.
\end{proof}

\subsection{Proof of Lemma \ref{lem:future_reward}}
\label{subsec:future_reward_proof}
\begin{lemma*}
    Consider the stochastic two-level model at a time period $t$ and scenario $\omega_t$ in the large-market setting with
population size parameter $k$.  Fix any time $t$ supply-location vector $\bfS_t$.
    Let $\bfS^*_{t+1}$ be the fluid time $t+1$ supply-location vector, as defined in equation (\ref{eq:fluid_supply_loc_vec}),
    and let $\bfS_{t+1}$ be the stochastic time $t+1$ supply-location under the SSP mechanism, 
    as defined in equation (\ref{eq:stochastic_supply_loc_vec}).
    Then there exists a sequence $\epsilon_k\to0$ as $k\to\infty$ such that equation (\ref{eq:supply_loc_vec_concentration}) holds, restated below:
    \begin{equation*}
    \bbP^k\left(\|\bfS^*_{t+1} - \bfS_{t+1}\|_1 \leq \epsilon_k\right) \geq 1-\epsilon_k.
    \end{equation*}
\end{lemma*}
\begin{proof}
    From the definitions of $\bfS^*_{t+1}$ and $\bfS_{t+1}$ in equations (\ref{eq:fluid_supply_loc_vec}) and 
    (\ref{eq:stochastic_supply_loc_vec}), the following upper bound on $\|\bfS^*_{t+1} - \bfS_{t+1}\|_1$ holds:
    \begin{equation}
        \label{eq:supply_loc_ub}
    \|\bfS^*_{t+1} - \bfS_{t+1}\|_1 \leq \sum_{\substack{(l,\omega_{t+1})\in\\ \calN_{\omega_{t+1}}^{t+1}}} \left(|\bar{M}_{(l,\omega_{t+1})} - M_{(l,\omega_{t+1})}| 
                                        + \left|\sum_{\substack{(r,\omega_t)\in\\ A^-_{\omega_{t+1}}(l)}} f^*_{(r,\omega_t)} - g_{(r,\omega_t,0)} - g_{(r,\omega_t, 1)}\right|\right).
    \end{equation}
    Let us focus on the inner summation.  From the details in Algorithm \ref{alg:matching_details} describing how the matching vector $\bfg$
    is constructed, each term in the the inner summation can be expanded as:
    \begin{align*}
        f^*_{(r,\omega_t)} - g_{(r,\omega_t,0)} - g_{(r,\omega_t, 1)} = 
        f^*_{(r,\omega_t)} - \bar{g}_{(r,\omega_t,0)} - \gamma_{(r,\omega_t,0)} - \bar{g}_{(r,\omega_t, 1)} + \delta_{(r,\omega_t,1)},
    \end{align*}
    where the values $\bar{g}_{(r,\omega_t,\delta)}$ are the candidate matching vector counts specified by equation (\ref{eq:candidate_match_vec}),
    and the values $\gamma_{(r,\omega_t,0)}$ and $\delta_{(r,\omega_t,1)}$ are the amounts by which the candidate matching vector is grown or decreased,
    as specified in Steps 3 and 4 of Algorithm \ref{alg:matching_details}.
    By definition of the candidate matching vector we have the upper bound
    $$
    |f^*_{(r,\omega_t)} - \bar{g}_{(r,\omega_t,0)}- \bar{g}_{(r,\omega_t,1)}| \leq 1/k,
    $$
    hence we can obtain the following bound on the inner summation appearing in the upper bound from equation (\ref{eq:supply_loc_ub}):
    \begin{align*}
        \left|\sum_{\substack{(r,\omega_t)\in\\ A^-_{\omega_{t+1}}(l,\omega_{t+1})}} f^*_{(r,\omega_t)} - g_{(r,\omega_t,0)} - g_{(r,\omega_t, 1)}\right| 
        \leq \frac{|A^-_{\omega_{t+1}}(l)|}{k} + \sum_{\substack{(r,\omega_t)\in\\ A^-_{\omega_{t+1}}(l)}}| \gamma_{(r,\omega_t,0)}  - \delta_{(r,\omega_t,1)}|.
    \end{align*}
    Adding the left hand side over all time $t+1$ locations, and using the fact that each time $t$ route $(r,\omega_t)$ appears
    in one and exactly one of the incoming-route-sets $A^-_{\omega_{t+1}}(l)$, we get the following:
    $$
    \sum_{\substack{(l,\omega_{t+1})\in\\ \calN_{\omega_{t+1}}^{t+1}}}\left|\sum_{\substack{(r,\omega_t)\in\\ A^-_{\omega_{t+1}}(l)}} f^*_{(r,\omega_t)} - g_{(r,\omega_t,0)} - g_{(r,\omega_t, 1)}\right| 
        \leq \frac{|\calA_{\omega_t}^t|}{k} + \sum_{\substack{(r,\omega_t)\in\\ \calA_{\omega_t}^t}}| \gamma_{(r,\omega_t,0)}  - \delta_{(r,\omega_t,1)}|.
    $$
    To analyze the right hand size of this expression, let us rearrange the sum so that the routes are grouped together by the location-time scenario from
    which they originate:
    $$
\sum_{\substack{(r,\omega_t)\in\\ \calA_{\omega_t}^t}}| \gamma_{(r,\omega_t,0)}  - \delta_{(r,\omega_t,1)}| = 
    \sum_{\substack{(l,\omega_t)\in\\ \calN_{\omega_t}^t}} \sum_{\substack{(r,\omega_t)\in\\ A^+_{\omega_t}(l)}}  | \gamma_{(r,\omega_t,0)}  - \delta_{(r,\omega_t,1)}| .
    $$
    The above grouping of routes by the location-time scenario from which they originate is the same grouping of routes used in Algorithm \ref{alg:matching_details}
    to obtain the values $ \gamma_{(r,\omega_t,0)}$ and $\delta_{(r,\omega_t,1)}$.
    For each $(l,\omega_t)\in\calN_{\omega_t}^t$ define the value
    \begin{equation}
        Y_{(l,\omega_t)} = \sum_{\substack{(r,\omega_t)\in\\ A^+_{\omega_t}(l)}}  | \gamma_{(r,\omega_t,0)}  - \delta_{(r,\omega_t,1)}| .
    \end{equation}
    From the above analysis we can simplify our initial upper bound in equation (\ref{eq:supply_loc_ub}) as follows:
    \begin{equation}
        \label{eq:supply_loc_ub_simplified}
        \|\bfS^*_{t+1} - \bfS_{t+1}\|_1 \leq \|\bar{\bfM}_{t+1} - \bfM_{t+1}\|_1 + \frac{|\calA_{\omega_t}^t|}{k} + \sum_{(l,\omega_t)\in \calN_{\omega_t}^t}Y_{(l,\omega_t)}.
    \end{equation}

    Let us now turn to bounding the probability that each value $Y_{(l,\omega_t)}$ is large.  Recall that Algorithm \ref{alg:matching_details}
    chooses to increase or decrease components of the candidate matching vector $\bar{\bfg}$ based on the total volume of candidate actions allocated
    to drivers positioned at $(l,\omega_t)$.  This total candidate volume is defined as
    $$
        L_{(l,\omega_t)} = \sum_{(r,\omega_t)\in A^+_{\omega_t}(l)}\left(\bar{g}_{(r,\omega_t,0)} + \bar{g}_{(r,\omega_t,1)}\right),
    $$
    and components of $\bar{\bfg}$ associated with outgoing routes in $A^+_{\omega_t}(l)$ are increased if there is excess supply, i.e. if $S_{(l,\omega_t)} < L_{(l,\omega_t)}$,
    and are decreased if there is excess demand, i.e. if $L_{(l,\omega_t)} \geq S_{(l,\omega_t)}$.
    Moreover, steps 3 and 4 of Algorithm 4 specify that the extent to which components of the candidate matching vector are
    increased or decreased should result in a matching vector which produces a volume of action allocations that exactly matches 
    the amount of active drivers at $(l,\omega_t)$.  In other words, steps 3 and 4 of Algorithm \ref{alg:matching_details} guarantee
    the following holds:
    $$
        Y_{(l,\omega_t)} = |L_{(l,\omega_t)} - S_{(l,\omega_t)}|.
    $$
    From the definition of the candidate matching vector, and from the fact that the optimal primal solution $\bff^*$ satisfies flow-conservation
    with respect to $S_{(l,\omega_t)}$, we get the upper bound:
    \begin{align*}
        Y_{(l,\omega_t)} & = |L_{(l,\omega_t)} - S_{(l,\omega_t)}|\\
        & \leq \left|\sum_{(r,\omega_t)\in A^+_{\omega_t}(l)} (\bar{g}_{(r,\omega_t,0)} + \bar{g}_{(r,\omega_t,1)} - f^*_{(r,\omega_t)})\right|\\
        & \leq \left|\sum_{(r,\omega_t)\in A^+_{\omega_t}(l)} (\bar{g}_{(r,\omega_t,1)} - \bar{d}_{(r,\omega_t)})\right| + 
            \left|\sum_{(r,\omega_t)\in A^+_{\omega_t}(l)} (\bar{g}_{(r,\omega_t,0)} - (f^* - \bar{d}_{(r,\omega_t)}))\right| \\
        & \leq \left|\sum_{(r,\omega_t)\in A^+_{\omega_t}(l)} (d_{(r,\omega_t)} - \bar{d}_{(r,\omega_t)})\right| + 
                \frac{|A^+_{\omega_t}(l)|}{k}.
    \end{align*}
    Adding the above bounds over all time $t$ LT-scenarios we get
    $$
        \sum_{(l,\omega_t)\in \calN_{\omega_t}^t} Y_{(l,\omega_t)} \leq \|\bfd_t - \bar{\bfd}_t\|_1 + \frac{|\calA_{\omega_t}^t|}{k},
    $$
    and so the upper bound (\ref{eq:supply_loc_ub_simplified}) simplifies further as
    \begin{equation}
        \label{eq:supply_loc_ub_simplified}
        \|\bfS^*_{t+1} - \bfS_{t+1}\|_1 \leq \|\bar{\bfM}_{t+1} - \bfM_{t+1}\|_1 + \frac{2|\calA_{\omega_t}^t|}{k} + \|\bfd_t - \bar{\bfd}_t\|_1.
    \end{equation}

    Let us now shift our attention to using the above inequality to bound the probability that $\|\bfS^*_{t+1} - \bfS_{t+1}\|_1$ is large.
    Let $\epsilon^k_{(r,\omega_t)}$ be the error term associated with the random dispatch requestion volume $d_{(r,\omega_t)}$ in the
    large-market setting with size $k$, provided by Lemma \ref{lem:dispatch_request_concentration}. Define the error term $\epsilon^{k}_{\bfd}$ by
    $$
        \epsilon^k_{\bfd} = |\calA_{\omega_t}^t|\max_{(r,\omega_t)\in\calA_{\omega_t}^t}\epsilon^k_{(r,\omega_t)},
    $$
    and proceed with the following series of inequalities:
    \begin{align*}
        \bbP^k\left(\|\bfd_t - \bar{\bfd}_t\|_1 \geq \epsilon^k_{\bfd}\right) & \leq 
            \bbP^k\left(\exists (r,\omega_t)\in\calA_{\omega_t}^t : |d_{(r,\omega_t)} - \bar{d}_{(r,\omega_t)}| \geq \frac{\epsilon^k_{\bfd}}{|\calA_{\omega_t}^t|}\right) \\
& \leq 
        \bbP^k\left(\exists (r,\omega_t)\in\calA_{\omega_t}^t : |d_{(r,\omega_t)} - \bar{d}_{(r,\omega_t)}| \geq \epsilon^k_{(r,\omega_t)}\right) \\
        & \leq \sum_{(r,\omega_t)\in \calA_{\omega_t}^t} \bbP^k\left(|d_{(r,\omega_t)} - \bar{d}_{(r,\omega_t)}| \geq \epsilon^k_{(r,\omega_t)}\right)\\
        & \leq \sum_{(r,\omega_t)\in \calA_{\omega_t}^t} \epsilon^k_{(r,\omega_t)}\\
        & \leq \epsilon^k_{\bfd}.
    \end{align*}
    A similar analysis shows that the bound
    $$
    \bbP^k\left(\|\bfM_{t+1} - \bar{\bfM}_{t+1}\|_1 \geq \epsilon^k_{\bfM}\right)  \leq \epsilon^k_{\bfM}
    $$
    holds, where the error term $\epsilon^k_{\bfM}$ is defined by
    $$
    \epsilon^k_{\bfM} = |\calN_{\omega_{t+1}}^{t+1}|\max_{(l,\omega_{t+1})\in\calN_{\omega_{t+1}}^{t+1}}\epsilon^k_{(l,\omega_{t+1})},
    $$
    where the terms $\epsilon^k_{(l,\omega_{t+1})}$ are the error terms associated with $M_{(l,\omega_{t+1})}$ in the large-market setting
    as described in Assumption \ref{assn:large_market}.

    Thus, taking 
    $$\epsilon^k = 2\max(\epsilon^k_\bfd,  \epsilon^k_\bfM) + \frac{2|\calA_{\omega_t}^t|}{k},$$
    we conclude
    \begin{align*}
        \bbP^k\left(\|\bfS^*_{t+1} - \bfS_{t+1}\|_1 \geq \epsilon^k\right) & \leq
        \bbP^k\left(\|\bar{\bfM}_{t+1} - \bfM_{t+1}\|_1 +  \|\bfd_t - \bar{\bfd}_t\|_1 \geq 2\max(\epsilon^k_\bfd,  \epsilon^k_\bfM)\right)\\
        & \leq \bbP^k\left(\|\bar{\bfM}_{t+1} - \bfM_{t+1}\|_1 \geq \max(\epsilon^k_\bfd,  \epsilon^k_\bfM)\right) + 
            \bbP^k\left(\|\bar{\bfd}_{t+1} - \bfd_{t+1}\|_1 \geq \max(\epsilon^k_\bfd,  \epsilon^k_\bfM)\right) \\
        & \leq \bbP^k\left(\|\bar{\bfM}_{t+1} - \bfM_{t+1}\|_1 \geq \epsilon^k_\bfM\right) + 
            \bbP^k\left(\|\bar{\bfd}_{t+1} - \bfd_{t+1}\|_1 \geq \epsilon^k_\bfd\right) \\
        & \leq \epsilon^k_\bfM + \epsilon^k_\bfd \\
        & \leq 2\max(\epsilon^k_\bfM, \epsilon^k_\bfd) + \frac{2|\calA_{\omega_t}^t|}{k}\\
        & = \epsilon^k.
    \end{align*}

\end{proof}

\subsection{Proof of Lemma \ref{lem:large_market_asymptotic_utility}}
\label{subsec:large_market_asymptotic_utility}
\begin{lemma*}
    Consider a driver positioned at an active location $(l,\omega_t)$ in the stochastic two-level model at a time period $t$ and scenario $\omega_t$. 
    Fix any supply-location vector $\bfS_t$ and let $\eta^*_{(l,\omega_t)}(\bfS_t)$ be the optimal dual variable associated with the
    $(l,\omega_t)$ constraint for the stochastic maximum flow problem (\ref{eq:dynamic_optimization}) over the  subnetwork induced by $\omega_t$ and
    the supply-location vector $\bfS_t$.
    Then, as the population size $k\to\infty$, the expected utility-to-go for a driver positioned at $(l,\omega_t)$ under the SSP mechanism
    converges to the optimal dual value $\eta^*_{(l,\omega_t)}(\bfS_t)$, i.e. the following holds:
    \begin{equation*}
    \lim_{k\to\infty}\calV^k_{(l,\omega_t)}(\bfS_t) = \eta^*_{(l,\omega_t)}(\bfS_t).
    \end{equation*}
\end{lemma*}
\begin{proof}
    Recall that the expected utility-to-go is defined by the recursive equation:
    \begin{equation*}
    \calV^k_{(l,\omega_t)}(\bfS_t) = \bbE^k\left[\delta\rho_t(r,\omega_t) - c_{(r,\omega_t)} + \calV^k_{(l^+(r),\omega_{t+1})}(\bfS_{t+1})\right].
    \end{equation*}
    The above expectation is taken over multiple sources of randomness:
    \begin{itemize}
    \item The random matching vector $\bfg$, which is produced by Algorithm \ref{alg:matching_details} and depends on the random dispatch-request vector
        $\bfd$.
    \item The random action allocated to a driver positioned at $(l,\omega_t)$, denoted by $(r,\omega_t,\delta)\sim \bfg(l,\omega_t)$.
    \item The random time $t+1$ scenario $\omega_{t+1}\sim\bbP(\cdot\mid\omega_t)$.
    \item The time $t+1$ supply-location vector $\bfS_{t+1}$, which is random because it depends on the time $t+1$ new-driver counts $\bfM_{t+1}$.
        The components of $\bfS_{t+1}$ are specified in equation (\ref{eq:stochastic_supply_loc_vec}).
    \end{itemize}

    With the above in mind, let us focus on showing that equation (\ref{eq:large_market_asymptotic_utility}) holds.
    We proceed via backwards induction on the time period $t$.
    Fix a time period $t$
    and assume that the property (\ref{eq:large_market_asymptotic_utility}) holds for all time periods larger than $t$.
    In the base case where $t=T$ is the final time period, let us adopt the convention that $\eta^*_{(l,\omega_{T+1})}(\bfS_{T+1})=0$
    and $\calV_{(l,\omega_{T+1})}^k(\bfS_{T+1}) = 0$, so that our induction assumption trivially holds in the base case.

    Fix an LT-scenario $(l,\omega_t)\in\calN_{\omega_t}^t$ and note that the complementary slackness conditions 
    guarantee the following equality holds, for every outgoing
    route $(r,\omega_t)\in A^+_{\omega_t}(l)$ for which the primal optimal solution allocates nonzero flow, i.e.
    for which $f^*_{(r,\omega_t)} > 0$:
    $$
    \eta^*_{(l,\omega_t)}(\bfS_t) = \rho_t(r,\omega_t) - c_{(r,\omega_t)} + \eta^*_{(l^+(r),\omega_{t+1})}(\bfS^*_{t+1}),
    $$
    where $\bfS^*_{t+1}$ is the supply-location vector which would arise if driver-movement in time period $t$ was
    carried out exactly as according to $\bff^*$, and if the new-driver counts at time $t+1$ were equal to their means.
    The components of $\bfS^*_{t+1}$ are specified in equation (\ref{eq:fluid_supply_loc_vec}).

    In addition, note that the construction of the matching vector $\bfg$, as specified in Algorithm \ref{alg:matching_details},
    guarantees that actions $(r,\omega_t,\delta)$ are only allocated to routes $(r,\omega_t)$ for which the primal optimal
    solution allocates nonzero flow.  That is, for every matching vector $\bfg$ that the platform might produce,
    and for every action $(r,\omega_t,\delta)$ in the support of the action-allocation distribution $\bfg(l,\omega_t)$,
    the primal optimal solution allocations nonzero flow along $(r,\omega_t)$, i.e. $f^*_{(r,\omega_t)} > 0$.
    Since the complementary slackness conditions hold for every route satisfying this condition, we can take
    an expectation over action suggestions for free to conclude that the following holds with equality:
    \begin{equation}
        \label{eq:dual_var_recursion}
    \eta^*_{(l,\omega_t)}(\bfS_t) = \bbE^k\left[\rho_t(r,\omega_t) - c_{(r,\omega_t)} + \eta^*_{(l^+(r),\omega_{t+1})}(\bfS^*_{t+1})\right],
    \end{equation}
    where the above expectation is with respect to the matching vector $\bfg$ and a random action allocation $(r,\omega_t,\delta)\sim\bfg(l,\omega_t)$.

    Taking the difference between equation (\ref{eq:utility_recursion}) and (\ref{eq:dual_var_recursion}),
    we obtain the following expression:
    $$
        \calV^k_{(l,\omega_t)}(\bfS_t) - \eta^*_{(l,\omega_t)}(\bfS_t) = \bbE^k\left[\delta\rho_t(r,\omega_t) - \rho_t(r,\omega_t)
                                                            + \calV^k_{(l^+(r),\omega_{t+1})}(\bfS_{t+1}) - \eta^*_{(l^+(r),\omega_{t+1})}(\bfS^*_{t+1})\right].
    $$
    We will show equation (\ref{eq:large_market_asymptotic_utility}) holds by showing the following equations hold:
    \begin{equation}
        \label{eq:immediate_reward_lim}
        \lim_{k\to\infty}\bbE^k\left[\delta\rho_t(r,\omega_t) - \rho_t(r,\omega_t)\right] = 0
    \end{equation}
    and
    \begin{equation}
        \label{eq:future_reward_lim}
        \lim_{k\to\infty}\bbE^k\left[\calV^k_{(l^+(r),\omega_{t+1})}(\bfS_{t+1}) - \eta^*_{(l^+(r),\omega_{t+1})}(\bfS^*_{t+1})\right].
    \end{equation}

    To establish (\ref{eq:immediate_reward_lim}), let 
    $$
        p^k = \bbP^k(\delta\rho_t(r,\omega_t)\neq\rho_t(r,\omega_t))
    $$
    denote the probability that the SSP mechanism allocates an action-suggestion $(r,\omega_t,\delta)$ 
    for which $\delta\rho_t(r,\omega_t)\neq\rho_t(r,\omega_t)$ holds, to a driver positioned at $(l,\omega_t)$,
    in the large-market setting with size $k$, and decompose each term in the sequence as follows:
    \begin{align*} 
        \bbE^k\left[\delta\rho_t(r,\omega_t) - \rho_t(r,\omega_t)\right] &=
        \bbE^k\left[\delta\rho_t(r,\omega_t) - \rho_t(r,\omega_t)|\delta\rho_t(r,\omega_t)\neq\rho_t(r,\omega_t)\right]p_k\\
        & \ \ + \bbE^k\left[\delta\rho_t(r,\omega_t) - \rho_t(r,\omega_t)|\delta\rho_t(r,\omega_t)=\rho_t(r,\omega_t)\right](1-p_k)\\
        & = \bbE^k\left[\delta\rho_t(r,\omega_t) - \rho_t(r,\omega_t)|\delta\rho_t(r,\omega_t)\neq\rho_t(r,\omega_t)\right]p_k.
    \end{align*} 
    In Lemma \ref{lem:immediate_reward} we know the sequence $p_k\to 0$, and from our rider-value distribution assumption Assumption \ref{assn:rider_value_dist}
    we know the prices $\rho_t(r,\omega_t)$ are bounded above by $V_{max}$ and below by $0$.  Hence the above expression converges to $0$
    as $k\to\infty$, establishing equation (\ref{eq:immediate_reward_lim}).

    To establish equation (\ref{eq:future_reward_lim}), let $\epsilon_k$ be the value established by Lemma \ref{lem:future_reward}
    for which the following high-probability concentration property (\ref{eq:supply_loc_vec_concentration}) holds, and define
    $$
    q_k = \bbP^k\left(\|\bfS^*_{t+1} - \bfS_{t+1}\|_1 \leq \epsilon_k\right).
    $$
    Lemma \ref{lem:future_reward} establishes that $q_k\geq 1-\epsilon_k$, in particular $q_k\to 1$ as $k\to\infty$.
    Now, decompose each term in the sequence as follows:
    \begin{align}
        & \bbE^k\left[\calV^k_{(l^+(r),\omega_{t+1})}(\bfS_{t+1}) - \eta^*_{(l^+(r),\omega_{t+1})}(\bfS^*_{t+1})\right]\nonumber \\
    = & \bbE^k\left[\calV^k_{(l^+(r),\omega_{t+1})}(\bfS_{t+1}) - \eta^*_{(l^+(r),\omega_{t+1})}(\bfS_{t+1}) \right]\nonumber \\
        &  \ + \bbE^k\left[\eta^*_{(l^+(r),\omega_{t+1})}(\bfS_{t+1}) - \eta^*_{(l^+(r),\omega_{t+1})}(\bfS^*_{t+1}) \mid \|\bfS^*_{t+1} - \bfS_{t+1}\|_1 > \epsilon_k\right](1- q_k)\nonumber\\
        \label{eq:ev_three_terms}
      & \ + \bbE^k\left[\eta^*_{(l^+(r),\omega_{t+1})}(\bfS_{t+1}) - \eta^*_{(l^+(r),\omega_{t+1})}(\bfS^*_{t+1}) \mid \|\bfS^*_{t+1} - \bfS_{t+1}\|_1 \leq \epsilon_k\right] q_k.
    \end{align}

    We proceed by analyzing each term in the above expression separately.  For the first term, 
    $\bbE^k\left[\calV^k_{(l^+(r),\omega_{t+1})}(\bfS_{t+1}) - \eta^*_{(l^+(r),\omega_{t+1})}(\bfS_{t+1}) \right]$, we claim that this converges to $0$ as
    $k\to\infty$ by our induction hypothesis.  Our induction hypothesis states that $\calV^k_{(l^+(r),\omega_{t+1})}(\bfS_{t+1}) \to\eta^*_{(l^+(r),\omega_{t+1})}(\bfS_{t+1})$ for a fixed
    supply-location vector $\bfS_{t+1}$, while the supply-location vector $\bfS_{t+1}$ appearing in this expression is a random variable whose distribution depends on $k$.
    However, an application of the dominated convergence theorem lets us bring the limit inside the expectation to conclude that the entire expectation
    converges to $0$ as $k\to\infty$.
    \mccomment{does this pass the ``enough details'' check?}

    For the second term, note that the values $\eta^*_{(l^+(r),\omega_{t+1})}(\bfS^*_{t+1})$
    and $\calV^k_{(l^+(r),\omega_{t+1})}(\bfS_{t+1})$ are both bounded above and below by a constant that does not depend on $k$.  For example, we can take
    $(V_{max} + c_{max})(T-t+1)$ as an upper bound on both values, and we can take $-c_{max}(T-t+1)$ as a lower bound on both values, where $V_{max}$ is our
    uniform upper bound on the support of rider-value distributions, as exists by Assumption \ref{assn:rider_value_dist}, 
    and $c_{max}$ is a uniform upper bound on the absolute value of the costs $|c_{(r,\omega_t)}|$ over all routes $(r,\omega_t)$.
    Thus, since the following expectation is bounded and the sequence $1-q_k\to 0$ as $k\to\infty$ we conclude the following holds:
    $$
\lim_{k\to\infty}\bbE^k\left[\calV^k_{(l^+(r),\omega_{t+1})}(\bfS_{t+1}) - \eta^*_{(l^+(r),\omega_{t+1})}(\bfS^*_{t+1}) \mid \|\bfS^*_{t+1} - \bfS_{t+1}\|_1 > \epsilon_k\right](1- q_k) = 0.
$$

    Finally, for the third term, we use Lemma \ref{lem:marginal_utility} which establishes that $\eta^*_{(l^+(r),\omega_{t+1})}(\bfS)$ is continuous as a
    function of $\bfS$. Since the third term conditions on the event $\|\bfS^*_{t+1} - \bfS_{t+1}\|_1 \leq \epsilon_k$, and since $\epsilon_k$ converges to $0$
    as $k\to\infty$, by continuity of $\eta^*_{(l^+(r),\omega_{t+1})}(\cdot)$ we conclude:
$$
\lim_{k\to\infty}\bbE^k\left[\eta^*_{(l^+(r),\omega_{t+1})}(\bfS_{t+1}) - \eta^*_{(l^+(r),\omega_{t+1})}(\bfS^*_{t+1}) \mid \|\bfS^*_{t+1} - \bfS_{t+1}\|_1 \leq \epsilon_k\right] q_k = 0.
$$
    Thus all 3 terms in the decomposition (\ref{eq:ev_three_terms}) go to $0$ in $k$, so we have established that equation (\ref{eq:future_reward_lim}) holds.

    Equations (\ref{eq:immediate_reward_lim}) and (\ref{eq:future_reward_lim}) are sufficient to establish that
    $$
    \lim_{k\to\infty}\calV^k_{(l,\omega_t)}(\bfS_t) - \eta^*_{(l,\omega_t)}(\bfS_t) = 0,
    $$
    concluding the induction step and establishing the claim that equation (\ref{eq:large_market_asymptotic_utility}) holds.
\end{proof}

\subsection{Proof of Theorem \ref{thm:stochastic_ic}}
\label{subsec:stochastic_ic}
\begin{theorem*}
    Consider a driver positioned at an active location $(l,\omega_t)$ in the stochastic two-level model at a time period $t$ and scenario $\omega_t$. 
    Then for any supply-location vector $\bfS_t$ there exists a sequence of error terms $\epsilon_k\to 0$ such that the probability the driver is allocated
    an $\epsilon_k$-incentive-compatible trip-suggestion under the SSP mechanism goes to $1$ as $k\to\infty$.
\end{theorem*}
\begin{proof}
    Let us use the notation $\calU_{(l,\omega_t)}^k(\bfS_t)$ to denote the maximum utility that a driver positioned at $(l,\omega_t)$
    can collect under the SSP mechanism, without the assumption that the driver always follows their trip-suggestion. 

    We proceed via backwards induction on the time $t$.  When $t=T$, we know that if the driver is allocated a trip-suggestion
    $(r,\omega_t,\delta)$ that is either a dispatch-trip or that has trip-price $0$, then no other action can improve their welfare.
    This follows because the SSP mechanism only allocates trip-suggestions to routes $(r,\omega_t)$ for which the optimal flow $f^*_{(r,\omega_t)}$
    is nonzero, and complementary slackness establishes the welfare of a trip-suggestion along such a route satisfying the above property
    is equal to $$\eta^*_{(l,\omega_t)} = \delta\rho_{(r,\omega_t)} - c_{(r,\omega_t)} = \rho_{(r,\omega_t)} - c_{(r,\omega_t)}.$$
    welfare is optimized.  Thus, for the last time period $t=T$, we can choose error terms $\epsilon_k=0$ for all $k$, and the probability
    of receiving an incentive-compatible dispatch, i.e.
    $$
    q_k = 1-\bbP^k(\delta\rho_{(r,\omega_t)}\neq\rho_{(r,\omega_t)})
    $$
    goes to $1$ as $k$ goes to $\infty$ as established by Lemma \ref{lem:immediate_reward}.
    Also note that the maximum utility that a driver can collect in the final time period from deviated can be bounded by the corresponding
    optimal dual variable as follows:
    \begin{align*}
        \calU_{(l,\omega_t)}^k(\bfS_t) &\leq \max_{(r,\omega_t)\in A^+_{\omega_t}(l)} \rho_{(r,\omega_t)} - c_{(r,\omega_t)} = \rho^*_{(r,\omega_t)}(\bfS_t).
    \end{align*}

    For our induction hypothesis at earlier time periods $t<T$, assume that for each realization of the time $t+1$ supply-location
    vector $\bfS_{t+1}$ and for each location $l$ we can find error terms $\epsilon_k^{t+1}(\bfS_{t+1}, l)$ and probabilities
    $q_k^{t+1}(\bfS_{t+1},l)$ such that $q_k^{t+1}(\bfS_{t+1},l)\to 1$ as $k\to\infty$,  $\epsilon_k^{t+1}(\bfS_{t+1}, l)\to 0$ as $k\to\infty$,
    and $q_k^{t+1}(\bfS_{t+1},l)$ is the probability our driver is allocated a trip-suggestion for which there is less than $\epsilon_k^{t+1}(\bfS_{t+1}, l)$
    incentive to deviate.
    
    Also   
    assume that the maximum utility for deviating at future time periods is bounded by the corresponding optimal dual variable as follows:
    $$
        \calU_{(l,\omega_{t+1})}^k(\bfS_{t+1}) \leq \rho^*_{(l,\omega_{t+1})}(\bfS_{t+1}) + \delta_k(\bfS_{t+1}) 
    $$
    where $\delta_k(\bfS_{t+1})$ is a term that goes to $0$ as $k\to\infty$ (in the final time period $t=T$ we showed this is satisfied by $\delta_k(\bfS_T)=0$ for all $k$).

    Now, to analyze incentive in time $t$, define  $q_k(\bfS_t)$ as the probability the driver is allocated a trip-suggestion satisfying $\delta\rho_{(r,\omega_t)} = \rho_{(r,\omega_t)}$
    in the time period $t$ under supply-location vector $\bfS_t$.
    We can decompose the always-follow utility-to-to for a driver as
    \begin{align*}
        \calV_{(l,\omega_t)}^k(\bfS_t) &= \bbE^k\left[\delta\rho_{(r,\omega_t)} - c_{(r,\omega_t)} + \calV^k_{(l^+(r),\omega_{t+1})}(\bfS_{t+1})\mid  \delta\rho_{(r,\omega_t)} = \rho_{(r,\omega_t)}\right] q_k(\bfS_t) \\
        & \ \ \ + \bbE^k\left[\delta\rho_{(r,\omega_t)} - c_{(r,\omega_t)} + \calV^k_{(l^+(r),\omega_{t+1})}(\bfS_{t+1})\mid  \delta\rho_{(r,\omega_t)} \neq \rho_{(r,\omega_t)}\right](1- q_k(\bfS_t)).
    \end{align*}
    Let 
    $$\alpha_k(\bfS_t)=\bbE^k\left[\calV^k_{(l^+(r),\omega_{t+1})}(\bfS_{t+1}) - \eta^*_{(l^+(r),\omega_{t+1})}(\bfS_{t+1}^*)\mid \delta\rho_{(r,\omega_t)} = \rho_{(r,\omega_t)}\right],$$
    where $\bfS_{t+1}^*$ are the time $t+1$ supply-locations that would arise under the fluid optimal solution.
    Then, conditioning on the event that the driver is allocated a trip-suggestion satisfying $\delta\rho_{(r,\omega_t)} = \rho_{(r,\omega_t)}$, the utility they collect
    is equal to the following:
    \begin{align*}
        & \bbE^k\left[\delta\rho_{(r,\omega_t)} - c_{(r,\omega_t)} + \calV^k_{(l^+(r),\omega_{t+1})}(\bfS_{t+1})\mid  \delta\rho_{(r,\omega_t)} = \rho_{(r,\omega_t)}\right]\\
        = &\bbE^k\left[\delta\rho_{(r,\omega_t)} - c_{(r,\omega_t)} + \eta^*_{(l^+(r),\omega_{t+1})}(\bfS_{t+1}^*)\mid  \delta\rho_{(r,\omega_t)} = \rho_{(r,\omega_t)}\right] - \alpha_k(\bfS_t)\\
        = & \eta^*_{(l,\omega_t)}(\bfS_t) - \alpha_k(\bfS_t).
    \end{align*}

    We claim that the sequence of error terms satisfy $\alpha_k(\bfS_t)\to 0$ as $k\to\infty$. This property follows from continuity of $\eta^*_{(l^+(r),\omega_{t+1})}(\cdot)$ established
    in Lemma \ref{lem:marginal_utility}, the fact that the supply-location random variable $\bfS_{t+1}$ converges in probability to $\bfS_{t+1}^*$ as $k\to\infty$ established by
    Lemma \ref{lem:future_reward}, the fact
    that $\calV^k_{(l^+(r),\omega_{t+1})}(\bfS_{t+1})\to\eta^*_{(l^+(r),\omega_{t+1})}(\bfS_{t+1})$ is true for any fixed supply-location vector $\bfS_{t+1}$ established 
    by Lemma \ref{lem:large_market_asymptotic_utility}, and an application of the dominated convergence theorem to exchange the order of the expectation and the limit.
    \mccomment{does this pass the details check?}

    Next, we can upper bound the maximum utility from deviating also in terms of the optimal dual variables:
    \begin{align*}
        \calU^k_{(l,\omega_t)}(\bfS_t) & \leq \max_{(r,\omega_t)\in A^+_{\omega_t}(l)} \rho_{(r,\omega_t)} - c_{(r,\omega_t)} + \bbE^k\left[\calU_{(l^+(r),\omega_{t+1})}^k(\tilde{\bfS}_{t+1}) \right],
    \end{align*}
    where $\tilde{\bfS}_{t+1}$ is the random time $t+1$ supply-location vector which occurs when the driver deviates.
    From our induction hypothesis it follows that 
    $$
        \calU^k_{(l,\omega_t)}(\bfS)  \leq \max_{(r,\omega_t)\in A^+_{\omega_t}(l)} \rho_{(r,\omega_t)} - c_{(r,\omega_t)} + 
        \bbE^k\left[\eta^*_{(l^+(r),\omega_{t+1})}(\tilde{\bfS}_{t+1}) + \delta_k(\tilde{\bfS}_{t+1})) \right].
    $$
    Let $r^*$ be the ODT which achieves the max in the above expression, and define
    $$
    \delta_k(\bfS_t) =
        \bbE^k\left[\eta^*_{(l^+(r^*),\omega_{t+1})}(\tilde{\bfS}_{t+1})  - \eta^*_{(l^+(r^*),\omega_{t+1})}(\bfS^*_{t+1}) + \delta_k(\tilde{\bfS}_{t+1})) \right].
    $$
    Thus, we have the upper bound
   \begin{align*} 
       \calU^k_{(l,\omega_t)}(\bfS_t)  &\leq \max_{(r,\omega_t)\in A^+_{\omega_t}(l)} \rho_{(r,\omega_t)} - c_{(r,\omega_t)} + 
        \bbE^k\left[\eta^*_{(l^+(r),\omega_{t+1})}({\bfS}^*_{t+1})\right] + \delta_k({\bfS}_{t})) \\
        & = \eta^*_{(l,\omega_t)}(\bfS_t) + \delta_k(\bfS_t).
   \end{align*}
   We claim that $\delta_k(\bfS_t)$ goes to $0$ as $k\to\infty$ for the same reasons above that $\alpha_k(\bfS_t)\to 0$: the random supply-location vector $\tilde{\bfS}_{t+1}$
   converges in probability to $\bfS^*_{t+1}$ (the deviation  does not affect this limit because the size of the driver is $1/k$), the optimal dual variable is a continuous function
   of the supply-location vector, the induction hypothesis states that $\delta_k(\tilde{\bfS}_{t+1}))\to 0$ for any fixed supply-location vector $\tilde{\bfS}_{t+1}$, and an
   application of the dominated convergence theorem lets us exchange the order of the limit and the expectation. \mccomment{does this pass the details check?}

   Finally, define the error term $$\epsilon_k(\bfS_t) = \alpha_k(\bfS_t) + \delta_k(\bfS_t).$$
   We have shown that with probability $q_k(\bfS_t)$ the driver is allocated a trip-suggestion such that following the trip-suggestion
   yields utility at least $\eta^*_{(l,\omega_t)}(\bfS_t) - \alpha_k(\bfS_t)$, and so $\epsilon_k(\bfS_t)$ is the gap between following this trip-suggestion
   and the maximum utility from deviating of $\eta^*_{(l,\omega_t)} + \delta_k(\bfS_t)$.
   Since $\alpha_k(\bfS_t)$ and $\delta_k(\bfS_t)$ both go to $0$ as $k\to\infty$, this concludes the proof.
\end{proof}

}
\Xomit{
\section{Stability Lemma}
In this section we prove a Lemma about the stability of solutions to the dynamic optimization formulation
(\ref{eq:dynamic_optimization}).  Consider a time $t$ and a scenario $\omega_t$, and let $\calW_{\omega_t}(\bfS)$
denote the value of the dynamic optimization formulation (\ref{eq:dynamic_optimization}) with respect to some
supply-location vector $\bfS$. 
Recall a supply-location vector $\bfS = (S_{(l,\omega_t)} : (l,\omega_t)\in\calN_{\omega_t}^t)$ specifies driver-counts $S_{(l,\omega_t)}$
at each active location, after new drivers have entered the marketplace for time period $t$.

We re-state the optimization problem (\ref{eq:dynamic_optimization}) below for completeness:
$$
\calW_{\omega_t}(\bfS)\equiv\sup\left\{\calU_{\omega_t}(\bff) : 
    \bff \mbox{ satisfies (\ref{eq:flow_conservation_dynamic}) w.r.t }\bfS, \bff\in\bbR^{\calA_t^{\geq t}}_+\right\},
$$
where the objective function $\calU_{\omega_t}(\bff)$ is equal to expected welfare:
$$
\calU_{\omega_t}(\bff) \equiv \sum_{\substack{(r,\omega_\tau)\\\in\calA_{\omega_t}^{\geq t}}}
\bbP(\omega_\tau|\omega_t)U_{(r,\omega_\tau)}(f_{(r,\omega_\tau)}).
$$

Let us introduce the notation $F^*(\bfS)$ to denote the set of optimal stochastic flows with respect to $\bfS$:
$$
F^*(\bfS) \equiv \left\{\calU_{\omega_t}(\bff) = \calW_{\omega_t}(\bfS) : 
\bff \mbox{ satisfies (\ref{eq:flow_conservation_dynamic}) w.r.t }\bfS, \bff\in\bbR^{\calA_t^{\geq t}}_+\right\},
$$
and let us introduce the notation $F^\delta(\bfS)$ to denote the set of $\delta$-optimal solutions:
$$
F^\delta(\bfS) \equiv \left\{\calU_{\omega_t}(\bff) \geq \calW_{\omega_t}(\bfS) - \delta : 
\bff \mbox{ satisfies (\ref{eq:flow_conservation_dynamic}) w.r.t }\bfS, \bff\in\bbR^{\calA_t^{\geq t}}_+\right\},
$$

Our stability lemma states, roughly, that if $\bfS$ and $\bfS'$ are two supply-location vectors that are close to each other,
then there are optimal solutions in $F^*(\bfS)$ and $F^*(\bfS')$ which are close to each other:

\begin{lemma}
\label{lem:stability}
Fix a supply-location vector $\bfS$ and let $\{\bfS^k: k=1,2,\dots\}$ be a sequence of supply-location vectors which
converge to $\bfS$.
Assume each reward function $U_{(r,\omega_\tau)}(\cdot)$ is Lipschitz continuous with constant $\alpha$.
Consider the same premise as Lemma \ref{lem:marginal_utility}. 
Then the following statement is true:
$$
\lim_{k\to\infty} \inf\left\{\|\bff - \bff'\|_1 : \bff\in F^*(\bfS), \bff'\in F^*(\bfS^k)\right\} = 0.
$$
\end{lemma}

\begin{proof}
We prove this lemma in two steps.  The first step is to show that if $\bfS$ is $\epsilon_k$ away from $\bfS^k$, then
there is an optimal solution $\bff_k\in F^*(\bfS^k)$ which is $O(\epsilon_k)$ distance from an $O(\epsilon_k)$-optimal
solution for $\bfS$.

For each $k$ let $\bff_k\in F^*(\bfS^k)$ be an optimal solution for $\bfS^k$ and let $\epsilon_k = \|\bfS - \bfS^k\|_1$.
A simple argument shows we can construct from $\bff_k$ a new stochastic flow $\bff$ such that:
\begin{enumerate}
\item $\bff$ is feasible with respect to $\bfS$
\item For each remaining time period $\tau=t,t+1,\dots,T$, and for each possible time $\tau$ scenario $\omega_\tau\in\Omega_\tau(\omega_t)$,
the difference between $\bff$ and $\bff_k$ on active routes with respect to $\omega_\tau$ is at most $\epsilon_k$, i.e:
\begin{equation}
\label{eq:flow_distance_property}
\sum_{(r,\omega_\tau)\in\calA_{\omega_\tau}^\tau} |f_{(r,\omega_\tau)} - f_{k,(r,\omega_\tau)}| \leq \|\bfS - \bfS^k\| = \epsilon_k.
\end{equation}
\end{enumerate}
Let us quickly describe how we construct $\bff$ from $\bff_k$.  We start at time $t$ and scenario $\omega_t$, and at each active location 
$(l,\omega_t)\in\calN_{\omega_t}^t$, we look at the difference betwee $S_{(l,\omega_t)}$ and $S_{k,(l,\omega_t)}$, and adjust the flow
values on outgoing routes $(r,\omega_t)\in N^+(l,\omega_t)$ accordingly. This adjustment results in an $\epsilon_k$ difference between $\bff$
and $\bff_k$ on the active routes with respect to $\omega_t$. We proceed by looking at the time $t+1$ problem. 
For each scenario $\omega_{t+1}$ we get some supply-location vector $\bfS_{t+1}$ which arises from using the flow $\bff$ at time $t$, and 
a supply-location vector $\bfS_{t+1}^k$ which arises from using the flow $\bff_k$ at time $t$.  However, from flow conservation we have
the difference between teh supply-location vectors $\|\bfS_{t+1}-\bfS_{t+1}^k\|$ is at most the difference between the time-$t$ flows
prescribed by $\bff$ and $\bff^k$, which by construction is at most $\epsilon_k$.  This observation is how we get property 
(\ref{eq:flow_distance_property}) to hold.

Next, we bound the objective value for $\bff$ as follows:
\begin{align*}
|\calU(\bff) - \calU(\bff_k)| & = \sum_{\substack{(r,\omega_\tau)\\\in\calA_{\omega_t}^{\geq t}}}
\bbP(\omega_\tau|\omega_t)\left(U_{(r,\omega_\tau)}(f_{(r,\omega_\tau)}) - U_{(r,\omega_\tau)}(f_{k,(r,\omega_\tau)})\right) \\
& = \sum_{\tau=t}^T \sum_{\omega_\tau\in\Omega_\tau(\omega_t)} \bbP(\omega_\tau|\omega_t)
\sum_{\substack{(r,\omega_\tau)\\\in\calA_{\omega_\tau}^{\tau}}}
\left(U_{(r,\omega_\tau)}(f_{(r,\omega_\tau)}) - U_{(r,\omega_\tau)}(f_{k,(r,\omega_\tau)})\right) \\
& \leq \sum_{\tau=t}^T \sum_{\omega_\tau\in\Omega_\tau(\omega_t)} \bbP(\omega_\tau|\omega_t)
\sum_{\substack{(r,\omega_\tau)\\\in\calA_{\omega_\tau}^{\tau}}}
\alpha |f_{(r,\omega_\tau)} - f_{k,(r,\omega_\tau)}|\\
& \leq \sum_{\tau=t}^T \sum_{\omega_\tau\in\Omega_\tau(\omega_t)} \bbP(\omega_\tau|\omega_t)
\alpha \epsilon_k \\
&= \alpha(T-t)\epsilon_k.
\end{align*}
In the first upper bound we use Lipschitz continuity of the reward function $U_{(r,\omega_\tau)}(\cdot)$,
and in the second upper bound we use the bound (\ref{eq:flow_distance_property}).

Notice that the above analysis would hold if we had started with an optimal solution for $\bfS$ and constructed
a feasible solution for $\bfS^k$ in an analogous manner. This would yield a feasible solution for $\bfS^k$ which achieves
utility at most $\alpha(T-t)\epsilon_k$ away from the optimal utility for $\bfS$.  And we have just exhibited a
feasible solution for $\bfS$ which achieves utility at most $\alpha(T-t)\epsilon_k$ away from the optimal utility for $\bfS^k$.
Thus the optimal utility for $\bfS$ and $\bfS^k$ are at most $\alpha(T-t)\epsilon_k$ away from each other, i.e.
$$
|\calW_{\omega_t}(\bfS) - \calW_{\omega_t}(\bfS^k)| \leq \alpha(T-t)\epsilon_k.
$$

So, now we know that $\bff$ is close to $\bff_k$ in the sense of equation (\ref{eq:flow_distance_property}),
we know that the utility achieved by $\bff$ is at most $\alpha(T-t)\epsilon_k$ away from the utility of $\bff_k$,
and by virtue of the above upper bound we know the utility of $\bff_k$ is at most $\alpha(T-t)\epsilon_k$ away from
the optimal utility for $\bfS$. Thus, $\bff$ achieves utility that is at most $2\alpha(T-t)\epsilon_k$ away from optimal.
In symbols:
\begin{align*}
|\calU_{\omega_t}(\bff) -\calW_{\omega_t}(\bfS)| & \leq 
|\calU_{\omega_t}(\bff) -\calW_{\omega_t}(\bfS^k)| + | \calW_{\omega_t}(\bfS^k) - \calW_{\omega_t}(\bfS)| \leq 2\alpha(T-t)\epsilon_k. 
\end{align*}
This concludes the goal of step 1.

For step 2 of the proof we show that as the optimality gap shrinks, the distance to an optimal
flow in $F^*(\bfS)$ also shrinks.
To this end, define the following function:
\begin{equation}
\label{eq:d_delta_limit}
d(\delta) = \sup_{\bff'\in F^\delta(\bfS)}\inf_{\bff\in F^*(\bfS)}\|\bff' - \bff\|.
\end{equation}
The value $d(\delta)$ gives the maximum distance between a $\delta$-optimal point from the set of optimal points.
Notice that $d(0)=0$ and the $d(\delta)$ is decreasing as $\delta$ decreases.
Our goal in step 2 is to show that the following limit holds:
$$
\lim_{\delta\to 0^+} d(\delta) = 0.
$$

Suppose for contradiction this limit does not hold. 
For each $i=1,2,\dots$ let $\bfg_i\in F^{1/i}(\bfS)$ be a stochastic flow which achieves the supremum in the
definition of $d(1/i)$.  Since the space of feasible solutions with respect to $\bfS$ is compact, the sequence $\{\bfg_i\}$
has at least one limit point.  Let $\bar{\bfg}$ be a limit point for this sequence, and let $\{\bfg_{i(j)} : j = 1,2,\dots\}$
be the subsequence of $\{\bfg_i\}$ which converges to $\bar{\bfg}$.  However, since each $\bfg_i$ is $1/i$-optimal,
it follows that $\bar{\bfg}$ must be an optimal solution, i.e. $\bar{\bfg}\in F^*(\bfS)$.

Now, for each $\delta > 0$, let $j = j(\delta)$ be the largest index in the subsequence $\{\bfg_{i(j)}\}$ such that $1/j > \delta$.
We then have the following:
\begin{align*}
\lim_{\delta\to 0^+} d(\delta) & \leq \lim_{\delta\to 0^+} d(1/j(\delta)) \\
& = \lim_{\delta\to 0^+} \inf_{\bff\in F^*(\bfS)} \|\bfg_{j(\delta)} - \bff\|\\
& \leq \lim_{\delta\to 0^+} \|\bfg_{j(\delta)} - \bar{\bfg}\|\\
& = 0.
\end{align*}
Thus the limit (\ref{eq:d_delta_limit}) holds.

Let us return to our original desired bound.
Let $\bff_k\in F^*(\bfS^k)$ be an optimal flow for each $\bfS^k$, let $\delta_k=2\alpha(T-t)\epsilon_k$ be
the optimality gap we achieved for each $k$, and let $\{\bff_k'\}$ be the sequence
of stochastic flows we constructed from $\{\bff_k\}$ in the first step, so that 
\begin{enumerate}
\item each $\bff_k'$ is feasible w.r.t. $\bfS$, and
\item each $\bff_k'$ is $\delta_k$ from optimal. 
\end{enumerate}
Using the notation for $\delta$-optimal solutions we introduced above, the two above properties
state that $\bff_k'\in F^{\delta_k}(\bfS)$ holds for each $k$.

We then have:
\begin{align*}
& \lim_{k\to\infty} \inf\left\{\|\bff - \bff'\|_1 : \bff\in F^*(\bfS), \bff'\in F^*(\bfS^k)\right\} \\
\leq & \lim_{k\to\infty} \inf\left\{\|\bff - \bff_k\|_1 : \bff\in F^*(\bfS)\right\} \\
\leq & \lim_{k\to\infty} \inf\left\{\|\bff - \bff_k'\|_1 + \|\bff_k' - \bff_k\|_1 : \bff\in F^*(\bfS)\right\} \\
\leq & \lim_{k\to\infty} O(\delta_k) +  \inf\left\{\|\bff - \bff_k'\|_1 : \bff\in F^*(\bfS)\right\} \\
\leq & \lim_{k\to\infty} O(\delta_k) + d(O(\delta_k))\\
= & 0,
\end{align*}
as desired.

\end{proof}

\section{Continuous Marginal Utility Lemma}
In this section we prove that the optimal dual variables are continuous functions of the supply-location
vector.
\begin{lemma}
\label{lem:marginal_utility}
Fix a time period $t$ and scenario $\omega_t$. Let $\bfS = (S_{(l,\omega_t)} : (l,\omega_t)\in\calN_{\omega_t}^t)$ 
be a supply-location vector for time $t$ and sceneario $\omega_t$. We make use of the following notation:
\begin{itemize}
    \item Let
        $$\bfeta^*(\bfS)\equiv\left(\eta^*_{(l,\omega_\tau)}(\bfS) : (l,\omega_\tau)\in\calN_{\omega_t}^{\geq t}\right)$$ 
        denote the optimal dual variables for the dynamic optimization problem (\ref{eq:dynamic_optimization})
        under the supply-location vector $\bfS$.
    \item Let
        $$\bff^*(\bfS)\equiv\left(f^*_{(r,\omega_\tau)}(\bfS) : (r,\omega_\tau)\in\calA_{\omega_t}^{\geq t}\right)$$
        denote the optimal primal solution for the dynamic optimization problem (\ref{eq:dynamic_optimization})
        under the supply-location vector $\bfS$.
\end{itemize}
    Under Assumption \ref{assn:unique_soln}, the optimization problem (\ref{eq:dynamic_optimization}) has a unique solution for any $\bfS$, so
    both $\bfeta^*(\bfS)$ and $\bff^*(\bfS)$ are well-defined.

Let $\bfS^k$ for $k=1,2,\dots,\infty$ be a sequence of supply-location vectors which converge to $\bfS$ as $k\to\infty$.
Fix an active LT-scenario $(l,\omega_t)\in\calN_{\omega_t}^t$ for which there is nonzero supply under $\bfS$, i.e. $S_{(l,\omega_t)} > 0$.
Then the sequence of optimal dual variables associated with that location under each $\bfS^k$ converge to the optimal dual variable
under $\bfS$, i.e. the following is true:
$$
    \lim_{k\to\infty}\eta_{(l,\omega_t)}^*(\bfS^k) = \eta_{(l,\omega_t)}^*(\bfS).
$$
\end{lemma}
\begin{proof}
Notice that under Assumption \ref{assn:unique_soln}, asserting uniqueness of the primal optimal solutions for the dynamic optimization
    (\ref{eq:dynamic_optimization}), Lemma \ref{lem:stability} implies that the sequence of primal optimal solutions under each $\bfS^k$
    converge to the primal optimal solution under $\bfS$, i.e.
    $$
    \lim_{k\to\infty}\bff^*(\bfS^k) = \bff^*(\bfS).
    $$

We use the complementary slackness conditions (\ref{eq:kkt_conditions}) to get convergence of the dual variables.
First assume that $t=T$ is the final time period, and pick an outgoing route $(r,\omega_t)\in N^+(l,\omega_t)$
for which nonzero flow is allocated under $\bfS$, i.e. $f^*_{(r,\omega_t)}(\bfS) > 0$.
From convergence of the primal solutions it follows that $f^*_{(r,\omega_t)}(\bfS^k) > 0$ for large enough values of
$k$, and for such $k$ the complementary slackness conditions assert the following:
$$
    \eta^*_{(r,\omega_t)}(\bfS^k) = \frac{d}{df} U_{(r,\omega_t)}(f^*_{(r,\omega_t)}(\bfS^k)).
$$

Under the rider-value distribution assupmtions stated in Assumption \ref{assn:rider_value_dist}, Lemma \ref{lem:reward_fn}
states that the reward function derivative $\frac{d}{df} U_{(r,\omega_t)}(\cdot)$ is continuous at all points greater than $0$.
Thus, the sequence of optimal dual variables converges to the following limit:
$$
    \lim_{k\to\infty}\eta^*_{(r,\omega_t)}(\bfS^k) = \lim_{k\to\infty}\frac{d}{df} U_{(r,\omega_t)}(f^*_{(r,\omega_t)}(\bfS^k))
    = \frac{d}{df} U_{(r,\omega_t)}(f^*_{(r,\omega_t)}(\bfS)) = \eta^*_{(r,\omega_t)}(\bfS),
$$
where the final equality holds by complementary slackness for the optimization problem under $\bfS$.
Thus, the claim holds when $t=T$ is the final time period.

A simple backwards induction argument establishes the claim in the case where $t<T$.
In this case, complementary slackness conditions take the form
$$
    \eta^*_{(r,\omega_t)}(\bfS^k) = \frac{d}{df} U_{(r,\omega_t)}(f^*_{(r,\omega_t)}(\bfS^k)) + \bbE\left[\eta^*_{(l^+(r),\omega_{t+1})}(\bfS^k)\right].
$$
However, the dual variable for the future-location $(l^+(r),\omega_{t+1})$ under $\bfS^k$, i.e. the dual variable $\eta^*_{(l^+(r),\omega_{t+1})}(\bfS^k)$,
takes the same value as the dual variable for the optimization problem starting at time $t+1$ under the supply-location vector $\bfS^k_{t+1}$, where $\bfS^k_{t+1}$
gives driver locations at time $t+1$ under scenario $\omega_{t+1}$ assuming that all drivers move according to the primal optimum $\bff^*(\bfS^k)$ in time period
    $t$. \mccomment{Maybe we want to add more detail to this step?  I think it's fine for now but I'll just leave a comment in case we want to revisit.}
A consequence of the fact that primal solutions $\bff^*(\bfS^k)$ converge to $\bff^*(\bfS)$ is that each of the time $t+1$ supply-location vectors
$\bfS^{k}_{t+1}$ converge to the supply-location vector $\bfS_{t+1}$, hence by the backwards induction hypothesis the following limit holds:
$$
    \lim_{k\to\infty}\eta^*_{(l^+(r),\omega_{t+1})}(\bfS^k_{t+1}) = \eta^*_{(l^+(r),\omega_{t+1})}(\bfS_{t+1}),
$$
and so by the same logic as above we conclude
\begin{align*}
    \lim_{k\to\infty}\eta^*_{(r,\omega_t)}(\bfS^k) &=\lim_{k\to\infty} \left(\frac{d}{df} U_{(r,\omega_t)}(f^*_{(r,\omega_t)}(\bfS^k)) + 
                                                                            \bbE\left[\eta^*_{(l^+(r),\omega_{t+1})}(\bfS^k_{t+1})\right]\right)\\
    &= \frac{d}{df}U_{(r,\omega_t)}(f^*_{(r,\omega_t)}(\bfS)) + \bbE\left[\eta^*_{(l^+(r),\omega_{t+1})}(\bfS_{t+1})\right]\\
    &= \eta^*_{(r,\omega_t)}(\bfS).
\end{align*}
\end{proof}

}

\section{General Convex Analysis Properties}
In this section we obtain useful convex analysis properties.  We change notation from the rest of the paper,
and consider the following generic convex optimization problem
\begin{equation}
\label{eq:generic_cvxopt}
\inf\left(f(x) \mid g_i(x)\leq 0, i=1,2,\dots,m, x\in\bbR^n\right),
\end{equation}
where $f,g_1,\dots,g_m$ are convex functions from $\bbR^n$ to $\bbR$.
We assume our convex program (\ref{eq:generic_cvxopt}) satisfies the conditions described in Assumption \ref{assn:generic_cvxopt}.
\begin{assumption}
\label{assn:generic_cvxopt}
Assume the following conditions hold:
\begin{enumerate}
\item $f, g_1,\dots,g_m$ are continuously differentiable at every point $x\in\bbR^n$.
\item The feasible region $\{x \mid g_i(x)\leq 0, i=1,2,\dots,m,x\in\bbR^n\}$ is a bounded compact set.  In particular, there is a
constant $\gamma$ such that the feasible region lies in $\bar{B}(0,\gamma)$, i.e. the closed ball centered at $0$ with radius $\gamma$.
\item The gradients $\nabla f(x), \nabla g_1(x),\dots,\nabla g_m(x)$ have norm smaller than some constant $C$ for all $x\in\bbR^n$.
\end{enumerate}
\end{assumption}

The Lagrangian associated with the optimization problem (\ref{eq:generic_cvxopt}) is the function $L:\bbR^n\times\bbR^m_+\to\bbR\cup\{+\infty\}$
defined by
\begin{equation}
\label{eq:generic_lagrangian}
L(x;\lambda) = f(x) + \lambda^Tg(x),
\end{equation}
where $g(x)$ is the vector in $\bbR^m$ with $g_i(x)$ as its $i$th component.  
Note the min-max theorem states the following relads \mccomment{cite}:
\begin{equation}
\inf_{x\in\bbR^n}\sup_{\lambda\in\bbR^m_+} L(x;\lambda) = \sup_{\lambda\in\bbR^m_+}\inf_{x\in\bbR^n}L(x;\lambda).
\end{equation}

\begin{definition}
\label{def:lagrange_multiplier}
Let $x\in\bbR^n$.  A nonnegative vector $\lambda\in\bbR^m_+$ is said to be a Lagrange multiplier vector for $x$ if it satisfies the
following conditions:
\begin{itemize}
\item Complementary slackness: $\lambda_ig_i(x)=0$ holds for every $i=1,2,\dots,m$.
\item Stationarity: $\nabla_x L(x;\lambda) = 0$.  In other words, $x$ is a global minimizer of $L(\cdot;\lambda)$.
\end{itemize}
We say $\lambda$ is an $\epsilon$-approximate Lagrange multiplier vector for $x$ if it satisfies the following conditions:
\begin{itemize}
\item Approximate complementary slackness: $|\lambda_ig_i(x)|\leq\epsilon$ holds for every $i=1,2,\dots,m$.
\item Approximate stationarity: $\|\nabla_x L(x;\lambda)\|_2 \leq \epsilon$.
\end{itemize}
\end{definition}

For the next Lemma, we consider an optimal solution $x^*$ and a Lagrange multiplier vector $\lambda^*$ for $x^*$.
Part of our proof is concerned with the function mapping $x\in\bbR^n$ to the gradient of the Lagrangian $\nabla_x L(x;\lambda^*)$.
In particular, we care about how large the gradient can vary when evaluated at two points that are close to one another.
To reason about this maximum perturbation effect, for $\delta>0$ define
\begin{equation}
\label{eq:lagrangian_gradient_max_perturbation}
\kappa(\delta) = \max \left(\|\nabla_x L(x;\lambda^*) - \nabla_x L(y;\lambda^*)\|_2 \mid x,y\in\bar{B}(0,\gamma), \|x-y\|_2\leq\delta\right)
\end{equation}
which gives the maximum norm of the difference between the gradients of any two points in $\bar{B}(0,\gamma)$ whose distance from each other is at most $\delta$.
From Assumption \ref{assn:generic_cvxopt} we know the gradient $\nabla_x L(x;\lambda^*)$ is continuous in $x$, and we know that every continuous function
over a compact set is uniformly continuous, so it follows that the gradient function, restricted to the closed ball $\bar{B}(0,\gamma)$, is uniformly continuous.
Therefore, the maximum perturbation $\kappa(\delta)$ goes to $0$ as $\delta\to 0$.

\begin{lemma}
\label{lem:approx_lagrange_optimality}
Let $\bar{x}$ be a feasible solution for (\ref{eq:generic_cvxopt}) and let $\bar{\lambda}\in\bbR^m_+$.  If $\bar{\lambda}$ is a Lagrange
multiplier vector for $\bar{x}$ then $\bar{x}$ is an optimal solution for (\ref{eq:generic_cvxopt}).  If $\bar{\lambda}$ is an $\epsilon$-approximate
Lagrange multiplier vector for $\bar{x}$ then $\bar{x}$ is an $\epsilon'$-optimal solution for (\ref{eq:generic_cvxopt}), where 
$\epsilon'=\epsilon(n+2\gamma)$.
\end{lemma}
\begin{proof}
We focus on the case where $\bar{\lambda}$ is an $\epsilon$-approximate Lagrange multiplier vector for $\bar{x}$.
Let $x^*$ and $\lambda^*$ be optimal primal and dual variables.  Observe the following chain of inequalities:
\begin{equation*}
f(x^*) = L(x^*;\lambda^*) \geq L(x^*;\bar{\lambda}) \geq L(\bar x;\bar \lambda) + \nabla_x L(\bar x; \bar\lambda)^T(x^* - \bar x),
\end{equation*}
where the first bound follows from optimality of $\lambda^*$ and the second line follows by convexity of $L(\cdot;\bar \lambda)$.
Therefore we have the upper bound
$$
f(\bar x) \leq f(x^*) + \| \nabla_x L(\bar x; \bar\lambda)\|_2\|(x^* - \bar x)\|_2 + |\bar\lambda ^T g(\bar x)|.
$$
By approximate complementary slackness we have $|\bar\lambda ^T g(\bar x)|\leq n\epsilon$, by approximate stationarity we have
$\| \nabla_x L(\bar x; \bar\lambda)\|_2\leq \epsilon$, and by feasibility of $x^*$ and $\bar x$ we have $\|(x^* - \bar x)\|_2 \leq 2\gamma$.
Therefore we obtain the bound
$$
f(\bar x) \leq f(x^*) + \epsilon(n + 2\gamma) = f(x^*) + \epsilon'
$$
as claimed.
\end{proof}

\begin{lemma}
\label{lem:approx_lagrange_multiplier}
Let $x^*$ be an optimal solution for (\ref{eq:generic_cvxopt}) and let $\lambda^*$ be a Lagrange multiplier vector for $x^*$.
Let $\bar{x}$ be a feasible solution for (\ref{eq:generic_cvxopt}).  If $\bar{x}$ is an optimal solution then $\lambda^*$ is a Lagrange multiplier
vector for $\bar{x}$.  If $\bar{x}$ is an $\epsilon$-optimal solution, then $\lambda^*$ is an $\epsilon'$-approximate Lagrange multiplier vector, where
$\epsilon'=\sqrt{\sqrt{\epsilon} + C\kappa(C\sqrt{\epsilon})}$.
\end{lemma}
\begin{proof}
 Observe the following chain of inequalities:
\begin{eqnarray}
f(x^*) &= & L(x^*;\lambda^*)\nonumber\\
& = & \min_{x\in \bbR^n} L(x;\lambda^*)\nonumber\\
& \leq & L(\bar x;\lambda^*)\label{eq:stationarity_step} \\
& = & f(\bar x) + \lambda^{*T}g(\bar x)\nonumber\\
& \leq & f(\bar x) \label{eq:csc_step} 
\end{eqnarray}
The first two lines follow from the complementary slackness and stationarity conditions which hold between $x^*$ and $\lambda^*$.  Line
(\ref{eq:csc_step}) follows from feasibility of $\bar{x}$ and nonnegativity of $\lambda^*$.

We first consider the case where $\bar{x}$ is an optimal solution. In this case, we have $f(\bar x)=f(x^*)$ so every line in the above chain
of inequalities holds with equality.  In particular, line (\ref{eq:stationarity_step}) holding with equality shows that $\bar x$ is a global
minimizer of $L(\cdot;\lambda^*)$, so the stationarity condition $\nabla_x L(\bar x;\lambda^*)$ holds, and line (\ref{eq:csc_step}) holding
with equality shows that complementary slackness holds between $\bar x$ and $\lambda^*$. Therefore $\lambda^*$ is a Lagrange multiplier vector
for $\bar x$.

Next, we consider the case wehre $\bar{x}$ is an $\epsilon$-optimal solution.  In this case, we have $f(\bar x) \leq f(x^*) + \epsilon$, so $\epsilon$
is an upper bound on the difference between any two consecutive terms in the above chain of inequalities.  Applying this upper bound to line (\ref{eq:csc_step}) 
we conclude that $|\lambda^{*T}g(\bar x)| \leq \epsilon$.  Since every term $\lambda^*_ig_i(\bar x)$ is nonpositive, it follows that 
$|\lambda^*_ig_i(\bar x)| \leq |\lambda^{*T}g(\bar x)|\leq \epsilon\leq \epsilon'$, so approximate complementary slackness holds between $\lambda^*$ and $\bar x$.

We show in Lemma \ref{lem:approx_stationarity} that $\bar x$ satisfies approximate stationarity with respect to $\lambda^*$.  
We use the result in Lemma \ref{lem:approx_stationarity} by taking $h(x)$ to mean $L(x;\lambda^*)$. The result of Lemma \ref{lem:approx_stationarity}
then provides the bound
$$
\|\nabla_x L(\bar x;\lambda^*)\|_2^2 \leq \sqrt{\epsilon} + C\kappa(C\sqrt{\epsilon}). 
$$
Therefore, $\|\nabla_x L(\bar x;\lambda^*)\|_2\leq \epsilon'$, establishing that $\lambda^*$ is an $\epsilon'$-approximate Lagrange
multiplier vector for $\bar x$.
\end{proof}

\begin{lemma}
\label{lem:approx_stationarity}
Let $h:\bbR^n\to\bbR$ be a continuously differentiable convex function.  Let $x^*$ be a global minimizer of $h$, let $\bar x$ satisfy $h(\bar x) < h(x^*) + \epsilon$,
and assume that $\bar x$ and $x^*$ both lie in $\bar{B}(0,\gamma)$, i.e. the closed ball centered at $0$ with radius $\gamma$. 
For $\delta > 0$, define
$$
\kappa(\delta) = \max \left(\|\nabla h(x) - \nabla h(y)\|_2 \mid x,y\in\bar{B}(0,\gamma), \|x-y\|_2\leq\delta\right)
$$
to be the maximum norm of the difference between the gradients of any two points in $\bar{B}(0,\gamma)$ whose distance from each other is at most $\delta$.
Assume $C$ is a constant upper bound on $\|\nabla h(x)\|_2$.
Then
\begin{equation}
\label{eq:approx_opt_gradient_bound}
\|\nabla h(\bar x)\|_2^2 \leq \sqrt{\epsilon} + C\kappa(C\sqrt{\epsilon}). 
\end{equation}
\end{lemma}
\begin{proof}
Define
$$
x' = \bar x - t\nabla h(\bar x).
$$
for some $t>0$.
From convexity of $h$ we have
$$
h(x') + \langle \nabla h(x'), \bar x - x'\rangle \leq h(\bar x).
$$
Rearranginge the above, and using the definition of $x'$, 
\begin{equation*}
t\langle \nabla h(x'), \nabla h(\bar x)\rangle \leq h(\bar x) - h(x') \leq \epsilon.
\end{equation*}
Now consider the following chain of inequalities:
\begin{eqnarray*}
\|\nabla h(\bar x) \|_2^2 & = & \langle \nabla h(\bar x), \nabla h(\bar x)\rangle \\
& = & \langle \nabla h(x'), \nabla h(\bar x)\rangle + \langle \nabla h(\bar x) - \nabla h(x'), \nabla h(\bar x)\rangle\\
& \leq &\frac{\epsilon}{t} + \langle \nabla h(\bar x) - \nabla h(x'), \nabla h(\bar x)\rangle\\
& \leq &\frac{\epsilon}{t} + \|\nabla h(\bar x) - \nabla h(x')\|_2 \|\nabla h(\bar x)\|_2\\
& \leq &\frac{\epsilon}{t} + \|\nabla h(\bar x) - \nabla h(x')\|_2 C\\
& \leq & \frac{\epsilon}{t} + C\kappa(\|\bar x - x'\|_2)\\
& \leq & \frac{\epsilon}{t} + C\kappa(\|t\nabla h(\bar x)\|_2) \\
& \leq & \frac{\epsilon}{t} + C\kappa(Ct).
\end{eqnarray*}
Setting $t=\sqrt{\epsilon}$, we obtain the claimed bound (\ref{eq:approx_opt_gradient_bound}).
\end{proof}

\Xomit{
\section{Scratch Work on $\beta$-smooth Objective Function}
\subsection{A Parameterization of the State-Dependent Optimization Function with a $\beta$-smooth
Objective Function}
When analyzing incentives in the stochastic two-level model, we will want to leverage the equivalence
between driver utility-to-go values and partial derivatives of the state-dependent optimization function,
to bound the error in incentives that can arise due to small unexpected stochastic perturbations to the market state.
Specifically, we will want to show that the partial derivatives of the state-dependent optimization function
are Lipschitz continuous (in the convex analysis literature this is a property known as $\beta$-smoothness).

$\beta$-smoothness is a stronger property than continuous differentiability. We will show that the 
state-dependent optimization function is $\beta$-smooth by a connection to smoothness of the objective function
used in the state-dependent optimization problem.
However, the objective function in the current parameterization of the state-dependent optimization problem
does not satisfy this smoothness property.
In this section, we will provide an equivalent parameterization of the state-dependent optimization problem with
a $\beta$-smooth objective function.

First, let us illustrate the problem by considering the current parameterization.  Recall the function
$\calU_{\omega_t}(\bff,\bfg)$ gives the time $t$ reward generated by the trips specified by our decision
variables $\bff$ and $\bfg$, which correspond to total trip volumes and dispatch trip volumes, respectively, across each route.
This reward function is defined by
$$
\calU_{\omega_t}(\bff,\bfg) = \sum_{(\ell,d)\in\calL^2}U_{(\ell,d,\omega_t)}(g_{(\ell,d)}) - \sum_{(\ell,d)\in\calL^2}c_{(\ell,d)}f_{(\ell,d)} 
    - \sum_{\ell\in\calL} A(\bfg^T\bfone_\ell,\bff^T\bfone_\ell).
$$
Lemma \ref{lem:reward_fn} shows the derivatice $d/dg U_{(\ell,d,\omega_t)}(g)$ is continuous for all $g\in\bbR$, and Assumption \ref{assn:rider_value_dist}
guarantees that the derivative is Lipschitz continuous.
The linear cost is also fine, because the derivative is a constant.

It is the add-passenger disutility cost function $A(\bfg^T\bfone_\ell,\bff^T\bfone_\ell)$ that violates the $\beta$-smoothness assumption (in fact,
this function is not continuously differentiable in general). Recall this function is equal to
$$
A(\bfg^T\bfone_\ell,\bff^T\bfone_\ell) = \frac{C}{2}\frac{\left(\bfg^T\bfone_\ell\right)^2}{\bff^T\bfone_\ell}.
$$
The partial derivative with respect to $g_{(\ell,d)}$ is equal to
$$
\frac{\partial}{\partial g_{(\ell,d)}}A(\bfg^T\bfone_\ell,\bff^T\bfone_\ell) = C \frac{\bfg^T\bfone_\ell}{\bff^T\bfone_\ell}.
$$
This partial derivative is discontinuous at $(\bff,\bfg) = (0,0)$, even over the feasible space where $0\leq g_{(\ell,d)} \leq f_{(\ell,d)}$ is satisfied.
Indeed, there is a sequence of feasible points where $\bfg^T\bfone_\ell$ is always $0$, and $\bff^T\bfone_\ell$ converges to $0$ from above, so that
the sequence of partial derivatives converges to $0$.  Conversely, taking $\bfg^T\bfone_\ell=\bff^T\bfone_\ell$ as equal and having them both converge
to $0$ from above, produces a sequence of feasible points converging to $(0,0)$ such that the corresponding sequence of partial derivatives converges
to $C>0$.

To overcome this barrier, we introduce a new parameterization of our state-dependent optimization problem.  In this new parameterization, the
decision variables $\bfg$ are made implicit, and are solved for within a new function we incorporate into the objective.
For each $(\ell,\omega_t)$, define
\begin{equation}
\label{eq:reparameterization_reward}
R_{(\ell,\omega_t)}(\bff) = \sup_{0\leq\bfg \leq \bff} \sum_{d\in\calL}U_{(\ell,d,\omega_t)}(g_{(\ell,d)}) - A(\bfg^T\bfone_\ell,\bff^T\bfone_\ell).
\end{equation}
We only define $R_{(\ell,\omega_t)}$ over decision variables $\bff$ where all components are nonnegative.
In the case where $\bff^T\bfone_\ell=0$, we take $A(\bfg^T\bfone_\ell,\bff^T\bfone_\ell)$ to be $0$.
\Xomit{
For the purpose of full generality, we want the function $R_{(\ell,\omega_t)}(\bff)$ to be well-defined even if $\bff$ has negative components.
To this end, notice that we don't include a lower bound of $0$ on the decision variable $\bfg$ in the above optimization problem.  We also
adopt the convention that $A(\bfg^T\bfone_\ell,\bff^T\bfone_\ell)=0$ whenever $\bff^T\bfone_\ell \leq 0$.}
Our objective function is defined as
$$
\calU_{\omega_t}(\bff) = \sum_{(\ell,d)\in\calL^2} c_{(\ell,d,\omega_t)}f_{(\ell,d)} + \sum_{\ell\in\calL} R_{(\ell,\omega_t)}(\bff),
$$
and our reparameterized optimization problem is defined below, for any scenario $\omega_t$ and supply-location vector $\bfS$:
\begin{align}
\sup_{\bff} \ \  & \calU_{\omega_t}(\bff) + \calU_{\omega_t}^{>t}(\bff)\label{eq:reparameterization_opt}\\
\mbox{such that}\ \  & f_{(\ell,d)} \geq 0 \ \forall (\ell,d)\in\calL^2\\
& S_\ell = \sum_{d\in\calL}f_{(\ell,d)} \ \forall \ell\in\calL.
\end{align}
The following Lemma shows that this parameterization leads to the same optimization problem. The proof follows directly from the definition
of both optimization problems.
\begin{lemma}
For any scenario $\omega_t$ and supply-location vector $\bfS=(S_\ell\geq : \ell\in\calL)$, the value of the state-dependent optimization problem 
$\Phi_{\omega_t}(\bfS)$ is equal to the optimal value of the problem (\ref{eq:reparameterization_opt}).
\end{lemma}
\Xomit{
\begin{proof}
The proof follows by taking 
\Xomit{
Because the optimization problem (\ref{eq:reparameterization_opt}) just moves the optimization over decision variables $\bfg$
into the objective function, we just have to check that dropping the lower bound of $0$ on $\bfg$ in the inner optimization
problem (\ref{eq:reparameterization_reward}) does not change the objective function.
For each location $\ell$, define $$R_{(\ell,\omega_t)}(\bff,\bfg) = \sum_{d\in\calL}U_{(\ell,d,\omega_t)}(g_{(\ell,d)}) - A(\bfg^T\bfone_\ell,\bff^T\bfone_\ell)$$
to be the objective function of the inner optimization problem (\ref{eq:reparameterization_reward}).

Let $\bfg$ be an arbitrary decision variable with possibly negative components.  Define $\Delta$ to be the vector with absolute values of the
negative components of $\bfg$, i.e.
$$
\Delta_{(\ell,d)} = \begin{cases} |g_{(\ell,d)}| & \mbox{if }g_{(\ell,d)} < 0,\\
0 &\mbox{else},
\end{cases}
$$
and define $\bar{\bfg} = \bfg + \Delta$
}
\end{proof}
}

\subsection{Gradient of State-Dependent Optimization Function is Lipschitz Continuous}
In this section we show the gradient of the state-dependent optimization function is Lipschitz continuous.
The property we want to establish is known as $\beta$-smoothness in the convex analysis literature
\mccomment{\url{https://sites.math.washington.edu/~ddrusv/crs/Math_516_2020/bookwithindex.pdf}}.

The existing characterizations of $\beta$-smooth functions apply to convex functions $f:\bbR^n\to\bbR$ that are differentiable
on all of $\bbR^n$.  To use these characterizations, we define an extension of the state-dependent optimization function so that its domain is all of
$\bbR^\calL$.

\begin{lemma}
\label{lem:convex_fn_extension}
Let $f:\bbR^n\to\bbR$ be a convex function.  Let $E=\{x\in\bbR^n : f(x)<\infty\}\subseteq\bbR^n$ be the domain of $f$.
Assume that $f$ is continuously differentiable on the interior of $E$.
Define $\bar{f}:\bbR^n\to\bbR$ by
$$
\bar{f}(x) = \sup_{y\in\mathrm{int}(E)} f(y) + (x-y)^T\nabla f(y).
$$
If the gradient of $\bar{f}$ is bounded on the interior of $E$, then $\bar{f}$ is finite and continuosly differentiable on all of $\bbR^n$,
and the gradient of $\bar{f}$ is the same as the gradient of $f$ in the interior of $E$.
\end{lemma}

\begin{lemma}
For any time $t$, any scenario $\omega_t\in\Omega_t$, and any supply-location vector $\bfS\in\mathrm{int}(\bbR^\calL_+)$,
 let $\Lambda_{\omega_t}(\bfS) = -\Phi_{\omega_t}(\bfS)$ be the negative of the state-dependent optimization function.
Define the extension of $\Lambda_{\omega_t}(\cdot)$ to all of $\bbR^\calL$:
$$
\bar{\Lambda}_{\omega_t}(\bfS) = \sup_{\bfS_0\in\mathrm{int}(\bbR^\calL_+)} \Lambda_{\omega_t}(\bfS_0) + (\bfS - \bfS_0)^T\nabla\Lambda_{\omega_t}(\bfS_0).
$$
Then $\bar{\Lambda}_{\omega_t}(\cdot)$ is $\beta$-smooth, in the sense that there exists a constant $\beta>0$ such that
$\|\nabla\bar{\Lambda}_{\omega_t}(\bfS_1) - \nabla\bar{\Lambda}_{\omega_t}(\bfS_2)\|\leq\beta\|\bfS_1-\bfS_2\|$, for all $\bfS_1,\bfS_2\in\bbR^\calL$.
\end{lemma}
\begin{proof}
Note that the function $\Lambda_{\omega_t}(\cdot)$ is convex, has domain $\bbR^\calL_+$, is continuously 
differentiable on the interior of $\bbR^\calL_+$ and the gradient is bounded.
Therefore Lemma \ref{lem:convex_fn_extension} applies and the extension $\bar{\Lambda}_{\omega_t}(\cdot)$ is differentiable on all of $\bbR^\calL$.

We proceed via backwards induction.  Assume that for every time $t+1$ scenario $\omega_{t+1}\in\Omega_{t+1}$, there exists a $\beta$ such that
$\bar{\Lambda}_{\omega_{t+1}}(\cdot)$ is $\beta$-smooth.

The book 
\mccomment{\url{https://sites.math.washington.edu/~ddrusv/crs/Math_516_2020/bookwithindex.pdf}}
confirms that a continuously differentiable function is $\beta$-smooth and convex if and only if
the following condition holds for every $\bfS_1,\bfS_2$:
$$
0\leq\langle \nabla\bar{\Lambda}_{\omega_t}(\bfS_1) - \nabla\bar{\Lambda}_{\omega_t}(\bfS_2),\bfS_1 - \bfS_2\rangle \leq \beta\|\bfS_1 - \bfS_2\|.
$$
The lower bound follows immediately from convexity of $\bar{\Lambda}_{\omega_t}$. To obtain the upper bound, it suffices to work with the original function
$\Lambda_{\omega_t}$ and show that there exists a $\beta>0$ for which the following bound holds for all $\bfS_1,\bfS_2\in\mathrm{int}(\bbR^\calL_+)$.
$$
\langle\nabla{\Lambda}_{\omega_t}(\bfS_1) - \nabla{\Lambda}_{\omega_t}(\bfS_2),\bfS_1 - \bfS_2\rangle \leq \beta\|\bfS_1 - \bfS_2\|^2.
$$
Because the function $\Lambda_{\omega_t}$ is just the negative of $\Phi_{\omega_t}$, this is equivalent to establishing the following:
\begin{equation}
\label{eq:beta_smoothness_condition}
\langle\nabla{\Phi}_{\omega_t}(\bfS_2) - \nabla{\Phi}_{\omega_t}(\bfS_1),\bfS_1 - \bfS_2\rangle \leq \beta\|\bfS_1 - \bfS_2\|^2.
\end{equation}

The optimality conditions (\ref{eq:dynamicopt_stationarity_f}) characterize the partial derivative of $\Phi_{\omega_t}$. For $\bfS\in\mathrm{int}(\bbR^\calL)$, 
let $(\bff,\bfg)\in F_{\omega_t}^*(\bfS)$ be an optimal primal solution, and $(\bm{\alpha},\bm{\beta},\bm{\gamma},\bfeta)\in D^*_{\omega_t}(\bfS)$
be an optimal dual.
Then for any location $\ell$, and any destination $d$ where $f_{(\ell,d)} > 0$:
\begin{align*}
\frac{\partial}{\partial S_\ell}\Phi_{\omega_t}(\bfS) &= \frac{\partial}{\partial f_{(\ell,d)}}\calW_{\omega_t}(\bff,\bfg)+\alpha_{(\ell,d)} + \gamma_{(\ell,d)}\\
&=\frac{\partial}{\partial f_{(\ell,d)}}\calU_{\omega_t}(\bff,\bfg)+\frac{\partial}{\partial f_{(\ell,d)}}\calU^{>t}_{\omega_t}(\bff)+\alpha_{(\ell,d)} + \gamma_{(\ell,d)}\\
&=\frac{\partial}{\partial f_{(\ell,d)}}\calU_{\omega_t}(\bff,\bfg)+
\bbE\left[\frac{\partial}{\partial S_d}\Phi_{\omega_{t+1}}(\bfS_{\omega_{t+1}}(\bff))\right]+\alpha_{(\ell,d)} + \gamma_{(\ell,d)}.
\end{align*}
Therefore, the left side of the bound (\ref{eq:beta_smoothness_condition}) can be rewritten as
\begin{align}
\langle\nabla{\Phi}_{\omega_t}(\bfS_2) - \nabla{\Phi}_{\omega_t}(\bfS_1),\bfS_1 - \bfS_2\rangle =& 
\langle \nabla_{\bff}\calU_{\omega_t}(\bff_2,\bfg_2)   - \nabla_{\bff}\calU_{\omega_t}(\bff_1,\bfg_1)   , \bff_1 - \bff_2
  \rangle\label{eq:first_term}\\
&  +\langle \nabla_{\bff}\calU^{>t}_{\omega_t}(\bff_2)   - \nabla_{\bff}\calU^{>t}_{\omega_t}(\bff_1)   , \bff_1 - \bff_2
  \rangle\label{eq:second_term}\\
&+ \langle\bfalpha_2 + \bfgamma_2 -\bfalpha_1 - \bfgamma_1,\bff_1 - \bff_2\rangle,\label{eq:third_term}
\end{align}
where $(\bff_1,\bfg_1)\in F^*_{\omega_t}(\bfS_1)$, $(\bff_2,\bfg_2)\in F^*_{\omega_t}(\bfS_2)$, 
$(\bm{\alpha}_1,\bm{\beta}_1,\bm{\gamma}_1,\bfeta_1)\in D^*_{\omega_t}(\bfS_1)$, and 
$(\bm{\alpha}_2,\bm{\beta}_2,\bm{\gamma}_2,\bfeta_2)\in D^*_{\omega_t}(\bfS_2)$.

Now, observe that the single-period objective function $\calU_{\omega_t}(\cdot)$ is $\beta$-smooth \mccomment{add this to Lemma somewhere},
so we can bound the first term (\ref{eq:first_term}) as desired:
$$ \langle \nabla_{\bff}\calU_{\omega_t}(\bff_2,\bfg_2)   - \nabla_{\bff}\calU_{\omega_t}(\bff_1,\bfg_1)   , \bff_1 - \bff_2
  \rangle \leq \beta_1\|\bfS_1-\bfS_2\|^2$$ for some $\beta_1>0$.

From our backwards induction assumption, we can bound the middle term (\ref{eq:second_term}) as
$$
\langle \nabla_{\bff}\calU^{>t}_{\omega_t}(\bff_2)   - \nabla_{\bff}\calU^{>t}_{\omega_t}(\bff_1)   , \bff_1 - \bff_2\rangle \leq \beta_2\|\bfS_1-\bfS_2\|^2
$$
for some $\beta_2 > 0$.

\mccomment{Continue from here, using Lipschitz continuity of the gradient of the objective, $\nabla_{\bff}\calU_{\omega_t}(\bff_2,\bfg_2)$.}
To bound the third term (\ref{eq:third_term}), observe that the complementary slackness conditions provide the equality
$$
\langle\bfalpha_2 + \bfgamma_2 -\bfalpha_1 - \bfgamma_1,\bff_1 - \bff_2\rangle = \langle\bfalpha_2 + \bfgamma_2 ,\bff_1\rangle + \langle\bfalpha_1 +\bfgamma_1, \bff_2\rangle.
$$
From the optimality of $(\bff_1,\bfg_1)$ with respect to $(\bm{\alpha}_1,\bm{\beta}_1,\bm{\gamma}_1,\bfeta_1)$ we have the inequality
$$
 L_1(\bff_1,\bfg_1) \leq L_1(\bff_2,\bfg_2),
$$
where $L_1(\bff,\bfg)=L(\bff,\bfg;\bm{\alpha}_1,\bm{\beta}_1,\bm{\gamma}_1,\bfeta_1)$ is the Lagrangian relaxation of the (convex, negated) state-dependent optimization function (\ref{eq:state_dep_opt_cvx}).
We also have
$$
L_1(\bff_2,\bfg_2) + \nabla_{\bff}L_1(\bff_2,\bfg_2)^T(\bff_1-\bff_2) + \nabla_{\bfg}L_1(\bff_2,\bfg_2)^T(\bfg_1-\bfg_2) \leq L_1(\bff_1,\bfg_1),
$$
therefore
\begin{equation}
\label{eq:lagrange_grad_bound}
\nabla_{\bff}L_1(\bff_2,\bfg_2)^T(\bff_1-\bff_2) + \nabla_{\bfg}L_1(\bff_2,\bfg_2)^T(\bfg_1-\bfg_2) \leq 0.
\end{equation}
Note that
$$
\nabla_{\bff}L_1(\bff_2,\bfg_2) = -\nabla_{\bff}\calW_{\omega_t}(\bff_2,\bfg_2) - \bm{\alpha}_1 - \bm{\gamma}_1 + (\eta^1_\ell : (\ell,d)\in\calL^2),
$$
therefore
$$
\nabla_{\bff}L_1(\bff_2,\bfg_2)^T(\bff_1-\bff_2) = \langle\nabla_{\bff}\calW_{\omega_t}(\bff_2,\bfg_2), \bff_2-\bff_1\rangle+\langle\bm{\alpha}_1+\bm{\gamma}_1 ,\bff_2-\bff_1 \rangle
+ \langle\bfeta_1,\bfS_1-\bfS_2\rangle.
$$
Recall that $\frac{\partial}{\partial S_\ell}\Phi_{\omega_t}(\bfS) = \frac{\partial}{\partial f_{(\ell,d)}}\calW_{\omega_t}(\bff,\bfg)+\alpha_{(\ell,d)} + \gamma_{(\ell,d)}$,
so we have
$$
\langle\nabla_{\bff}\calW_{\omega_t}(\bff_2,\bfg_2), \bff_2-\bff_1\rangle = \langle\bm{\alpha}_2 + \bm{\gamma}_2,\bff_1 - \bff_2\rangle + 
\langle \nabla\Phi_{\omega_t}(\bfS_2),\bfS_2-\bfS_1\rangle.
$$

And
$$
\nabla_{\bfg}L_1(\bff_2,\bfg_2) = -\nabla_{\bfg}\calW_{\omega_t}(\bff_2,\bfg_2) - \bm{\beta}_1 + \bm{\gamma}_1.
$$
We also have a symmetric version of (\ref{eq:lagrange_grad_bound}) where the roles of $(\bff_1,\bfg_1)$ and $(\bff_2,\bfg_2)$ are reversed.

\mccomment{earlier attempt below}
From the optimality of $(\bff_1,\bfg_1)$ with respect to $(\bm{\alpha}_1,\bm{\beta}_1,\bm{\gamma}_1,\bfeta_1)$ we have the inequality
$$
L(\bff_1,\bfg_1;\bm{\alpha}_1,\bm{\beta}_1,\bm{\gamma}_1,\bfeta_1) \leq L(\bff_2,\bfg_2;\bm{\alpha}_1,\bm{\beta}_1,\bm{\gamma}_1,\bfeta_1),
$$
where $L(\cdot;\cdot)$ is the Lagrangian relaxation of the (convex, negated) state-dependent optimization function (\ref{eq:state_dep_opt_cvx}).
Expanding this inequality:
\begin{align*}
& \calW_{\omega_t}(\bff_1,\bfg_1) + \bm{\alpha}_1^T\bff_1 + \bm{\beta}_1^T\bfg_1 + \bfgamma_1^T(\bff_1 - \bfg_1) + \sum_\ell \eta^1_\ell\left(S_\ell^1 - \bff_1^T\bfone_\ell\right)\\
&\geq \calW_{\omega_t}(\bff_2,\bfg_2) + \bm{\alpha}_1^T\bff_2 + \bm{\beta}_1^T\bfg_2 + \bfgamma_1^T(\bff_2 - \bfg_2) + \sum_\ell \eta^1_\ell\left(S_\ell^1 - \bff_2^T\bfone_\ell\right).
\end{align*}
We also have a symmetric inequality where the roles of $(\bff_1,\bfg_1)$ and $(\bff_2,\bfg_2)$ are reversed.
Adding both inequalities produces the following:
\begin{align*}
0 \geq & \bm{\alpha}_1^T\bff_2 + \bm{\beta}_1^T\bfg_2 + \bfgamma_1^T(\bff_2 - \bfg_2) \\
& + \bm{\alpha}_2^T\bff_1 + \bm{\beta}_2^T\bfg_1 + \bfgamma_2^T(\bff_1 - \bfg_1)\\
&+ \sum_\ell \eta^1_\ell\left(S_\ell^1 - \bff_2^T\bfone_\ell\right) \\
 &+ \sum_\ell \eta^2_\ell\left(S_\ell^2 - \bff_1^T\bfone_\ell\right)
\end{align*}
where several terms disappear thanks to complementary slackness.
\end{proof}

}

\end{document}